\documentclass[12pt,a4paper]{report}

\usepackage[utf8]{inputenc} 
\usepackage[T1]{fontenc} 
\usepackage[Bjornstrup]{fncychap}
\usepackage{amsmath,amssymb}
\usepackage[english]{babel}
\usepackage{color}
\usepackage{graphicx}        
\usepackage{xcolor}
\usepackage{geometry}
\usepackage{subcaption}
\usepackage{apptools}
\AtAppendix{\counterwithin{lem}{section}}
\AtAppendix{\counterwithin{defn}{section}}
\AtAppendix{\counterwithin{prop}{section}}
\AtAppendix{\counterwithin{thm}{section}}
\usepackage{hyperref}
\usepackage{authblk}
\usepackage{amsthm}
\usepackage{mathptmx}
\usepackage{newtxtext,newtxmath}
\usepackage[titletoc,page]{appendix}
\usepackage{enumitem}
\usepackage{leftindex}
\usepackage{pifont}
\newtheorem{thm}{Theorem}[chapter]

\newtheorem*{rmk*}{Remark}
\newtheorem*{rmks*}{Remarks}
\newtheorem{lem}{Lemma}[chapter]
\newtheorem{defn}{Definition}[chapter]
\newtheorem*{defn*}{Definition}
\newtheorem{defns}{Definitions}
\newtheorem*{defns*}{Definitions}
\newtheorem{prop}{Proposition}[chapter]
\newtheorem*{prop*}{Proposition}
\newtheorem*{cor*}{Corollary}
\newtheorem{ans}{Ansatz}
\newtheorem*{ques}{Question}
\newtheorem*{ex}{Example}
\newtheorem*{exs}{Examples}
\newtheorem*{ill*}{Illustration}

\begin{document}
	\begin{titlepage} 
		
		\centering 
		
		\scshape 
		
		\vspace*{\baselineskip} 
		
		
		\rule{\textwidth}{1.6pt}\vspace*{-\baselineskip}\vspace*{2pt} 
		\rule{\textwidth}{0.4pt} 
		
		\vspace{0.75\baselineskip} 
		
		{\LARGE Topos and Stacks \\ of\\ Deep Neural Networks\\} 
		
		\vspace{0.75\baselineskip} 
		
		\rule{\textwidth}{0.4pt}\vspace*{-\baselineskip}\vspace{3.2pt} 
		\rule{\textwidth}{1.6pt} 
		
		\vspace{1\baselineskip} 
		
		
		
		\vspace*{3\baselineskip} 
		
		
		
		\vspace{0.5\baselineskip} 
		
		{\scshape\Large Jean-Claude Belfiore \\} 
		
		\vspace{0.5\baselineskip} 
		
		\textit{Huawei Advanced Wireless Technology Lab. \\Paris Research Center} 
		
		\vfill 
		
		{\scshape\Large Daniel Bennequin \\} 
		
		\vspace{0.5\baselineskip} 
		
		\textit{Huawei Advanced Wireless Technology Lab. \\Paris Research Center\\} 
		\textit{University of Paris Diderot, Faculty of Mathematics} %
		
		\vfill 
		
		
		
		\vspace{0.3\baselineskip} 
		
		
		
	\end{titlepage}
%
%
%
	\date{}
	

\chapter*{Abstract}
	Every known artificial Deep Neural Network (DNN) corresponds to an object in a canonical Grothendieck’s topos;
	its learning dynamic corresponds to a flow of morphisms in this topos. Invariance structures in the layers (like CNNs or LSTMs)
	correspond to Giraud’s stacks. This invariance is supposed to be responsible of the generalization property,
	that is extrapolation from learning data under constraints. The fibers represent pre-semantic categories (Culioli \cite{Culioli}, Thom \cite{Thom1972}),
	over which artificial languages are defined, with internal logics, intuitionist, classical or linear (Girard \cite{GIRARD19871}).
	Semantic functioning of a network is its ability to express theories in such a language for answering questions in output
	about input data. Quantities and spaces of semantic information are defined by analogy with the homological interpretation
	of Shannon’s entropy (Baudot \& Bennequin \cite{Baudot-Bennequin}).  They generalize the measures found by Carnap and Bar-Hillel \cite{CBH52}.
	Amazingly, the above semantical structures are classified by geometric fibrant objects in a closed model category of Quillen \cite{quillen1967homotopical},
	then they give rise to homotopical invariants of DNNs and of their semantic functioning. Intentional type theories (Martin-Löf \cite{ML})
	organize these objects and fibrations between them. Information contents and exchanges are analyzed by Grothendieck’s derivators \cite{grothendieck1990derivateurs}.

\tableofcontents

\chapter*{Preface}
\section*{Introduction}
This text presents
a general theory of semantic functioning of deep neural networks, DNNs, based on topology, more precisely, Grothendieck's topos,
Quillen's homotopy theory, Thom's singularity theory and the pre-semantic of Culioli in enunciative linguistic.\\

The theory is based on the existing networks, transforming data, as images, movies or written texts, to answer questions, achieve actions or take decisions. Experiments, recent and past, show that the deep neural networks, which have learned
under constrained methods, can achieve surprising semantic performances \cite{xie}, \cite{bao-basu}, \cite{basu-bao},
\cite{logic-DNN}, \cite{DBLP:journals/corr/abs-2012-08508}, \cite{DBLP:journals/corr/KarpathyF14}, \cite{MaoYuille2015}, \cite{DBLP:journals/corr/abs-1806-01830},
\cite{zhuo-2019}, \cite{goyal2020recurrent}. However, the exploitation of more explicit invariance
structures and adapted languages, are in great part a task for the future. Thus the present text is a mixture of an analysis of the functioning
networks, and of a conjectural frame to make them able to approach more ideal semantic functioning.\\

Note that categories, homology and homotopy were recently applied in several manners to semantic information. An example is
the application of category theory to the design of networks,
by Fong and Spivak \cite{fong2018seven}. For a recent review on many applications of category theory to Machine Learning, see \cite{shiebler2021category}.
Other examples are given by the general notion of Information Networks based on Segal spaces by Yuri Manin and Matilde Marcolli, \cite{Manin-Marcolli} and the \v{C}ech homology reconstruction of the environment by place fields
of Curto and collaborators, \cite{curto}. Let us also mention the characterization of entropy, by Baez, Fritz, Leinster, \cite{baez2011characterization},
and the use of sheaves and cosheaves for studying information networks, Ghrist, Hiraoka 2011 \cite{Ghrist2011ApplicationsOS}, Curry 2013
\cite{curry2013sheaves}, Robinson and Joslyn \cite{robinson2017sheaves},
and Abramsky et al. specially for Quantum Information \cite{abramsky2011sheaf}.
Persistent homology for detecting structures in data must also be cited in this context, for instance Port, Karidi, Marcolli 2019, \cite{port2019topological}
on syntactic structures, and Carlsson et al. on shape recognition \cite{carlsson2005persistence}.
More in relation with Bayes networks, there are the three recent PhD theses of Juan-Pablo Vigneaux \cite{Vigneaux}, Olivier Peltre \cite{peltre} and Gr\'egoire Sergeant-Perthuis
\cite{gregoire-these}.\\

With respect to these works, we look at a notion of information which is a (toposic) topological invariant of the situation which involves three dimensions of dynamics:
\begin{enumerate}[label=\arabic*)]
	\item a logical flow along the network;
	\item in the layers, the action of categories;
	\item the evocations of meaning in languages.
\end{enumerate}

The resulting notion of information generalizes the suggestion of Carnap and Bar-Hillel $1952$ in these three dynamical directions. Our inspiration came from the toposic interpretation of Shannon's
entropy in \cite{Baudot-Bennequin} and \cite{vigneaux-TAC}. A new fundamental ingredient is the interpretation of internal implication (exponential) as
a \emph{conditioning} on theories, analogous to the conditioning in probabilities.
We distinguish between the theoretically accessible information, concerning all the theories in a fibred languages, and the practically accessible information, that corresponds to the
semantic functioning of concrete neural networks, associated to a feed-forward dynamics which depends on a learning process.\\
\indent The main results in this text are,
\begin{itemize}[label=\ding{52}]
	\item theorems \ref{thm:backprop} and \ref{thm:dnn} characterizing the topos associated to a DNN
	\item theorem \ref{thm:semanticflow} giving a geometric sufficient condition for a fluid circulation of semantics in this topos
	\item theorems \ref{thm:fibrations-dnn} and \ref{thm:M-L}, characterizing the  fibrations (in particular the fibrant objects) in a closed model category made by
	the stacks of the DNNs having a given network architecture
	\item the tentative definition of Semantic Information quantities and spaces in sections \ref{sec:sem-info} and \ref{sec:homotopy}
	\item theorem \ref{thm:activities} on the generic structures and dynamics of LSTMs.
\end{itemize}

\indent Specific examples, showing the nature of the semantic information that we present here, are at the end of section \ref{sec:homotopy} extracted from the exemplar toy language of Carnap and Bar-Hillel and the mathematical interpretation of the pre-semantic of Culioli in relation with the artificial memory cells of sections \ref{sec:memories} and \ref{sec:presemantics}.\\

\indent {\bf Chapter \ref{chap:architecture}} describes the nature of the sites and the topos associated to deep neural networks, said $DNNs$, with their dynamics, feedforward and backward
(backpropagation) learning.\\
\indent {\bf Chapter \ref{chap:stacks}} presents the different stacks of a $DNN$, which are fibred categories over the site of the $DNN$, incorporating symmetries and logics for approaching the wanted
semantics in functioning. Usual examples are $CNNs$ for translation symmetries, but also other ones regarding logic and semantics
(see experiments in \emph{Logical Information Cells I} \cite{logic-DNN}).
Thus the logical structure of the classifying topos of such a stack is described. We introduce hypotheses on the stack and the language objects that allow a transmission of theories downstream and of propositions upstream in the network.
The $2$-category of the stacks over a given architecture is shown to constitute  a closed model theory of injective type, in the sense of Quillen (also Cisinski and Lurie).
The fibrant objects, which are difficult to
characterize in general, are determined in the case of the Grothendieck sites of $DNNs$. Interestingly, they correspond to the hypothesis
guarantying the transmission of theories. Using the work of Arndt and Kapulkin \cite{AK} we show that the above model theory gives
rise to a Martin-Löf type theory associated to every $DNN$. Semantics in the sense of topos (Lambek) is added by considering objects in the classifying topos
of the stack.\\
\indent In {\bf chapter \ref{chap:dynamics}}, we start exploring the notion of semantic information and semantic functioning in $DNNs$, by using homology and homotopy theory.
 Then we define semantic conditioning of the theories
by the propositions,
and compute the corresponding ringed cohomology of the functions of these theories; this gives a numerical notion of semantic ambiguity, of semantic mutual
information and of semantic Kullback-Leibler divergence. Then we generalize the homogeneous bar-complex to define a bi-simplicial set
$I^{\bullet}_\star $ of classes of theories and propositions
histories over the network,  by taking homotopy colimits. We introduce a class of increasing and concave functions from $I^{\bullet}_\star $ to an external model category $\mathcal{M}$;
and with them, we obtain natural homotopy types of semantic information, associated to coherent semantic functioning of a network with respect to
a semantic problem; they satisfy properties conjectured by Carnap and Bar-Hillel in 1952 \cite{CBH52} for the sets of semantic information. On the simple example
they studied we show the interest of considering spaces of information, in particular groupoids, in addition to the more usual combinatorial dimension of
logical content of propositions.\\
\indent {\bf Chapter \ref{chap:unfolding}} describes examples of memory cells, as the long and short terms memory cells (LSTM), and shows that the natural groupoids for their stack have as fundamental group the group of Artin's
braids with three strands $\mathfrak{B}_3$. Generalizations are proposed, for semantics closer to the semantic of natural languages, in appendix \ref{app:nat-lang}.\\
\indent Finally {\bf chapter \ref{chap:3-category}} introduces possible applications of topos, stacks and models to the relations between several $DNNs$: understanding
the modular structures of networks, defining and studying the
obstructions to integrate some semantics or to solve problems in some contexts. Examples could be taken from the above mentioned experiments
on logical information cells, and from recent attempts of several teams in artificial intelligence: Hudson \& Manning \cite{HM18}, Santoro, Raposo et al. \cite{San-Rap}, Bengio and Hinton,
using memory modules, linguistic analysis modules, attention modules and relation modules, in addition to convolution $CNNs$, for answering
questions about images and movies (also see \cite{DBLP:journals/corr/RaposoSBPLB17}, \cite{zhuo-2019}, \cite{hjelm2020representation}).\\

\indent Most of the figures mentioned in the text can be found in the chapter by Bennequin and Belfiore \emph{On new mathematical concepts for Artificial Intelligence},
in the Huawei volume on \emph{Mathematics for Future Computing and Communication}, edited by Liao Heng and Bill McColl \cite{heng_mccoll_2021}. We also refer to this chapter for the elements of category theory that are necessary to understand this text,
the  definitions and
first properties of topos and Grothendieck topos, and the presentation of elementary type theories.\\
\indent Chapter $9$ in \cite{heng_mccoll_2021}, by Ge Yiqun and Tong Wen, \emph{Mathematics, Information and Learning}, explains the large place of topology in the notions of
semantic information.\\

In a forthcoming preprint, entitled \emph{A search of semantic spaces}, we will compute spaces of semantic information for several elementary languages,
along the lines indicated in section \ref{sec:homotopy}, and develop further the Galois point of view on the information flow in a network.
The notions of intentional signification, meaning
and knowledge are discussed from a philosophical point of view, and adapted to artificial semantic and its intelligibility.\\

In another following preprint, \emph{A mathematical theory of semantic communication}, we plan to present the application
of the above stacks of functioning DNNs and their information spaces, to the problem of semantic communication. In particular we show how the invariance
structures in the fibers, made by categories acting on artificial languages, give a way to understand generalization properties of DNNs, for extrapolation, not only interpolation.\\
\indent Analytical aspects, as equivariant standard DNNs approximation of functions, or gradient descent respecting the invariance,
are developed in this context.

\section*{Acknowledgements}

The two authors deeply thank Olivia Caramello and Laurent Lafforgue for the impulsion they gave to this research, for their constant
encouragements and many helpful suggestions. They also warmly thank Merouane Debbah for his deep interest, the help and the support he gave,
Xavier Giraud, for the concrete experiments he realized with us, allowing to connect the theory with the lively spontaneous behavior
of artificial neural networks, and Zhenrong Liu (Louise) for her constant and very kind help at work.\\
\indent D.B. gives special thanks to Alain Berthoz, with whom he has had the chance to work and dream since many years on a conjectural topos
geometry (properly speaking stacks) for the generation and control of the variety of humans voluntary movements.
He also does not forget that the presence of natural invariants of topos in Information theory was discovered during a common work with Pierre Baudot,
that he heartily thanks. D.B. had many inspiring discussions on closely related subjects with his former students, in particular
Alireza Bahraini, Alexandre Afgoustidis, Juan-Pablo Vigneaux, Olivier Peltre and Grégoire Sergeant-Perthuis, that he friendly thanks, with gratitude.

\chapter{Architectures}\label{chap:architecture}

\noindent Let us show how every (known) artificial deep neural network ($DNN$) can be described by a family of objects in a well defined topos.\\

\section{Underlying graph}
\begin{defn*}
	An oriented graph $\Gamma$ is \emph{directed} when the relation $a\leq b$ between vertices, defined by the existence of an oriented path, made by concatenation of oriented edges, is a partial ordering on the set $V(\Gamma)=\Gamma_{(0)}$ of vertices.  A graph is said \emph{classical} if there exists at most one edge between two vertices, and no loop at one vertex (also named tadpole). A classical directed graph can have non-oriented cycles, but no oriented cycles.
\end{defn*}

\indent The layers and the direct connections between layers in an artificial neural network constitute a finite oriented graph $\Gamma$, which is directed, and classical.\\
\indent The minimal elements correspond to the initial layers, or input layers, and the maximal elements to the
final layers, or output layers, all the other correspond to hidden layers, or inner layers. In the case of $RNNs$ (as when we look at feedback connections in the brain)
we apparently see loops, however they are not loops in space-time, the graph which represents the functioning of the network must be seen in the space-time (not necessary Galilean but causal), then the loops disappear and the graph appears directed and classical (see figure \ref{fig:rnn}). Apparently there is no exception to these rules in the world of $DNNs$.\\
\vspace{8mm}
\noindent
\begin{figure}[ht]\hspace{-17mm}
	\begin{subfigure}{.5\textwidth}
		\centering
		\includegraphics[height=4.5cm]{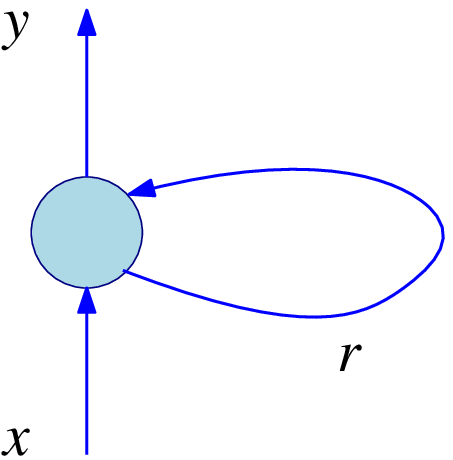}\vspace{2mm}
		\caption{Original RNN}
		\label{rnn}
	\end{subfigure}
	\begin{subfigure}{.5\textwidth}
		\centering
		\includegraphics[height=4.5cm]{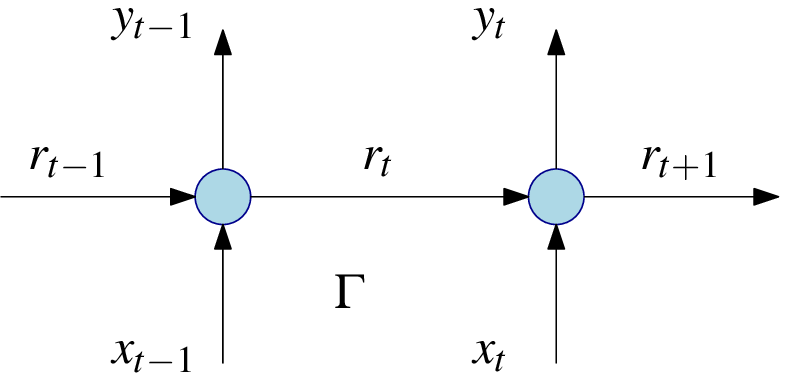}\vspace{2mm}
		\caption{Unfolded RNN in Space-Time}
		\label{rnn-unfolded}
	\end{subfigure}
	\caption{RNN with space-time unfolding }
	\label{fig:rnn}
\end{figure}

\begin{rmk*}
	\normalfont Bayesian networks are frequently associated to oriented or non-oriented graphs, which can be non-directed,
	and have oriented loops. However, the underlying random variables are associated to vertices and to edges, the variable of an edge
	$ab$ being the joint variable of the variables of $a$ and $b$. More generally, an hypergaph is considered, made by a subset $\mathcal{A}$
	of the set $\mathcal{P}(I)$ of subsets of a given set $I$. In this situation, we have a poset, where the natural partial ordering relation is the
	opposite of the inclusion, i.e. it goes from the finer variable to the coarser one.
\end{rmk*}

\section{Dynamical objects of the chains}\label{sec:dynamic-chain}

\indent The simplest architecture of a network is a chain, and the feed-forward functioning of the network, when it has learned, corresponds to a
covariant functor $X$ from  the category $\mathcal{C}^{o}(\Gamma)$ freely generated by the graph to the category of sets, ${\sf Set}$: to a layer $L_k;k\in  \Gamma$ is associated the set $X_k$ of possible activities of the population of neurons in $L_k$, to the edge $L_k\mapsto L_{k+1}$ is associated the map
$X^{w}_{k+1,k}:X_k\rightarrow X_{k+1}$ which corresponds to the learned weights $w_{k+1,k}$; then to each arrow in $\mathcal{C}^{o}(\Gamma)$, we associate the composed map.\\
\indent But also the weights can be encoded in a covariant functor $\Pi$ from $\mathcal{C}^{o}(\Gamma)$ to ${\sf Set}$: for $L_k$ we define $\Pi_k$ as the product of all the sets $W_{l+1,l}$ of weights for $l\geq k$, and to the edge $k\mapsto k+1$ we associate the natural forgetting projection $\Pi_{k+1,k}:\Pi_k\rightarrow \Pi_{k+1}$. (The product over an empty set is the singleton $\star $ in ${\sf Set}$, then for the output layer $L_n$ the last projection is the unique possible map from $\Pi_{n-1}$ to $\star $.) In what follows, we will note $\mathbb{W}=\Pi$, for remembering that it describes the functor of weights, but the notation $\Pi$ is less confusing for denoting the morphisms in this functor.\\
\indent The cartesian products $X_k\times \Pi_k$ together with the maps
\begin{equation}
X_{k+1,k}\times \Pi_{k+1,k}\left(x_k,(w_{k+1,k},w'_{k+1})\right)=\left(X^{w}_{k+1,k}(x_k),w'_{k+1}\right)
\end{equation}
also defines a covariant functor $\mathbb{X}$; it represents all the possible feed-forward functioning of the network, for every potential weights. The natural projection from $\mathbb{X}$ to $\mathbb{W}=\Pi$ is a natural transformation of functors. It is remarkable that, in supervised learning, the Backpropagation algorithm is represented by a flow of natural transformations of the functor $\mathbb{W}$ to itself. We give a proof below in the general case, not only for a chain, where it is easier.\\
\noindent Remark a difference with Spivak et al. \cite{FST}, where backpropagation 
is a functor, not a natural transformation.\\
\indent In fact, the weights represent mappings between two layers, individually they correspond to morphisms in a functor $X^{w}$, then it should
have been more intuitive if they had been coded by morphisms, however globally they are better encoded by the objects in the
functor $\mathbb{W}$, and the morphisms in this functor are the erasure of the weights along the arrows that correspond to them.
This appears as a kind of dual representation of the mappings $X^{w}$.\\

\indent As we want to respect the convention of Topos theory,  \cite{SGA4}, we introduce the category $\mathcal{C} = \mathcal{C}(\Gamma)$ which is opposed to $\mathcal{C}^{0}(\Gamma)$; then $X^{w}$, $\mathbb{W}=\Pi$ and
$\mathbb{X}$ become contravariant functors from this category $\mathcal{C}$ to ${\sf Sets}$, i.e. presheaves over $\mathcal{C}$, i.e. objects in the topos $\mathcal{C}^{\wedge}$ \cite{heng_mccoll_2021}.
This is this topos which is associated to the neural network which has the shape of a chain (multi-layer perceptron). Observe that the arrows between sets continue to follow the natural dynamical ordering, from the initial layer to the final layer, but the arrows in the category (the site) $\mathcal{C}$ are going now in the opposite direction.\\
\indent The object $X^{w}$ can be naturally identified with a subobject of $\mathbb{X}$,
we call this singleton the fiber of $pr_2: \mathbb{X}\rightarrow \mathbb{W}$ over the singleton $w$ in $\mathbb{W}$,
(that is a morphism in $\mathcal{C}^{\wedge}$ from the final object $\textbf{1}$ (the constant functor equal to the point $\star $ at each layer)
to the object $\mathbb{W}$),
which is a system of weights for each edge of the graph $\Gamma$.\\

\indent In this simple case of a chain, the classifying object of subobjects $\Omega$, which is responsible of the logic in the topos \cite{proute-logique},
is given by the subobjects of $\textbf{1}$; more precisely, for every $k\in \mathcal{C}$, $\Omega(k)$
is the set of subobjects of the localization $\textbf{1}|k$, made by the arrows in $\mathcal{C}$ going to $k$.
All these subobjects are increasing sequences $(\emptyset,...,\emptyset,\star ,...,\star )$.
This can be interpreted as the fact that a proposition in the language (and internal 
semantic theory) of the topos is more and more determined when we approach the last layer.
Which corresponds well to what happens in the internal world of the network, and also, in most cases,
to the information about the output that an external observer can deduce from the activity in the inner layers \cite{logic-DNN}. 

\section{Dynamical objects of the general DNNs}\label{sec:dynamic-general}

\indent However, many networks, and most today's networks, are far from being simple chains. The topology of $\Gamma$ is very complex,
with many paths going from a layer to
a deeper one, and many inputs and outputs at a same vertex. In these cases, the functioning and the weights are not defined by functors
on $\mathcal{C}(\Gamma)$ (the category opposite to the category freely generated by $\Gamma$).
But a canonical modification of this category allows to solve the problem: at each layer $a$
where more than one layer sends information, say $a',a",...$, i.e. where there exist irreducible arrows $aa',aa",...$ in $\mathcal{C}(\Gamma)$ (edges in $\Gamma^{\rm op}$),
we perform a surgery: between $a$ and $a'$ ({\em resp.} $a$ and $a"$, a.s.o.)  introduce two new objects $A^{\star }$ and $A$, with arrows
$a'\rightarrow A^{\star }$, $a"\rightarrow A^{\star }$, ..., and $A^{\star }\rightarrow A$, $a\rightarrow A$, forming a fork,
with tips in $a',a",...$ and handle $A^{\star }Aa$ (more precisely if not too pedantically, the arrows $a'A^{\star }, a"A^{\star },...$ are the tines,
the arrow $A^{\star }A$ is the tang, or socket, and the arrow $aA$ is the handle) (see figure \ref{fig:fork}).
By reversing arrows, this gives a new oriented graph $\boldsymbol{\Gamma}$, also without oriented cycles,
and the category $\mathcal{C}$ which replaces $\mathcal{C}(\Gamma)$ is the category $\mathcal{C}(\boldsymbol{\Gamma})$,
opposite of the category which is freely generated by $\boldsymbol{\Gamma}$.

\begin{figure}[h]
	\begin{center}
		\includegraphics[width=9cm]{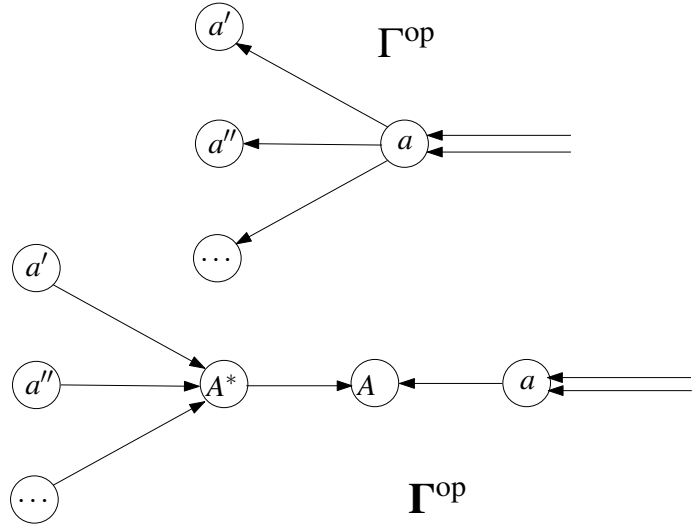}
		\caption{\label{fig:fork}From the initial graph to the Fork}
	\end{center}
\end{figure}

\begin{rmk*}
	\normalfont In $\boldsymbol{\Gamma}$, the complement of the unions of the tangs is a forest.
	Only the convergent multiplicity in $\Gamma$ gives rise to forks, not the divergent one.
	In the category $\mathcal{C}$, this convergence ({\em resp.} divergence) corresponds to a divergence ({\em resp.} convergence) of the arrows.
\end{rmk*}

\noindent When describing concrete networks (see for instance $RNN$, and $LSTM$ or $GRU$ memory cells that we will study in chapter \ref{chap:unfolding}),
ambiguity can appear with the input layers: they can
be considered as input or as tips when several inputs join for connecting a deeper layer $a$. The better attitude is to duplicate them; for instance two input layers
$x_t, h_{t-1}$ going to $h_t, y_t$, we introduce $X_t, x'_t$, $H_{t-1}, h'_{t-1}$, then a fork $A^{\star }, A$, and in $\mathcal{C}$, arrows $x'_t\rightarrow X_t$, $h'_{t-1}\rightarrow H_{t-1}$ for representing the input data, arrows of fork $x'_t\rightarrow A^{\star }$, $h'_{t-1}\rightarrow A^{\star }$, $A^{\star }\rightarrow A$, and arrows of information transmissions $h_t\rightarrow A$ and $y_t\rightarrow A$, representing the output of the memory cell.\\

\noindent With this category $\mathcal{C}$, it is possible to define the analog of the presheaves $X^{w}$, $\mathbb{W}=\Pi$ and
$\mathbb{X}$ in general.\\

\noindent First $X^{w}$: at each old vertex, the set $X^{w}_a$ is as before the set of activities of the neurons of the corresponding layer; over a point
like  $A^{\star }$ and $A$ we put the product of all the incoming sets $X^{w}_{a'}\times X^{w}_{a"},...$. The map from $X_A$ to $X_a$ is the dynamical transmission in the network, joining the information coming from all the inputs layers $a',a",...$ at $a$, all the other maps are given by the structure: the projection on its factors from $X^{w}_{A^{\star }}$, and the identity over the arrow $A^{\star }A$. It is easy to show, that given a collection of activities $\varepsilon^{0}_{in}$ in all the initial layers of the network, it results a unique section of the presheaf $X^{w}$, a singleton, or an element of $\lim_\mathcal{C}X^{w}$, which induces $\varepsilon^{0}_{\rm in}$. Thus, dynamically, each arrow of type $a\rightarrow A$ has replaced the set of arrows from $a$ to $a',a",...$. \\

\indent It is remarkable that the main structural part (which is the projection from a product to its components) can be interpreted by the fact that the
presheaf is a sheaf for a natural Grothendieck topology $J$ on the category $\mathcal{C}$: in every object $x$ of $\mathcal{C}$ the only covering is the full category $C|x$, except when $x$ is of the type of $A^{\star }$,
where we add the covering made by the arrows of the type $a'\rightarrow A^{\star }$ \cite{SGA4}.\\

\noindent The sheafification process, associating a sheaf $X^{\star }$ over $(\mathcal{C},J)$  to any presheaf $X$ over $\mathcal{C}$ is easy to describe:
no value is changed except at a place $A^{\star }$, where $X_{A^{\star }}$ is replaced by the product $X^{\star }_{A^{\star }}$  of the $X_{a'}$, and the map
from $X^{\star }_A=X_A$ to $X^{\star }_{A^{\star }}$ is replaced by the product of the maps from $X_A$ to the $X_{a'}$ given by the functor $X$. In particular,
important for us, the sheaf $C^{\star }$ associated to a constant presheaf $C$ replaces $C$ in $A^{\star }$ by a product $C^{n}$ and
the identity $C\rightarrow C$ by the diagonal map $C\rightarrow C^{n}$ over the arrow $A^{\star }A$.\\

\noindent Let us now describe the sheaf $\mathbb{W}$ over $(\mathcal{C},J)$ which represents the set of possible weights of the $DNN$ (or $RNN$ a.s.o.). First consider at each vertex $a$ of the initial graph $\Gamma$, the set $W_a$ of weights describing the allowed maps from the product $X_{A}=\prod_{a'\leftarrow a}X_{a'}$ to $X_a$, over the projecting layers $a',a",...$ to $a$. Then consider at each layer $x$ the (necessarily connected) subgraph $\Gamma_x$  (or $x|\Gamma$) which is the union of the connected oriented paths in $\Gamma$ from $x$ to some output layer (i.e. the maximal branches issued from $x$ in $\Gamma$); take for $\mathbb{W}(x)$ the product of the $W_y$ over all the vertices in $\Gamma_x$. (For the functioning, it is useful to consider the part $\boldsymbol{\Gamma}_x$ (or $x|\boldsymbol{\Gamma}$) which is formed from $\Gamma_x$, by adding the collections of points $A^{\star },A$ when necessary, and the arrows containing them in $\boldsymbol{\Gamma}$.) At every vertex of type $A^{\star }$ or $A$ of $\boldsymbol{\Gamma}$, we put the product $\mathbb{W}_A$ of the sets $\mathbb{W}_{a'}$ for the afferent $a',a",...$ to $a$. If $x'x$ is an oriented edge of $\boldsymbol{\Gamma}$, there exists a natural projection $\Pi_{xx'}:\mathbb{W}(x')\rightarrow \mathbb{W}(x)$. This defines a sheaf over  $\mathcal{C}=\mathcal{C}(\boldsymbol{\Gamma})$.\\
\indent The crossed product $\mathbb{X}$ of the $X^{w}$ over $\mathbb{W}$ is defined as for the simple chains. It is an object of the topos
of sheaves over $\mathcal{C}$ that represents all the possible functioning of the neural network.\\

\section{Backpropagation as a natural (stochastic) flow in the topos}

\noindent Nothing is loosed in generality if we put together the inputs ({\em resp.} the output) in a product space $X_0$ ({\em resp.} $X_n$); this corresponds to the
introduction of an initial vertex $x_0$ and a  final vertex $x_n$ in $\Gamma$, respectively connected to all the existing initial or final vertices.\\
\indent We also assume that the spaces of states of activity $X_a$ and the spaces of weights $W_{aA}$ are smooth manifolds, and that the maps $(x,w)\mapsto X^{w}(x)$
defines smooth maps on the corresponding product manifolds.\\
In particular it is possible to define tangent objects in the topos of the network $T\left(\mathbb{X}\right)$ and $T\left(\mathbb{W}\right)$, and smooth natural
transformations between them.\\

Supervised learning consists in the choice of an energy function
\begin{equation}
(\xi_0,w)\mapsto F(\xi_0;\xi_{n}(w,\xi_0));
\end{equation}
then in the search of the absolute minimum of the mean $\Phi=\mathbb{E}(F)$ of this energy over a measure on the inputs $\xi_0$; it is a real function on the whole set of weights $W=\mathbb{W}_0$. For simplicity, we assume that $F$ is smooth, and we do not enter the difficult point of effective numerical gradient descent algorithms, we just want to develop
the formula of the linear form $dF$ on $T_{w_0}W$, for a fixed input $\xi_0$ and a fixed system of weights $w_0$. The gradient will depend on the choices of a Riemannian metric on $W$. And the gradient of $\Phi$ is the mean of the individual gradients. \\
We have
\begin{equation}
dF(\delta w)=F^{\star }d\xi_n(\delta w),
\end{equation}
then it is sufficient to compute $d\xi_n$.\\
The product formula is
\begin{equation}
\mathbb{W}_0=\prod_{a\in \Gamma} W_{aA},
\end{equation}
where $a$ describes all the vertices of $\Gamma$, $Aa$ is the corresponding edge in $\boldsymbol{\Gamma}$.
Then it is sufficient to compute $d\xi_n(\delta w_a)$ for $\delta w_a\in T_{w_0}W_{aA}$, assuming that all the other vectors $\delta w_{bB}$
are zero, except $\delta w_a$ which denotes the weight over the edge $Aa$.\\

For that, we consider the set $\Omega_a$ of directed paths $\gamma_a$ in $\Gamma$ going from $a$ to the output layer $x_n$. Each such path
gives rise to a zigzag in $\boldsymbol{\Gamma}$ :
\begin{equation}
...\leftarrow B'\rightarrow b'\leftarrow B\rightarrow b \leftarrow ...
\end{equation}
which gives a feed-forward composed map, by taking over each $B\rightarrow b$ the map $X^{w_{bB}}$ from the product $X_B$ to the
manifold $X_b$, where
everything is fixed by $\xi_0$ and $w_0$ except on the branch coming from $b'$, where $w_a$ varies, and by taking over each $b'\leftarrow B$ the injection $\rho_{Bb'}$
defined by the other factors $X_{b"},X_{b{'''}},...$ of $X_B$. This composition is written
\begin{equation}
\phi_{\gamma_a}=\prod_{b_k\in \gamma_a} X^{w_0}_{b_kB_k}\circ \rho_{B_kb_{k-1}}\circ X^{w}_{aA};
\end{equation}
going from the manifold $W_a\times X_A$ to the manifold $X_n$. In  the above formula, $k$ starts with $1$, and $b_0=a$.\\

\indent Two different elements $\gamma'_a$, $\gamma"_a$ of $\Omega_a$ must coincide after a given vertex $c$, where they join from
different branches $c'c$, $c"c$ in $\Gamma$; they pass through $B$ in $\boldsymbol{\Gamma}$; then we can define the sum $\phi_{\gamma'_a}\oplus\phi_{\gamma"_a}$,
as a map from $W_{aA}^{\oplus2}\times X_A$ to $X_n$, by composing the maps between the $X's$ after $b$, from $b$ to $x_n$, with the two maps $\phi_{\gamma'_a}$ and $\phi_{\gamma"_a}$ truncated at $B$. We name this operation the cooperation, or cooperative sum, of $\phi_{\gamma'_a}$ and $\phi_{\gamma"_a}$.\\
Cooperation can be iterated in associative and commutating manner to any subset of $\Omega_a$, representing a tree issued from $x_n$, embedded in $\Gamma$,
made by all the common branches between the pairs of paths from $a$ to $x_n$. The full cooperative sum is the map
\begin{equation}
\bigoplus \phi_{\gamma_a}:X_A\times\bigoplus_{\gamma_a\in \Omega_a}W_{aA}\rightarrow X_n.
\end{equation}
For a fixed $\xi_0$, and all $w_{bB}$ fixed except $w_{aA}$, the point $\xi_n(w)$ can be described as the composition
of the diagonal map with the total cooperative sum
\begin{equation}
w_a\mapsto (w_a,...w_a)\in \bigoplus_{\gamma_a\in \Omega_a}W_{aA}\rightarrow X_n.
\end{equation}
This gives
\begin{equation}
d\xi_n(\delta w_a)=\sum_{\gamma_a\in \Omega_a} d\phi_{\gamma_a}\delta w_a;
\end{equation}
which implies the backpropagation formula:\\
\begin{lem}\label{lem:backprop}
	\begin{equation}
		d\xi_n(\delta w_a)=\sum_{\gamma_a\in \Omega_a}\prod_{b_k\in \gamma_a} DX^{w_0}_{b_kB_k}\circ D\rho_{B_kb_{k-1}}\circ\partial_wX^{w}_{aA}.\delta w_a
	\end{equation}
	going from the tangent space $T_{w_a^{0}}(W_a)$ to the tangent space $T_{\xi_n^{0}}(X_n)$.
	In  this expression, $k$ starts with $1$, and $b_0=a$.
\end{lem}

To get the backpropagation flow, we compose to the left with $F^{\star }=dF$, which gives a linear form, then apply the chosen metric on the manifold $W$, which
gives a vector field $\beta(w_0|\xi_0)$. Let us assume that the function $F$ is bounded from below on $X_0\times W$ and coercive (at least proper). Then the flow of $\beta$
is globally defined on $W$. From it we define a one parameter group of natural transformations of the object $\mathbb{W}$.\\

In practice, a sequence $\Xi_m;m\in [M]$ of finite set of inputs $\xi_0$ (benchmarks) is chosen randomly, according to the chosen measure on the initial data, and
the gradient is taken for the sum
\begin{equation}
F_m=\sum_{\Xi_m}F_{\xi_0},
\end{equation}
then the flow is integrated (with some important cooking) for a given time, before the next integration with $F_{m+1}$.\\
\indent This changes nothing to the result:\\

\begin{thm}\label{thm:backprop}
	Backpropagation is a flow of natural transformations of $\mathbb{W}$, computed from collections of singletons in $\mathbb{X}$.
\end{thm}

\vspace{8mm}
\noindent
\begin{figure}[ht]
	\begin{subfigure}[b]{.6\textwidth}
		\centering
		\includegraphics[width=0.9\textwidth]{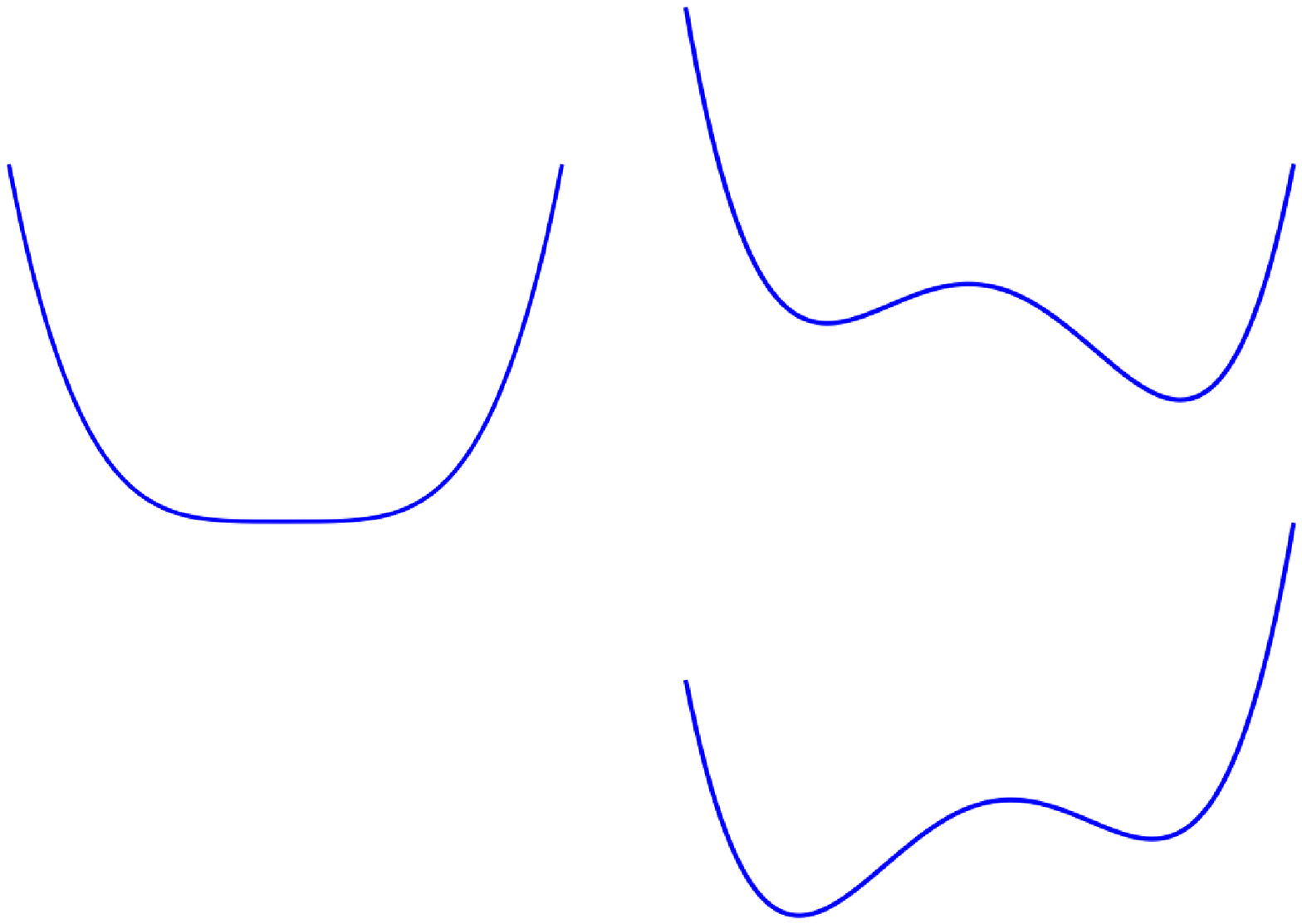}\vspace{3mm}
		\caption{Dynamics of $X^w$}
		\label{subf:quartic}
	\end{subfigure}
	\begin{subfigure}[b]{.4\textwidth}
		\centering
		\includegraphics[width=0.9\textwidth]{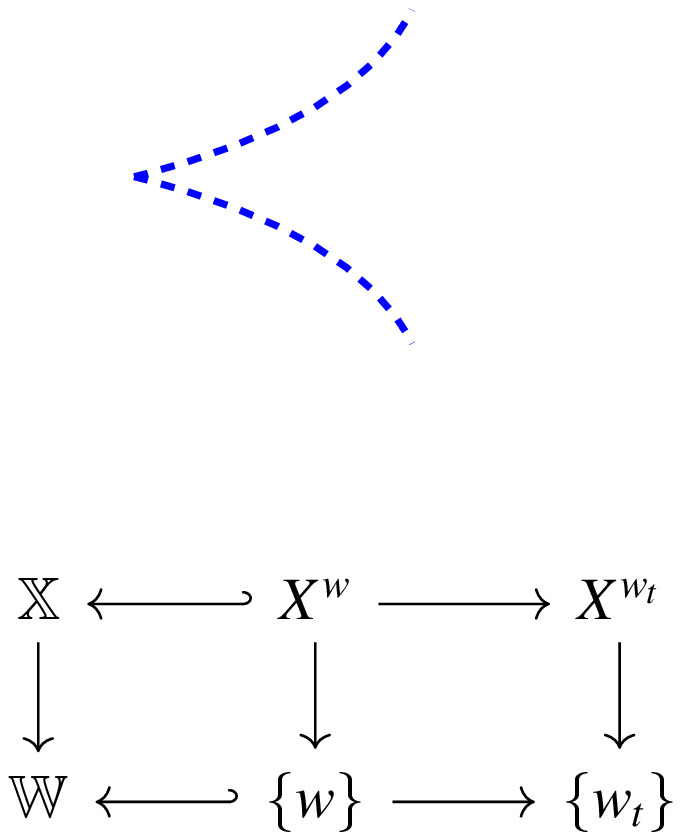}\vspace{3mm}
		\caption{Illustration of theorem \ref{thm:backprop}}
		\label{subf:com_diag}
	\end{subfigure}
	\caption{Examples of bifurcations}
	\label{fig:bifurcation}
\end{figure}

\noindent Figure \ref{fig:bifurcation} shows a bifurcation  $\Sigma$ in $\mathbb{W}$, $\mathbb{X}\rightarrow \mathbb{W}$. Subfigure \ref{subf:quartic} shows three forms of potentials for dynamics of $X^{w}$ on the left part when, in the upper-right part, we can see the regions of a planar projection of $\mathbb{W}$, where the learned dynamics has the corresponding shape.

\begin{rmk*}
	\normalfont Frequently, the function $F$ takes the form of a Kullback-Leibler divergence \[D_{KL}(P(\xi_n)|P_n)\]
	and can be rewritten as a free energy, which can itself be replaced by a Bethe free energy over inner variables,
	which are probabilistic laws on the weights. This is where information quantities could enter \cite{peltre}.
\end{rmk*}

\section{The specific nature of the topos of DNNs}

\noindent We wonder now to what species the topos $\mathcal{C}^{\sim}$ of a $DNN$ belongs. 

\begin{defns*}
	Let $\mathbf{X}$ denotes the set of vertices of $\boldsymbol{\Gamma}$ of type $a$ or of type $A$ (see figure \ref{fig:fork}). We introduce the full subcategory $\mathcal{C}_{\mathbf{X}}$ of $\mathcal{C}$ generated by $\mathbf{X}$. 
\end{defns*}	
	
	There only exists one arrow from a vertex of type $a'$ to a vertex of type $A$ through $A^{\star }$ (but a given $a'$ can join different $A^{\star }$ then different $A$), only one arrow from a vertex of type $a$ to its preceding $A$ (but $A$ can belong to several vertices $a$). Moreover there exists only one arrow from
	a vertex $c$ to a vertex $b$ when $b$ and $c$ are on a chain in $\mathcal{C}$ which does not contain a fork. And no other arrows exist in $\mathcal{C}_{\mathbf{X}}$.
	By definition of the forks, a point $a$ (i.e. a handle) cannot join another point than its tang $A$, and an input or a tang $A$ is the center of a convergent star.\\
	\indent Any maximal chain in $\mathcal{C}^{\rm op}_{\mathbf{X}}$ joins an input entry or a $A$-point (i.e. a tang), to a vertex of type $a'$ (i.e. a tip) or to an output layer.
	Issued from a tang $A$ it can pass through a handle $a$ or a tip $a'$, because nothing forbids a tip to join a vertex $b$.\\
	\indent If $x,y$ belong to $\mathbf{X}$, we note $x\leq y$ when there exists a morphism from $x$ to $y$; then it is equivalent to write $x\rightarrow y$ in the category $\mathcal{C}_{\mathbf{X}}$.

\begin{prop}
	\begin{enumerate}[label=(\roman*)]
		\item $\mathcal{C}_{\mathbf{X}}$ is a poset.
		\item Every presheaf on $\mathcal{C}$ induces a presheaf on $\mathcal{C}_{\mathbf{X}}$.
		\item For every presheaf on $\mathcal{C}_{\mathbf{X}}$, there exists a unique sheaf on $\mathcal{C}$ which induces it.
	\end{enumerate}
\end{prop}
\begin{proof}
	\begin{enumerate}[label=(\roman*)]
		\item let $\gamma_1,\gamma_2$ be two different simple directed paths in $\mathcal{C}_{\mathbf{X}}$ going from a point $z$ in $\mathbf{X}$ to a point $x$ in $\mathbf{X}$, there must exists
		a first point $y$ where the two paths disjoin, going to two different points $y_1$, $y_2$. This point $y$ cannot be a handle (type $a$), nor an input, nor a tang (type $A$), then it is
		an output or a tip. It cannot be an output, because a fork would have been introduced here to manage the divergence. If the two points $y_1,y_2$ were tangs, they were the ending points of the paths, which is impossible. But at least one of them is a tang, say $A_2$, because a tip cannot diverge to two ordinary vertices, if not, there should be a fork here. Then one of them, say $y_1$, is an ordinary vertex and begins a chain, without divergence until it attains an input or a tang $A_1$. Therefore $A_1=A_2$, but this gives an oriented loop in the initial graph $\Gamma$, which was excluded from the beginning for a $DNN$. This final argument directly forbids the existence of $x,\neq y$ with $x\leq y$ and $y\leq x$. Then $\mathcal{C}_{\mathbf{X}}$ is a poset.
		\item is obvious. 
		\item remark that the vertices of $\boldsymbol{\Gamma}$ which are eliminated in $\mathbf{X}$ are the $A^{\star }$. Then consider a presheaf $F$ on $\mathbf{X}$, the sheaf condition over $\mathcal{C}$ tells that $F(A^{\star })$ must be the product of the entrant $F(a'),...$,
		then the product map $F(A)\rightarrow F(A^{\star })$ of the maps $F(A)\rightarrow F(a')$ gives a sheaf.
	\end{enumerate}
\end{proof}

\begin{cor*}
	$\mathcal{C}^{\sim}$ is naturally equivalent to the category of presheaves $\mathcal{C}^{\wedge}_{\mathbf{X}}$.
\end{cor*}

\begin{rmk*}
	\normalfont In Friedman \cite{Friedmann}, it was shown that every topos defined by a finite site, where objects do not possess non unit endomorphisms,
	has this property to be equivalent
	to a topos of presheaves over a finite full subcategory of the site: this is the category generated by the objects that have only the trivial full covering.
	Then we are in a particular case of this theorem. The special fact, that we get a site which is a poset, implies many good properties for the topos \cite{Bell}, \cite{caramello2009duality}.
\end{rmk*}

\noindent In what follows, $\mathbf{X}$ will often denote the poset $\mathcal{C}_{\mathbf{X}}$.\\
\begin{defns}\label{defn:alexandrov}
	The (lower) Alexandrov topology on $\mathbf{X}$, is made by the subsets $U$ of $\mathbf{X}$ such that ($y\in U$ and $x\leq y$) imply $x\in U$.\\
	A basis for this topology is made by the collections $U_\alpha$ of the $\beta$ such that $\beta\leq \alpha$. In fact, consider the intersection
	$U_x\cap U_{x'}$; if $y\leq x$ and $y\leq x'$, we have $U_y\subseteq U_x\cap U_{x'}$,
	then $U_x\cap U_{x'}=\bigcup_{y\in U_x\cap U_{x'}}U_y$. \\
	In our examples the poset $\mathbf{X}$ is in general not stable by intersections or unions of subsets of $\mathbf{X}$, but the intersection and union
	of the sets $U_x$, $U_y$ for $x,y\in \mathbf{X}$ plays this role.\\
	We note $\Omega$ or $\Omega(\mathbf{X})$ when there exists a possibility of confusion, the set of (lower) open sets on $\mathbf{X}$.\\
	A sheaf in the topological sense over the Alexandrov space $\mathbf{X}$ is a sheaf in the sense of topos over the category $\Omega(\mathbf{X})$, where
	arrows are the inclusions, equipped with the Grothendieck topology, generated by the open coverings of open sets.
\end{defns}

\begin{prop}\label{prop:alexandrov}
	(see \cite[Theorem 1.1.8, the comparison lemma]{caramello-18} and \cite[p. 210]{Bell}): every presheaf of sets over the category $\mathcal{C}_\mathbf{X}$ can
	be extended to a sheaf on $\mathbf{X}$ for the Alexandrov topology, and this extension is unique up to a unique isomorphism.
\end{prop}
\begin{proof}
	Let $F$ be a presheaf on $\mathcal{C}_{\mathbf{X}}$; for every $x\in \mathbf{X}$, $F(U_x)$ is equal to $F(x)$. For any open set $U=\bigcup_{x\in U}U_x$ we define $F(U)$ as the limit over $x\in U$ of the sets $F(x)$ (that is the
	set of families $s_x;x\in U$
	in the sets $F(x);x\in U$, such that for any pair $x,x'$ in $U$ and any element $y$ in $U_x\cap U_{x'}$,
	the images of $s_x$ and $s_{x'}$ in $F(y)$ coincide. This defines a presheaf for the lower topological topology.\\
	This presheaf is a sheaf:
	\begin{enumerate}[label=\arabic*)]
		\item if $\mathcal{U}$ is a covering of $U$, and if $s,s'$ are two elements of $F(U)$ which give the same elements over $V$ for all $V\in \mathcal{U}$,
		the elements $s_x, s'_x$ that are defined by $s$ and $s'$ respectively in every $F(x)$ for $x\in U$ are the same, then
		by definition, $s=s'$.
		\item To verify the second axiom of a sheaf, suppose that a collection $s_V$ is defined for $V$ in the covering $\mathcal{U}$ of $U$, and
		that for any intersection $V\cap W$,$V,W \in \mathcal{U}$ the restrictions of $s_V$ and $s_W$ coincide, then by restriction to any $U_x$
		for $x\in U$ we get a coherent section over $U$.
		\item 	For the uniqueness, take a sheaf $F'$ which extends $F$, and consider the open set $U=\bigcup_{x\in U}U_x$, any element $s'$ of
		$F'(U)$ induces a collection $s'_x\in F(U_x)=F(x)$ which is coherent, then defines a unique element $s=f_U(s')\in F(U)$. These maps $f_U;U\in \Omega$
		define the required isomorphism.
	\end{enumerate}
\end{proof}

\begin{cor*}
	The category $\mathcal{C}^{\sim}$ is equivalent to the category $\emph{Sh}(\mathbf{X})$ of sheaves of $\mathbf{X}$, in the ordinary topological sense, for the (lower) Alexandrov topology.
\end{cor*}

\noindent Consequences from \cite[pp.408-410]{Bell}: the topos $\mathcal{E}=\mathcal{C}^{\sim}$ of a neural network is coherent. It possesses sufficiently many points, i.e. geometric functors ${\sf Set}\rightarrow \mathcal{C}^{\sim}$, such that equality of morphisms in $\mathcal{C}^{\sim}$ can be tested on these points.\\
\noindent In fact, such an equality can be tested on sub-singletons, i.e. the topos is generated by the subobjects of the final object $\mathbf{1}$. This property is called \emph{sub-extensionality} of the topos $\mathcal{E}$.\\
Moreover $\mathcal{E}$ (as any Grothendieck topos) is defined over the category of sets, i.e. there exists a unique geometric functor $\mu:\mathcal{E}\rightarrow{\sf Set}$. This functor is given by the global sections of the sheaves over $\mathbf{X}$. In this case, as shown in \cite{Bell}, the equality of subobjects (i.e. propositions) in every object of the form $\mu^{\star }(S)$ (named sub-constant objects) is decidable.\\
The two above properties characterize the so-called \emph{localic topos} \cite{Bell}, \cite{MacLaneMoerdijk1992}.\\

\noindent The points of $\mathcal{E}$ correspond to the ordinary points of the topological space $\mathbf{X}$; they are also the points of the poset
$\mathcal{C}_{\mathbf{X}}$. For each such point $x\in \mathbf{X}$, the functor $\epsilon_x:{\sf Set}\rightarrow \mathcal{E}$ is the right adjoint of the functor sending any sheaf $F$ to its fiber $F(x)$.\\

\noindent In the neural network, the minimal elements for the ordering in $\mathbf{X}$ are the output layers plus some points $a'$ (tips), the maximal ones are the input layers, and the
points of type $A$ (tangs).
However, for the standard functioning and for supervised learning, in the objects $\mathbb{X}$, $\mathbb{W}$, the fibers in $A$ are identified with
the products of the fibers in the tips $a',a",...$, and play the role of transmission to the branches of type $a$. Therefore the feed-forward functioning does not reflect the
complexity of the set $\Omega$. The backpropagation learning algorithm also escapes this complexity. \\
\begin{rmks*}
	\normalfont If $A$ were not present in the fork, we should have added the empty covering of $a$ in order to satisfy the axioms of a Grothendieck topology, and this would have been disastrous, implying that every sheaf must have in $a$ the value $\star $ (singleton). A consequence is the existence of more general sheaves than the ones that correspond to usual feed-forward dynamics, because they can have a value $X_A$ different from the product of the $X_{a'}$ appearing in $A^{\star }$, equipped with a map $X_{A^{\star }A}:X_A \rightarrow \prod X_{a'}$ and $X_{aA}:X_A\rightarrow X_a$. Then, depending on the value of $\varepsilon^{0}_{\rm in}$ and of the other objects and morphisms, a propagation can happen or not. This  opens the door to new types of networks, having a part of spontaneous activities (see chapter \ref{chap:dynamics}).
\end{rmks*}

\begin{rmk*}
	\normalfont Several evidences show that the natural neuronal networks in the brain
	of the animals are working in this manner, with spontaneous activities,  internal modulations and complex variants of supervised and unsupervised learning, involving memories, spontaneous activities, genetically and epi-genetically programmed activations and desactivations, which optimize the survival at the level of the evolution of species.
\end{rmk*}

\begin{rmk*}
	\normalfont Appendix \ref{app:localic} gives an interpretation due to Bell of the class of topos we encounter here, named \emph{localic topos}, in terms
	of a categorical version of fuzzy sets, called  \emph{sets with fuzzy identities} taking values in a given Heyting algebra.
\end{rmk*}

For the topos of a $DNN$, the Heyting algebra $\Omega$ is the algebra of open subsets of the poset $\mathbf{X}$. However, we can go further in the characterization of this topos
by using the particular properties of the poset $\mathbf{X}$, and of the algebra $\Omega$.

\begin{thm}\label{thm:dnn}
	The poset $\mathbf{X}$ of a DNN is made by a finite number of trees, rooted in the maximal points and which are joined in the minimal points.
\end{thm}

More precisely, the \emph{minimal} elements are of two types: the outputs layers $x_{n,j}$ and the tips of the forks, i.e. the points of type $a'$; the \emph{maximal} elements are also of two types: the input layers $x_{0,i}$ and the tangs of the forks (i.e. the points $A$). Moreover, the tips and the tanks are joined by an irreducible arrow,
but a tip can join several tanks and some ordinary point (of type $a$ but not being an input $x_{0,i}$), and a tank can be joined by several tips and
other ordinary points (but not being an output $x_{n,j}$)  as it is illustrated in figure \ref{fig:poset}.

\begin{figure}[ht]
	\begin{center}
		\includegraphics[width=13cm]{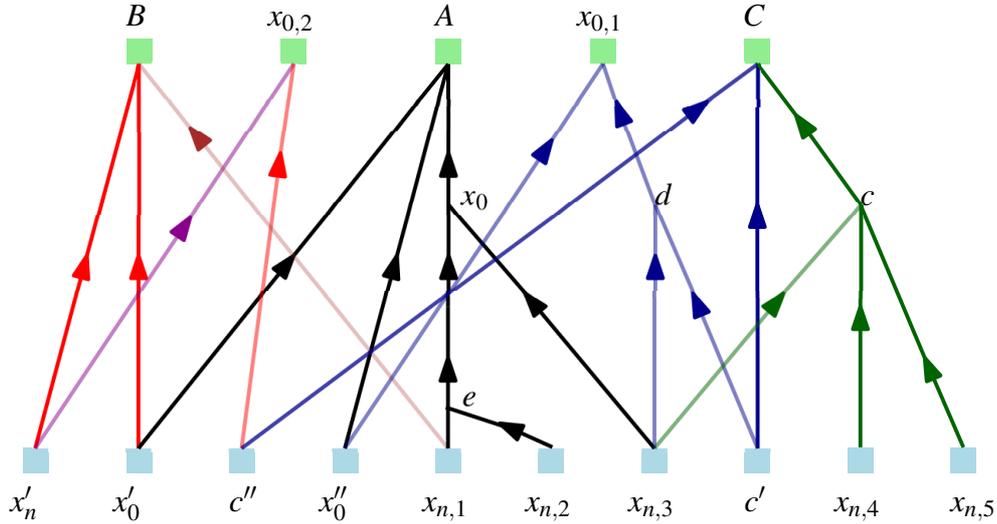}
		\caption{\label{fig:poset}Poset of a DNN}
	\end{center}
\end{figure}

\begin{rmk*}
	\normalfont The only possible divergences happen at tips, because they can joint several tanks and additional ordinary points in $\mathbf{X}$.
\end{rmk*}

\begin{rmk*}
	\normalfont Appendix \ref{app:top-of-dnn} gives an interpretation of the type of toposes we may obtain for $DNNs$ in terms of spectrum of commutative rings.
\end{rmk*}

Any object in the category $\mathcal{C}_\mathbf{X}^{\wedge}$ can be interpreted as a dynamical network, because it describes a
flow of maps between sets $\left\{F_x;x\in \mathbf{X}\right\}$, along the arrows between the layers, and each of these sets can be interpreted as a space of states,
not necessarily made by vectors. However what matters for the functioning of the network is the correspondence between input data,
that are elements in the product $F_{{\rm in}}$ of spaces over input layers, and output states, that are elements in the product $F_{\rm out}$
of the spaces over output layers. This correspondence is described by the limit of $F$ over $\mathcal{C}_\mathbf{X}$, i.e. $H^{0}(F)=H^{0}(\mathcal{C}_\mathbf{X};F)$, \cite{maclane:71}.
This contains the graphs of ordinary applications from $F_{\rm in}$ to $F_{\rm out}$, when taking the products at the forks,
but in general, except for chains, that are models of simple reflexes, this limit is much wider, and a source of innovation (see the above remarks on spontaneius activity and  section \ref{section:spontaneous} below).

\chapter{Stacks of DNNs}\label{chap:stacks}

\section{Groupoids, general categorical invariance and logic}

\indent In many interesting cases, a restriction on the structure of the functioning $X^{w}$, or the learning in $\mathbb{W}$,
comes from a geometrical or semantic invariance, which is extracted (or expected) from the input data and/or the problems that the network has to solve as output.

\indent The most celebrate example is given by the convolutional networks $CNNs$. These networks are made for analyzing images; it can be for finding
something precise in an image in a given class of images, or it can be for classifying special forms.
The images are assumed to be by nature invariant by planar translation, then it is imposed to a large number
of layers to accept a non trivial action of the group $G$ of $2D$-translations and to a large number of connections
between two layers to be compatible with the actions, which implies that the underlying linear part when it exists is made
by convolutions with a numerical function on the plane. This does not forbid that in several layers, the action of $G$ is trivial,
to get invariant characteristics under translations, and here, the layers can be fully connected. The \emph{Resnets} today have such a structure,
with non-trivial architectures, as described in the preceding chapter.\\
\indent Other Lie groups and their associated convolutions were recently used for DNNs, (see Cohen et al. \cite{cohen2019gauge}, \cite{cohen2020general},
\cite{bronstein2021geometric}). Diverse case of equivariant deep learning are presented in \cite{shiebler2021category}. For example, studies of graph networks
involve invariance and equivariance under groupoids of isomorphisms between graphs \cite{maron2019invariant}.\\
Cohen et al. \cite{cohen2019gauge} underline the analogy with Gauge theory in Physics. In the same spirit, Bondesan and Welling \cite{bondesan2021hintons} give an interpretation
of the excited states in DNNs in terms of particles in Quantum Field Theory.\\

DNNs that analyze images today, for instance in object detection, have several channels of convolutional maps, max pooling
and fully connected maps, that are joint together to take a decision. It looks as a structure for localizing the translation
invariance, as it happens in the successive visual areas in the brains of animals. Experiments show that in the first layers,
kinds of wavelet kernels are formed spontaneously to translate contrasts, and color opposition kernels are formed to construct
color invariance.\\

\indent A toposic manner to encode such a situation consists to consider contravariant functors from the category $\mathcal{C}$ of the network
with values in the topos $G^{\wedge}$ of $G$-sets, in place of taking values in the category ${\sf Set}$ of sets. Here the group $G$ is identified with the category
with one object and whose arrows are given by the elements of $G$, then a $G$-set, that is a set with a left action of $G$, is viewed as a set valued sheaf over $G$.
The collection of these functors, with morphisms given by the equivariant natural transformations,
form a category $\mathcal{C}^{\sim}_G$, which was shown to be itself a topos by Giraud \cite{giraud-classifying}. We will prove this fact
in the following section \ref{sec:objects-classifiers}: there exists a category $\mathcal{F}$, which is fibred in groups isomorphic to $G$ over $\mathcal{C}$,
$\pi:\mathcal{F}\rightarrow \mathcal{C}$, and satisfies the axioms of a stack, equipped with a canonical topology $J$ (the least fine such that $\pi$ is cocontinuous \cite[7.20]{stack-project}, i.e.
a comorphism
of site \cite{caramello2021relative}),
in such a manner that the topos $\mathcal{E}=\mathcal{F}^{\sim}$ of sheaves of sets over the site $(\mathcal{F}, J)$, is naturally equivalent to the category $\mathcal{C}^{\sim}_G$.
This topos is named the \emph{classifying topos} of the stack. \\
\indent The construction of Giraud is more general; it extends to any stack over $\mathcal{C}$, not necessarily in groups or in groupoids. In this chapter,
we will consider this more general situation, given by a functor $F$ from $\mathcal{C}^{\sf op}$ to the  category ${\sf Cat}$ of small categories, then corresponding to a fibred category
$\mathcal{F}\rightarrow\mathcal{C}$. But we will not consider the issue of non-trivial topologies, because, as we have shown in chapter \ref{chap:architecture}, the topos of $DNNs$ are topos of presheaves.
Then we determine the inner logic of the classifying topos, from fibers to fibers, to describe later the possible (optimal) flows of information in functioning
networks.\\
\indent The case of groupoids has the interest that the presheaves on a groupoid
form a Boolean topos, then ordinary logic is automatically incorporated.\\
\begin{rmks*}\normalfont 
	\begin{enumerate}[label=\arabic*)]
		\item The logic in the topos of a groupoid consists of simple Boolean algebras; however, things appear more interesting when we
		remember the meaning of the atoms $Z_i;i\in K$, because they are made of irreducible $G_a$-sets. We interpret that as a part
		of the semantic point of view, in the languages of topos and stacks.
		\item In the experiments reported in \cite{logic-DNN} as in $CNNs$, the irreducible linear representations of groups appear spontaneously among
		the dynamical objects.
		\item In every language we can talk of the future, the uncertain past, and introduce hypotheses, this does not mean that
		we are leaving the world of usual Boolean logic, we are just considering externally some intuitionist Heyting algebra, this can be done within ordinary set theory, as is done topos theory in Mathematics, in the fibers, defined by groupoids.
 	\end{enumerate}
\end{rmks*}

\noindent Appendix \ref{app:classifying} gives a description of the classifying object of a groupoid, that is well known by specialists of category theory.\\

However, other logics, intuitionist, can also have an interest.
In more recent experiments done with Xavier Giraud on data representing time evolution, we used simple posets in the fibers.

The notion of invariance goes further than groupoids.\\

\indent Invariance is synonymous of action (like group action), and is understood here in the categorical sense: a category $\mathcal{G}$ acts
on another category $\mathcal{V}$ when a (contravariant) functor from $\mathcal{G}$ to $\mathcal{V}$ is given.
The example that justifies this terminology is when $\mathcal{G}$ is  a group $G$, and $\mathcal{V}$ the Abelian category
of vector spaces and linear maps over a commutative field $\mathbb{K}$. In the latter case, we obtain a linear representation of the group $G$.\\

In any category $\mathcal{V}$, there exists a notion which generalizes the notion of {\em element} of a set.
Any morphism $\varphi:u\rightarrow v$ in $\mathcal{V}$ can be viewed as an {\em element} of the object $v$ of $\mathcal{V}$.
\begin{defn*}
	Suppose that $\mathcal{G}$ acts through the functor $f:\mathcal{G}\rightarrow \mathcal{V}$
	and that $v=f(a)$, then the \emph{orbit} of $\varphi$  under $\mathcal{G}|a$ is the functor from the left slice category $\mathcal{G}|a$
	to the right slice category $u|\mathcal{V}$, that associates to any morphism $a'\rightarrow a$ the
	element $u\rightarrow f(a)\rightarrow f(a')$ of $f(a')$ in $\mathcal{V}$ and to an arrow $a"\rightarrow a'$ over $a$ the
	corresponding morphism $f(a')\rightarrow f(a")$, from $u\rightarrow f(a')$ to $u\rightarrow f(a")$.
\end{defn*}

\noindent In the classical example of a group representation, $u=\mathbb{K}$ and the morphism $\varphi$ defines a vector $x$ in the space $V_e$.
The group $G$ is identified with $G|e$ and the vector space $V_e$, identified with ${\sf Hom}(K,V_e)$, contains the whole orbit of $x$.\\

\noindent In a stack, the notion of action of categories is extended to the notion of fibred action of a fibred category
$\mathcal{F}$ to a fibred category $\mathcal{N}$:\\
\begin{defn*}
	Suppose we are given a sheaf of categories $F:\mathcal{C}\rightarrow {\sf Cat}$, that we consider as a general structure of invariance,
	and another sheaf $M:\mathcal{C}\rightarrow {\sf Cat}$. An action of $F$ on $M$ is a family of contravariant functors
	$f_U:\mathcal{F}_U\rightarrow \mathcal{M}_U$ such that, for any morphism $\alpha:U\rightarrow U'$ of $\mathcal{C}$, we have
	\begin{equation}
		f_U\circ F_\alpha=M_\alpha\circ f_{U'}.
	\end{equation}
	This is the {\em equivariance formula} generalizing group equivariance as it can be found in \cite{Kondor2018NbodyNA}
	for instance. It is equivalent to morphisms of stacks, and allows to define the orbits of sections $u_U\rightarrow f_U(\xi_U)$
	in the sheaf $u|\mathcal{M}$ under the action of the relative stack $\mathcal{F}|\xi$.
\end{defn*}

Remark that Eilenberg and MacLane, when they invented categories and functors in \cite{eilenberg1945general},
were conscious to generalize the Klein's program in Geometry (Erlangen program).\\

\noindent In the next sections, we will introduce languages with types taken from presheaves over the fibers of the stack, where we define
the terms of theories
and  propositions of interest for the functioning of the DNN. Then the above notion of invariance will concern the action of a kind of pre-semantic
categories on the languages and the possible sets of theories, that the network could use and express in functioning.\\

\indent This view is a crucial point for our applications of topos theory to DNNs, because it is in this framework that logical reasoning,
and more generally semantics, in the neural network, can be set: in a stack, the different layers interpret the logical propositions
and the sentences of the output layers. As we will see, the interpretations are expected to become more and more faithful when approaching the output,
however the information flow in the whole networks is interesting by itself.\\

This shift from groups to groupoids, then to  categories, then to more general semantic, by taking presheaves in groupoids or categories, is a fundamental addition
to the site $\mathcal{C}$. The true topos associated to a network is the classifying topos $\mathcal{E}$ over $\mathcal{F}$;
it incorporates much more structure than the visible architecture of layers, it takes into account invariance
(which appears here to be part of the semantic, or better pre-semantic). More generally, it can concern
the domain of natural human semantics that the network has to understand in his own artificial world.

Moreover, as we will show below, working in this setting gives access to more flexible type theories, like the Martin-Löf intensional types,
and goes into the direction of homotopy type theory according to Hofmann and Streicher \cite{HS}, Hollander \cite{Hollander2001AHT}, Arndt and Kapulkin \cite{AK},
enlarged by objects and morphisms in classifying topos in the sense of Giraud.\\

\section{Objects classifiers of the fibers of a classifying topos}\label{sec:objects-classifiers}

In this section we study the propagation of logical theories through a stack (equipped with a scindage in the sense of Giraud). In particular we find a
sufficient condition for free propagation downstream  and upstream, that was apparently not described before; it asks that gluing functors are fibrations,
plus a supplementary geometrical condition, always satisfied in the case of groupoids, (see theorem \ref{thm:semanticflow}).\\
\indent The application to the dynamics of functioning $DNNs$ is presented in the next section \ref{sec:theories}, with the notion of
semantic functioning. It is developed in the next chapter
\ref{chap:dynamics}. Examples are presented in the following chapters \ref{chap:unfolding}, with long and short term memory cells and variants of them, and their tentative
relation with cognitive linguistic, then in \ref{chap:3-category}, with more general networks moduli.\\

In the general case and in more canonical toposic terms, the logic in the stack $\mathcal{F}$ over $\mathcal{C}$ is studied by Olivia Caramello and Riccardo Zanfa \cite{caramello2021relative}; see also the available notes written for "Topos Online",
24-30 june 2021.\\
\indent Also the contributions of Shulman, \cite{shulman2010stack}, \cite{Shulman_2019}, and the slides of his talk "Large categories and quantifiers in topos theory", January 26 2021, Cambridge Category Seminar are of interest.\\

Among the equivalent points of view on stacks and classifying topos \cite{giraud-descente}, \cite{giraud-cohomologie}, and \cite{giraud-classifying}), the most concrete one starts with a contravariant functor $F$ from the category $\mathcal{C}$ to the $2$-category of small categories ${\sf Cat}$. 
(This corresponds to an element of the category ${\sf Scind}(\mathcal{C})$ in the book of Giraud \cite{giraud-cohomologie}.) To each object
$U\in \mathcal{C}$ is associated a small category $\mathcal{F}(U)$, and to each morphism $\alpha:U\rightarrow U'$ is associated a covariant functor $F_\alpha:F(U')\rightarrow F(U)$, also denoted $F(\alpha)$, satisfying the axioms of a presheaf over $\mathcal{C}$. If $f_U:\xi\rightarrow \eta$ is a morphism in $F(U)$, the functor $F_\alpha$ sends it to a morphism $F_\alpha(f_U):F_\alpha(\xi)\rightarrow F_\alpha(\eta)$ in $F(U')$.\\
\indent The corresponding fibration $\pi:\mathcal{F}\rightarrow\mathcal{C}$, written $\nabla F$ by Grothendieck, has for objects the pairs $(U,\xi)$ where $U\in \mathcal{C}$ and $\xi\in F(U)$, sometimes shortly written $\xi_U$, and for morphisms the elements of
\begin{equation}
{\sf Hom}_\mathcal{F}((U,\xi),(U',\xi'))=\bigcup_{\alpha\in {\sf Hom}_\mathcal{C}(U,U')}{\sf Hom}_{F(U)}(\xi,F(\alpha)\xi').
\end{equation}
For every morphism $\alpha:U\rightarrow U'$ of $\mathcal{C}$, the set ${\sf Hom}_{F(U)}(\xi,F(\alpha)\xi')$ is also denoted\\ ${\sf Hom}_\alpha((U,\xi),(U',\xi'))$; it is the subset of morphisms in $\mathcal{F}$ that lift $\alpha$.\\

The functor $\pi$ sends $(U,\xi)$ on $U$. We will write indifferently $F(U)$ or $\mathcal{F}_U$ the fiber $\pi^{-1}(U)$.\\

A section $s$ of $\pi$ corresponds to a family $s_U\in \mathcal{F}_U$ indexed by $U\in \mathcal{C}$, and a family of morphisms
$s_\alpha\in {\sf Hom}_{F(U)}(s_U,F(\alpha)s_{U'})$ indexed by $\alpha\in {\sf Hom}_\mathcal{C}(U,U')$ such that, for any pair of compatible morphisms $\alpha,\beta$,
we have
\begin{equation}\label{compatibilitysection}
s_{\alpha\circ\beta}=F_\beta(s_\alpha) \circ s_\beta.
\end{equation}

As shown by Grothendieck and Giraud \cite{giraud-descente}, a presheaf $A$ over $\mathcal{F}$ corresponds to a family of presheaves $A_U$ on the categories $\mathcal{F}_U$ indexed by $U\in \mathcal{C}$, and a family
$A_\alpha$ indexed by $\alpha\in {\sf Hom}_\mathcal{C}(U,U')$, of natural transformations from $A_{U'}$ to $F_\alpha^{\star }A_U$. (Here $F_\alpha^{\star }$ denotes the pullback of presheaf associated to the functor $F_\alpha:F(U')\rightarrow F(U)$, that is, for $A_U:F(U)\rightarrow {\sf Set}$, the
composed functor $A_U\circ F_\alpha$.)\\
Moreover, for any compatible  morphisms $\beta:V\rightarrow U$, $\alpha:U\rightarrow U'$, we must have
\begin{equation}\label{compatibiltysheaf}
A_{\alpha\circ \beta}=F_\alpha^{\star }(A_\beta)\circ A_\alpha.
\end{equation}
\noindent If $\xi$ is an object of $\mathcal{F}_U$, we define $A(U,\xi)=A_U(\xi)$, and if $f:\xi_U\rightarrow F_\alpha\xi'_{U'}$ is a morphism of $\mathcal{F}$
between $\xi_U\in \mathcal{F}_U$ and $\xi'_{U'}\in \mathcal{F}_{U'}$ lifting $\alpha$, we take
\begin{equation}
A(f)=A_U(f)\circ A_\alpha:A_{U'}(\xi')\rightarrow A_U(F_\alpha(\xi'))\rightarrow A_U(\xi).
\end{equation}
The relation $A(f\circ g)=A(g)\circ A(f)$ follows from \eqref{compatibiltysheaf}.\\
\indent A natural transformation $\varphi:A\rightarrow A'$ corresponds to a family of natural transformations \[\varphi_U:A_U\rightarrow A'_U,\] such that,
for any arrow $\alpha:U\rightarrow U'$ in $\mathcal{C}$,
\begin{equation}
F_\alpha^{\star }\varphi_U\circ A_\alpha=A'_\alpha\circ \varphi_{U'}:A_{U'}\rightarrow F_\alpha^{\star }A'_U.
\end{equation}
This describes the category $\mathcal{E}$ of presheaves over $\mathcal{F}$ from the family of categories $\mathcal{E}_U$
of presheaves over the fibers $\mathcal{F}_U$ and the family of functors $F_\alpha^{\star }:\mathcal{E}_U\rightarrow \mathcal{E}_{U'}$.\\
Note that for two consecutive morphisms $\beta:V\rightarrow U$, $\alpha:U\rightarrow U'$, we have $F_{\alpha\beta}^{\star }=F_{\alpha}^{\star }\circ F_{\beta}^{\star }$.\\

The category $\mathcal{E}$ is fibred over the category $\mathcal{C}$, it corresponds to the functor $E$ from $\mathcal{C}$ to $Cat$, which
associates to $U\in \mathcal{C}$  the category $\mathcal{E}_U$ and to an arrow $\alpha:U\rightarrow U'$, the functor
$F^{\alpha}_!:\mathcal{E}_{U'}\rightarrow \mathcal{E}_U$, which is the \emph{left adjoint} of $F_\alpha^{\star }$. This functor extends $F_\alpha$
through the Yoneda embedding, \cite[Chap. I, Presheaves]{SGA4}.\\
For two consecutive morphisms $\beta:V\rightarrow U$, $\alpha:U\rightarrow U'$, we have $F^{\alpha\beta}_!=F^{\beta}_!\circ F^{\alpha}_!$.\\
\indent Let $\eta_\alpha:F^{\alpha}_!\circ F_\alpha^{\star }\rightarrow Id_{\mathcal{E}_U}$ the counit of the adjunction; a natural transformation
$A_\alpha:A_{U'}\rightarrow F_\alpha^{\star } A_U$ gives a natural transformation $A^{\star }_\alpha:F^{\alpha}_!A_{U'}\rightarrow A_U$, by taking
$A^{\star }_\alpha=(\eta_\alpha\otimes Id)F^{\alpha}_!(A_\alpha)$. This gives another way to describe the elements of $\mathcal{E}$, through
the presheaves over $\mathcal{F}$.
\begin{rmk*}
	\normalfont A section $(s_U,s_\alpha)$ defines a presheaf $A$, by taking
	\begin{equation}
		A_U(\xi)={\sf Hom}_{\mathcal{F}_U}(\xi,s_U);
	\end{equation}
	and $A_\alpha=s_\alpha^{\star }\circ F_\alpha$, according to the following sequence:
	\begin{equation}
		{\sf Hom}(\xi',s_{U'})\rightarrow {\sf Hom}(F_\alpha \xi',F_\alpha(s_{U'}))\rightarrow {\sf Hom}(F_\alpha \xi',s_U).
	\end{equation}
	The identity \eqref{compatibiltysheaf} follows from the identity \eqref{compatibilitysection}.\\
	This construction generalizes in the fibered situation the Yoneda objects in the absolute situation.\\
	A morphism of sections gives a morphism of presheaves.
\end{rmk*}

In each topos $\mathcal{E}_U$ there exists a classifying object $\boldsymbol{\Omega}_U$, such that the natural transformations ${\sf Hom}_U(X_U,\boldsymbol{\Omega}_U)$
naturally correspond to the subobjects of $X_U$; the presheaf $\boldsymbol{\Omega}_U$ has for value in $\xi_U\in \mathcal{F}_U$ the set of subobjects in $\mathcal{E}_U$
of the Yoneda presheaf $\xi_U^{\wedge}$ defined by $\eta \mapsto {\sf Hom}(\eta, \xi_U)$, with morphisms given by composition to the right.\\
The set $\boldsymbol{\Omega}_U(\xi_U)$ can also be identified with the set
of subobjects of the final sheaf $\mathbf{1}_{\xi_U}$ over the slice category $\mathcal{F}_U|\xi_U$.\\

\begin{rmk*}
\normalfont In general, the object of parts $\boldsymbol{\Omega}^{X}$ of an object $X$ in a presheaf topos $\mathcal{D}^{\wedge}$ over a category $\mathcal{D}$, is the presheaf
given in $x\in \mathcal{D}_0$ by the set of subsets of the product set $(\mathcal{D}|x)\times X(x)$ and by the maps induced by $X(f)$ for $f\in \mathcal{D}_1$.
Observe that $\boldsymbol{\Omega}^{X}$ realizes an equilibrium
between the category of basis through $\mathcal{D}|x$ and the set theoretic nature of the value $X(x)$.\\
A special case is when $X=\mathbf{1}$, the final object, made by
a singleton $\star $ at each $x\in \mathcal{D}_0$, and the unique possible maps for $f\in \mathcal{C}_1$. The presheaf
$\boldsymbol{\Omega}^{\mathbf{1}}$ is denoted by $\boldsymbol{\Omega}$. Its value in $x\in \mathcal{D}_0$, is the set $\boldsymbol{\Omega}_x$ of subsets
of the Yoneda object $x^{\wedge}$.\\
It can be proved that a subobject $Y$ of an object $X$ of $\mathcal{D}^{\wedge}$ corresponds to a unique morphism $\chi_Y:X\rightarrow\boldsymbol{\Omega}$
such that at any $x\in \mathcal{D}_0$, we have $Y(x)=\chi_Y^{-1}(\top)$.\\
The exponential presheaf $\boldsymbol{\Omega}^{X}$ is characterized by the natural family
of bijections
\begin{equation}
{\sf Hom}_{\mathcal{D}^{\wedge}}(Y\times X, \boldsymbol{\Omega})\approx {\sf Hom}_{\mathcal{D}^{\wedge}}(Y, \boldsymbol{\Omega}^{X}),
\end{equation}
which expresses the universal property of the classifier $\boldsymbol{\Omega}$.\\
We will also frequently consider the set of subobjects of $\mathbf{1}$ over the whole category $\mathcal{D}$, and we simply
denote it by the letter $\Omega$. It is named the Heyting algebra of the topos $\mathcal{C}^{\wedge}$. See appendix A for more details.
\end{rmk*}

As just said before, the functor $F_\alpha^{\star }:\mathcal{E}_U\rightarrow \mathcal{E}_{U'}$ which associates $A\circ F_\alpha$ to $A$,  possesses
a left adjoint $F^{\alpha}_!:\mathcal{E}_{U'}\rightarrow \mathcal{E}_U$
which extends the functor $F_\alpha$ on the Yoneda objects. For any object $\xi'$
in $\mathcal{F}_{U'}$, note $\xi=F_\alpha(\xi')$; the functor $F^{\alpha}_!$ sends $(\xi')^{\wedge}$ to $\xi^{\wedge}$, and sends
a subset of $(\xi')^{\wedge}$ to a subset of $\xi^{\wedge}$. This is not because $F^{\alpha}_!$ is necessarily left exact, but because we are working
with Grothendieck topos, where subobjects are given by families of coherent subsets.\\
Moreover $F^{\alpha}_!$ respects the
ordering between these subsets, then it induces a poset morphism between the posets of subobjects
\begin{equation}
\Omega_\alpha(\xi'):\Omega_{U'}(\xi')\rightarrow \Omega_U(F_\alpha(\xi'))=F_\alpha^{\star }\Omega_U(\xi');
\end{equation}
the functoriality of $\Omega_U$, $\Omega_{U'}$ and $F_\alpha$ implies that these maps constitute a natural transformation between presheaves
\begin{equation}
\Omega_\alpha:\Omega_{U'}\rightarrow F_\alpha^{\star }\Omega_U.
\end{equation}
The naturalness of the construction insures the formula \eqref{compatibiltysheaf} for the composition of morphisms.
Consequently, we obtain a presheaf $\boldsymbol{\Omega}_\mathcal{F}$. \\
\indent Moreover the final object $\mathbf{1}_\mathcal{F}$ of the classifying topos $\mathcal{E}=\mathcal{F}^{\wedge}$ corresponds to the collection of
final objects $\mathbf{1}_U;U\in \mathcal{C}$ and to the collection of morphisms $\mathbf{1}_{U'}\rightarrow F_\alpha^{\star }\mathbf{1}_U;\alpha\in {\sf Hom}_{\mathcal{C}}(U,U')$,
then we have:
\begin{prop}
	The classifier of the classifying topos is the sheaf $\boldsymbol{\Omega}_\mathcal{F}$ given by the
	classifiers $\Omega_U$ and the pullback morphisms $\Omega_\alpha$, which can be summarized by the formula
	\begin{equation}
		\boldsymbol{\Omega}_\mathcal{F}=\nabla_{U\in \mathcal{C}}\Omega_Ud\Omega_\alpha.
	\end{equation}
\end{prop}

In general the functor $F_\alpha^{\star }$ is not \emph{geometric}; by definition, it is so if and only if its left adjoint \[\left(F_\alpha\right)_!=\left(F^{\alpha}\right)_!\], which
is right exact (i.e. commutes with the finite colimits), is also \emph{left exact} (i.e. commutes with the finite limits).
Also by definition, this is the case if and only if the morphism $F_\alpha$ is a \emph{morphism of sites} from $\mathcal{F}_U$ to $\mathcal{F}_{U'}$,
\cite[IV 4.9.1.1.]{SGA4}, not to be confused with a comorphism, \cite[III.2]{SGA4}, \cite{caramello2021relative}. \\
Important for us: it results from the work of Giraud in \cite{giraud-classifying}, that $F_\alpha^{\star }$ is geometric when $F_\alpha$ is itself a stack,
and when finite limits exist in the sites $\mathcal{F}_U$ and $\mathcal{F}_{U'}$ and are preserved by $F_\alpha$.
(We will see in the next section, that these stacks $\mathcal{F}\rightarrow \mathcal{C}$, made by stacks between fibers,
correspond to some admissible contexts in a dependent type theory, when $\mathcal{C}$ is the site of a $DNN$.)\\
When $F_\alpha^{\star }$ is geometric, a great part of the logic in $\mathcal{E}_{U'}$ can be transported to $\mathcal{E}_{U}$:\\
\indent Let us write $f=F_\alpha^{\star }$ and $f^{\star }=(F_\alpha)_!$ its left adjoint, supposed to be left exact, therefore exact, as just mentionned. This functor $f^{\star }$
preserves the monomorphisms, and the final elements of the slices categories. Then it induces a map between the sets of subsets, called the inverse
image or pullback by $f$, for any object $X'\in \mathcal{E}_{U'}$:
\begin{equation}
f^{\star }:{\rm Sub}(X')\rightarrow {\rm Sub}(f^{\star }X').
\end{equation}
When $X'$ describes the Yoneda objects $(\xi')^{\wedge}$, this gives the morphism $\Omega_\alpha:\Omega_{U'}\rightarrow F_\alpha^{\star }\Omega_U$.\\
As it is shown in MacLane-Moerdijk \cite[p. 496]{MacLaneMoerdijk1992}, this map is a morphism of lattices, it preserves the ordering and the operations $\wedge$ and $\vee$.
If $h:Y'\rightarrow X'$ is a morphism in $\mathcal{E}_{U'}$,  the reciprocal image $h^{\star }$ between the sets of subsets has a
left adjoint $\exists_h$ and a right adjoint $\forall_h$. The morphism $f^{\star }$ commutes with $\exists_h$, but in general not with $\forall_h$, for which
there is only an inclusion:
\begin{equation}
f^{\star }(\forall_hP')\leq \forall_{f^{\star }h}(f^{\star }P').
\end{equation}

To have an equality, the morphism $f$ must be geometric and \emph{open}.
This is equivalent to the existence of  a left adjoint, in the sense of posets morphisms, for $\Omega_\alpha$, \cite[Theorem 3, p. 498]{MacLaneMoerdijk1992}.\\
In {MacLaneMoerdijk1992}, this natural transformation $\Omega_\alpha$ is denoted $\lambda_\alpha$, and its left adjoint when it exists is denoted $\mu_\alpha$.

When this left adjoint in the sense of Heyting algebras exists, we have, by adjunction, the counit and unit morphisms:
\begin{align}
\mu\circ \lambda & \leq {\sf Id}: \Omega_{U'} \rightarrow \Omega_{U'};\\
\lambda\circ \mu & \geq {\sf Id}: F^{\star }\Omega_{U} \rightarrow F^{\star }\Omega_{U}.
\end{align}

\noindent If $f$ is geometric and open, the map $f^{\star }$ also commutes with the negation $\neg$ and with the (internal) implication $\Rightarrow$.\\
If openness fails, only inequality (external implication) holds for the universal quantifier.
\begin{rmk*}
	\normalfont When $\mathcal{F}_{U'}$ and $\mathcal{F}_{U}$ are the posets of open sets of (sober) topological spaces $\mathcal{X}'$ and $\mathcal{X}$,
	and when $F_\alpha$ is given by the direct image of a continuous \emph{open} map $\varphi_\alpha:\mathcal{X}_{U'}\rightarrow \mathcal{X}_U$,
the functor $F^{\star }_\alpha$ is geometric and open. This extends to locale, \cite{MacLaneMoerdijk1992}.
\end{rmk*}

\indent When $F_\alpha^{\star }$ is geometric and open, it transports the predicate calculus of formal theories from $\mathcal{E}_{U'}$
to $\mathcal{E}_{U}$, as exposed in the book of Mac Lane and Moerdijk, \cite{MacLaneMoerdijk1992}. This is expressed by the following result,
\begin{prop}
	Suppose that all the $F_\alpha; \alpha:U\rightarrow U'$ are open morphisms of sites (in the direction from $F(U)$ to $F(U')$, then,
	\begin{enumerate}[label=(\roman*)]
		\item the pullback $\Omega_\alpha$ commutes with all the operations of predicate calculus;
		\item any theory at a layer $U'$, i.e. in $\mathcal{E}_{U'}$, can be read and translated in a deeper layer $U$, in $\mathcal{E}_U$, in particular at the output layers.
	\end{enumerate}
	
\end{prop}

\noindent In the sequence we will be particularly interested by the case where all the $\mathcal{F}_U$ are groupoids and the $F_\alpha$
are morphisms of groupoids, in this case, the algebras of subobjects ${\rm Sub}_{\mathcal{E}}(X)$ are boolean, then, in this case,
the following lemma implies
that, as soon as $F_\alpha^{\star}$ is geometric, it is open:
\begin{lem}\label{bool-lat}
	In the boolean case the morphism of lattices $f^{\star}:{\rm Sub}(X')\rightarrow {\rm Sub}(f^{\star}X')$ is a morphism of algebras which commutes with the universal quantifiers $\forall_h$.
\end{lem}
\begin{proof}
	Since $f^{\star}$ is right and left exact, it sends $0=\bot$ to $0=\bot$ and $X'=\top$
	to $X=\top$. Therefore, for every $A\in {\rm Sub}_{\mathcal{E}'}(X')$, $f^{\star}(X'\setminus A')=X\setminus f^{\star}(A')$, i.e.
	$f^{\star}$ commutes with the negation $\neg$. This negation establishes a duality between $\exists$ and $\forall$, then
	$f^{\star}$ commutes with the universal quantifier. More precisely:
	\begin{equation}
		f^{\star}(\neg(\forall x', P'(x')))=f^{\star}(\exists a', \neg P'(a'))=\exists a, f^{\star}(\neg P') (a)=\neg (\forall x f^{\star}(P') (x)),
	\end{equation}
	then by commutation with $\neg$, and $\neg\neg={\sf Id}$, we have
	\begin{equation}
		f^{\star}(\forall x', P'(x'))=\forall x, f^{\star}(P') (x).
	\end{equation}
\end{proof}

\indent Let us mention here a difficulty: in the case of groups or groupoids, $F_\alpha^{\star}$ is geometric if and only if $F_\alpha$ is an equivalence of categories
(then an isomorphism in the case of groups).
This is because a morphism of group is flat if and only if it is an isomorphism, \cite[4.5.1.]{SGA4}. The main problem is with the preservation of products.

However, it is remarkable that for any kind of group homomorphisms $F:G'\rightarrow G$, in every algebra of subobjects the map $f^{\star }$ induced by $F_!$
preserves "locally" and "naturally" all the logical operations:

\begin{lem}\label{morph-heyt}
For every object $X'$ in $B_{G'}$, note $X=F_!(X')$, then $f^{\star}$ induces a map of lattices $f^{\star}:{\rm Sub} (X')\rightarrow {\rm Sub} (X)$,
that is bijective. It preserves the order $\leq$, the elements $\top$ and $\bot$, and the operations $\wedge$ and $\vee$, therefore it is a morphism of
Heyting algebras. Moreover, for any natural transformation $h:Y'\rightarrow X'$, it commutes with both
the existential quantifier $\exists_h$ and the universal
quantifiers $\forall_h$.
\end{lem}
\begin{proof}
As said in \cite[4.5.1]{SGA4}, if $F:G'\rightarrow G$ is a morphism of groups, the functor $F_!$ from $B_{G'}$ to $B_G$
is given on $X'$ by the \emph{contracted product}
$F_!(X')=G\times_{G'}X'$, that is the set of orbits of the action of $G'$ on the $G$-set $G_d\times X'$.\\
The algebra ${\rm Sub} (X')$ is the boolean algebra generated by the primitive representations of $G'$
on the orbits $G'x'$ of the elements of $X'$. But each orbit $G'x'$ is sent in $X$ to an orbit of $G$, that is the
product of $G/F(H'_{x'})$ with the singleton $\{G'x'\}$, where $H'_{x'}$ is the stabilizer of $x'$. These sets describe
the orbits of the action of $G$ on $X$, then the elements of ${\rm Sub}(X)$.\\
The commutativity with $h^{\star }$ for a $G'$-morphism $h:Y'\rightarrow X'$ is evident, the rest follows from the bijection
property, orbitwise.
\end{proof}

\noindent Therefore, even if $F^{\star}$ is not a geometric morphism, it is legitimate to say that in some sense, it is \emph{open}, because all
logical properties are preserved by the induced morphisms between the local Heyting algebras. We could say that $F^{\star}$ is "weakly geometric and open".\\

This can be easily extended to morphisms of groupoids. The left adjoint $F_!$ admits a description
which is analogous to the contracted product of groups. Lemma \ref{morph-heyt} holds true. The only difference is that $f^{\star}$ is not a bijection,
but it is a surjection when $F$ is surjective on the objects. More details  and the
generalization of the above results to fibrations of categories that are themselves fibrations in groupoids over posets
will be given in the text \emph{Search of semantic spaces}.\\

\indent In the reverse direction of the flow, it is important that a proposition in the fiber over $U$ can be understood over $U'$.\\
\indent Hopefully, this can always be done, at least in part: the functor $F_\alpha^{\star}$
is left exact and has a right adjoint $F_\star ^{\alpha}:\mathcal{E}_{U'}\rightarrow \mathcal{E}_{U}$, which can be described as a right Kan extension \cite{SGA4}:
for a presheaf $A'$ over $\mathcal{F}_{U'}$, the value of the presheaf $F_\star ^{\alpha}(A'_{U'})$ at $\xi_U\in \mathcal{F}_U$ is the limit of $A'_{U'}$ over the slice category
$F_\alpha|\xi_U$,
whose objects are the pairs $(\eta',\varphi)$ where $\eta'\in \mathcal{F}_{U'}$ and $\varphi: F_\alpha(\eta')\rightarrow  \xi_U$ is a morphism in $\mathcal{F}_U$,
and whose morphisms from $(\eta',\varphi)$ to $(\zeta',\phi)$ are the morphisms $u:\eta'\rightarrow\zeta'$ such that $\varphi=\phi\circ F_\alpha(u)$.\\
Therefore, if we denote $\rho$ the forgetting functor from $F_\alpha|\xi_U$ to $\mathcal{F}_{U'}$, we have
\begin{equation}
F_\star ^{\alpha}(A')(\xi_U)=H^{0}(F_\alpha|\xi_U;\rho^{\star}A'),
\end{equation}
that is the set of sections of the presheaf $\rho^{\star }A'$ over the slice category.
\begin{rmk*}
	\normalfont In the case where $F_\alpha:\mathcal{F}_{U'}\rightarrow \mathcal{F}_{U}$ is a morphism of groupoids, this set is the set of sections of $A'$
	over the connected components of $F_\alpha^{-1}(\xi_U)$.
\end{rmk*}

Therefore the functor $g=F_\star^{\alpha}$ is always geometric in our situation of presheaves. By definition, this proves that $F_\alpha$
is a comorphism of sites. Consequently, as shown in \cite{MacLaneMoerdijk1992}, the pullback of subobjects defines
a natural transformation of presheaves over $\mathcal{F}_{U'}$:
\begin{equation}
\lambda'_\alpha:\Omega_U\rightarrow F_\star ^{\alpha}\Omega_{U'};
\end{equation}
which corresponds by the adjunction of functors $F_\alpha^{\star}\dashv F_\star^{\alpha}$, to a natural transformation of sheaves over $\mathcal{F}_{U}$:
\begin{equation}
\tau'_\alpha:F_\alpha^{\star}\Omega_U\rightarrow \Omega_{U'}.
\end{equation}
\begin{lem}\label{lem:fib-open}
	If $F_\alpha$ is a fibration (not necessarily in groupoids), it is an open morphism of sites, and the functor $F^{\alpha}_\star$ is
	open \cite{giraud-classifying}.
\end{lem}
\begin{proof}
	This results directly from \cite[Proposition 1, pp. 509-513]{MacLaneMoerdijk1992}. Precisely this proposition says
	that a morphism of sites $F:\mathcal{F}'\rightarrow \mathcal{F}$ induces an open geometric morphism $F_\star:{\sf Sh}(\mathcal{F}', J')\rightarrow {\sf Sh}(\mathcal{F}, J)$
	between the categories of sheaves, as soon as the following three conditions are satisfied:
	\begin{enumerate}[label=(\roman*)]
		\item $F$ has the property of lifting of the coverings:
		\begin{equation}
			\forall \xi'\in \mathcal{F}', \forall S\in J(F(\xi')),\exists T'\in J'(\xi'), F(T')\subseteq S;
		\end{equation}
		where $F(T')$ is the sieve generated by the images of the arrows in $T'$;
		\item $F$ preserves the covers, i.e.
		\begin{equation}
			\forall \xi'\in \mathcal{F}', \forall S'\in J'(\xi'), F(S')\in J(F(\xi'));
		\end{equation}
	\item for every $\xi'\in \mathcal{F}'$, the sliced morphism $F| \xi': \mathcal{F}'| \xi'\rightarrow \mathcal{F}| F(\xi')$ is surjective
	on the objects.
	\end{enumerate}
	\indent The two first conditions are true for the canonical topology of a stack \cite{giraud-classifying}. They are obvious in our case of presheaves. Condition (iii) is part of the definition of fibration (pre-fibration).
\end{proof}

If in addition $F$ itself is surjective on the objects, as it will be the case in our applications, the maps of
algebras $g^{\star}_X: {\rm Sub}(X)\rightarrow {\rm Sub}(f^{\star}X)$ are injective and the geometric open morphism $g=F_\star$ is surjective on the objects \cite[page 513]{MacLaneMoerdijk1992}.

\begin{lem}\label{lem:open-morphism}
	When $F_\alpha$ is a fibration, the relation between $\lambda_\alpha=\Omega_\alpha: \Omega_{U'}\rightarrow F_\alpha^{\star }\Omega_U$ and
	$\lambda'_\alpha: \Omega_U\rightarrow F_\star ^{\alpha}\Omega_{U'}$, is given by the adjunction of posets morphisms:\\
	\begin{equation}
		\Omega_\alpha\dashv \tau'_\alpha;
	\end{equation}
	where $\tau'_\alpha:F_\alpha^{\star }\Omega_U\rightarrow \Omega_{U'}$ is the dual of $\lambda'_\alpha$.\\
	The morphism $\Omega_\alpha$ is the left adjoint of the morphism $\tau'_\alpha$. Moreover, $\tau'_\alpha$
	is an injective section of the surjective morphism $\Omega_\alpha$.
\end{lem}
\begin{proof}
	If $F_\alpha$ is a fibration, $F_\alpha^{\star }\Omega_U$ is isomorphic to $\Omega_U$, it is the sub-algebra of $\Omega_{U'}$
	formed by the subobjects of ${\bf 1}_{U'}$ that are invariant by $F_\alpha$, i.e. by $\lambda_\alpha:\Omega_{U'}\rightarrow F_\alpha^{\star }\Omega_U$.\\
	The map $\tau'_\alpha$ associates to an element $P$ of $\Omega_U$ the element $P\circ F_\alpha$, seen as a sub-sheaf of $1_{U'}$, that is
	an element of $\Omega_{U'}$ saturated by $F_\alpha$. Therefore, for every $P'\in \Omega_{U'}$,  the element $\tau'_\alpha\circ \lambda_\alpha (P')$
	of $\Omega_{U'}$ is the saturation of $P'$, then it contains $P'$. This gives a natural transformation
	\begin{equation}
		\eta: {\sf Id}_{\Omega_{U'}}\rightarrow \tau'_\alpha\circ \Omega_\alpha.
	\end{equation}
	In the other direction, $\tau'_\alpha$ is a section over $\Omega_{U'}$ of the map $\lambda_\alpha$, i.e.
	$\Omega_\alpha\circ \tau'_\alpha=Id_{F_\alpha^{\star }\Omega_U}$. Which gives a natural transformation
	\begin{equation}
		\epsilon:  \Omega_\alpha\circ\tau'_\alpha \rightarrow {\sf Id}_{F_\alpha^{\star }\Omega_U}.
	\end{equation}
	In the following lines, we forget the indices $\alpha$ everywhere, and show that $\eta$ and $\epsilon$ are respectively the unit and counit of an
	adjunction of posets morphisms.\\
	Let $P'$ and $Q$, be respectively elements of $\Omega_{U'}$ and $\Omega_{U}$, if we have a morphism from $\lambda P'$
	to $Q$, by applying $\tau'$, we obtain a morphism from $\tau'\circ\lambda P'$ to $\tau'Q$, then a morphism from $P'$
	to $\tau'Q$. All that is equivalent to the following implications:
	\begin{equation}
		(\lambda P'\leq Q)\Rrightarrow (P'\leq \tau'\lambda P'\leq \tau'Q).
	\end{equation}
	In the other direction,
	\begin{equation}
		(P'\leq \tau'Q)\Rrightarrow (\lambda P'\leq \lambda\tau'P'\leq Q).
	\end{equation}
	Therefore
	\begin{equation}
		(P'\leq \tau'Q)\Lleftarrow\Rrightarrow (\lambda P'\leq Q).
	\end{equation}
	Which is the statement of lemma \ref{lem:open-morphism}.
\end{proof}

\noindent From the above lemmas, we conclude the following result (central for us):

\begin{thm}\label{thm:semanticflow}
When for each $\alpha:U\rightarrow U'$ in $\mathcal{C}$, the functor $F_\alpha$ is a fibration, the logical formulas and their truth in the topos propagate from $U$ to $U'$
by $\lambda'_\alpha$
(feedback propagation in the DNN),
and if in addition $F_\alpha$ is a morphism of groupoids (surjective on objects and morphisms), the logic in the topos also propagates from $U'$ to $U$, by $\lambda_\alpha$ (feed-forward functioning in the DNN).\\
Moreover, the map $\lambda_\alpha$ is the left adjoint of the transpose $\tau'_\alpha$ of the map $\lambda'_\alpha$.
And we have, for any $\alpha:U\rightarrow U'$ in $\mathcal{C}$,
\begin{equation}\label{strongstandard}
\lambda_\alpha\circ \leftindex^{t}\lambda'_\alpha={\sf Id}_{\Omega_{U'}}.
\end{equation}
\end{thm}
~

\begin{defn}\label{dfn:semanticflow}
When the conclusion of the above theorem holds true, even if the $F_\alpha$ are not fibrations, we say that the stack
$\pi:\mathcal{F}\rightarrow \mathcal{C}$ satisfies the \emph{strong standard hypothesis} (for logical propagation). Without the equation \eqref{strongstandard},
we simply say that the standard hypothesis is satisfied.
\end{defn}

\noindent In this case, the logic is richer in $U'$ than in $U$, like a fibration of Heyting algebras of subobjects of objects.\\

\indent To finish this section, let us describe the relation between the classifier $\Omega_\mathcal{F}$ and the classifier $\Omega_\mathcal{C}$ of the basis
category $\mathcal{C}$ of the fibration $\pi:\mathcal{F}\rightarrow \mathcal{C}$.\\
\indent As reminded above, proposition $2.1$ in \cite{giraud-cohomologie}, gives sufficient conditions for guarantying that the functor $\pi^{\star }$ is geometric.
But, even in the non-geometric case, when the fibers are groupoids, the morphism has locally (at the level of subobjects) the logical properties of an open geometric
morphism, (see lemmas \ref{bool-lat} and \ref{morph-heyt} ) and lemma \ref{lem:fib-open} says
that the functor $\pi_\star $, which is its right adjoint, is geometric and open. We can then apply lemma \ref{lem:open-morphism}, and get an
adjunction $\lambda_\pi\dashv \tau'_\pi$, where
\begin{equation}
\lambda_\pi:\Omega_\mathcal{F}\rightarrow \pi^{\star }\Omega_{\mathcal{C}},
\end{equation}
is a surjective morphism of lattices,
and
\begin{equation}
\tau'_\pi:\pi^{\star }\Omega_{\mathcal{C}}\rightarrow \Omega_{F},
\end{equation}
is the section by invariant objects.\\
\indent When $\pi$ is fibration of groupoids, $\pi^{\star }$ is open, and $\lambda_\pi$ is a morphism of Heyting algebras.
In this case, there exists a
perfect lifting of the theories in $\mathcal{C}$ to the theories in $\mathcal{F}$.\\

\section{Theories, interpretation, inference and deduction}\label{sec:theories}

\noindent Main references are Bell \cite{Bell}, Lambek and Scott \cite{Lambek1981IntuitionistTT}, \cite{lambek1988introduction} , MacLane and M\oe{rdijk} \cite{MacLaneMoerdijk1992}.\\

The formal languages, that we will mainly consider, are the typed languages of type theory, in the sense of Lambek and Scott \cite{Lambek1981IntuitionistTT}. In particular, in
such a type theory we have a notion of deduction, conditioned by a set $S$ of propositions, named axioms, which is denoted by $\vdash_S$. This is
a relation between two propositions, $P\vdash_SQ$, which satisfies the usual axioms, structural, logical, and set theoretical, also named rules of inference,
of the form
\begin{equation}
(P_1\vdash_SQ_1,P_2\vdash_SQ_2,...,P_n\vdash_SQ_n)/P\vdash_SQ,
\end{equation}
meaning that the truth (or validity) of the left (said upper) conjunction of deductions implies the truth of the right deduction (said lower).\\
The conditional validity of a proposition $R$ is noted $\vdash_SR$.\\
\indent A (valid) proof of $\vdash_SR$ is an oriented classical graph without oriented cycles,  whose vertices are labelled by
valid inferences,
and whose oriented edges are identifying one of the upper terms of its final extremity to the lower term of its initial extremity, and having
only one final vertex whose lower term is $\vdash_SR$. The initial vertices have
left terms that are empty or belonging to the set $S$. \\
A theory $\mathbb{T}$ in a formal language $\mathbb{L}$ is the set of propositions that can be asserted to be true if
some axioms are assumed to be true, this means that these propositions are deduced by valid proofs from the axioms.\\

A language $\mathbb{L}$ is interpreted in a topos $\mathcal{E}$ when some objects of $\mathcal{E}$ are associated to every type, the
object $\Omega_\mathcal{E}$ corresponding to the logical type $\Omega_{\mathbb{L}}$,
when some arrows $A\rightarrow B$ are associated to the variables (or terms) of $B$ in the context $A$, all that being compatible with the respective
definitions of products, subsets, exponentials, singleton, changes of contexts (substitutions), and logical rules, including the predicate
calculus, which includes the two projections (existential and universal) on the side
of topos \cite{Bell}, \cite{Lambek1981IntuitionistTT}.\\
\indent A theory $\mathbb{T}$ is represented in $\mathcal{E}$ when all its axioms are true in $\mathcal{E}$. The fact that all the deductions
are valid in $\mathcal{E}$ is the statement of the \emph{soundness theorem} of $\mathbb{T}$ in $\mathcal{E}$.
\begin{rmk*}
	\normalfont The \emph{completeness theorem} says that, for any language and any theory, there exists
	a minimal "elementary topos" $\mathcal{E}_{\mathbb{T}}$, which in general is not a Grothendieck topos, where the converse of the soundness
	theorem is true; validity
	in $\mathcal{E}_{\mathbb{T}}$  implies validity in $\mathbb{T}$. The different interpretations in a topos $\mathcal{E}$
	of a theory $\mathbb{T}$ form a category $\mathcal{M}(\mathbb{T},\mathcal{E})$, which is equivalent to the category of "logical functors"
	from $\mathcal{E}_{\mathbb{T}}$ to  $\mathcal{E}$. This equivalence needs precisions given by Lambek and Scott, in particular to fix representant
	of subobjects, which is automatic in a Grothendieck topos.
\end{rmk*}

\noindent As suggested by Lambek, an interpretation of a type theory in a topos constitutes a \emph{semantic} of this theory. \\

\noindent If a formal language $\mathbb{L}$ can be interpreted in a topos $\mathcal{E}$, and if $F:\mathcal{E}\rightarrow\mathcal{F}$ is a left exact
functor from $\mathcal{E}$ to a topos $\mathcal{F}$, the interpretation is transferred to $\mathcal{F}$. The condition for transporting any
theory $\mathbb{T}$ by $f$ is that it admits a right adjoint $f:\mathcal{F}\rightarrow \mathcal{E}$ which is geometric and open.\\
A geometric functor allows the transportation of the restricted family of geometric theories as in \cite{caramello2009duality}, \cite{caramello-18} or \cite{MacLaneMoerdijk1992}.
\begin{rmk*}
	\normalfont If $\mathbb{T}$ is a geometric theory, there is a Grothendieck topos $\mathcal{E}'_{\mathbb{T}}$
	which classifies the interpretations of $\mathbb{T}$, i.e. for every Grothendieck topos $\mathcal{E}$ the category
	of geometric functors from $\mathcal{E}$ to $\mathcal{E}'_{\mathbb{T}}$ is equivalent to $\mathcal{M}(\mathbb{T},\mathcal{E})$ \cite{caramello2009duality}, \cite{caramello-18},\cite{MacLaneMoerdijk1992}. A {\em logical functor} is the left adjoint of a {\em geometric functor}.
\end{rmk*}

\indent In many applications of $DNNs$, a network has to proceed to a semantic analysis of some data. Our aim now is to precise what this means, and how
we, observers, can have access to the internal process of this analysis.\\
\indent As before, the network is presented as a dynamic object $\mathbb{X}$ in a topos, with learning object of weights $\mathbb{W}$, and the considered topos $\mathcal{E}$ is the classifying
topos of a fibration $\pi:\mathcal{F}\rightarrow\mathcal{C}$.\\
\indent In the applications, the logic is richer in $U'$ than in $U$ when there is a morphism
$\alpha:U\rightarrow U'$ in $\mathcal{C}$.  We suppose given a family of typed language $\mathbb{L}_U;U\in \mathcal{C}$, interpreted in the
topos $\mathcal{E}_U;U\in \mathcal{C}$ of the corresponding layers.\\
We say that the functors $f=g^{\star }=F_\alpha^{\star }$ propagate these languages
backward, when for each morphism $\alpha:U\rightarrow U'$ in $\mathcal{C}$, there exists a natural transformation
\begin{equation}
\mathbb{L}_\alpha:\mathbb{L}_{U'}\rightarrow F_\alpha^{\star }\mathbb{L}_U,
\end{equation}
which extends $\Omega_\alpha=\lambda_\alpha$, implying that the types define objects or morphisms in $\mathcal{E}$,
in particular $0_U$, $1_U$.\\
And we say that the left adjoint functor $f^{\star }$ propagates the languages feed-forward, when for each morphism
$\alpha:U\rightarrow U'$ in $\mathcal{C}$, there exists a natural transformation
\begin{equation}
\mathbb{L}'_\alpha:\mathbb{L}_{U}\rightarrow F^{\alpha}_{\star }\mathbb{L}_{U'},
\end{equation}
which extends $\lambda'_\alpha$, implying that the types define objects or morphisms in the fibration $\mathcal{E}'$,
defined by the right adjoint functors $F^{\alpha}_{\star }$.\\

We assume that the standard hypothesis \ref{dfn:semanticflow} is satisfied for the extensions
$\mathbb{L}_\alpha$ and $\mathbb{L}'_\alpha$.\\

\noindent Note that in the case of stacks of DNNs, there exist two kinds of functors $F_\alpha:\mathcal{F}_{U'}\rightarrow\mathcal{F}_{U}$ over $C$, the ordinary ones, flowing from the input to the output, and the added canonical projections from the fiber at a fork $A$ to the fibers of their tines $a'$, $a"$, .... The second kind of functors are
canonically fibrations, but for the other functors, this is a condition we can require for a good semantic functioning (see theorem \ref{thm:semanticflow}).\\

\noindent Let $\mathbb{L}$ denote the corresponding presheaf in languages over $\mathcal{C}$, $\Omega_{\mathbb{L}}$ its logical type, and
for each $U\in \mathcal{C}$, we note $\Omega_{\mathbb{L}_U}$ the value of this logical type at $U$.
For each $U\in \mathcal{C}$, we write $\Theta_U$ the set of possible sets of axioms in $\mathbb{L}_U$, that is
$\Theta_U=\mathcal{P}(\Omega_{\mathbb{L}_U})$. This is also the set of theories.\\

\noindent We take as output ({\em resp.} input) the union of the output ({\em resp.} output) layers. In supervised and reinforcement learning,
we can tell that, for every input $\xi_{\sf in}\in \Xi_{\sf in}$ in a set of inputs for learning,
a theory $\mathbb{T}_{\sf out}(\xi)$ in $\mathbb{L}_{\sf out}$ is imposed at the output of the network., i.e. some propositions are asked to be true, other are
asked to  be false.\\
The set of theories in the language $\mathbb{L}_{\sf out}$ is denoted $\Theta_{\sf out}$. Then the objectives of the functioning is a map
$\mathbb{T}_{\sf out}:\Xi_{\sf in}\rightarrow \Theta_{\sf out}$.

\begin{defn*}
	A \emph{semantic functioning} of the dynamic object $X^{w}$ of possible activities in the network,
	with respect to the mapping $\mathbb{T}_{\sf out}$, is a family of quotient
	sets $D_U$ of $X_U^{w}$, $U\in \mathcal{C}$, equipped with a map $S_U: D_U\rightarrow \Theta_U$, such that for every
	$\xi_{\sf in}\in \Xi_{\sf in}$ and every $U\in \mathcal{C}$,
	the image
	$S_U(\xi_U)$ generates a theory which is coherent with $\mathbb{T}_{\sf out}\left(\xi_{\sf in}\right)$, for the transport in both directions along any path.
\end{defn*}

\begin{rmk*}
	\normalfont In the known applications, the richer logic relies on a richer language with more propositions and less axioms, present near the input layers, but the opposite happens to expressed theories; they are more constrained
	in the deepest layers, with more axioms in general.
\end{rmk*}

\indent In the examples we know \cite{logic-DNN}, the quotient $D_U$ (from \emph{discretized cells})
is given by the activity of some special neurons in the layer $L_U$, which saturate at a finite number of values,
associated to propositions in the Heyting algebras $\Omega_{\mathbb{L}_U}$. In this case, the definition of semantic functioning can be
made more concrete: for each neuron $a\in L_U$, each quantized value of activity $\epsilon_a$ implies the validity of a proposition $P_a(\epsilon_a)$
in $\Omega_{\mathbb{L}_U}$; this defines the map $S_U$. Then the definition of semantic functioning asks that, for each input $\xi_{\sf in}\in \Xi_{\sf in}$, the generated activity defines
values $\epsilon_a(\xi_{\sf in})$ of the special neurons, such that the generated set of propositions $P_a(\epsilon_a)$, implies
the validity of a given proposition in $\Omega_{\mathbb{L}_{\sf out}}$, which is valid for $\mathbb{T}_{\sf out}(\xi_{\sf in})$.\\
\indent In particular, we saw experimentally that the inner layers understand the language $\mathbb{L}_{\sf out}$, which is an
indication that the functors $f=g^{\star }=F_\alpha^{\star }$ propagate the languages backward.\\
\indent This gives a crude notion of \emph{logical information} of a given layer, or any subset $E$ of neurons in the union of the sets $D_U$: it is
the set of propositions predicted to hold true in $\mathbb{T}_{\sf out}(\xi_{\sf in})$ by the activities in $E$. If all the involved sets are finite,
the amount of information given by the set $E$ can be defined as the ratio of the number of predicted propositions over the number of wanted decisions, and a mean of this
ratio can be taken over the entries $\xi_{\sf in}$.\\
\begin{rmk*}
	\normalfont The above notion of semantic functioning and semantic information can be extended to sets of global activities $\Xi$, singletons sections
	of $X^{w}$, more general that the ones used for learning.
\end{rmk*}

Our experiments  in \cite{logic-DNN} have shown that the number of hidden layers, or the complexity of the architecture, strongly influences the nature
of the semantic functioning. This implies that the semantic functioning, then the corresponding accessible semantic information, depend on the characteristics of
the dynamic $X^{w}$, for instance the non-linearities for saturation and quantization, and of the characteristics of the learning, the influence of
the non-linearities of the gradient of backpropagation on the optimal weights $w\in W$. Therefore, it appears a notion of \emph{semantic learning},
which is a flow of natural transformations between dynamic objects $X^{w_t}$, increasing the semantic information.\\
\indent In the mentioned experiments, the semantic behavior appears only for sufficiently deep networks, and for non-linear activities.\\

\section{The model category of a DNN and its Martin-Löf type theory}\label{modelML}

In this section, we study the collection of stacks over a given layers architecture, with fibers in a given category, as groupoids,
and we show that it possesses a natural structure of closed model category of Quillen, giving both a theory of homotopy and an
intensional type theory, where the above stacks with free logical propagation, described by theorem \ref{thm:semanticflow}, correspond respectively to fibrant objects and
admissible contexts.\\

Consider two fibrations $(\mathcal{F}_U,F_\alpha)$ and $(\mathcal{F}'_U,F'_\alpha)$ over $\mathcal{C}$; a morphism $\varphi$ from the first to the second
is given by a collection of functors $\varphi_U:\mathcal{F}_U\rightarrow \mathcal{F}'_U$ such that for any arrow $\alpha:U\rightarrow U'$ of $\mathcal{C}$,
$\varphi_U\circ F_\alpha=F'_\alpha\circ \varphi_{U'}$. With the fibrations in groupoids, this gives a category ${\sf Grpd}_{\mathcal{C}}$. Natural transformations between two morphisms give it a structure of strict $2$-category.\\
\indent We consider this category fibred over $\mathcal{C}_\mathbf{X}$. Remind that the Grothendieck topology on $\mathcal{C}_\mathbf{X}$ that we consider is chaotic \cite{SGA4}.
If we consider an equivalent site, with a non-trivial topology, homotopical constraints appear for defining stacks \cite{giraud-classifying}, \cite{Hollander}.
However the category of stacks ({\em resp.} stacks in groupoids) is equivalent to the category obtained from $\mathcal{C}_\mathbf{X}$.\\
\indent Hofmann and Streicher \cite{HS}, have proved that the category ${\sf Grpd}$ of groupoids gives rise to a Martin-Löf type theory \cite{ML}, by taking for types the fibrations in groupoids, for terms their sections, for substitutions the pullbacks, and they have defined non-trivial (non-extensional) identity types in this theory.

\indent Hollander \cite{Hollander2001AHT}, \cite{Hollander}, using Giraud's work and homotopy limits, constructed a Quillen model theory on the category of fibrations ({\em resp.} stacks) in groupoids over any site $\mathcal{C}$, where the fibrant objects are the stacks, the cofibrant objects are generators,
and the weak equivalences are the homotopy equivalence in the fibers (see also Joyal-Tierney and Jardine cited in Hollander \cite{Hollander}).
These results were extended to the category of general stacks, not only in groupoids, over a site by Stanculescu \cite{Stanculescu}.

\indent  Awodey and Warren \cite{AWODEY_2009} observed that the construction of Hofmann-Streicher is based on the most natural closed model category structure in the sense of Quillen on ${\sf Grpd}$,
and proposed an extension of the construction to more general model categories. Thus they established the connection between Quillen's models and Martin-Löf intensional theories,
which was soon extended to a connection between more elaborate Quillen's models and Voedvosky univalent theory. 

\indent Arndt and Kapulkin, in \emph{Homotopy Theoretic Models of Type Theory} \cite{AK}, have proposed additional axioms on a closed model theory that are sufficient to formally deduce a Martin-Löf theory. This was extended later by Kapulkin and Lumsdaine \cite{KLV}, to obtain models of Voedvosky theory,
by using more simplicial techniques. Here, we will follow their approach, without
going to the special properties of HoTT, that are functions extensionality, Univalence axiom and Higher inductive type formations.\\

\indent In what follows, we focus on the model structure of groupoids and stacks in groupoids, which are the most useful models for our applications.
However, many things also work with ${\sf Cat}$ in place of ${\sf Grpd}$, and  some other model categories $\mathcal{M}$. The complication is due to the difference between fibrations ({\em resp.} stacks) in the sense of Giraud and Grothendieck and the fibrations in the sense of Quillen's models, which is not the case with groupoids.
For ${\sf Cat}$, there exists a unique closed model structure, defined by Joyal and Tierney, such that the weak equivalences are the equivalence of categories \cite{CSP-2012}. It is named for this reason the \emph{canonical model structure on ${\sf Cat}$}; in this structure, the cofibrations are the functors injective on objects and
the fibrations are the \emph{isofibrations}. An isofibration is a functor $F:\mathcal{A}\rightarrow \mathcal{B}$, such that every isomorphism of $\mathcal{B}$ can be lifted to an isomorphism of $\mathcal{A}$. Any fibration of category is an iso-fibration, but the converse is true only for groupoids. A different model theory was defined by Thomason \cite{CTGDC_1980__21_3_305_0}, which is better understandable in terms of $\infty$-groupoids
and $\infty$-categories.

The axioms of Quillen \cite{quillen1967homotopical} concern three subsets
of morphisms in a category $\mathcal{M}$, supposed to be (at least finitely) complete and cocomplete, the set ${\sf Fib}$ of fibrations, the set ${\sf Cofib}$ of cofibrations
and the set ${\sf WE}$ of weak equivalences. An object $A$ of $\mathcal{M}$ is said fibrant ({\em resp.} cofibrant) if $A\rightarrow {\bf 1}$, the final object ({\em resp.} $\emptyset\rightarrow A$ from the initial object) is a fibration
({\em resp.} a cofibration).
\begin{defns*}
	Two morphisms $i:A\rightarrow B$ and $p:C\rightarrow D$ in a category are said \emph{orthogonal}, written (non-traditionally)
	$i\rightthreetimes p$, if for any pair of morphisms $u:A\rightarrow C$ and $v:B\rightarrow D$, such that $p\circ u=v\circ i$, there exists a morphism
	$j:B\rightarrow C$ such that $j\circ i=u$ and $p\circ j=v$. The morphism $j$ is named a lifting, left lifting of $i$ and a right lifting of $p$.\\
	\indent Two sets $\mathcal{L}$ and $\mathcal{R}$ are said be the orthogonal one of each other if $i\in  \mathcal{L}$ is equivalent to $\forall p\in \mathcal{R},
	i\rightthreetimes p$ and $p\in \mathcal{R}$ is equivalent to $\forall i\in  \mathcal{L},i\rightthreetimes p$.
\end{defns*}

\noindent The three axioms of Quillen for a closed category $\mathcal{M}$ of models are:
\begin{enumerate}[label=\arabic*)]
	\item given two morphisms $f:A\rightarrow B$, $g:B\rightarrow C$, define $h=g\circ f$; if two of the morphisms $f,g,h$ belong to ${\sf WE}$, then
	the third one belongs to ${\sf WE}$;
	\item every morphism $f$ is a composition $f=p\circ i$ of an element $p$ of ${\sf Fib}$ and an element $i$ of ${\sf Cofib}\cap {\sf WE}$,
	and a composition $p'\circ i'$ of an element $p'$ of ${\sf Fib}\cap {\sf WE}$ and an element $i'$ of ${\sf Cofib}$;
	\item the sets ${\sf Fib}$ and  ${\sf Cofib}\cap {\sf WE}$ are the orthogonal one of each other and the sets ${\sf Fib}\cap {\sf WE}$ and  ${\sf Cofib}$ also.
\end{enumerate}

\noindent An element of ${\sf Fib}\cap {\sf WE}$ is named a \emph{trivial fibration}, and an element of ${\sf Cofib}\cap {\sf WE}$ is named a \emph{trivial cofibration}.\\

These axioms (and some more general) allowed Quillen to develop a convenient homotopy theory in $\mathcal{M}$, and to define a
homotopy category $Ho\mathcal{M}$ (see his book, \emph{Homotopical Algebra}, \cite{quillen1967homotopical}). The objects of $Ho\mathcal{M}$ are the fibrant
and cofibrant objects of $\mathcal{M}$,
and its morphisms are the homotopy classes of morphisms in $\mathcal{M}$; two morphisms $f,g$ from $A$ to $B$ are homotopic if there exists an object $A'$, equipped
with a weak equivalence $\sigma: A'\rightarrow A$ and two morhisms $i_0,i_1$ from $A$ to $A'$ such that $\sigma\circ i_0=\sigma\circ i_1$, and
a morphism $h:A'\rightarrow B$, such that $h\circ i_0=f$ and $h\circ i_1=g$. In the category $Ho\mathcal{M}$, the weak equivalences of $\mathcal{M}$ are inverted.\\

\noindent A particular example is the category of sets with surjections as fibrations, injections as cofibrations and all maps as equivalences. Another
trivial structure, which exists for any category is no restriction for ${\sf Fib}$ and ${\sf Cofib}$ but isomorphisms for ${\sf WE}$.\\
As we already said, an important example is the category of groupoids ${\sf Grpd}$, with the usual fibrations in groupoids, with all the functors injective on the objects as cofibrations,
and the usual homotopy equivalence (i.e. here category equivalence) as weak equivalences.\\
We also mentioned the canonical structure on ${\sf Cat}$, that is the only one where weak homotopy corresponds to the usual equivalence of category.\\
Other fundamental examples are the topological spaces ${\sf Top}$ and the simplicial sets ${\sf SSet}=\Delta^{\wedge}$, with Serre and Kan fibrations for $\sf Fib$ respectively.\\
The closed model theory of Thomason 1980 \cite{CTGDC_1980__21_3_305_0} on ${\sf Cat}$ is deduced by the above structure on ${\sf SSet}$, by using the nerve construction
and the square of the right adjoint functor f the barycentric subdivision. In this structure the weak equivalences are not reduced to the category equivalences
and the cofibrant objects are constrained \cite{cisinski2006prefaisceaux}; this theory is weakly equivalent to the Kan structure on ${\sf SSet}$.
Then in this structure, a category is considered through its weak homotopy type (the weak homotopy type of its nerve).\\

\indent We now call on a general result of Lurie 's book, \cite[appendix A.2.8, prop. A.2.8.2]{lurie2009higher}, which establishes the existence of two canonical
closed model structures on the category of functors $\mathcal{M}_\mathcal{C}={\sf Fun}(\mathcal{C}^{\sf op},\mathcal{M})$ when $\mathcal{M}$ is a model category.
(Caution, Lurie consider diagrams, i.e. $\mathcal{C}$ and not $\mathcal{C}^{\rm op}$.) An additional hypothesis is made on $\mathcal{M}$, that it is combinatorial
in the sense of Smith (see Rosicky in \cite{Rosick2009OnCM}), i.e. locally presentable (i.e. accessible by a regular cardinal), and generated by cofibrant objects,
which are both satisfied by ${\sf Grpd}$ and by ${\sf Cat}$. Moreover $\mathcal{M}$ is supposed to have all small limits and small colimits,
which is also the case for ${\sf Grpd}$ (or ${\sf Cat}$); as ${\sf Set}$, both are cartesian closed categories; every object is fibrant and cofibrant.\\
\indent The two Lurie structures are respectively obtained by defining the sets ${\sf Fib}$ or ${\sf Cofib}$ in the fiberwise manner, as for the set ${\sf WE}$,
and by taking respectively the set ${\sf Cofib}$ or ${\sf Fib}$ of morphism which satisfy the required lifting properties, respectively on the left and on the right, i.e. the orthogonality of Quillen.\\
The structure obtained by fixing ${\sf Fib}$ ({\em resp.} ${\sf Cofib}$) by the behavior in the fibers, is named the \emph{projective} structure, or \emph{right}
one ({\em resp.} the \emph{injective} one, or \emph{left} one). 

{\bf Caution:} depending on the authors, the term right and left may be exchanged.\\
The model structure of Hollander on ${\sf Grpd}_{\mathcal{C}}$ (or Stanculescu for ${\sf Cat}_{\mathcal{C}}$) is the right Lurie model. She called this
model a left model.\\
\indent A model category is said \emph{right proper} when the pullback of any weak equivalence
along an element of $Fib$ is again a weak equivalence. Dually, \emph{left proper} is when push-forward of weak equivalence along
cofibrations is again in ${\sf WE}$.\\
\noindent In the right proper case, the injective (left) structure of Lurie was defined before by D-C. Cisinski in
"{\em Images directes cohomologiques dans les catégories de modèles}" \cite{AMBP_2003__10_2_195_0}.\\

\noindent The cofibrations in the right model ({\em resp.} the fibrations in the left model) depend on the category $\mathcal{C}$. They certainly deserve to be better
understood. \\
See the discussion of Cisinski, in his book {\em Higher Categories and Homotopical Algebra}, \cite[section $2.3.10$]{cisinski_2019}.

\begin{prop}\label{prop:fib-cofib}
	If $\mathcal{C}$ has sufficiently many points, the elements of ${\sf Fib}$ for the left Lurie structure are fibrations in the fibers (i.e. elements of
	${\sf Fib}$ for the right structure) and the elements of ${\sf Cofib}$ for the right structure are injective on the objects in the fibers (i.e. elements of ${\sf Cofib}$ for the left structure.
\end{prop}
\begin{proof}
	Suppose that a morphism $\varphi$ is right orthogonal to any trivial cofibration $\psi$ of the left Lurie structure; for every point $x$ in $\mathcal{C}$, this
	gives an orthogonality in the model ${\sf Grpd}$, then over $x$, $\varphi_x$ induces a fibration in groupoids. From the hypothesis, this implies that in every fiber over $\mathcal{C}$, $\varphi$ is a fibration, then an element of ${\sf Fib}$ for the right Lurie structure.\\
	The other case is analog.
\end{proof}

\noindent However in general, even if $\mathcal{C}$ is a poset, not all fibrations in the fibers are in ${\sf Fib}$ for the left model structure, and not all the injective in fibers
are in ${\sf Cofib}$ for the right model. This was apparent in Hollander \cite{Hollander2001AHT}.\\

\noindent Trying to determine the obstruction for a local fibration ({\em resp.} local cofibration) to be orthogonal to functors that are locally injective on the objects
({\em resp.} local fibrations) and locally homotopy equivalence, we see that 
the intuitionistic structure of $\Omega_{\mathcal{C}}$ enters the game, through the global constraints
on the complement of presheaves:
\begin{lem}
	The category $\mathcal{C}$ being the oriented segment $1\rightarrow 0$ and the category $\mathcal{M}$ being ${\sf Set}$
	(then $\mathcal{M}_\mathcal{C}$ is the topos of the Shadoks \cite{proute-shadok}); in the left Lurie model the fibrant objects are the (non-empty) surjective maps $f:F_0\rightarrow F_1$.
\end{lem}
\begin{proof}
	A trivial cofibration is a natural transformation
	\begin{equation}
		\eta:(h:H_0\rightarrow H_1)\rightarrow (h':H'_0\rightarrow H'_1);
	\end{equation}
	such that $\eta_0$ and $\eta_1$ are injective.\\
	Suppose given a natural transformation $u=(u_0,u_1)$ from $h$ to 
	$f:F_0\rightarrow F_1$; the lifting problem is the extension of $u$
	to $u'$ from $h'$ to $f$. If $H_1$ is empty, there is no problem. If not, we choose a point $\star_0$ in $H_0$ and note $\star_1=h(\star_0)$.
	If  $x'_1\in H'_1$ does'nt belong to $H_1$ we define $u'_1(x'_1)=u_1(\star_1)$, and for any $x'_0$ such that $h'(x'_0)=x'_1$, we define
	$u'_0(x'_0)=u_0(\star _0)$. Now the problem comes with the points $x"_0$ in $H'_0\backslash H_0$ such that $h'(x"_0)\in H_1$ (a shadok with an egg);
	their image by $u_1$ is defined, then $u'_1(h'(x"_0))$ is forced to be in the image of $F_0$ by $f$. If $f$ is not surjective there exists
	$\eta$ such that the lifting is impossible. But, if $f$ is surjective there is no obstruction: we define $u'_0(x"_0)$ to be any point $y_0$
	in $F_0$ such that $f(y_0)=u_1(h'(x"_0))$ in $F_1$.
\end{proof}
\begin{lem}
	Also $\mathcal{M}={\sf Set}$, but $\mathcal{C}$ being the (confluence) category $\bigwedge$ with three objects $0,1,2$ and two non-trivial
	arrows $1\rightarrow 0$ and $2\rightarrow 0$. In the left Lurie model, the fibrant objects are the pairs $(f_1:F_0\rightarrow F_1, f_2:F_0\rightarrow F_2)$,
	such that the product map $(f_1,f_2)$ is surjective.
\end{lem}
\begin{proof}
	Following the path of the preceding proof, with an injective transformation $\eta$ from a triple $H_0,H_1,H_2$ to a triple $H'_0,H'_1,H'_2$,
	we are in trouble with the elements $x"_0\in H'_0$ that $h'_1$ or $h'_2$ sends into $H_1$ or $H_2$ respectively. Under the hypothesis of bi-surjectivity, we know
	where to define $u'_0(x"_0)$. But if this hypothesis is not satisfied, impossibility happen in general for $\eta$.
\end{proof}

\begin{lem}\label{lem:012}
	Also $\mathcal{M}={\sf Set}$, but $\mathcal{C}$ being the (divergence) category $\bigvee$ with three objects $0,1,2$ and two non-trivial
	arrows $0\rightarrow 1$ and $0\rightarrow 2$. In the left Lurie model, the fibrant objects are the pairs $(f_1:F_1\rightarrow F_0, f_2:F_2\rightarrow F_0)$,
	such that separately $f_1$ and $f_2$ are surjective.
\end{lem}
\begin{proof}
	following the path of the preceding proof, with an injective transformation $\eta$ from a triple $H_0,H_1,H_2$ to a triple $H'_0,H'_1,H'_2$,
	we are in trouble with the elements $x"_1\in H'_1$ ({\em resp.} $x"_2\in H'_2$) that $h'_1$ ({\em resp.} $h'_2$) sends into $H_0$. As in the proff of the lemma 1, the problem is solved under
	the hypothesis of surjectivity, but it cannot be solved without it.
\end{proof}

\noindent More generally, we can determine the fibrant objects of the left Lurie model (injective) for every closed model category $\mathcal{M}$, and a finite poset $\mathcal{C}$
which has the structure of a DNN, coming with a graph, with unique directed paths:

\begin{thm}\label{thm:fibrations-dnn}
	When $\mathcal{C}$ is the poset of a $DNN$, for any combinatorial category of model, the fibrations of $\mathcal{M}_\mathcal{C}$ for
	the injective (left) model structure are made by the natural transformations $\mathcal{F}\rightarrow\mathcal{F}'$ between functors in $\mathcal{C}$ to $\mathcal{M}$,
	that induce fibrations in $\mathcal{M}$ at each object of $\mathcal{C}$, such that the functor $\mathcal{F}$ is also a fibration in $\mathcal{M}$ along each arrow
	of $\mathcal{C}$ coming from an internal of minimal vertex (ordinary vertex, output or tip), and
	a fibration along each of the arrows issued from a minimal vertex (output and tip), and a multi-fibration at each confluence point, in particular
	at the maximal vertices (input or tank).
\end{thm}

\noindent By \emph{multi-fibration} $f_i,i\in I$ from an object $F_A$ of $\mathcal{M}$ to a family of objects $F_i,i\in I$ of $\mathcal{M}$, we mean a fibration
(element of ${\sf Fib}$) from $F_A$ to the product $\prod_{i\in I}F_i$.
\begin{proof}
	We proceed by recurrence on the number of vertices. For an isolated vertex, this is the
	definition of fibration in $\mathcal{M}$. Then consider an initial vertex (tank or input) $A$ with incoming arrows $s_i:i\rightarrow A$ for $i\in I$ in the graph poset
	$\mathcal{C}$, and note $\mathcal{C}^{\star }$ the category with the star $A,s_i$ deleted. A trivial cofibration in $\mathcal{M}_\mathcal{C}$ is a natural transformation
	$\eta;\mathcal{H}\rightarrow \mathcal{H}'$ between contravariant functors in $\mathcal{C}\rightarrow\mathcal{M}$, which is at each vertex injective on objects
	and an element of $\mathcal{WE}$. Let us consider a morphism $(u,u')$ in in $\mathcal{M}_\mathcal{C}$ from $\eta$ to a morphism
	$\varphi:\mathcal{F}\rightarrow \mathcal{F}'$, where $\mathcal{F}$ belongs to $\mathcal{M}_\mathcal{C}$.\\
	\indent Suppose that $\varphi$ satisfies the hypotheses of the theorem. From the recurrence hypothesis, there exists a lifting
	$\theta^{\star }:(\mathcal{H}')^{\star }\rightarrow\mathcal{F}^{\star }$ between the restrictions of the functors to $\mathcal{C}^{\star }$; it is in particular defined on the
	objects $H'_i,i\in I$ to the objects $F_i,i\in I$.\\
	Consider the functor from $H'_A$ to the product $\prod_iF_i$, which is obtained by composing the horizontal arrows of $\eta$, from $H'_A$ to the
	product $H'=\prod_iH'_i$ with $\theta'$. The fact that $F_A\rightarrow \prod_iF_i$ is a multi-fibration in $\mathcal{M}$ and the fact that
	$\eta_A:H_A\rightarrow H'_A$ is a trivial
	cofibration in $\mathcal{M}$ imply the existence of a lifting $\theta_A:H'_A\rightarrow F_A$, which is given on $H_A$.\\
	\indent Conversely, if the hypothesis of multi-fibration is not satisfied, there exists elements $\eta_A:H_A\rightarrow H'_A$ in ${\sf Cofib}\cap {\sf WE}$ of $M$,
	such that the lifting $\theta_A$ of $H'_A$ to $F_A$ does'nt exist, by the axiom $(3)$ of closed models. To finish the proof, we note that the necessity
	to be a fibration at each vertex in $C$ is given by proposition \ref{prop:fib-cofib}.
\end{proof}
\begin{cor*}
	Under the same hypotheses, the fibrant objects of $\mathcal{M}_\mathcal{C}$ for
	the injective (left) model structure are made by the functors that are a fibration in $\mathcal{M}$ at each internal of minimal vertex (ordinary vertex, output or tip), and
	a fibrant object at the minimal (output and tip), and a multi-fibration at each confluence point (see lemma \ref{lem:012}), in particular at the maximal vertices (input or tank).
\end{cor*}

\noindent One interest of this result is that it will describe the allowed contexts in the associated Martin-Löf theory when it exists, as we will see just below.\\
Another interest is for the behavior of the classifying object $\boldsymbol{\Omega}_\mathcal{F}$: in the case of ${\sf Grpd}_\mathcal{C}$ the fibrant objects are all good
for the induction theory in logic over the network (see theorem \ref{thm:semanticflow}). In the case of ${\sf Cat}_\mathcal{C}$, with the canonical structure, we will see below that
it is not the case, only a subclass of fibrant objects are good, which are made by composition of Giraud-Grothendieck fibrations.\\
Last by not least, this corollary allows to enter the homotopy theory of the stacks, according to Quillen \cite{quillen1967homotopical}, because it associates
objects up to homotopy with the stacks that have a fluid semantic functioning as in theorem \ref{thm:semanticflow}.\\

\noindent In ${\sf Grpd}_{\mathcal{C}}$ the final object ${\bf 1}$ ({\em resp.} the initial object $\emptyset$) is the constant
functor on $\mathcal{C}$ with values a singleton, ({\em resp.} the empty set). It follows that any object is cofibrant.\\

\indent The additional axioms of Arndt and
Kapulkin for a \emph{Logical Model Theory} are as follows:
\begin{enumerate}[label=(\arabic*)]
	\item\label{item:right-adjoint} for any element $f\in {\sf Fib}$, $f:B\rightarrow A$, the pullback functor $f^{\star }:\mathcal{M}|A\rightarrow \mathcal{M}|B$, once restricted to
	the fibrations, possesses a right-adjoint, denoted $\Pi_f$.
	\item\label{item:pullback} The pullback of a trivial cofibration, i.e. an element of ${\sf Cofib}\cap {\sf WE}$, along an element of ${\sf Fib}$ is again a trivial cofibration.
\end{enumerate}
\begin{rmk*}
	\normalfont In Arndt and Kapulkin \cite{AK}, the first axiom is written without the restriction of the adjunction to fibrations, however they remark later \cite[section 4.1, acknowledging an anonymous reviewer]{AK} that this restricted axiom is sufficient for the application below.
\end{rmk*}

The second axiom is satisfied if separately ${\sf Cofib}$ and ${\sf WE}$ are stable by pullback along a fibration. As we already said, a model category satisfying the second
property for ${\sf WE}$ is called \emph{right proper}.\\

When every object in $\mathcal{M}$ is fibrant ({\em resp.} cofibrant) the theory is right ({\em resp.} left) proper \cite{Hirschhorn2003}. This is the case for ${\sf Grpd}$ (and ${\sf Cat}$).
And Lurie proved that his two model structures on diagrams (or phe-sheaves) are right (reps. left) proper as soon as $\mathcal{M}$ is so. Then in our case, all the considered
models are right proper and left proper. This was shown by Hollander \cite{Hollander2001AHT} for ${\sf Grpd}_\mathcal{C}$.\\

The injectivity on objects in the fibers and the equivalence of category in the fibers are preserved
by every pullback, thus condition \ref{item:pullback} is satisfied for the left injective structure. This is the structure we choose. What happens to the right structure? \\
Arndt and Kapulkin noticed the example of the injective structure \cite[Prop. 27, p.12]{AK} and its Bousfield-Kan localizations;
this gives in particular the injective model structures for the category of stacks over any site (see Hirschhorn, \emph{Localization of Model Categories}
\cite{Hirschhorn2003}).\\

The existence of a right adjoint and a left adjoint of the pullback of fibrations
in categories, as it holds for presheaves of sets, was proved by Giraud in 1964 \cite[section I.2.]{giraud-descente}. 

Then, by proposition \ref{prop:fib-cofib}, for $\mathcal{M}={\sf Grpd}$, both left and right structures satisfy the condition $(1)$. For $\mathcal{M}={\sf Cat}$ this
is true only if $f$ is a fibration in the geometric sense, not only an isofibration.
What happens to other models categories $\mathcal{M}$?\\

As noticed by Arndt and Kapulkin, the left adjoint of $f^{\star }:\mathcal{M}|A\rightarrow \mathcal{M}|B$ always exists, it is written $\Sigma_f$,
and the right properness implies that
it respects ${\sf WE}$.\\

If $\mathcal{M}$ satisfies the axioms \ref{item:right-adjoint} and \ref{item:pullback}, Arndt and Kapulkin generalized the constructions of Seely \cite{seely_1984}, Hofmann and Streicher \cite{HS}, and Awodey$-$Warren \cite{AWODEY_2009} to define a M-L theory:\\

\noindent A \emph{context} is a fibration $\Gamma\rightarrow\mathcal{C}$, that is a fibrant object. A \emph{type} $\mathcal{A}$ in this context is a fibration $\mathcal{A}\rightarrow \Gamma$.
The declaration (judgment) of a type is written $\Gamma\vdash \mathcal{A}$. A \emph{term} $a:A$ is a section $\Gamma\rightarrow\mathcal{A}$. It is denoted $\Gamma\vdash a:\mathcal{A}$.\\
A \emph{substitution} $x/a$ is given by a change of base $F^{\star }$ for a functor $F:\Delta\rightarrow \Gamma$ in $\mathcal{M}_{\mathcal{C}}$, not necessarily a fibration.\\
The adjoint functor $\Sigma_f$ and $\Pi_f$ of $f^{\star }$, allows to define new types of objects: given $\Gamma$ and $f:\mathcal{A}\rightarrow \Gamma$, and $g:\mathcal{B}\rightarrow\mathcal{A}$, we get $\Sigma_f(g):\Sigma_{x:\mathcal{A}}\mathcal{B}(x)\rightarrow\Gamma$ and $\Pi_f(g):\Pi_{x:\mathcal{A}}\mathcal{B}(x)\rightarrow\Gamma$. They respectively replace the union over $\mathcal{A}$ and the product over $\mathcal{A}$.\\
\indent On the types, logical operations are applied, $\mathcal{A}\wedge\mathcal{B}$, $\mathcal{A}\vee\mathcal{B}$, $\mathcal{A}\Rightarrow\mathcal{B}$,
$\bot$ is empty, $\exists x, B(x)$, $\forall x, B(x)$. The rules for these operations satisfy the usual axioms.\\
More types, like the integers or the real numbers or the well ordering  can be added, with specific rules. \\

As remarked by  Arndt and Kapulkin, it is not necessary to have a fully closed model 
theory to get a Martin-Löf type theory \cite[remarks pp. 12-15]{AK}.
They noticed that $M-L$ type theories are probably associated to fibration-categories (or categories with fibrant objects) in the sense of Brown \cite{Brown_1973} (see also \cite{Uemura_2017}).
In these categories, cofibrations are not considered, however a nice homotopy theory can be developed.\\
\indent We have the following result concerning the weak factorization system made by cofibrations and trivial fibrations in the canonical model ${\sf Cat}$:
\begin{lem}
	A canonical trivial fibration in ${\sf Cat}$ is a geometric fibration.
\end{lem}
\begin{proof}
	Consider an isofibration $f:\mathcal{A}\rightarrow \mathcal{B}$ that is also an equivalence of category. Take $a\in \mathcal{A}$ and
	$f(a)=b\in \mathcal{B}$ and a morphism $\varphi:b'\rightarrow b$ of $\mathcal{B}$; because $f$ is surjective on the objects, there exists $a'\in A$ such that $f(a")=b'$,
	and because $f$ is an equivalence the map from ${\sf Hom}(a',a)$ to ${\sf Hom}(b",b)$ is a bijection, then there exists a unique morphism $\psi:a'\rightarrow a$
	such that $f(\psi)=\varphi$. In the same manner, every morphism $b"\rightarrow b'$ has a unique lift $a"\rightarrow a'$, and conversely any morphism $\psi':a"\rightarrow a'$
	defines a composed morphism $\chi:a"\rightarrow a$ and a morphism image $\varphi':b"\rightarrow b'$ that define the same morphism $\varphi\circ \varphi"$ from
	$b"$ to $b$. As the morphisms from $a"$ to $a$ are identified by $f$ with the morphisms from $b"$ to $b$, this gives a natural bijection between the morphisms
	$\psi'$ from $a"$ to $a'$ and the pairs $(\chi,\varphi')$ in ${\sf Hom}(a",a)\times {\sf Hom}(b",b')$ over the same element in ${\sf Hom}(b",b)$. Therefore $\psi$ is a strong cartesian
	morphism over $\varphi$.
\end{proof}

\noindent The same proof shows that a canonical trivial fibration is a geometric op-fibration,
that is by definition a fibration between the opposite categories.\\

\indent In the case where $\mathcal{C}$ is the poset of a $DNN$ and $\mathcal{M}$ is the category ${\sf Cat}$, we say that a model fibration $f:A\rightarrow B$,
in $\mathcal{M}_\mathcal{C}$
is a \emph{geometric fibration} if it is a Grothendieck-Giraud fibration, and if all the iso-fibrations that constitute the fibrant object $A$
are Grothendieck-Giraud fibrations (see theorem \ref{thm:fibrations-dnn}).

\begin{thm}\label{thm:M-L}
	Let $\mathcal{C}$ be a poset of $DNN$, there exists a canonical $M-L$ structure where contexts and types correspond to the geometric fibrations in the $2$-category of
	contravariant functors ${\sf Cat}_\mathcal{C}$, and such that base change substitutions correspond to its $1$-morphisms.
\end{thm}
\begin{proof}
	We follow the lines of Arndt and Kapulkin \cite[theorem $26$]{AK}. The main point is to
	prove that if $f:A\rightarrow B$ is a geometric fibration in $\mathcal{M}_\mathcal{C}$, the pullback functor $f^{\star }:{\sf Cat}|A\rightarrow {\sf Cat}_B$,
	has a left adjoint $f_!=\Sigma_f$
	and a right adjoint $f_\star =\Pi_f$ that both preserve the geometric fibrations. For the first case it is the stability of Grothendieck-Giraud fibrations by composition. For the second one,
	this is Giraud theorem of bi-adjunction \cite{giraud-cohomologie}.
\end{proof}

There exist several equivalent interpretations of such  a type theory, as for the intuitionistic theory of Bell, Lambek Scott et al. (see Martin-Löf,
{\em Intuitionistic Type Theory}, \cite{ML}). For instance the types are sets, the terms
are elements, or a type is a proposition and a term is a proof, or a type is a problem (a task) and a term is a method for solving it. (For each interpretation, things are local over a context.)\\
In particular, \emph{Identity types} are admitted, representing equivalence of elements, proofs or methods that are not strict equalities, like homotopies, or invertible natural equivalences.\\
The types of identities, as in Hofmann and Streicher \cite{HS}, are fibrations ${\sf Id}_A:I_\mathcal{A}\rightarrow \mathcal{A}\times \mathcal{A}$ equipped with a cofibration
$r:\mathcal{A}\rightarrow I_\mathcal{A}$ (with a section) such that ${\sf Id}_A\circ r=\Delta$, the diagonal morphism. They are considered as paths spaces.\\
For instance, given a groupoid $A$, ${\sf Id}_A=(\{0\leftrightarrow 1\}\Rightarrow A=A^{\{0\leftrightarrow 1\}}$ is an identity type.\\
Axioms of inference for the types are expressed by rules of formation, introduction and determination, specific to each type \cite{ML}.\\

Let us compare to the semantics in a topos $\mathcal{C}^{\wedge}$: a context is an object $\Gamma$ which is a presheaf with values in ${\sf Set}$, so a fibration in sets
over $\mathcal{C}$ and a type is another object $A$; to get something over $\Gamma$ we can consider the projection $\Gamma\times A \rightarrow \Gamma$. A section
corresponds to a morphism $a:\Gamma\rightarrow A$, which is rightly a term of type $A$, $\Gamma\vdash a:\mathcal{A}$.\\
A substitution corresponds to a morphism $F:\Delta\rightarrow \Gamma$, and defines a pullback of trivial fibrations $\Delta\times A\rightarrow\Delta$.\\
If we have a morphism $g:B\rightarrow \Gamma\times A$ in the topos, we can define its existential image $\exists_\pi g(B)$ and its universal image $\forall_\pi g(B)$
as subobjects of $\Gamma$, which can be seen as a trivial fibrations over $\Gamma$.\\
\indent Therefore, we have analogs of M-L type theory in Set theory, but with trivial fibrations only and without fibrant restriction.\\

\section{Classifying the M-L theory ?}\label{sec:class-ML}

In what precedes the category ${\sf Grpd}$ has replaced the category ${\sf Set}$; it is also cartesian closed. Also we have seen that all small
limits and colimits exist in ${\sf Grpd}_{\mathcal{C}}$ (Giraud, Hollander, Lurie). However every natural
transformation between two functors with values in ${\sf Grpd}$ is invertible. Thus in the $2$-category, the morphisms in ${\sf Hom}_{\sf Grpd}(G,G')$ are like
homotopies. In fact they become homotopies when passing to the nerves.\\

\indent Let us introduce the categories of presheaves on every fibration in groupoids $\mathcal{A}\rightarrow\mathcal{C}$,
i.e. the classifying topos $\mathcal{E}_\mathcal{A}$ of the stack $\mathcal{A}$. Their objects are fibered in groupoids over $\mathcal{C}$, because the fibers $\mathcal{E}_U$ for $U\in \mathcal{C}$ are such (they take their values in ${\sf IsoSet}$), but their
morphisms, the natural transformations between functors, are taken in the sense of sets, not invertible.\\

\noindent In what follows we combine the type theory of topos with the groupoidal $M-L$ type theory.\\
\noindent We propose new types, associated to every object $X_{\mathcal{A}}$ in every $\mathcal{E}_\mathcal{A}$.\\

\noindent The fibration $\mathcal{A}\rightarrow\Gamma$ itself can be identified with the final object $\mathbf{1}_\mathcal{A}\in \mathcal{E}_\mathcal{A}$ in the context $\Gamma$.\\

\noindent Sections of $\mathcal{A}\rightarrow\Gamma$ are particular cases of objects. For the terms in an object $X_A$, we take any natural transformation from the
object $S$ corresponding to a section $\Gamma\rightarrow\mathcal{A}$
to the object $X_A$ in $\mathcal{E}_\mathcal{A}$.\\
A simple section is a term to $\mathbf{1}_\mathcal{A}$, the final object, which is a usual M-L type.\\

\noindent Due to the adjunction for the topos of presheaves, the construction $\Sigma$ and $\Pi$ extend to the new types.\\

\noindent Now a classifier of subobjects $\Omega_\mathcal{A}$ is available for any M-L type $\mathcal{A}$.\\
We define relative subobjects using the correspondence $\lambda_\pi:\Omega_\mathcal{A}\rightarrow \pi^{\star }\Omega_\Gamma$.\\

This extension of M-L theory allows to define languages and semantics over DNNs with internal structure in the model category $\mathcal{M}$.

\chapter{Dynamics and homology}\label{chap:dynamics}

\section{Ordinary cat's manifolds}\label{sect:cats}

Some limits, in the sense of category theory, of the dynamical object $X^{w}$ of $\mathcal{C}^{\sim}$
describe the sets of activities in the $DNN$ which correspond to some decisions taken by its output (the so called \emph{cat's manifolds} in the
folklore of Deep Learning).\\
\indent Here we consider the case of supervised learning or the case of reinforcement learning, because the success or the failure of an action integrating
the output of the network is also a kind of metric.\\
\indent For instance, consider a proposition $P_{out}$ about the input $\xi_{\rm in}$ which depends on the final states $\xi_{\rm out}$. It can be seen as a function $P$ on the product
$X_B=\prod_bX_b$
of the spaces of states over the output layers to the boolean field $\Omega_{{\sf Set}}=\{ 0,1\}$, taking the value $1$
if the proposition is true, $0$ if not. Our aim is to better understand the involvement of the full network in this decision; it is caused by the input data in a deterministic
manner, but it results from the chosen weights and from the full functioning of the $DNN$. One of the many ways to express the situation in terms of category is to enlarge
$\mathcal{C}$ (or $\boldsymbol{\Gamma}$) by several terminal layers (see figure \ref{fig:proposition}):
\begin{enumerate}[label=\arabic*)]
	\item a layer $B^{\star}$ which makes the product of the output layers, as we have done with forks,
	followed by the layer $B$ (remark that this can be replaced by $B$ only, with an arrow from $b\in x_{\rm out}$);
	\item a layer $\omega_b$ with one cell and two states in a set $\Omega_b$,
	as in $\Omega_{{\sf Set}}$,
	with one arrow from $\omega_b$ to $B$, for translating the proposition $P$, followed by a last layer $\omega_1$, with one arrow $\omega_b\rightarrow \omega_1$, the state's space
	$X_{\omega_1}$
	being a singleton $\star_1$, and the map $\star_1\rightarrow \Omega_b$ sending the singleton to ${\bf 1}$. This gives a category $\mathcal{C}_+$ enlarging $\mathcal{C}$
	by a fork with handle $B\leftarrow \omega_b \rightarrow \omega_1$, and a unique extension $X_+^{w}$, depending on $P$, of the functor $X^{w}$
	from $\mathcal{C}^{\rm op}$ to ${\sf Set}$ in a presheaf over $\mathcal{C}_+$.
\end{enumerate}

\begin{figure}[ht]
	\begin{center}
		\includegraphics[width=11cm]{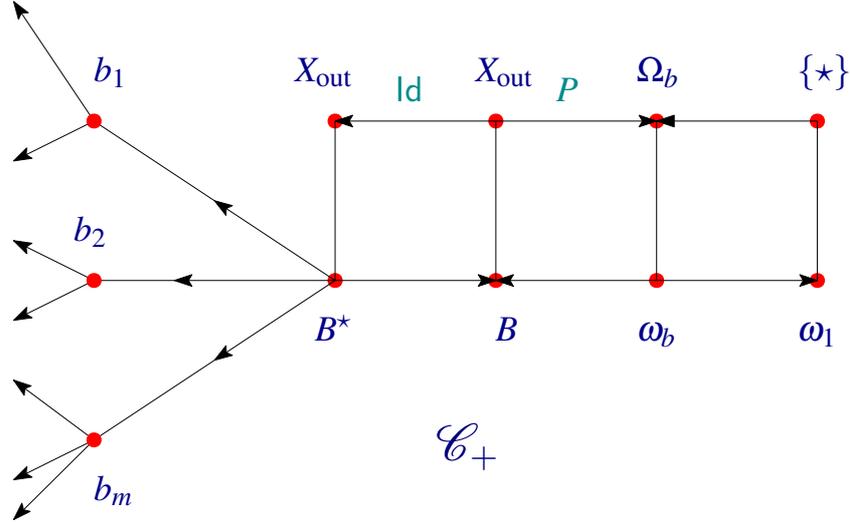}
		\caption{\label{fig:proposition}Interpretation of a proposition : categorical representation}
	\end{center}
\end{figure}

The space of sections singletons of $X_+^{w}$ is identified naturally with the space of sections of $X^{w}$ such that the output
satisfies $P_{\rm out}$, i.e. the subset of the product of all the $X^{w}(a)$ when $a$ describes $\mathcal{C}$ made by the coherent activities giving
the assertion "$P$ is true" at the output. In this picture, we also can consider that $P$ is the weight over the arrow $B\leftarrow \omega_b$, and note
$X_+^{w,P}$ the extension of $X^{w}$.\\
\noindent In other terms, the subset of activities of $X$ which affirm the proposition $P_{\rm out}$ is given by a value of the right Kan extension of $X_+$ along
the unique functor $p_+:\mathcal{C}_+^{\rm op}\rightarrow \star $:
\begin{equation}
M(P_{\rm out})(X)=\mathbf{R}{\rm Kan}_{\mathcal{C}_+}(X_+)={\rm lim}_{a\in \mathcal{C}_+^{\rm op}}X_+^{w}(a):
\end{equation}
\indent In the $AI$ folklore, the set $M(P_{\rm out})(X)$ is named a \emph{cat's manifold}, alluding to the case where the network has to decide if yes or no a cat is present in the image. $M(P_{\rm out})(X)$ can be identified with a subset of the product $X_{in}$ of the input layers.
It has to be compared with the assertion "$P$ is true" made by an observer, then studied in function of the weights $\mathbb{W}$ of the dynamics. \\
However, in general, $M(P_{\rm out})(X)$ cannot be identified with a product of subsets in the $X_a$'s, for $a\in \mathcal{C}$; it is a global invariant.\\
In fact, it is a particular case of a set of cohomology:
\begin{equation}
M(P_{\rm out})(X)=H^{0}(\mathcal{C}_+;X_+).
\end{equation}
If the proposition $P_{\rm out}$ is always true, $M$ coincides with the set of section of $X=X^{w}$, which can be identified with the product of the entry layers activities:
\begin{equation}
\Gamma(X)=H^{0}({C};X)\cong \prod_{a\in x_{\rm in}}X_a
\end{equation}\\

\noindent The construction of $\mathcal{C}_+$ and the extension of $X$ by $X_+$ can be seen as a conditioning. The map $X_+(\omega_b\rightarrow B)$ is equivalent to
a proposition, the characteristic map of a subset of $X_B$. In this case we have
\begin{equation}
H^{0}(\mathcal{C}_+;X_+)\subset H^{0}(\mathcal{C};X).
\end{equation}\\

In the same manner, we define the manifold of a theory $\mathbb{T}_{\rm out}$ expressed in a typed language $\mathbb{L}_{\rm out}$ in the
output layers, by replacing the above objects $\omega_b$, $\omega_1$, and the presheaf $X_+(P)$ over them, by larger sets and
$X_+(\mathbb{T})$, as the set of sections of $X_+(\mathbb{T})$ over the whole $\mathcal{C}_+$.\\

We will revisit the notion of cat's manifold when considering the homotopy version of semantic information.\\

\section{Dynamics with spontaneous activity}\label{section:spontaneous}

In our approach of networks functioning, the feed-forward dynamic coincides with the limit set $H^{0}(X)$. The coincidence with the traditional
notion of propagation from the input to the output relies on the particular choice of morphisms at the centers of forks (named tanks), product on one side and
isomorphism on the other. But this can be
generalized to other morphisms: the only condition being that the inner sources $A$ and the input from the outer world $I$ determine a unique section
of the object $X^{w}$ over $\mathcal{C}$. In concrete terms, this happens if and only if the maps from $A$ and $I$ give coherent values
at any tip of each fork.\\
This tuning involves the values in entry $\xi_0\in \Xi$, the values of the inner sources $\sigma_0\in \Sigma$ and the
weights, in particular from an $A$ to the $a', a",\ldots$'s. Therefore it depends on the learning process.\\

Then a possibility for defining coherent dynamical objects with spontaneous activities is to start with standard objects $X^{w}$, satisfying
the restriction of products and isomorphisms, then to introduce small deformations of the projections maps, and obtain the global dynamics
by using algorithms which
realize the Implicit Function Theorem in the Learning process.\\

Another possibility, closer to the natural networks in animals, and more readable, is to keep unchanged the projections to the tips $a'$,$\ldots$,
and to introduce new dynamical entries $y_A$ for each tang
$A$, then to send a message to the handle $a$ according to the following formula
\begin{equation}
x_a=f_a^{A}(x_{a'},x_{a"},\ldots;y_A).
\end{equation}
The state in $A$ being described by $(x_{a'},x_{a"},\ldots,y_A)$.\\
In such a manner the coherence is automatically verified. Each collection of inputs and tangs modulations generates a unique section.\\
These spontaneous entries can be learned by backpropagation, as the weights, by minimizing a functional, or realizing a task with
success.\\

It is important to remark that in natural brains, even for very small animals, having no more than several hundred neurons, the part of spontaneous activity is
much larger than the part due to the sensory inputs. This activity comes from internal rhythms, regulatory activities of the autonomous
system, internal motivations more or less planed. The neural network transforms them in actions or more general decisions. To make them
efficient, corrections are necessary, due to reentrant architectures. \\
\indent However these natural networks in general do not learn using fully supervised methods; they depend on reinforcement, by success of
actions, or by unsupervised methods, involving maximization of mutual information quantities. This will require much further works to
achieve this degree of integration in artificial networks. Also evolution plays a fundamental role, in particular by specifying the processes of weights transformations. 
But certainly experiments can be easily conducted in this direction, with simple networks as in \emph{Logical Information Cells} \cite{logic-DNN}, the experimental companion article.\\

\section{Fibrations and cofibrations of languages and theories}\label{sec:fibrations-and-cofibrations}

In this section, we define several sheaves and cosheaves over the stack $\mathcal{F}$, that are naturally associated to languages and theories,
defining moduli over monoids in the classifying topos $\mathcal{E}$, by using the semantic conditioning (see theorem \ref{thm:pistarpistar}), in such a manner that their homology or homotopy
invariants, in the sense of topos, give tentative semantical Information
quantities and spaces. In the most elementary cases, we recover in the following sections \ref{sec:sem-info}, \ref{sec:homotopy}, the definitions of Carnap and Bar-Hillel \cite{CBH52}, and their known generalizations
\cite{bao-basu}, \cite{basu-bao}, already used in Deep Learning (see for instance \cite{xie}), but at the end of this chapter, we will also show new promising elements
of information.\\
In this section and the following ones, we use the semantic functioning in usual $DNNs$ to define their semantic information content.
\noindent Taking into account the internal dimensions given by the stacks $\mathcal{F}$ over $\mathcal{C}$, several levels of information
emerge. Without closing the subject, they reflect different meaning of the word {\em information}.

\indent A first level concerns the pertinent types, or objects, to introduce in order to understand how the network performs a semantic task,
in addition to the types coming from $\mathbb{L}_{\rm out}$, that are put at
the hand by the observer, and guide the learning process, by backpropagation or reinforcement. A first conjecture, that we will not study in the
present text, is that new objects appear in cohomological forms, as
obstructions for integrating correctly the input data in the output theory. It is not excluded that this can appear spontaneously in the network,
but more probably it requires the intervention of the observer, for changing the functional (the metric) or the data, which induces a variation of the weights.
We will describe below in section \ref{chap:dynamics} examples of semantic groupoids which could generate or constrain these
obstructions. More precisely, we expect that the new objects are vanishing cycles, in the sense of Grothendieck, Deligne, Laumon \cite{Illusie2016produits}, for convenient
maps of sites, localized in the fibers $\mathcal{F}_U$, at points $(U,\xi)$.\\
In some regions of the weights, the network should become able to develop a semantic functioning about the new objects, formalized by the
languages $\mathbb{L}_U;U\in \mathcal{C}$ similarly to what happens
with singularities of functions or varieties, with imposed reality conditions. The analogy is made more precise in chapter \ref{chap:unfolding}.\\

\indent A second level, perhaps not independent of the first one, concerns the information contained in some theories about other theories, or about
decisions to take or actions to do, for instance $\mathbb{T}_{U'}$ in some layer, considered in relation to $\mathbb{T}_{U}$, when $\alpha:U\rightarrow U'$,
or $\mathbb{T}_{\rm out}$. As we saw, the expression of these theories in functioning networks depends on the given section $\epsilon$ of $X^{w}$.
However, we expect that the notion of information also allows to compare the theories made by different networks about a some class of problems.\\
\indent The semantic information that we want to make more precise must be attached to the communication between layers and the communication between networks,
and attached to some problems to solve, for a view of the necessity to introduce interaction in a satisfying view of information. See Thom in \cite{thom1983mathematical}.\\

\indent Some theories will be more informative than others, or more redundant, then we will be happy to attach quantitative notions of amount
of information to the notion of semantic information. However, efficient numerical measures should also take care of the expression of theories
by some axioms. Some systems of axioms are more economical than others, or more redundant than others. Redundancy is more the matter of
axioms, ambiguity is more the matter of theories. In the present approach, the notion of ambiguity comes first.\\

\indent In Shannon information theory, \cite{ShannonWeaver49}, the fundamental quantity is the entropy, which is in fact a measure
of the ambiguity of the expressed knowledge with respect to an individual fact, for instance a message. Only some algebraic combinations of entropies
can be understood as an information in the common sense, for instance the mutual information \[I(X;Y)=H(X)+H(Y)-H(X,Y).\]
\indent Here the theories $\mathbb{T}_U,U\in \mathcal{C}$ are seen as possible models, analogous to the probabilistic models $\mathbb{P}_X,X\in \mathcal{B}$ in Bayesian networks.
The variables
of the Bayesian network are analogous to the layers of the neural networks; the values of the variables are analogs of the states of the neurons
of the layers. In some version of Bayes analysis, for instance presented by Pearl \cite{pearl-book}, the Bayes
network is associated to a directed graph, but in some other versions it is an hypergraph \cite{NIPS2000_61b1fb3f}, or a more general poset \cite{bennequin2020extrafine}.\\
\indent In the case
of the probabilistic models, Shannon theorems have revealed the importance of entropy and of mutual information. It has been shown in \cite{Baudot-Bennequin} and \cite{vigneaux-TAC}), that
the entropy is a universal class of cohomology of degree one of the topos of presheaves over the Bayesian network, seen as a poset $\mathcal{B}$,
equipped with a cosheaf $\mathcal{P}$ of probabilities (covariant functor of sets). The
operation of joining variables gives a presheaf $\mathcal{A}$ in monoids over $\mathcal{B}$. On the other hand,
the numerical functions on $\mathcal{P}$ form a sheaf $\mathcal{F}_\mathcal{P}$, which becomes an $A$-module by considering the mean conditioning of Shannon. The entropy
belongs to the ${\sf Ext}_{\mathcal{A}}^{1}(K;\mathcal{F}_\mathcal{P})$ with coefficients in this module.
Moreover, in this framework, higher mutual information quantities \cite{McGill1954MultivariateIT}, \cite{HKT}
belong to the homotopical algebra of cocycles of higher degrees
\cite{Baudot-Bennequin}.\\
\indent We conjecture that something analog appears in the case of $DNNs$ and theories $\mathbb{T}$, and of axioms for them.\\

\indent The first ingredient in the case of probabilities was the operation of
marginalization of a probability law, interpreted as the definition of a covariant functor (a copresheaf); it
can be replaced here by the transfers of theories associated to the functors  $F_\alpha:\mathcal{F}_{U'}\rightarrow \mathcal{F}_{U}$, and to
the morphisms $h$ in the fibers $\mathcal{F}_U$ from objects $\xi$ to objects $F_\alpha(\xi')$, as we saw in the section $3$. For logics, this transfer can go
in two directions,
depending
on the geometry of $F_\alpha$, from $U'$ to $U$, and from $U$ to $U'$, as seen in section \ref{sec:objects-classifiers}.\\

\indent We start with the transfer from $U'$ to $U$, having in mind the flow of information
in the downstream direction to the output of the $DNN$; when it exists, a non-supervised learning should also correspond to this direction. However,
the learning by backpropagation or by reinforcement goes from the output layers to the inner layers, then the inner
layers have to understand something of the imposed language $\mathbb{L}_{\rm out}$ and the useful theories $\mathbb{T}_{\rm out}$ for concluding.
Therefore we will also laterconsider this backward or upstream direction.\\

\noindent For an arrow $(\alpha,h):(U,\xi)\rightarrow (U',\xi')$,  the map
\begin{equation}
\Omega_{\alpha,h}: \Omega_{U'}(\xi')\rightarrow \Omega_U(\xi),
\end{equation}
is obtained by composing the map $\lambda_\alpha=\Omega_\alpha$ at $\xi'$, from $\Omega_{U'}(\xi')$ to $\Omega_{U}(F_\alpha \xi')$ with the
map $\Omega_U(h)$ from $\Omega_{U}(F_\alpha \xi')$ to $\Omega_U(\xi)$.\\
More generally, for every object $X'$ in $\mathcal{E}_{U'}$, the map $F^{\alpha}_!$ sends the subobjects of $X'$ to the subobjects of $F^{\alpha}_!(X')$,
respecting the lattices structures. Then for any natural transformation over $\mathcal{F}_U$, $h:X\rightarrow F^{\alpha}_!(X')$, we get a transfer
\begin{equation}
\Omega_{\alpha,h}: \Omega_{U'}^{X'}\rightarrow \Omega_U^{X}.
\end{equation}
The object $X$ or $X'$ is seen as a local context in the topos semantics.\\
We assume in what follows that this mapping extends to the sentences in the typed languages $\mathbb{L}_U$, where the dependency on $\xi$ reflects the variation
of meaning in the included notions. In particular, the morphisms in the topos $\mathcal{E}_U$ express such variations. At the level of theories,
this induces in general a weakening, something which is implied at $(U,\xi)$ by the propositions at $(U',\xi')$, or more generally
at the context $X$ by telling what is true, or expected, in the context $X'$.\\
In what follows we note by $\mathcal{A}=\Omega^{\mathbb{L}}$ this presheaf of sentences in $\mathbb{L}$ over $\mathcal{F}$,
and by $\mathbb{L}_{\alpha,h}$, or $\pi_{\alpha,h}^{\star}$, its transition maps, extending $\Omega_{\alpha,h}$.\\

Under the strong standard hypotheses on the fibration $\mathcal{F}$, for instance if it defines a fibrant object in the injective groupoids models,
i.e. any $F_\alpha$ is a fibration,
(see definition \ref{dfn:semanticflow} above, following lemma \ref{lem:open-morphism} of section \ref{sec:objects-classifiers}) there exists a right adjoint of $\Omega_{\alpha,h}$:
\begin{equation}
\Omega'_{\alpha,h}: \Omega_{U}^{X}\rightarrow \Omega_{U'}^{X'}.
\end{equation}
It is given by extension of the operators $\lambda'_\alpha$, associated to $F_\alpha^{\star}$, in the place of $F^{\alpha}_!$, plus a transposition.\\
In what follows we note by $\mathcal{A}'=\leftindex^{t}\Omega^{\mathbb{L}}$ this copresheaf of sentences over $\mathcal{F}$, and by $^{t}\mathbb{L}'_{\alpha,h}$,
or simply $\pi^{\alpha,h}_\star$, its transition maps. The extended strong hypothesis requires that $\pi^{\star}_{\alpha,h}\circ \pi_\star^{\alpha,h}={\sf Id}$.\\

For fixed $U$ and $\xi\in \mathcal{F}_U$, the operation $\wedge$
gives a monoid structure on the set $\mathcal{A}_{U,\xi}=\mathcal{A}'_{U,\xi}$, which is respected
by the maps $\mathbb{L}_{\alpha,h}$ and $^{t}\mathbb{L}'_{\alpha,h}$.\\
Moreover, $\mathcal{A}_{U,\xi}$ has a natural structure of poset category, given by
the external implication $P\leq Q$, for which
$\mathbb{L}_{\alpha,h}$ and $^{t}\mathbb{L}'_{\alpha,h}$ are functors.\\
There exists a right adjoint of the functor $R\mapsto R\wedge Q$; this is the internal implication, $P\mapsto (Q\Rightarrow P)$.
Then, by definition, $\mathcal{A}_{U,\xi}=\mathcal{A}'_{U,\xi}$ is a \emph{closed monoidal category}. In fact this is the only structure
that is essentially needed for the information theory below; this allows the linear generalization of appendix \ref{app:nat-lang}.\\

The maps $\pi^{\star}$ and $\pi_\star$ give a fibration $\widetilde{\mathcal{A}}$ over
$\mathcal{F}$, and a cofibration $\widetilde{\mathcal{A}}'$ over $\mathcal{F}$, in the sense of Grothendieck \cite{AST_2005__301__R1_0}:

\noindent a morphism $\gamma$ in $\widetilde{\mathcal{A}}$ from $(U,\xi,P)$ to $(U',\xi',P')$, lifting a morphism $(\alpha,h)$ in $\mathcal{F}$
from $(U,\xi)$ to $(U',\xi')$, is given by an arrow $\iota$ in $\Omega^{\mathbb{L}_U}$ from $P$ to $\mathbb{L}_{\alpha,h}(P')=\pi_{\alpha,h}^{\star}P'$, that is
an external implication
\begin{equation}\label{aarrow}
P\leq \mathbb{L}_{\alpha,h}(P').
\end{equation}
\noindent Similarly, an arrow in the category $\widetilde{\mathcal{A}}'$
lifting the same morphism $(\alpha,h)$ in $\mathcal{F}$, is
an implication
\begin{equation}\label{aprimearrow}
^{t}\mathbb{L}'_{\alpha,h}(P)\leq P'.
\end{equation}

\noindent Remark that \emph{a priori} the left adjunction $\pi_{\alpha,h}^{\star}\dashv \pi^{\alpha,h}_\star$ does not imply something between $P$ and $\mathbb{L}_{\alpha,h}(P')$
when \eqref{aprimearrow} is satisfied. However, under the strong hypothesis $\pi^{\star}\circ\pi_\star={\sf Id}$, the relation \eqref{aprimearrow} implies the relation
\eqref{aarrow}. Then in this case, $\widetilde{\mathcal{A}}'$ is a subcategory of $\widetilde{\mathcal{A}}$.

\begin{rmk*}
	\normalfont An important particular case, where our standard hypotheses are satisfied, is when the $\Omega^{\mathbb{L}_{U,\xi}}=\mathcal{A}_{U,\xi}$
	are the sets of open sets of a topological spaces $Z_{U,\xi}$, and when there exist continuous open maps $f_{\alpha}:Z_{U',\xi'}\rightarrow Z_{U,\xi}$ lifting the
	functors  $F_\alpha$, such that the maps $\pi^{\star}$ and $\pi_\star$ are respectively the direct images and the inverse images. The strong hypothesis holds when the
	$f_\alpha$ are topological fibrations.
\end{rmk*}

\noindent $\widetilde{\mathcal{A}}$ and $\widetilde{\mathcal{A}}'$ belong to augmented model categories using monoidal posets \cite{Raptis2010HomotopyTO}. See section \ref{sec:class-ML}.\\

\noindent For $P\in \Omega^{\mathbb{L}_{U,\xi}}=\mathcal{A}_{U,\xi}$, we note $\mathcal{A}_{U,\xi,P}$ the set of proposition $Q$ such that $P\leq Q$.
They are sub-monoidal categories of $\mathcal{A}_{U,\xi}$. Moreover they are closed, because $P\leq Q, P\leq R$ implies $P\wedge Q=P$, then
$P\wedge Q\leq R$, then $P\leq(Q\Rightarrow R)$.\\
When varying $P$, these sets form a presheaf over $\mathcal{A}_{U,\xi}=\mathcal{A}'_{U,\xi}$.\\
\begin{lem}\label{lem:monoid-presheaf}
	The monoids $\mathcal{A}_{U,\xi,P}$, with the
	functors $\pi^{\star }$ between them, form a presheaf  over the category $\widetilde{\mathcal{A}}$.
\end{lem}

\begin{proof}
	Given a morphism $(\alpha,h,\iota): \mathcal{A}_{U,\xi,P}\rightarrow \mathcal{A}_{U',\xi',P'}$ in $\widetilde{\mathcal{A}}$, the symbol
	$\iota$ means $P\leq \pi^{\star }P'$, then, from $P'\leq Q'$, we deduce $P\leq\pi^{\star }P'\leq \pi^{\star }Q'$.
\end{proof}

Lemma \ref{lem:open-morphism} in section \ref{sec:objects-classifiers} established the existence of a counit $\eta:\pi^{\star }\pi_\star \rightarrow {\sf Id}_U$,
for every morphism $(\alpha,h):(U,\xi)\rightarrow (U',\xi')$, then for every $P\in A_{U,\xi}$, we have $\pi^{\star }\pi_\star P\leq P$.\\
\indent Under the stronger hypothesis on the fibration $\mathcal{F}$, that $\eta={\sf Id}_{\Omega^{\mathbb{L}}}$, i.e. $\pi^{\star }\pi_\star P= P$, lemma \ref{lem:monoid-presheaf} holds also true for the category $\widetilde{\mathcal{A}}'$.\\

\begin{defn*}
	$\Theta_{U,\xi}$ is the set of theories expressed in the algebra $\Omega^{\mathbb{L}_U}$
	in the context $\xi$. Under our standard hypothesis on $\mathcal{F}$, both $\mathbb{L}_\alpha$ and $^{t}\mathbb{L}_\alpha$
	send theories to theories.
\end{defn*}
\begin{defn*}
	$\Theta_{U,\xi,P}$ is the subset of theories which imply the truth of proposition $\neg P$, i.e. the subset of theories excluding
	$P$.
\end{defn*}
\noindent Remind that $\neg P\equiv (P\Rightarrow \bot)$ is the largest proposition $R$ such that $R\wedge P\leq\bot$.\\
	It is always true that $P\leq P'$ implies $\neg P'\leq \neg P$, but the reciprocal implication in general requires a boolean logic.\\
	Then, for fixed $U,\xi$, the sets $\Theta_{U,\xi,P}$ when $P$ varies in $\mathcal{A}_{U,\xi}$, form a presheaf over $\mathcal{A}_{U,\xi}$; if $P\leq Q$, any theory excluding
	$Q$ is a theory excluding $P$.
\begin{lem}\label{lem:theories-morph}
	Under the standard hypotheses on the fibration $\mathcal{F}$, without necessarily axiom (\ref{strongstandard}),
	the sets $\Theta_{U,\xi,P}$ with the morphisms $\pi^{\star}$,  form a presheaf over $\widetilde{\mathcal{A}}$.
\end{lem}
\begin{proof}
	Let us consider a morphism $(\alpha,h,\iota): \mathcal{A}_{U,\xi,P}\rightarrow \mathcal{A}_{U',\xi',P'}$, where $\iota$ denotes $P\leq\pi^{\star }P'$;
	we deduce $\pi^{\star }\neg P'=\neg\pi^{\star }P'\leq \neg P$; then $T'\leq \neg P'$ implies $\pi^{\star }T'\leq \pi^{\star }\neg P'\leq \neg P$.
\end{proof}
\begin{cor*}
	Under the standard hypotheses on the fibration $\mathcal{F}$ plus the stronger one,
	the sets $\Theta_{U,\xi,P}$ with morphisms $\pi^{\star }$, also form a presheaf over $\widetilde{\mathcal{A}}'$.
\end{cor*}

\noindent What happens to $\pi_\star$?\\

\indent It is in general false that the collection $\mathcal{A}_{U,\xi,P}$ with the functors  $\pi^{\alpha,h}_\star $ forms a copresheaf
over  $\widetilde{\mathcal{A}}'$. However, if we restrict ourselves to the smaller category $\widetilde{\mathcal{A}}_{\rm strict}'$, with the same objects but with morphisms
from $\mathcal{A}_{U,\xi,P}$ to $\mathcal{A}_{U',\xi',P'}$ only when $P'=\pi_\star^{\alpha,h}P$, this is true.\\
\begin{proof}
	If $P\leq Q$, $\pi_\star P\leq \pi_\star Q$, then $P'\leq \pi_\star Q$.\\
\end{proof}
\indent The same thing happens to the collection of the $\Theta_{U,\xi,P}$ with the morphism $\pi_\star $: over the restricted category $\widetilde{\mathcal{A}}_{\rm strict}'$,
they form a copresheaf.
\emph{Proof}: if $T\leq \neg P$, we have $\pi_\star T\leq \pi_\star \neg P= \neg\pi_\star P=\neg P'$.\\
However for the full category $\widetilde{\mathcal{A}}'$ ({\em resp.} the category $\widetilde{\mathcal{A}}$),
the argument does not work: from $\pi_\star P\leq P'$ ({\em resp.} $P\leq \pi^{\star }P'$), it follows that $\neg P'\leq \neg \neg \pi_\star P=\pi_\star \neg P$
({\em resp.} $\pi_\star P\leq \pi_\star \pi^{\star }P'$ then $\neg \pi_\star \pi^{\star }P'\leq \pi_\star P$, then by adjunction $\neg P'\leq \neg \pi_\star P=\pi_\star \neg P$);
then $T\leq\neg P$ implies $\pi_\star T\leq \pi_\star \neg P$, not $\pi_\star T\leq \neg P'$.\\

\noindent To summarize what is positive with $\pi_\star  $,
\begin{lem}
	Under the strong standard hypothesis of defition \ref{dfn:semanticflow}, the collections $\mathcal{A}_{U,\xi,P}$ and $\Theta_{U,\xi,P}$ with the
	morphisms $\pi_\star $, constitute copresheaves over $\widetilde{\mathcal{A}}_{\rm strict}'$.
\end{lem}

\noindent Note that the fibers $\mathcal{A}_{U,\xi,P}$ are not sub-categories of $\widetilde{\mathcal{A}}_{\rm strict}'$, they are subcategoris
of $\widetilde{\mathcal{A}}'$ and $\widetilde{\mathcal{A}}$.\\
\begin{defn*}
	A theory $\mathbb{T}'$ is said \emph{weaker} than a theory $\mathbb{T}$ if its axioms are true
	in $\mathbb{T}$. We note $\mathbb{T}\leq \mathbb{T}'$, as we made for weaker probabilistic models. This applies to theories
	excluding a proposition $P$, in $\Theta_{U,\xi,P}$.
\end{defn*}

\indent With respect to propositions in $\mathcal{A}_{U,\xi}$, if we take the joint $R$ by the operation "and"  of all the axioms
$\vdash R_i;i\in I$ of $\mathbb{T}$, and the analog $R'$ for $\mathbb{T}'$, the above relation corresponds to $R\leq R'$.\\
Remark: a weaker theory can also be seen as a simpler or more understandable theory; for instance in $\Theta_\lambda$,
the maximal theory $\vdash(\neg P)$ is dedicated to exclude $P$, and the propositions implying $P$.\\

\noindent Be careful that in the sense of sets of truth assertions, the pre-ordering by inclusion of the theories
goes in the reverse direction. For instance $\{\vdash\bot\}$ is the strongest theory, in it everything is true,
thus every other theory is weaker.\\

\noindent Now we introduce a notion of \emph{semantic conditioning}.\\
\begin{defn}\label{defn:sem-cond}
	For fixed $U,\xi$, $P\leq Q$ in $\Omega^{\mathbb{L}_{U,\xi}}$, and $\mathbb{T}$ a theory
	in the language $\mathbb{L}_{U,\xi}$, we define a new theory by the internal implication:
	\begin{equation}\label{semanticconditioning}
		Q.\mathbb{T}=(Q\Rightarrow \mathbb{T}).
	\end{equation}
	More precisely: the axioms of $Q.\mathbb{T}$ are the assertions $\vdash(Q\Rightarrow R)$ where $\vdash R$ describes the axioms
	of $\mathbb{T}$. We consider $Q.\mathbb{T}$ as the \emph{conditioning} of $\mathbb{T}$ by $Q$, in the logical or semantic sense, and frequently we write the resulting theory $\mathbb{T}|Q$.
	
\end{defn}

\noindent At the level of propositions, the operation $\Rightarrow$ is the right adjoint in the sense of the Heyting algebra of the relation $\wedge$, i.e.
\begin{equation}
(R\wedge Q\leq P)\quad iff \quad (R\leq (Q\Rightarrow P)).
\end{equation}
\begin{prop}
	The conditioning gives an action of the monoid $\mathcal{A}_{U,\xi,P}$ on the set of theories
	in the language $\mathbb{L}_{U,\xi}$.
\end{prop}
\begin{proof}
	\begin{equation}
		\begin{split}
			(R\wedge Q'\wedge Q\leq P)\quad {\rm iff} \quad  (R\wedge Q')\leq (Q\Rightarrow P)\\ \quad {\rm iff} \quad  (R\leq (Q'\Rightarrow(Q\Rightarrow P)).
		\end{split}
	\end{equation}
	
	\noindent Note that $Q\Rightarrow P$ is also the maximal proposition $Q'$ (for $\leq$) such that $Q\wedge Q'\leq P$.\\
	
	\noindent Therefore the theory $Q\Rightarrow \mathbb{T}$ is the largest one among all theories $\mathbb{T}'$ satisfying
	\begin{equation}
		Q\wedge \mathbb{T}'\leq \mathbb{T}.
	\end{equation}
	
	\noindent This implies that $\mathbb{T}|Q$ is weaker than $\mathbb{T}$ and than $\neg Q$.
	\begin{enumerate}[label=\arabic*)]
		\item In $Q\wedge \mathbb{T}$, the axioms are of the form $\vdash (Q\wedge R)$ where $\vdash R$ is an axiom of $\mathbb{T}$,
		and from $\vdash (Q\wedge R)$, we deduce $\vdash R$.
		\item Here $Q$ ({\em resp.} $\neg Q$) is understood as the theory with unique axiom $\vdash Q$ ({\em resp.} $\vdash \neg Q$),
		then if $\vdash (Q\wedge \neg Q)$ we have $\vdash\bot$ and all theories are true.
	\end{enumerate}
\end{proof}

\begin{rmk*}
	\normalfont The theory $\mathbb{T}|Q=(Q\Rightarrow \mathbb{T})$ can also be written $\mathbb{T}^{Q}$, by definition of the internal exponential,
	as the action by conditioning is also the internal exponential.
\end{rmk*}

\noindent {\em Notation:} for being lighter, in what follows, we will mostly denote the propositions by the letters $P,Q,R,P',...$
and the theories by the next capital letters $S,T,U,S',...$.\\

\noindent The operation of conditioning was considered by Carnap and Bar-Hilled \cite{CBH52}, in the case of Boolean theories, studying the
content of propositions
and looking for a general notion of sets of semantic Information. In this case $Q\Rightarrow T$ is equivalent to
$T\vee\neg Q=(T\wedge Q)\vee\neg Q$ (see the companion text on logicoprobabistic information for more details \cite{logico-probabilistic}).\\
Their main formula for the concept of information was
\begin{equation}\label{carnapbarhillelequation}
{\rm Inf}(\mathbb{T}|P)={\rm Inf}(\mathbb{T}\wedge P)\backslash {\rm Inf}(P);
\end{equation}
assuming that ${\rm Inf}(A\wedge B)\supseteq {\rm Inf}(A)\cup {\rm Inf}(B)$.\\
\begin{prop}\label{prop:conditioning}
	The conditioning by elements of $\mathcal{A}_{U,\xi, P}$, i.e. propositions $Q$ implied by $P$,
	preserves the set $\Theta_{U,\xi, P}$ of theories excluding $P$.
\end{prop}
\begin{proof}
	Let $T$ be a theory excluding $P$ and $Q\geq P$; consider a theory $T'$ such that $Q\wedge T'\leq T$,
	we deduce $T'\wedge P\leq T$, thus $T'\wedge P\leq T\wedge P$. But $T\wedge P\leq\bot$, then $T'\wedge P\leq\bot$. But
	$Q\Rightarrow T$ is the largest theory such that $Q\wedge T'\leq T$, therefore $Q\Rightarrow T$ excludes $P$, i.e. asserts $\neg P$.
\end{proof}
\begin{rmk*}
	\normalfont Consider the sets $\Theta'_{U,\xi,P}$ of theories which imply the validity of the proposition $P$. These sets constitute a cosheaf
	over the category $\widetilde{\mathcal{A}}_{\rm strict}'$ for $\pi_\star $ and a sheaf for  $\pi^{\star }$.
	However, the formula
	\eqref{semanticconditioning} does'nt give an action of the monoid $\mathcal{A}_{U,\xi,P}$ on the set $\Theta'_{U,\xi,P}$, even in the boolean case,
	where $(Q\Rightarrow T)=T\vee\neg Q$.
\end{rmk*}

\indent We can also consider the set of all theories over the largest category $\widetilde{\mathcal{A}}$, without further localization;
they also form a sheaf for $\pi^{\star }$
and a cosheaf $\Theta$ for $\pi_\star $, which are stable by the conditioning.\\
\indent When necessary, we note $\Theta_{\rm loc}$ the presheaf for $\pi^{\star }$ made by the $\Theta_{U,\xi,P}$ over $\widetilde{\mathcal{A}}$.\\

\noindent The naturality over $\widetilde{\mathcal{A}}'_{strict}$ of the action of the monoids relies on the following formulas, for every arrow $(\alpha,h):(U,\xi)\rightarrow (U',\xi')$
in $\mathcal{F}$, we have the arrows $(U,\xi,P)\rightarrow (U',\xi',\pi_\star P)$ in $\widetilde{\mathcal{A}}_{\rm strict}'$; in the presheaf of monoids $\mathcal{A}_{U,\xi,P}$,
for the morphism $\pi^{\star }$, and the presheaf $\Theta_{U,\xi,P}$ with morphisms $\pi_\star $:
\begin{equation}
(\pi^{\star }Q').T=\pi^{\star }\left[Q'.\pi_\star (T)\right].
\end{equation}

\noindent This holds true under the strong hypothesis $\pi^{\star }\pi_\star ={\sf Id}$.\\

\indent If we want to consider functions $\phi$ of the theories, two possibilities appear: $\pi_\star $ for $\Theta$ with $\pi^{\star }$
for the monoids $\mathcal{A}$, or the opposite $\pi^{\star }$ for $\Theta$ with $\pi_{\star }$
for the monoids $\mathcal{A}$.
But Both cases give a cosheaf over $\widetilde{\mathcal{A}}$, however only the second one gives a functional module $\Phi$ over $\mathcal{A}$, even
with the strong standard hypothesis,
\begin{thm}\label{thm:pistarpistar}
	Under the strong hypothesis, in particular $\pi^{\star }\pi_\star ={\sf Id}$, 
	and over the restricted category $\widetilde{\mathcal{A}}_{\rm strict}'$, the cosheaf $\Phi'$
	made by the measurable functions (with any kind of fixed values) on the theories $\Theta_{U,\xi,P}$, with the morphisms $\pi^{\star }$,
	is a cosheaf of modules over the monoidal cosheaf $\mathcal{A}'_{\rm loc}$, made by the monoidal
	categories $\mathcal{A}_{U,\xi,P}$, with the morphisms $\pi_\star $.
\end{thm}
\begin{proof}
	Consider a morphism $(\alpha,h, \iota):A_{U,\xi,P}\rightarrow A_{U',\xi',\pi_\star P}$, a theory $T'$ in
	$\Theta_{U',\xi',\pi_\star P}$, a proposition $Q$ in $A_{U,\xi,P}$, and an element $\phi_P$ in $\Phi'_{U,\xi,P}$, we have
	\begin{align*}
		\pi_\star Q.(\Phi'_\star \phi_P)(T')&=(\Phi'_\star \phi_P)(T'|\pi_\star Q)\\
		&=\phi_P[\pi^{\star }(T'|\pi_\star Q)]\\
		&=\phi_P[\pi^{\star }(\pi_\star Q\Rightarrow T')]\\
		&=\phi_P[\pi^{\star }\pi_\star Q\Rightarrow \pi^{\star }T']\\
		&=\phi_P[Q\Rightarrow \pi^{\star }T']\\
		&=\phi_P[\pi^{\star }(T')| Q].
	\end{align*}
\end{proof}
\begin{rmk*}
	\normalfont The same kind of computation shows that, in the case of the sheaf $\Phi$ of functions on the cosheaf $\Theta$
	with $\pi_\star $ and the sheaf $\mathcal{A}_{\rm loc}$ with $\pi^{\star }$, we have, for the corresponding elements $Q',T,\phi'$,
	\begin{equation}
		\pi^{\star }Q'.\Phi^{\star }(\phi')(T)=\phi'(\pi_\star \pi^{\star }Q'.\pi_\star T);
	\end{equation}
	which is not the correct equation of compatibility, under our assumption. It should be true for the other direction,
	if $\epsilon=\pi_\star \pi^{\star }={\sf Id}_{\mathcal{A}_{U',\xi'}}$.
\end{rmk*}

\noindent However, there exists an important case where both hypotheses $\pi^{\star }\pi_\star ={\sf Id}_U$ and $\pi_{\star }\pi^{\star }={\sf Id}_{U'}$
hold true; it the case where the languages over the objects $(U,\xi)$ are all isomorphic. In terms of the intuitive maps
$f_\alpha$, this means that they are homeomorphisms. This case happens in particular when we consider the restriction
of the story to a given layer in a network.\\

\begin{rmk*}
	\normalfont Considering lemmas \ref{lem:monoid-presheaf} and \ref{lem:theories-morph}, we could forget the functional point of view with a space $\Phi$.
	In this case we do not have an Abelian situation, but we have
	a sheaf of sets of theories $\Theta_{\rm loc}$, on which the sheaf of monoids $\mathcal{A}_{\rm loc}$ acts by conditioning,
\end{rmk*}

	\begin{prop}\label{prop:compatibility}
	The presheaf $\Theta_{\rm loc}$ for $\pi^{\star }$ is compatible with the monoidal action of
	the presheaf $\mathcal{A}_{\rm loc}$, both considered on the category $\widetilde{\mathcal{A}}$ (then over $\widetilde{\mathcal{A}}'$ by restriction, under
	the strong standard hypothesis on $\mathcal{F}$).
\end{prop}
\begin{proof}
	If $T'\leq \neg P'$ and $P\leq \pi^{\star }P'$, we have $\neg \pi^{\star }P'\leq \neg P$, therefore
	$\pi^{\star }T'\leq \neg P$.
\end{proof}

In the Bayesian case, the conditioning is expressed algebraically by the Shannon mean formula on the
functions of probabilities:
\begin{equation}\label{shannonconditioning}
Y.\phi(\mathbb{P}_X)=\mathbb{E}_{Y_\star \mathbb{P}_X}(\phi(\mathbb{P}|Y=y))
\end{equation}
This gives an action of the monoid of the variables $Y$ coarser than $X$, as we find here for the fibers $A_{U,\xi,P}$
and the functions of theories $\Phi_{U,\xi,P}$.\\
\indent Equation \eqref{carnapbarhillelequation} was also inspired by Shannon's equation
\begin{equation}\label{shannonequation}
(Y.H)(X;\mathbb{P})=H((Y,X);\mathbb{P})-H(Y;Y_\star \mathbb{P}).
\end{equation}
However this set of equations for a system $\mathcal{B}$ can be deduced from the set of equations of invariance
\begin{equation}\label{shannoninvariantequation}
(H_X-H_Y)|Z=H_{X\wedge Z}-H_{Y\wedge Z}.
\end{equation}

In the semantic framework, two analogies appear with the bayesian framework: in one of them, in each layer, the role of random variables
is played by the propositions $P$; in the other one, their role is played by the layers $U$, augmented by the objects of a groupoid
(or another kind of category for contexts). The first analogy was chosen by Carnap and Bar-Hillel, and certainly will play a role
in our toposic approach too, at each $U,\xi$, to measure the logical value of functioning. However, the second analogy
is more promising for the study of DNNs, in order to understand the semantic adventures in the feedforward and feedback dynamics.\\

To unify the two analogies, we have to consider the triples $(U,\xi,P)$ as the semantic analog of random variables, with the covariant
morphisms of the category $\mathcal{D}=\widetilde{\mathcal{A}'}_{\rm strict}^{\rm op}$,
\begin{equation}
(\alpha,h,\pi^{\alpha,h}_\star ):(U,\xi,P)\rightarrow (U',\xi',P'=\pi^{\alpha,h}_\star P),
\end{equation}
as analogs of the marginals.\\
In fact, a natural extension exists and will be also studied, replacing the monoids $\mathcal{A}_{U,\xi,P}$ by the monoids $\mathcal{D}_{U,\xi,P}$
of arrows in $\mathcal{D}$ going to $(U,\xi,P)$, i.e. replacing $\widetilde{\mathcal{A}}_{\rm loc}$ by the left slice $\mathcal{D}\backslash \mathcal{D}$.
This will allow the use of combinatorial constructions over the nerves of $\mathcal{F}$ and $\mathcal{C}$.\\

If we consider the theories in $\Theta_{U,\xi,P}$ as the analogs of the probability laws, the analogs of the values of a variable $Q$
are the conditioned theories $T|Q$.\\
\indent When a functioning network is considered, the neural activities in $X^{w}_{U,\xi}$, can also be seen as
values of the variables, through a map $S_{U,\xi,P}:X^{w}_{U,\xi}\rightarrow \Theta_{U,\xi,P}$.\\

As defined in section \ref{sec:theories},
a \emph{semantic functioning} of the neural network $\mathbb{X}$ is given by a function
\begin{equation}\label{semanticfunctioning}
S_{U,\xi}:\mathbb{X}_{U,\xi}\rightarrow \Theta_{U,\xi}.
\end{equation}

\noindent The introduction of $P$, seen as logical localization, corresponds
to a refined notion of semantic functioning, a quotient of the activities made by the neurons that express
a rejection of this proposition. This generates a foliations in the individual layer's activities.\\
\begin{rmk*}
	\normalfont We could also consider the cosheaf $\Theta'$ or $\Theta'_{\rm loc}$ over $\widetilde{\mathcal{A}}_{\rm strict}'$, and obtain the cosheaf $\Sigma'$,
	of all possible maps $S_{U,\xi}:X_{U,\xi}\rightarrow \Theta'_{U,\xi};U\in \mathcal{C}, \xi\in \mathcal{F}_U$,
	where the transition from $U,\xi$ to $U',\xi'$ over $\alpha,h$ is given by the contravariance of $X$ and by the covariance of $\Theta'$:
	\begin{equation}
		\Sigma'_{\alpha,h}(s_U)_{U',\xi'}=^{t}\mathbb{L}'_{\alpha,h}\circ s_U\circ X_{\alpha,h}.
	\end{equation}
	However the above discussion shows that the compatibility with the conditioning would require $\pi_\star \pi^{\star }={\sf Id}$, which appears to be too restrictive.
\end{rmk*}

\indent In addition, the network's feed-forward dynamics $X^{w}$ makes appeal to a particular class of inputs $\Xi$,
and is more or less adapted by learning to the expected theories $\Theta_{\rm out}$ at the output. Therefore a convenient notion of information, if it exists,
must involve these ingredients.\\
\indent By using functions of the mappings $S_{U,\xi}$, we could not apply them to particular vectors in $X_{U,\xi}$. But using functions
on the $\Theta_{U,\xi}$ we can. Then this will be our choice. And this can give numbers (or sets or spaces) associated to a family of activities
$x_\lambda\in X_\lambda$, and to their semantic expression
$S_\lambda(x_\lambda)\in \Theta'_\lambda$. Moreover, we can take the sum over the set of $x$ belonging to $\Xi$, then a sum of
semantic information corresponding to the whole set of data and goals. Which seems preferable.\\

\noindent The relations
\begin{equation}
S_{U,\xi}\circ X^{\star }=\pi^{\star }\circ S_{U',\xi'},
\end{equation}
mean that the logical transmission of the theories expressed by $U'$ (in the context $\xi'$) coincide with the
theories in $U$ induced by the neuronal transmission from $U'$ to $U$.\\

If this coherence is verified, the object $\Sigma$ in the topos, replacing $\Sigma'$, could be taken
as the exponential object $\Theta^{\mathbb{X}}$ in the topos of presheaves over $\widetilde{\mathcal{A}}$. By definition, this is
equivalent to consider the parameterized families of functioning
\begin{equation}\label{coherentsemanticfunctioning}
S_{\lambda}:\mathbb{X}_{U,\xi}\times Y_\lambda\rightarrow \Theta_{U,\xi,P};
\end{equation}
where $Y$ is any object in the topos of presheaves over $\widetilde{\mathcal{A}}$.\\
\begin{rmk*}
	\normalfont In the experiments with small networks, we verified this coherence, but only
	approximatively, i.e. with high probability on the activities in $X$.
\end{rmk*}

On another hand, a semantic information over the network must correspond to the impact of the inner functioning on the output decision, given the inputs. For instance, it has to
measure how far from the output theory is the expressed theory at $U',\xi'$. We hope that this should be done by an analog of the \emph{mutual information}
quantities. If
we believe in the analogy with probabilities, this should be given by the topological coboundary of the family
of sections of the module $\Phi_\lambda;\lambda\in \widetilde{\mathcal{A}'}$ \cite{Baudot-Bennequin}\\
\indent Then we enter the theory of topological invariants of the sheaves of modules in a ringed topos. Here $\Phi$ over $\mathcal{D}$, or $\mathcal{D}\backslash\mathcal{D}$.\\

\noindent The category $\mathcal{D}=\widetilde{\mathcal{A}'}_{\rm strict}^{\rm op}$ gives birth to a refinement
of the cat's manifolds we have defined before in section \ref{sect:cats}:\\
Suppose, to simplify, that we have a unique initial point in $\mathcal{C}$; it corresponds to the output layer $U_{out}$.
Then look at a given $\xi_0\in \mathcal{F}_{\rm out}$, and a given proposition $P_{\rm out}$ in $\Omega_{\rm out}(\xi_0)=\mathcal{A}_{U_{\rm out},\xi_0}$;
it propagates in the inner layers through $\pi_\star $ in $P\in \mathcal{A}_{U,\xi}$ for any $U$ and any $\xi$ linked to $\xi_0$, and can
be reconstructed by $\pi^{\star }$ at the output, due to the hypothesis $\pi^{\star }\pi_\star ={\sf Id}$. Then we get
a section over $\mathcal{C}$ of the cofibration $\mathcal{D}^{\rm op}\rightarrow\mathcal{C}$. This can be
extended as a section of $\mathcal{D}^{\rm op}\rightarrow\mathcal{F}$, 
by varying $\xi_0$, when the $F_\alpha$ are fibrations, which is the
main case we have in mind.\\
Note that this does not give all the sections, because some propositions $P$ in a $\mathcal{A}_\lambda$
are not in the image of $\pi_\star $, even if all of them are sent by $\pi^{\star }$ to an element of a set $\Omega_{\rm out}(\xi_0)$.\\
However, these interesting sections are in bijection with the connected components of $\mathcal{D}^{\rm op}$.\\
\indent Let $\mathbb{K}$ be a commutative ring, and $c_P$ a non zero element of $\mathbb{K}$; we define the (measurable) function $\delta_P$
on the theories in the $\Theta_\lambda(P)$, taking the value $c_P$ over a point in the
above connected component of $\mathcal{D}$, and $0$ outside.\\
\indent Looking at the semantic functioning $S:X_{U,\xi}\rightarrow\Theta_\lambda$, we get a function $\delta_P$ on the sets
of local activities. This function takes the value $c_P$ on the set of activities that form theories excluding $P$.\\
\noindent Several subtle points appear:
\begin{enumerate}[label=\arabic*)]
	\item the function really depends on $P$, but when $P$ varies, it does not change when two
	propositions have the same negation $\neg P$;
	\item to conform with the before introduced notion of cat's manifold, we must assume that the activities in different
	layers which exclude $P$ in their axioms, are coherent, i.e. form a section of the object $X^{w}$.
\end{enumerate}
\noindent Without the coherence hypothesis between dynamics and logics, we have two different notions of cat's manifolds,
one dynamic and one linguistic or logical. In a sense, only the agreement deserves to be really named semantics.

\section{Semantic information. Homology constructions}\label{sec:sem-info}
\subsection*{Bar complex of functions of theories and conditioning by propositions.}

\noindent We start with the computation of the Abelian invariants, therefore with the module of functions $\Phi$ on $\Theta$
in the cases where conditioning act.\\

\noindent We consider first the most interesting case described by theorem \ref{thm:pistarpistar}, given by the presheaf $\Theta$
over the category $\mathcal{D}$, fibred over $\mathcal{F}$ which is itself fibred over $\mathcal{C}$. Note that over $\widetilde{\mathcal{A}}'_{\rm strict}$we get cosheaves,
thus we prefer to work over the opposite $\mathcal{D}$. Then $\mathcal{A}'_{\rm loc}$ with morphisms $\pi_\star $, becomes a sheaf of monoids over
$\mathcal{D}$, and $\Theta'_{\rm loc}$, with morphisms $\pi^{\star }$, becomes a cosheaf of sets over $\mathcal{D}$, in such a manner
that the functions $\Phi$ on $\Theta'_{\rm loc}$ constitute a sheaf of $\mathcal{A}'_{\rm loc}$ modules.\\

We suppose that the elements $\phi_\lambda$ in $\Phi_\lambda$ take their values in a commutative ring $K$
(with cardinality at most continuous).\\

The method of relative homological algebra, used for probabilities in Baudot, Bennequin \cite{Baudot-Bennequin}, and Vigneaux \cite{vigneaux-TAC}, cited above, can be applied here,
for computing ${\sf Ext}^{\star }_{\mathcal{A}'_{\rm loc}}(K,\Phi)$ in the toposic sense. The action of $\mathcal{A}'_{\rm loc}$ on $K$ is supposed trivial.\\

We note $\mathcal{R}=K\left[\mathcal{A}'_{\rm loc}\right]$ the cosheaf in $K$-algebras associated to the monoids $\mathcal{A}'_\lambda;\lambda\in \mathcal{A}'$.
The non-homogeneous bar construction gives a free resolution of the trivial constant module $K$:
\begin{equation}
0\leftarrow K\leftarrow B'_0\leftarrow B'_1\leftarrow B'_2\leftarrow ...
\end{equation}
where $B'_n;n\in \mathbb{N}$, is the free $\mathcal{R}$ module $\mathcal{R}^{\otimes (n+1)}$, with the action on the first
factor. In each object $\lambda=(U,\xi,P)$, the module $B'_n(\lambda)$ is freely generated over $K\left[\mathcal{A}'_\lambda\right]$  by the symbols
$[P_1|P_2|...|P_n]$, where the $P_i$ are elements of $\mathcal{A}'_\lambda$, i.e. propositions implied by $P$.
Then the elements of $B'_n(\lambda)$ are finite sums of elements $P_0[P_1|P_2|...|P_n]$.\\
The first arrow from $B'_0$ to $K$ is the coordinate along $[\emptyset]$.\\
The higher boundary operators are of the Hochschild type, defined on the basis by the formula
\begin{multline}
\partial [P_1|P_2|...|P_n]=P_1[P_2|...|P_n]+\sum_{i=1}^{n-1}(-1)^{i}[P_1|...|P_iP_{i+1}|...|P_n]+(-1)^{n}[P_1|P_2|...|P_{n-1}]
\end{multline}

\noindent For each $n\in \mathbb{N}$,  the vector space ${\sf Ext}^{n}_{\mathcal{A}'}(K,\Phi)$ is the $n$-th group of cohomology
of the associated complex ${\sf Hom}_{\mathcal{A}'}(B^{\star },\Phi)$, made by natural transformations which commutes with the
action of $K[\mathcal{A}']$.\\
\indent The coboundary operator is defined by
\begin{equation}
	\begin{split}
		\delta f_\lambda(T;Q_0|...|Q_n)=\\
		 &f_\lambda(T|Q_0;Q_1|...|Q_n)+\sum_{i=0}^{n-1}(-1)^{i+1}f_\lambda(T;Q_0|...|Q_iQ_{i+1}|...|Q_n)
		+(-1)^{n+1}f_\lambda(T;Q_0|...|Q_{n-1}).
	\end{split}
\end{equation}

\noindent A cochain of degree zero is a section $\phi_\lambda;\lambda\in \mathcal{D}$ of $\Phi$, that is, a collection
of maps $\phi_\lambda:\Theta'_\lambda\rightarrow K$, such that, for any morphism $\gamma: \lambda\rightarrow\lambda'$
in $\mathcal{D}^{\rm op}$,
and any $S'\in \Theta'_{\lambda'}$, we have
\begin{equation}\label{naturality}
\phi_{\lambda'}(S')=\phi_{\lambda}(\pi^{\star }S').
\end{equation}

If there exists a unique last layer $U_{\rm out}$, as in the chain, this implies that the functions $\phi_\mu$ are all
determined by the functions $\phi_{\rm out}$ on the sets of theories $S_{\rm out}$ in the final logic, excluding given
propositions, by definition of the sets $\Theta'_{U,\xi,P}$. And a priori these final functions are arbitrary.\\

\subsection*{Acyclicity and fundamental cochains}

\noindent To be a cocycle, $\phi$ must satisfy, for any $\lambda=(U,\xi,P)$, and $P\leq Q$,
\begin{equation}\label{degreezerococycle}
0=\delta \phi ([Q])(S)=Q.\phi_\lambda (S)-\phi_\lambda(S)=\phi_\lambda(Q\Rightarrow S)-\phi_\lambda(S).
\end{equation}

\noindent However, for any $P$ we have $P\leq\top$, and $S|\top=\top$; then the invariance \eqref{degreezerococycle}
implies that $\phi_\lambda$ is independent of $S$; it is equal to $\phi_\lambda(\top)$.\\

Then, a cocycle is a collection elements $\phi(\lambda)$ in $K$, satisfying $\phi_{\lambda'}=\phi_\lambda$ each time there exists an arrow from
$\lambda$ to $\lambda'$ in $\widetilde{\mathcal{A}}'_{\rm strict}$, thus forming a section of the constant sheaf over $\widetilde{\mathcal{A}}'_{\rm strict}$.\\

\noindent This gives:\\
\begin{prop}As
	\begin{equation}
		{\sf Ext}^{0}_{\mathcal{A}'}(K,\Phi)=H^{0}(\widetilde{\mathcal{A}}'_{\rm strict};K)=K^{\pi_0\left(\widetilde{\mathcal{A}}'_{\rm strict}\right)},
	\end{equation}
	then degree zero cohomology counts the propositions that are transported by $\pi_{\star }$
	from the output.
\end{prop}
The discussion at the end of section \ref{sec:fibrations-and-cofibrations} describes the relation between the zero cohomology of information and the cats manifolds, that was identified before with the degree zero cohomology in the sense of \v{C}ech.\\

\noindent A degree one cochain is a collection $\phi^{R}_\lambda$ of measurable functions
on $\Theta'_\lambda$, and $R\in \mathcal{A}'_\lambda$, which satisfies the naturality hypothesis:
for any morphism $\gamma: \lambda\rightarrow\lambda'$ in $\mathcal{D}^{\rm op}$,
and any $S'\in \Theta'_{\lambda'}$, we have
\begin{equation}\label{naturalitydegreeone}
\phi_{\lambda'}^{\pi_\star R}(S')=\phi^{R}_\lambda(\pi^{\star }S').
\end{equation}
The cocycle equation is
\begin{equation}\label{degreeonecocycle}
\forall U,\xi, \forall P, \forall Q\geq P, \forall R\geq P, \forall S\in \Theta'_{U,\xi,P},
\phi_\lambda^{Q\wedge R}(S)=\phi_\lambda^{Q}(S)+\phi_\lambda^{R}(Q\Rightarrow S).
\end{equation}

\noindent Let us define a family of elements of $K$ by the equation
\begin{equation}
\psi_\lambda(S)=-\phi^{P}_\lambda(S).
\end{equation}
Formula \eqref{naturalitydegreeone} implies formula \eqref{naturality}, then $\psi_\lambda$
is a zero cochain.\\
Take its coboundary
\begin{equation}
\delta \psi_\lambda ([Q])(S)=-\phi^{P}_\lambda(S)+Q.\phi^{P}_\lambda(S).
\end{equation}
using the cocycle equation and the fact that for any $Q\geq P$ we have $Q\wedge P=P$, this gives
\begin{equation}
\phi_\lambda^{Q}(S)=\phi_\lambda^{Q\wedge P}(S)-Q.\phi^{P}_\lambda(S)=-\delta \psi_\lambda ([Q])(S).
\end{equation}

\noindent Remark that the cochain $\psi$ is not unique, the formula $\psi=-\phi_\lambda^{P}$ is
only a choice. Two cochains $\psi$ satisfying $\delta \psi=\phi$ differ by a zero cocycle, that is a family of
numbers $c_\lambda$, dependent
on $P$ but not on $S$. Remind us that $P$ is part of the object $\lambda$.\\

\noindent Therefore every one cocycle is a coboundary, or in other terms:\\
\begin{prop}
	${\sf Ext}^{1}_{\mathcal{A}'}(K,\Phi)=0$.
\end{prop}

The same argument applies to every degree $n\geq 1$, giving, 
\begin{prop}
	${\sf Ext}^{n}_{\mathcal{A}'}(K,\Phi)=0$.
\end{prop}
\begin{proof}
	If $\phi_\lambda^{Q_1;...;Q_n}$ is a cocycle of degree $n\geq 1$, where $\lambda=(U,\xi,P)$, the formula
	\begin{equation}
		\psi_\lambda^{Q_1;...;Q_{n-1}}=(-1)^{n} \phi_\lambda^{Q_1;...;Q_{n-1};P}
	\end{equation}
	defines a cochain of degree $n-1$ such that $\delta\psi=\phi$.\\
	Extracting $\phi_\lambda^{Q_1;...;Q_n}$ from the  last term of the cocycle equation for $\phi$, applied to $Q_1,..., Q_{n+1}$ with $Q_{n+1}=P$, gives
	\begin{equation}
		(-1)^{n}\phi_\lambda^{Q_1;...;Q_n}=Q_1.\phi_\lambda^{Q_2;...;Q_n;P}+\sum_{i=1}^{n-1}\phi_\lambda^{Q_2;...;Q_iQ_{i+1};...;Q_n;P}
		+(-1)^{n}\phi_\lambda^{Q_2;...;Q_n\wedge P}.
	\end{equation}
	As $Q_n\wedge P=P$ in $\mathcal{A}_\lambda$, this is exactly the coboundary of $\psi$ applied to $Q_1;...;Q_n$.
\end{proof}

\begin{rmk*}
	\normalfont At first sight this is a deception; however, there is a morality here, because it tells that
	the measure of semantic information reflects a value of a theory at the output, depending on many
	elements that the network does not know, without knowing the consequences of this theory. Some of these consequences can be
	included in the metric for learning, some other cannot be.
\end{rmk*}

\noindent When a cochain $\psi$ as above is chosen, it defines the degree one cocycle $\phi$ by the formula
\begin{equation}
\phi_\lambda^{Q}(S)=\psi_\lambda(Q\Rightarrow S)-\psi_\lambda(S).
\end{equation}
The cochain $\psi$ satisfied \eqref{naturality}, and the coboundary $\phi$ the
equation \eqref{naturalitydegreeone}.\\
\indent All the arbitrariness is contained in the values of $\psi_{\rm out}$, which are function of $P$ and of the theory excluding $P$.
Now examine the role of a proposition $Q$ implied by $P$. It changes the value of $\phi$ according to the equation
\begin{equation}
\phi_{\rm out}(Q;T)=\phi^{Q}_{\rm out}(T)=\phi^{P}_{\rm out}(T)-\phi^{P}_{\rm out}(T|Q)=\psi_{\rm out}(T|Q)-\psi_{\rm out}(T),
\end{equation}
then it subtracts from $\psi_{\rm out}(T)$ the conditioned value $\psi_{\rm out}(T|Q)$. And this is transmitted inside the network
by the equation
\begin{equation}
\phi_{\lambda'}^{\pi_\star Q}(S')=\phi^{Q}_\lambda(\pi^{\star }S');
\end{equation}
which is equivalent to the simplest equation
\begin{equation}
\psi_{\lambda'}(S')=\psi_{\lambda}(\pi^{\star }S').
\end{equation}

\noindent Note that we are working under the hypothesis $\pi^{\star }\pi_\star ={\sf Id}$, then it can happen that a theory
$S'$, in
the inner layers cannot be reconstructed (by $\pi_\star $) from its deduction $\pi^{\star }S'$ in the outer layer. Thus the logic inside
is richer than the transmitted propositions, but the quantity $\psi_{\lambda'}(S')$ depends only on $\pi^{\star }S'$.\\
This corresponds fairly well with what we observed in the experiments about simple classification problems,
with architectures more elaborated than a chain, (see Logical cells II, \cite{logic-DNN-2}). In some cases, the inner layers
invent propositions that are not stated in the objectives. They correspond to proofs of these objectives.\\

\subsection*{Mutual information, classical and quantum analogies}

\noindent We propose now an interpretation of the functions $\phi$ and $\psi$, when $\mathbb{K}=\mathbb{R}$,
or an ordered ring, as $\mathbb{Z}$: the value $\phi^{P}_{\rm out}(S)$ measures the ambiguity of $S$ with respect to
$\neg P$, then it is natural to assume that the value of $\psi_{\rm out}(S)$ is growing with $S$, i.e. $S\leq T$ implies $\psi_{\rm out}(S)\leq \psi_{\rm out}(T)$.\\

\noindent Among the theories which exclude $P$, there is a minimal one, which
is $\bot$, without much interest, even it has  the maximal information in the sense of Carnap and Bar-Hillel,
and a maximal theory, which is $\neg P$ itself; it is the more precise, but with the minimal information, if
we measure information by the quantity of exclusions of propositions it can give. Thus $\psi$ does not
count the quantity of possible information, but the closeness to $\neg P$.\\

\noindent Consequently, $\phi_P^{Q}(S)$ is always a positive number, which is decreasing in $Q$ when $S$ is given. Therefore,
we can take $\psi$ negative, by choosing $\psi_\lambda=-\phi_\lambda^{P}$. 
In what follows we consider this choice for $\psi$.\\
The maximal value of $\phi_P^{Q}(S)$, for a given $S$ is attained for $Q=P$, in this case $S|P=\neg P$, then the maximal value
is $\phi_\lambda^{P}(S)-\phi_\lambda^{P}(\neg P)$.\\

\noindent The truth of the proposition $\neg Q$ can be seen as a theory excluding $P$ when $P\leq Q$. Like a counterexample of $P$.\\

\noindent Note the following formula for $P\leq Q$:
\begin{equation}
\phi_\lambda^{Q}(S)=\phi_\lambda^{P}(S)-\phi_\lambda^{P}(S|Q).
\end{equation}

\noindent Remind that the entropy function $H$ of a joint probability is also always positive, and we have
\begin{equation}
I(X;Y)=H(X)-H(X|Y),
\end{equation}
as it follows from the Shannon equation and the definition of $I$.\\
This also gives $I(X;X)=H(X)$.\\

\noindent Then we interpret $\phi_\lambda^{Q}(S)$ as a mutual information between $S$ and $\neg Q$, and $\phi_\lambda^{P}(S)$
itself as a kind of entropy, thus measuring an ambiguity: the ambiguity of what is expressed in the layer $\lambda$
about the exclusion of $P$ at the output.\\
This is in agreement with next formula,
\begin{equation}
\phi_\lambda^{\pi_\star Q}(S)=\phi_{\rm out}^{Q}(\pi^{\star }S).
\end{equation}

\begin{rmk*}
	\normalfont In Quantum Information Theory, where variables are replaced by orthogonal decomposition of an Hilbert space, and probabilities are replaced by adapted positive hermitian operators of trace one \cite{Baudot-Bennequin}, the Shannon entropy $H$	(entropy of the associated classical law) appears as (minus) the coboundary of a cochain which is the Von Neumann entropy $S=-\log_2\mathrm{Trace}(\rho)$,
	\begin{equation}
		H_Y(Y;\rho)=S_X(\rho)-Y.S_X(\rho).
	\end{equation}
	Thus in the present case, it is better to consider that theories are analogs of density matrices, propositions are analogs of the observables, the function
	$\psi$ is an analog of the opposite of the Von-Neumann entropy, and the ambiguity $\phi$ an analog of the Shannon entropy.
\end{rmk*}

\noindent
Let us see what we get for a functioning network $X^{w}$, possessing a semantic functioning $S_{U,\xi}:X_{U,\xi}\rightarrow \Theta_{U,\xi}$,
not necessarily assuming the naturality \eqref{coherentsemanticfunctioning}. We can even specialize by taking a family of neurons having
an interest in the exclusion of some property $P$, and look at a family
\begin{equation}
S_\lambda:X_{U,\xi}\rightarrow \Theta'_\lambda,
\end{equation}
where $\lambda=(U,\xi,P)$.\\

\noindent To a true activity $x$ of the network, we get $x_{U,\xi}$, then, we define
\begin{equation}
H_\lambda^{Q}(x)=\phi_\lambda^{Q}(S_\lambda(x_{U,\xi})).
\end{equation}
And we propose it as the ambiguity in the layer $U,\xi$, about the proposition $P$ at the output, when $Q$ is given
as an example.\\

\noindent To understand better the role of $Q$, we apply the equation \eqref{naturalitydegreeone}, which gives
\begin{equation}
H_{\lambda'}^{\pi_\star Q}(x')=\phi_\lambda^{Q}(\pi^{\star }S'(x')).
\end{equation}
Therefore, evaluated on a proposition $\pi_\star Q$ which comes from the output, the above quantity $I\left(x'\right)$ in the hidden
layer $U'$, is the mutual information of $\neg Q$ and the deduction in $U_{\rm out}$ by $\pi^{\star }$
of the theory $S'(x')$, expressed
in $U'$ in presence of the given section (feedforward information flow), coming from the input,
by the activity $x'\in X_{U'}$.\\
\begin{rmk*}
	\normalfont Consider a chain $(U,\xi)\rightarrow (U',\xi')\rightarrow (U",\xi")$. We denote by $\rho_\star $ and $\rho^{\star }$ the applications
	which correspond to the arrow $(U',\xi')\rightarrow (U",\xi")$. Therefore $(\pi')^{\star }=\pi^{\star }\rho^{\star }$ and
	$\pi'_\star =\rho_\star \pi_\star $.\\
	For any section
	$x$, and proposition $P$ in the output $(U,\xi)$, consider the particular case $P=Q$, where $(Q\Rightarrow S)=\neg P$
	for every theory excluding $P$:
	\begin{align*}
		H(x')-H(x")&=\phi_\lambda^{P}(\pi^{\star }S'(x'))-\phi_\lambda^{P}(\pi^{\star }S'(x')|P)-(\phi_\lambda^{P}((\pi')^{\star }S"(x"))\\
		&\quad -\phi_\lambda^{P}((\pi')^{\star }S"(x")|P))\\
		&=\phi_\lambda^{P}(\pi^{\star }S'(x'))-\phi_\lambda^{P}((\pi')^{\star }S"(x"))\\
		&=\psi_\lambda(\pi^{\star }\rho^{\star }S"(x"))-\psi_\lambda(\pi^{\star }S'(x'))
	\end{align*}
	
	\noindent This is surely negative in practice, because the theory $S'(x')$ is larger than the theory $\rho^{\star }S"(x")$.
	For instance, at the end, we surely have $S_{\rm out}=\neg P$, as soon as the network has learned.\\
	Consequently this quantity has a tendency to be negative. Then it is not like
	the mutual information between the layers. It looks more as a difference of ambiguities. Because the ambiguity is decreasing in a functioning network, in reality.\\
	This confirms that $H$ is a measure of ambiguity.\\
	Therefore, the mutual information should come out in a manner that involves a pair of layers.
\end{rmk*}

\noindent To obtain a notion of mutual information, we make an extension of the monoids $\mathcal{A}_{U,\xi,P}$,
which continues to act by conditioning on the sets $\Theta_{U,\xi,P}$.\\

\noindent For that, we consider a fibration over $\mathcal{A}'_{\rm strict}$ made 
by monoids $\mathcal{D}_\lambda$ which contain
$\mathcal{A}_\lambda$ as submonoids.\\
By definition, if $\lambda=(U,\xi,P)$, an object of $\mathcal{D}_\lambda$ is an 
arrow $\gamma_0=(\alpha_0,h_0,\iota_0)$ of $\widetilde{\mathcal{A}}'_{\rm strict}$, going from a triple $(U_0,\xi_0,P_0)$ to a triple $(U,\xi,\pi_\star P_0)$, where $P\leq \pi_\star P_0$, and a morphism from $\gamma_0=(\alpha_0,h_0,\iota_0)$ to $\gamma_1=(\alpha_1,h_1,\iota_1)$
is a morphism $\gamma_{10}$ from $(U_0,\xi_0,P_0)$ to $(U_1,\xi_1,Q_1=\pi^{\alpha_{10},h_{10}}_\star P_0)$ such that $Q_1\geq P_1$.\\
\indent For the intuition it is better to see the objects as arrows in the opposite category $\mathcal{D}$ of  $\widetilde{\mathcal{A}}'_{\rm strict}$, in such a manner
they can compose with the arrows $Q\leq R$ in the monoidal category $\mathcal{A}_\lambda$, then we get a variant of the right slice $\lambda|\mathcal{D}$, just
extended by $\mathcal{A}_\lambda$. The category
$\mathcal{D}_\lambda$ is monoidal and strict if we define the product by
\begin{equation}
\gamma_1\otimes\gamma_2=(U,\xi,\pi^{\gamma_1}_{\star }P_1\wedge \pi^{\gamma_0}_{\star }P_2).
\end{equation}
The identity being the truth $\top_\lambda$.\\

\noindent We also define the action of $\mathcal{D}_\lambda$ on $\Theta_\lambda$ as follows:\\
for every arrow $\gamma_0:\lambda_0\rightarrow \lambda_{\pi_\star P_0}$, where $\lambda_0=(U_0,\xi_0,P_0)$, and where $\lambda_{\pi_\star P_0}$
denotes $(U,\xi,\pi_\star P_0)$, assuming $\pi_\star P_0\geq P$, we define
\begin{equation}
\gamma_0.\mathbb{T}=(\pi^{\gamma_0}_{\star }P_0\Rightarrow \mathbb{T}).
\end{equation}
This gives an action of the monoid of propositions in $\mathcal{A}_{\lambda_0}$ which are implied by $P_0$, whose images by $\pi_\star $
are implied by $P$.\\
If $P_0\leq Q_0$ and $P_0\leq R_0$, we have $\pi^{\gamma_0}_{\star }(Q_0\wedge R_0)=\pi^{\gamma_0}_{\star }(Q_0)\wedge\pi^{\gamma_0}_{\star }(R_0)$.\\

\noindent The monoidal categories $\mathcal{D}_\lambda;\lambda\in \mathcal{D}$ form a natural presheaf $\mathcal{D}\backslash\mathcal{D}$ over $\mathcal{D}$. For
any morphism $\gamma=(\alpha,h,\iota)$ of  $\widetilde{\mathcal{A}}'_{strict}$, going from $(U,\xi,P)$ to $(U',\xi',\pi_\star P)$, and
any object $\gamma_0:\lambda_0\rightarrow \lambda_{\pi_\star P_0}$ in $\mathcal{D}_\lambda$, we define $\gamma_\star (\gamma_0)$ by the
composition $(\alpha,h)\circ(\alpha_0,\xi_0)$ and the proposition $\pi^{\gamma}_\star \circ\pi_\star P_0$ in $\mathcal{A}_{\lambda'}$.\\

\noindent The naturalness of the monoidal action on the theories
follows from $\pi_\gamma^{\star }\pi^{\gamma}_\star ={\sf Id}_U$:
\begin{align*}
\pi_\gamma^{\star }[\gamma_\star (\pi_\star P_0).T']&=\pi_\gamma^{\star }[\pi^{\gamma}_\star \pi_\star P_0\Rightarrow T']\\
&=\pi_\gamma^{\star }\pi^{\gamma}_\star \pi_\star P_0\Rightarrow \pi_\gamma^{\star }T'\\
&=\pi_\star P_0\Rightarrow \pi_\gamma^{\star }T'
\end{align*}
Then, defining $[\Phi_\star (\gamma)(\phi_\lambda)](T')=\phi_\lambda(\pi_\gamma^{\star }T')$, we get the following result\\
\begin{lem}
	\begin{equation}
		[\Phi_\star (\gamma)\phi_\lambda](\gamma_\star (\gamma_0).T')=\phi_\lambda(\gamma_0.\pi^{\star }T').
	\end{equation}
\end{lem}

\noindent Consequently
the methods of Abelian homological algebra can be applied \cite{2012homology}.\\

The (non-homogeneous) bar construction makes now appeal to symbols $[\gamma_1|\gamma_2|...|\gamma_n]$, where the $\gamma_i$ are elements of
$\mathcal{D}_\lambda$. The action of algebra pass through the direct image of propositions $\pi_{\star }P_i;i=1,...,n$.\\
\indent Things are very similar to what happened with the precedent monoids $\mathcal{A}'_\lambda$:\\
the zero cochains are families $\phi_\lambda$  of maps on theories satisfying
\begin{equation}
\psi_\lambda(\pi^{\star }T')=\psi_{\lambda'}(T'),
\end{equation}
where $\gamma:\lambda\rightarrow\lambda'$ is a morphism in $\widetilde{A}'_{\rm strict}$.\\
The coboundary operator is
\begin{equation}
\delta \psi_\lambda([\gamma_1])=\psi_\lambda(T|\pi_\star^{\gamma_1}P_1)-\psi_\lambda(T).
\end{equation}
Then the cohomology is defined as before. We get analog propositions. For instance, the
degree one cochains are collections of maps of theories $\phi_\lambda^{\gamma_1}$ satisfying
\begin{equation}
\phi_{\lambda}^{\gamma_1}(\pi^{\star }T')=\phi_{\lambda'}^{\gamma_\star \gamma_1}(T');
\end{equation}
the cocycle equation is
\begin{equation}
\phi_\lambda^{\gamma_1\wedge \gamma_2}=\phi_\lambda^{\gamma_1}+\gamma_1.\phi_\lambda^{\gamma_2}.
\end{equation}
One more time, the cocycles are coboundaries; the following formula is easily verified
\begin{equation}
\phi_\lambda^{\gamma_1}=(\delta\psi_\lambda)[\gamma_1]=\pi_{\star }P_1.\psi_\lambda-\psi_\lambda;
\end{equation}
where
\begin{equation}
\psi_\lambda=-\phi_\lambda^{{\sf Id}_\lambda}.
\end{equation}

\noindent The new interesting point is the definition of a \emph{mutual information}. For that we mimic the formulas
of Shannon theory: we apply a combinatorial operator to the ambiguity. Then we consider the
canonical bar resolution for ${\sf Ext}^{\star }_{\mathcal{D}}(\mathbb{K},\Phi)$, with the trivial action of
$\mathcal{A}'|\lambda; \lambda\in \widetilde{\mathcal{A}}$. The operator is the combinatorial coboundary $\delta^{t}$ at degree two, and it gives:
\begin{equation}\label{mutualinformation}
I_\lambda(\gamma_1;\gamma_2)=\delta^{t}\phi_\lambda[\gamma_1,\gamma_2]\\=\phi_\lambda^{\gamma_1}-\phi_\lambda^{\gamma_1\wedge \gamma_1}+\phi_\lambda^{\gamma_2}.
\end{equation}
This gives the following formulas
\begin{equation}
I_\lambda(\gamma_1;\gamma_2)=\phi_\lambda^{\gamma_1}-\gamma_2.\phi_\lambda^{\gamma_1}=\phi_\lambda^{\gamma_2}-\gamma_1.\phi_\lambda^{\gamma_2}.
\end{equation}

More concretely,  for two morphisms $\gamma_1:\lambda_1\rightarrow\lambda$ and $\gamma_2:\lambda_2\rightarrow\lambda$,
denoting by $P_1,P_2$ their respective coordinates on propositions, and by $\psi_\lambda=-\phi_\lambda^{\lambda}$ the canonical $0$-cochain, we have:
\[
I_\lambda(\gamma_1;\gamma_2)(T)=\psi_\lambda(T|\pi_{\star }P_2)+\psi_\lambda(T|\pi_{\star }P_1)-\psi_\lambda(T|\pi_{\star }P_1\wedge\pi_{\star }P_2)-\psi_\lambda(T)
\]
\begin{rmk*}
	\normalfont We decided that the interpretation of $\phi_\lambda$ is better when $\psi_\lambda$ is growing. Now, assuming
	the positivity of $I_\lambda$, we get a kind of concavity of $\psi_\lambda$.
\end{rmk*}

\noindent More generally, we say that a real function $\psi$ of the theories, 
containing $\vdash\neg P$, in a given language, is \emph{concave} ({\em resp.}
strictly concave), if for any pair of
such theories $T\leq T'$ and any proposition $Q\geq P$, the following expression is positive ({\em resp.} strictly positive),
\begin{equation}
I_P(Q;T,T')=\psi(T|Q)-\psi(T)-\psi(T'|Q)+\psi(T').
\end{equation}
Remark that this definition extends \emph{verbatim} to any closed monoidal category, because it uses only the pre-order and the exponential.\\
The positivity of the mutual information is the particular case where $T'=T|Q_1$.\\

\noindent This makes $\psi$ look like the function $\log\left(\ln P\right)$ for a domain $\bot< P \leq \neg P$, analog of the interval
$] 0,1[$ in the propositional context.\\

The functions $\psi_\lambda$ can always be chosen such that $\phi_\lambda^{P}=-\psi_\lambda$. Then the above interpretation
of $\phi$ as an informational ambiguity is compatible with an interpretation of $\psi(T)$ as  a measure of the \emph{precision} of the theory.\\

\subsection*{The Boolean case, comparing to Carnap and Bar-Hillel \cite{CBH52}}

In the finite Boolean case, \emph{the opposite of the content} defined by Carnap and Bar-Hillel gives such a function $\psi$, strictly increasing and concave.
Remind that the content set $C(T)$ is the set of elementary propositions that are excluded by the theory $T$. Here we assimilate a theory with the language 
and its axioms, and with a subset of a finite set $E$. If $T< T'$, there is less excluded points by $T'$ than by $T$, then
$-c(T')-(-c(T))> 0$. If $P\leq Q$, the content set of $T\vee \neg Q$ is the intersection of $C(T)$ and $C(\vdash \neg Q)=C(Q)^{c}$, and the content of
$T'\vee \neg Q$ the intersection of $C(T)$ and $C(\vdash \neg Q)=C(Q)^{c}$, then the complement of $C(T'\vee \neg Q)$ in $C(T')$ is contained
in the complement of $C(T\vee \neg Q)$ in $C(T)$. Consequently
\begin{equation}
\psi(T|Q)-\psi(T)-(\psi(T'|Q)-\psi(T'))=c(T)-c(T|Q)-(c(T')-c(T'|Q))\geq 0.
\end{equation}
It is zero when $T'\wedge(\neg Q)\leq T$.\\
A natural manner to obtain a strictly concave function is to apply the logarithm function to the function $(c_{\max}-c(T))/c_{\max}$.\\
Therefore a natural formula in the boolean case is
\begin{equation}
\psi_P(\mathbb{T})=\ln\frac{c(\bot)-c(\mathbb{T})}{c(\bot)-c(\neg P)}
\end{equation}
But we also could take a uniform normalization:
\begin{equation}
\psi_\bot(\mathbb{T})=\ln\frac{c(\bot)-c(\mathbb{T})}{c(\bot)}
\end{equation}
Amazingly, this was the definition of the amount of information (with a minus sign) of Carnap and Bar-Hillel \cite{CBH52}.\\
A generalization along their line consists to choose any strictly positive function $m$ of the elementary propositions and
to define the numerical content $c(T)$ as the sum of the values of $m$ over the elements excluded by $T$. This corresponds to
the attribution of more or less value to the individual elements.\\
We essentially recover the basis of the theory presented by Bao, Basu et al. \cite{bao-basu}, \cite{basu-bao}.\\
\begin{ques}
	Does a natural formula exist, that is valid in every Heyting algebra, or at least in a class of Heyting algebras larger than Boole algebras?
\end{ques}
\begin{ex}
	\normalfont The open sets of a topology on a finite set $X$. The analog of the content of $T$ is the cardinality
	of the closed set $X\setminus T$. Then a preliminary function $\psi$ is the cardinality of $T$ itself, which is naturally
	increasing with $T$. However simple examples show that this function can be non-concave. The set $T|Q\backslash T$ is made by the
	points $x$ of $X\backslash T$ having a neighborhood $V$ such that $V\cap V\subset T$, there exists no relation between this set
	and the analog set for $T'$ larger than $T$, but smaller than $\neg P$.
\end{ex}

\noindent However, appendix \ref{app:non-boolean} constructs a good function $\psi$ for the sites of DNNs and the injective finite sheaves.
This applies in particular to the chains $0\rightarrow 1\rightarrow ...\rightarrow n$.\\

\subsection*{A remark on semantic independency}

In their 1952 report \cite{CBH52}, Carnap and bar-Hillel gave a different justification than us for taking the logarithm of
a normalized version of the content. This was in the Boolean situation, $n=0$, but our appendix \ref{app:non-boolean} extends what they said to some non-Boolean situations.\\
\indent They had in mind that independent assertions must give an addition of the amounts of information of the
separate assertions. However, as they recognized themselves, the concept of semantic independency is not very clear \cite[page 12]{CBH52}. In fact they studied a particular case of typed language that they named $\mathcal{L}_n^{\pi}$,
where there exists one type of subjects with $n$ elements, $a,b,c,...$, that can have a given number $\pi$
of attributes (or predicate). The example is three humans, their gender (male or female), and their age (old or young). For every elementary proposition $Z_i$, i.e. a point inn $E$, they choose a number $m_P(Z_i)$
in $]0,1|$, and define, as in the preceding section with $\mu$, the function $m$ of any proposition $L$, by taking
the sum of the $m_i$ over the elements of $L$, viewed as a subset of $E$.\\
\indent Carnap and Bar-Hillel imposed several axioms on $m_P$, for instance the invariance under the natural action of the symmetry group
$\mathfrak{S}_n\times\mathfrak{G}_\pi$, where $\mathfrak{G}_\pi$ describes the symmetries between the predicates, and the normalization by $m(E)=1$. The \emph{content} is an additive normalization
of the opposite of $m$. The number $c(L)$ evaluates the quantity of elementary propositions excluded by $L$.\\
\indent At some moment, they introduce axiom $h$, \cite[page 14]{CBH52}, $m(Q\wedge R)=m(Q)m(R)$, if $Q$ and $R$ do not consider
any common predicate. This axiom was rarely considered in the rest of the paper. However it is followed by a definition:
two assertions $S$ and $T$ were said inductively independent (with respect to $m_P$) if an only if
\begin{equation}
m(S\wedge T)=m(S)m(T).
\end{equation}
This was obviously inspired from the theory of probabilities \cite{Carnap1950-CARLFO}, where primitive predicates
are considered in relation to probabilities.\\
\indent If we think of the example with the age and the gender, the axiom is not very convincing from the point of view of probability, because in most sufficiently large population of humans it
is not true that age and gender are independent. However, from a \emph{semantic point of view}, this is completely justified!\\

Now, if we come to the amount of information, taking the logarithm of the inverse of $m(T)$ to measure ${\sf inf}(T)$ makes that
independency (inductive) is equivalent to the additivity:
\begin{equation}
\psi(S\wedge T)=\psi(S)+\psi(T).
\end{equation}

Under this form, the definition still has a meaning, for any function $\psi$. Even with values in a category of models, with a good notion
of colimit, as the disjoint union of sets.\\

In Shannon's theory, with the set theoretic interpretation of Hu Kuo Ting, \cite{HKT}, we recover the same thing.\\

\subsection*{Comparison of information between layers}

Another way to obtain a comparison between layers, i.e. objects $(U,\xi)$, comes from the ordinary cohomology
of the object $\Phi$ in the topos of presheaves over the opposite category of $\widetilde{\mathcal{A}}'_{\rm strict}$, that we named $\mathcal{D}$.\\
\indent This cohomology can be computed following the method exposed by Grothendieck and Verdier in SGA 4 \cite{SGA4}, using a canonical resolution of $\Phi$.
This resolution is constructed from the nerve $\mathcal{N}(\mathcal{D})$, made by the sequences of arrows $\mathcal{\lambda}\rightarrow\lambda_1\rightarrow \lambda_2 ...$
in $\widetilde{\mathcal{A}}'_{\rm strict}$, then associated to the fibration by the slices category $\lambda|\mathcal{D}$ over $\mathcal{D}$.
Be carefull that in $\mathcal{D}$, the arrows are in reverse order.\\

The nerve $\mathcal{N}(\mathcal{D})$ has a natural structure of simplicial set whose $n$ simplices are sequences of composable arrows $(\gamma_1,...,\gamma_n)$
between objects $\lambda_0 \to \cdots \to \lambda_n$
in $\widetilde{\mathcal{A}}'_{\rm strict}$, and whose face operators $d_i;i=0,...,n$ are given by the following formulas:
\begin{align*}
d_0(\gamma_1,...,\gamma_n)&=(\gamma_2,...,\gamma_n)\\
d_i(\gamma_1,...,\gamma_n)&=(\gamma_1,...,\gamma_{i+1}\circ \gamma_i,...,\gamma_n) \text{ if } 0 < i < n\\
d_n(\gamma_1,...,\gamma_n)&=(\gamma_1,...,\gamma_{n-1}).
\end{align*}

This allows to define a \emph{canonical cochain complex} $(C^n(\mathcal{D}, \Phi),d)$ which cohomology is $H^{\star }(\mathcal{D},\Phi)$.\\

\noindent The $n$-cochains are
\begin{equation}\label{topospresheafcocycle}
C^n(\mathcal{D},\Phi) = \prod_{\lambda_0 \to \cdots \to \lambda_n} \Phi_{\lambda_n}
\end{equation}
and the coboundary operator $\delta: C^{n-1}(\mathcal{D},\Phi)\to C^n(\mathcal{D},\Phi)$ is given by
\begin{equation}\label{coboundarypresheafcohomology}
(\delta \phi)_{\lambda_0\to \cdots \to \lambda_n} = \sum_{i=0}^{n-1} (-1)^i \phi_{d_i(\lambda_0 \to \cdots \lambda_n)}
+(-1)^{n}\Phi_{\star }(\gamma_n)\phi_{d_n(\lambda_0\to \cdots \lambda_n)}.
\end{equation}

\noindent For instance at degree zero, this gives, for $\gamma:\lambda\rightarrow \lambda'$
\begin{equation}
\delta\phi^{0}_\gamma(S')=\phi^{0}_{\lambda'}(S')-\phi^{0}_\lambda(\pi^{\star }S').
\end{equation}

For our cocycle $\phi_\lambda^{Q}$, with $P\leq Q$, a more convenient sheaf over $\mathcal{D}$ is given
by the sets $\Psi_\lambda$ of functions of the pairs $(S,Q)$, with $S$ excluding $P$ and $P$ implying $Q$,
with morphisms
\begin{equation}
\Psi_\star (\gamma)(S',Q')=\psi(\pi^{\star }S',\pi^{\star }Q').
\end{equation}
This gives
\begin{equation}
\delta\phi^{0}_\gamma(S',Q')=\phi^{0}_{\lambda'}(S',Q')-\phi^{0}_\lambda(\pi^{\star }S',\pi^{\star }Q').
\end{equation}
In our case, with $\phi^{0}_\lambda(S,Q)=\phi^{Q}_\lambda(S)$, we get the measure of the evolution
of the ambiguity along the network.\\

\noindent From now on, we change topic and consider the reverse direction of 
propagation of theories and propositions.\\

\subsection*{The particular case of natural isomorphisms}

\noindent Until the end of this subsection, we consider the particular case of isomorphisms between the logics in the layers, i.e.
$\pi^{\star }\pi_\star ={\sf Id}_{U}$ and $\pi_\star \pi^{\star }={\sf Id}_{U'}$.\\
As we will see, this is rather deceptive, giving a particular case of the preceding notion of ambiguity and information, obtained without the hypothesis of isomorphism,
then it can be skipped easily, but it seemed necessary to explore what possibilities were offered by the contravariant side of $\widetilde{A}$.\\

\indent In this case we are allowed to consider the sheaf of propositions $\mathcal{A}$ for $\pi^{\star }$ together
and the cosheaf of theories $\Theta$ for $\pi_\star $ over the category $\widetilde{\mathcal{A}}$. The action of $\mathcal{A}$
by conditioning on the sheaf $\Phi$ of measurable functions on $\Theta$ is natural, (see proposition \ref{prop:compatibility}).\\
Thus we can apply the same strategy as before, using the bar complex.\\

\noindent The zero cochains satisfy
\begin{equation}
\psi_{\lambda'}(\pi_\star T)=\psi_\lambda(T).
\end{equation}
This equation implies the naturality \eqref{naturality}.
However, there is a difference with the preceding framework, because we have more morphisms to take in account, i.e. the implications $P\leq P'$.
This implies that, for $U,\xi$ fixed, $\phi$ does not depend on $P$; there exists a function $\psi_{U,\xi}$ on all the theories
such that $\psi_\lambda$ on $\Theta(U,\xi,P)$ is its restriction.\\
\emph{Proof}: for any pair $P\leq Q$ in $\mathcal{A}_\lambda$,
and any theory which excludes $Q$ then $P$, we have
$\psi_P(S)=\psi_Q(S)$. Therefore $\psi_P=\psi_\bot$.\\
\noindent The equation of cocycle is the same as before, i.e. \eqref{degreezerococycle}. It
implies that $\psi_{U,\xi}$ is invariant by the action of $\mathcal{A}_\lambda$. In every case, boolean or not,
this implies that $\phi_{U,\xi}$ is also independent
of the theory $T$. Therefore the $H^{0}$ now simply counts the sections of $\mathcal{F}$.\\

\noindent The degree one cochains satisfy
\begin{equation}
\phi_{\lambda'}^{R'}(\pi_\star S)=\phi^{\pi^{\star }R'}_\lambda(S).
\end{equation}
In particular, for any triple $P\leq Q\leq R$,
and any $S\in \Theta_P$, we have
\begin{equation}
\phi_{U,\xi,Q}^{R}(S)=\phi_{U,\xi,P}^{R}(S),
\end{equation}
which allows us to consider only the elements of the form $\phi_\lambda^{P}$, that we denote simply $\phi_\lambda$.\\

\noindent The cocycle equation is as before, \eqref{degreeonecocycle}:
And taking $\psi_\lambda=-\phi_\lambda$ gives canonically a zero whose coboundary is $\phi$:
\begin{equation}
\phi_\lambda^{Q}(S)=\psi_\lambda(S)-\psi_\lambda(S| Q).
\end{equation}
Which defines the dependency of $\phi$ in $Q$.\\

\noindent The naturality, in the case of isomorphisms, for a connected network, with a unique
output layer, tells that everything can be computed in the output layer. The intervention of the layers
is illusory. Then it is sufficient to consider the case of one layer and logical calculus.\\
What follows is only a verification that things transport naturally to the whole category $\widetilde{A}$.\\

\noindent The extension of monoids is made via the left slices categories $\lambda| \mathcal{A}$;
the action of $\lambda| \mathcal{A}$ on $\Theta_\lambda$ is given by
\begin{equation}\label{arrowaction}
\gamma.\mathbb{T}=(\pi_\gamma^{\star }P'\Rightarrow \mathbb{T})=\mathbb{T}|\pi_\gamma^{\star }P'
\end{equation}
where $\gamma:\lambda\rightarrow \lambda'$, $\lambda=(U,\xi,P)$, $\lambda'=(U',\xi',P')$, $P\leq \pi^{\star }P'$,
and $\pi_\gamma=(\alpha,h)$ is the projected morphism of $\mathcal{F}$.\\
This defines an action of the monoid of propositions in $\mathcal{A}_{\lambda'}$ which are implied by $P'$.
If $P'\leq Q'$ and $P'\leq R'$, we have $\pi_\gamma^{\star }(Q'\wedge R')=\pi_\gamma^{\star }(Q')\wedge\pi_\gamma^{\star }(R')$.\\
\noindent A natural structure of monoid is given by
\begin{equation}
\gamma_1.\gamma_2=(U,\xi,\pi^{\star }\gamma_1\wedge \pi^{\star }\gamma_2).
\end{equation}
This works because, for a morphism $\gamma:\lambda\rightarrow \lambda'$, we have $P\leq \pi_\gamma^{\star }P'$.\\
The identity is the truth $\top_\lambda$.\\
\begin{lem}
	The naturality of the operations over $\mathcal{A}'$ follows from the further hypothesis: for every morphism $(\alpha,h)$,
	we assume that the counit $\pi^{\star }\pi_\star $ is equal to ${\sf Id}_{\mathbb{L}_{U,\xi}}$.
\end{lem}
\begin{proof}
	Consider an arrow $\rho:\lambda\rightarrow\lambda_1$; it gives a morphism $\rho^{\star }:\lambda_1|\mathcal{A}\rightarrow \lambda|\mathcal{A}$.\\
	For a morphism $\gamma_1:\lambda_1\rightarrow\lambda'_1$, $\rho^{\star }(\lambda_1)=\gamma_1\circ \rho$.\\
	If $\gamma_1:\lambda_1\rightarrow \lambda'_1$ is an arrow in $\mathcal{A}'$, where $\lambda'_1=(U'_1,\xi'_1,P'_1)$, and $T$ a theory in $\Theta_\lambda$, we have
	\begin{align*}
		\rho^{\star }(\gamma_1).T&=\pi_{\gamma_1\circ \rho}^{\star }P'_1\Rightarrow T\\
		&=\pi_\rho^{\star }\pi_{\gamma_1}^{\star }P'_1\Rightarrow \pi_\rho^{\star }(\pi_\rho)_\star T\\
		&=\pi_\rho^{\star }[\pi_{\gamma_1}^{\star }P'_1\Rightarrow (\pi_\rho)_\star T]\\
		&=\pi_\rho^{\star }[\gamma_1.(\pi_\rho)_\star T]\\
		&=\rho^{\star }(\gamma_1.\rho_\star T)
	\end{align*}
\end{proof}	
The monoids $\lambda|\widetilde{\mathcal{A}}$ is a presheaf over $\widetilde{\mathcal{A}}$, only in the case of isomorphisms, i.e. $\pi_\star\pi^{\star}={\sf Id}_{\lambda'}$.

The bar construction now makes appeal to symbols $[\gamma_1|\gamma_2|...|\gamma_n|$, where the $\gamma_i$ are arrows issued
from $\lambda$. The action of algebra pass through the inverse image of propositions $\pi^{\star}P_i$.\\
The zero cochains are families $\phi_\lambda$  of maps on theories satisfying
\begin{equation}
\psi_\lambda(T)=\psi_{\lambda'}(\pi_\star  T),
\end{equation}
where $\gamma:\lambda\rightarrow\lambda'$ is a morphism in $\widetilde{A}$.\\
The coboundary operator is
\begin{equation}
\delta \psi_\lambda([\gamma_1])=\psi_\lambda(T|\pi_\gamma^{\star }P_1)-\psi_\lambda(T).
\end{equation}
Then the cohomology is as before.\\
The one cochains are collections of maps of theories $\phi_\lambda^{\gamma_1}$ satisfying
\begin{equation}
\phi_{\lambda'}^{\gamma'_1}(\pi_\star T)=\phi_\lambda^{\gamma'_1\circ\gamma}(T).
\end{equation}
The cocycle equation is
\begin{equation}
\phi_\lambda^{\gamma_1\wedge \gamma_2}=\phi_\lambda^{\gamma_1}+\gamma_1.\phi_\lambda^{\gamma_2}.
\end{equation}
One more time, the cocycles are coboundaries; the following formula is easily verified
\begin{equation}
\phi_\lambda^{\lambda_1}=(\delta\psi_\lambda)[\lambda_1]=\pi^{\star}P_1.\psi_\lambda-\psi_\lambda;
\end{equation}
where
\begin{equation}
\psi_\lambda=-\phi_\lambda^{Id_\lambda}.
\end{equation}

\noindent The combinatorial coboundary $\delta^{t}$ at degree two gives:
\begin{equation}
I_\lambda(\gamma_1;\gamma_2)=\delta^{t}\phi_\lambda[\gamma_1,\gamma_2]\\=\phi_\lambda^{\gamma_1}-\phi_\lambda^{\gamma_1\wedge \gamma_1}+\phi_\lambda^{\gamma_2}.
\end{equation}
This gives the following formulas
\begin{equation}
I_\lambda(\gamma_1;\gamma_2)=\phi_\lambda^{\gamma_1}-\gamma_2.\phi_\lambda^{\gamma_1}=\phi_\lambda^{\gamma_2}-\gamma_1.\phi_\lambda^{\gamma_2}.
\end{equation}

More concretely,  for two morphisms $\gamma_1:\lambda_1\rightarrow\lambda$ and $\gamma_2:\lambda_2\rightarrow\lambda$,
denoting by $P_1,P_2$ their respective coordinates on propositions, and by $\psi_\lambda=-\phi_\lambda^{\lambda}$ the canonical $0$-cochain, we have:
\begin{equation}
I_\lambda(\gamma_1;\gamma_2)(T)=\psi_\lambda(T|\pi^{\star}P_1\wedge\pi^{\star}P_2)-\psi_\lambda(T|\pi^{\star}P_1)-\psi_\lambda(T|\pi^{\star}P_2)+\psi_\lambda(T)
\end{equation}

In a unique layer $U$, for a given context $\xi$, we get
\begin{equation}
I(P_1;P_2)(T)=\psi(T|P_1\wedge P_2)-\psi(T|P_1)-\psi(T|P_2)+\psi(T).
\end{equation}
This is the particular case of the mutual information we got before, see equation \eqref{mutualinformation}, because now, the generating function $\psi$
is the restriction to $\Theta(P)$ of a function that is defined on $\Theta=\Theta(\bot)$.

\section{Homotopy constructions}\label{sec:homotopy}
\subsection*{Abelian homogeneous bar complex of information}

We start by describing an homogeneous version of the information cocycles, giving first the differences of ambiguities, from which
the above ambiguity can be derived by reducing redundancy.
For that purpose we consider equivariant cochains as in \cite{Baudot-Bennequin}.\\
\indent The sets $\Theta_\lambda$, where
$\lambda=(U,\xi,P)$, are now extended by the symbols $[\gamma_0|\gamma_1|...|\gamma_n]$, where $n\in \mathbb{N}$, and the $\gamma_i;i=0,...,n$,  are objects
of the category $\mathcal{D}_\lambda$ or arrows in $\widetilde{\mathcal{A}}'_{\rm strict}$ abutting to $\lambda_R=(U,\xi,R)$
for $P\leq R$.\\
This extension with $n+1$ symbols is denoted by $\Theta^{n}_\lambda$. It represents the possible theories in the local language and
its context $U,\xi$, excluding the validity of $P$, augmented by the possibility to use counter-examples $\neg Q_i, i=0,...,n$.
There is a natural simplicial structure on the union $\Theta^{\bullet}_\lambda$
of these sets. The face operators $d_i;i=0,...,n$ being given by the following formulas:
\begin{align*}
d_0(\gamma_0,...,\gamma_n)&=(\gamma_1,...,\gamma_n)\\
d_i(\gamma_0,...,\gamma_n)&=(\gamma_0,...,\gamma_{i-1},\gamma_{i+1}...,\gamma_n) \text{ if } 0 < i < n\\
d_{n}(\gamma_0,...,\gamma_n)&=(\gamma_0,...,\gamma_{n-1}).
\end{align*}
By definition, the geometric realization of $\Theta^{\bullet}_\lambda$ is named the space of theories at $\lambda$
or localized at $\lambda$. Its homotopy type
is named the \emph{algebraic homotopy type} of theories, also at $\lambda$.\\

\noindent Remind that a simplicial set $K$ is a presheaf over the category $\Delta$, with objects $\mathbb{N}$ and morphisms
from $m$ to $n$, the  non decreasing maps from $[m]=\{ 1,...,m\}$ to $[n]=\{ 1,...,n\}$. The \emph{geometric realization} $|K|$
of a simplicial set $K$ is the topological space obtained
by quotienting the disjoint union of the products $K_n\times \Delta(n)$, where $K_n= K([n])$ and $\Delta(n)\subset
{\mathbb R}^{n+1}$ is   the geometric standard simplex, by the equivalence relation that identifies $(x,\varphi_\star(y))$
and $(\varphi^{\star}(x),y)$ for every nondecreasing map $\varphi:[m]\rightarrow [n]$, every $x\in K_n$ and every
$y\in \Delta(m)$; here $f^\star$ is $K(f)$ and $f_\star$ is the restriction to $\Delta(n)$ of the unique linear map from $\mathbb{R}^{n+1}$
to $\mathbb{R}^{m+1}$ that sends
the canonical vector $e_i$ to $e_{f(i)}$. In this construction, for $n\in \mathbb{N}$, $K_n$ is equipped with the discrete topology
and $\Delta(n)$ with its usual topology, then compact, the topology on the union over $n\in \mathbb{N}$ is the weak topology, i.e. a subset is closed
if and only if its intersection with each closed simplex is closed, and the realization is equipped with the quotient topology, the finest
making the quotient map continuous.
In particular, even it is not obvious at first glance, the realization of the simplicial set $\Delta^{k}$ is the standard simplex $\Delta(k)$.\\

\noindent Let $\mathbb{K}$ be commutative ring of cardinality at most continuous (conditions of measurability will be considered later).
We consider the rings $\Phi^{n}_\lambda;n\in \mathbb{N}$ of (measurable) functions  on the respective $\Theta^{n}_\lambda$ with values in $\mathbb{K}$.\\
The above simplicial structure gives a differential complex on the graded sum $\Phi^{\bullet}_\lambda$ of the $\Phi^{n}_\lambda;n\in \mathbb{N}$,
with the simplicial (or combinatorial) coboundary operator
\begin{equation}\label{coboundarycombinatorialcohomology}
(\delta_\lambda \phi)_\lambda^{\gamma_0| \cdots  |\gamma_n} = \sum_{i=0}^{n} (-1)^i \phi^{\gamma_0|  \cdots |\widehat{\gamma_i}|\cdots |\gamma_n}.
\end{equation}
We call \emph{algebraic cocycles} the elements in the kernel.\\

\noindent As we have seen, the arrows $\gamma_Q\in \mathcal{D}_\lambda$ can be multiplied, using the operation $\wedge$ on propositions in $\mathcal{A}_\lambda$,
and this defines an action of monoid on $\Theta_\lambda$ by the conditioning operation. Therefore we can define the \emph{homogeneous functions} or
\emph{homogeneous algebraic cochains} of degree
$n\in \mathbb{N}$ as the (measurable)  functions $\phi_\lambda^{\gamma_0;\gamma_1;...;\gamma_n}$ on $\Theta_\lambda$,
such that for any $\gamma_Q$ in $\mathcal{D}_\lambda$, abutting in $(U,\xi,Q)$, for $P\leq Q$, and any $T\in \Theta_\lambda$, thus excluding $P$,
\begin{equation}\label{homogeneousalgebraiccochain}
\phi_\lambda^{\gamma_Q\wedge\gamma_0;\gamma_Q\wedge\gamma_1;...;\gamma_Q\wedge\gamma_n}(T)
=\phi_\lambda^{\gamma_0;\gamma_1;...;\gamma_n}(T|Q).
\end{equation}

\noindent The above operator $\delta_\lambda$ preserves the homogeneous algebraic cochains. The kernel restriction of $\delta_\lambda$ defines
the \emph{homogeneous algebraic cocycles}.\\

\noindent A morphism $\gamma:\lambda\rightarrow\lambda'$ naturally
associates $\phi_\lambda^{\gamma_0|\gamma_1|...|\gamma_n}$ with $\phi_{\lambda'}^{\gamma'_0|\gamma'_1|...|\gamma'_n}$ through the formula
\begin{equation}\label{algebraicnaturality}
\phi_{\lambda}^{\gamma_0|\gamma_1|...|\gamma_n}(\pi^{\star }T')=\phi_{\lambda'}^{\gamma_\star \gamma_0|\gamma_\star \gamma_1|...|\gamma_\star \gamma_n}(T').
\end{equation}
Then the hypothesis $\pi^{\star }\pi_\star ={\sf Id}_{U',\xi'}$ allows to define a cosheaf $\Phi^{n}_\lambda;\lambda\in \mathcal{D}$ over $\mathcal{D}$, not a sheaf, by
\begin{equation}\label{algebraiccosheaf}
(\Phi_{\star }\phi_{\lambda'})^{\gamma_0|\gamma_1|...|\gamma_n}(T)=\phi_{\lambda'}^{\gamma_\star \gamma_0|\gamma_\star \gamma_1|...|\gamma_\star \gamma_n}(\pi_\star T).
\end{equation}

\noindent However the first equation \eqref{algebraicnaturality} is more precise, and we take it as a definition of \emph{natural algebraic homogeneous cochains}.\\
\begin{rmk*}
	\normalfont We cannot consider it as a sheaf because of a lack of definition of $\gamma^{\star }\gamma'_i$.
\end{rmk*}

The operation of conditioning preserves the naturality, in reason of the following 
identity, involving $\gamma:\lambda\rightarrow\lambda'$, $\gamma_Q\in \mathcal{D}_\lambda$,
$S'\in \Theta^{n}_\lambda$:
\begin{equation}
\pi_\gamma^{\star}[S'|\gamma_\star(\gamma_Q)]=\pi_\gamma^{\star}S'| \gamma_Q.
\end{equation}

\noindent Therefore we can speak of \emph{natural homogeneous algebraic cocycles}.\\

\noindent For $n=0$, the cochains are collections of functions $\psi_\lambda^{\gamma_0}$ of the theories in $\mathcal{A}_\lambda$
such that
\begin{equation}
\psi_\lambda^{\gamma_Q\wedge\gamma_0}(S)=\psi_\lambda^{\gamma_0}(S|Q),
\end{equation}
and such that, for any morphism $\gamma:\lambda\rightarrow\lambda'$,
\begin{equation}
\psi^{\gamma_0}_\lambda(\pi_\gamma^{\star}T')=\psi^{\gamma_\star\gamma_0}_{\lambda'}(T').
\end{equation}
From the first equation, we can eliminate $\gamma_0$. We define $\psi_\lambda=\psi_\lambda^{\top}$, and get
\begin{equation}\label{reduction}
\psi_\lambda^{\gamma_Q}(S)=\psi_\lambda(S|Q).
\end{equation}
The second equation, with the transport of truth, is equivalent to
\begin{equation}
\psi_\lambda(\pi_\gamma^{\star}T')=\psi_{\lambda'}(T').
\end{equation}
A cocycle corresponds to a collection of constant $c_\lambda$, which are natural, then to the functions of the
connected components of the category $\mathcal{D}$.\\
Thus we recover the same notion as in the preceding section.\\

\noindent In degree one, the homogeneous cochain $\phi_\lambda^{\gamma_0;\gamma_1}$
cannot be \emph{a priori} expressed through the collection of functions $\varphi_\lambda^{\gamma_0}=\phi_\lambda^{\gamma_0;\top}$, but,
if it is a cocycle, it can:
\begin{equation}
\phi_\lambda^{\gamma_0;\gamma_1}=\varphi_\lambda^{\gamma_0}-\varphi_\lambda^{\gamma_1};
\end{equation}
as this follows directly from the algebraic cocycle relation applied to $[\gamma_0|\gamma_1|\top_\lambda]$.\\

\noindent But we also have, by homogeneity
\begin{equation}
Q.\varphi^{\gamma_Q}=Q.\phi^{\gamma_Q|\top}=\phi^{\gamma_Q\wedge\gamma_Q|\gamma_Q\wedge\top}=\phi^{\gamma_Q|\gamma_Q}
=\varphi^{\gamma_Q}-\varphi^{\gamma_Q}=0.
\end{equation}
Then, the homogeneity equation gives the particular case
\begin{equation}
\varphi^{Q\wedge Q_O}-\varphi^{Q\wedge Q}=Q.\varphi^{\gamma_{Q_O}}-Q.\varphi^{\gamma_Q}=Q.\varphi^{\gamma_{Q_O}},
\end{equation}
therefore
\begin{equation}
\varphi^{Q\wedge Q_O}=\varphi^{\gamma_ Q}+Q.\varphi^{\gamma_{Q_O}};
\end{equation}
which is the cocycle equation we discussed in the preceding section, under the form of Shannon.\\

\begin{rmk*}
	\normalfont All that generalizes to any degree, in virtue of the comparison theorem between projective resolutions,
	proved in the relative case in MacLane "Homology" \cite{2012homology}, or in SGA 4 \cite{SGA4}, more generally, because the above homogeneous bar complex
	and in-homogeneous bar complex are such resolutions of the constant functor $\mathbb{K}$.
\end{rmk*}

\subsection*{Semantic Kullback-Leibler distance}

In \cite{Baudot-Bennequin}, it was also shown that the Kullback-Leibler distance (or divergence) $D_{KL}(X;\mathbb{P};\mathbb{P}')$
between two probability laws on a random variable $X$ defines a cohomology class in the above sense. The cochains depend on a sequence $\mathbb{P}_0,...,\mathbb{P}_n$
of probabilities and a sequence of variables $X_0,...,X_m$ less fine than a given variable $X$; the conditioning the $n+1$ laws
by the value $y$ of a variable $Y\geq X$ is integrated over $Y_\star \mathbb{P}_0$, for giving an action on the set of measurable functions of the $n+1$ laws,
then the homogeneity is defined as before, and the coboundary is the standard combinatorial
one, as before. For $n=1$, the universal degree one class is shown to be the difference of divergences.\\
Remind that the $K-L$ divergence is given by the formula
\begin{equation}
D_{KL}(X;\mathbb{P};\mathbb{P}')=-\sum_{x_i}p_i\log \frac{p'_i}{p_i}.
\end{equation}
\indent In our present case, we consider functions of $n+1$ theories and $m+1$ propositions, all works as for $n=0$.
In degree zero, the cochains are defined by functions $\psi_\lambda(S_0,S_1)$ satisfying
\begin{equation}
\psi_\lambda(\pi_\gamma^{\star }S'_0;...;\pi_\gamma^{\star }S'_n)=\psi_{\lambda'}(S'_0;...;S'_n),
\end{equation}
for any morphism $\gamma:\lambda\rightarrow\lambda'$.\\
The formula for the homogeneous cochain is
\begin{equation}
\psi_\lambda^{\gamma_Q}(S_0;...;S_n)=\psi_\lambda(S_0|Q;...;S_n|Q).
\end{equation}
The non-homogeneous zero cocycles are the functions of $P$ only, invariant by the transport $\pi_\star $.\\
In degree one, the cocycles are defined by any function $\varphi_\lambda^{Q}(S_0;...;S_n)$ which satisfies
\begin{equation}
\varphi_\lambda^{Q}(\pi_\gamma^{\star }S'_0;...;\pi_\gamma^{\star }S'_n)=\varphi_{\lambda'}^{\pi_\star (Q)}(S'_0;...;S'_n),
\end{equation}
for any morphism $\gamma:\lambda\rightarrow\lambda'$, and verifies the cocycle equation
\begin{equation}
\varphi_\lambda^{Q\wedge R}(S_0;...;S_n)=\varphi_\lambda^{Q}(S_0;...;S_n)+\varphi_\lambda^{R}(S_0|Q;...;S_n|Q).
\end{equation}
The  homogeneous cocycle associated to $\varphi$ is defined by
\begin{equation}
\phi_\lambda^{\gamma_{Q_0};\gamma_{Q_1}}(S_0;...;S_n)=\varphi_\lambda^{Q_0}(S_0;...;S_n)-\varphi_\lambda^{Q_1}(S_0;...;S_n).
\end{equation}
As for $n=0$, there exists a function $\psi_\lambda(S_0;...;S_n)$ such that for any $Q\in \mathcal{A}_\lambda$, i.e. $Q\geq P$,
we have
\begin{equation}
\varphi_\lambda^{Q}(S_0;...;S_n)=\psi_\lambda(S_0|Q;...;S_n|Q)-\psi_\lambda(S_0;...;S_n).
\end{equation}
\indent In the particular case $n=1$, we can consider a basic real function $\psi_\lambda(S)$, seen as a logarithm of theories as before, and define
\begin{equation}
\psi_\lambda(S_0;S_1)=\psi_\lambda(S_0\wedge S_1)-\psi_\lambda(S_0).
\end{equation}
If the function $\psi_\lambda(S)$ is supposed increasing in $S$ (for the relation of weakness $\leq$, as before), this gives a negative function.\\
We obtain
\begin{equation}
\phi_\lambda^{Q}(S_0;S_1)=\psi_\lambda(S_0\wedge S_1|Q)-\psi_\lambda(S_0\wedge S_1)-\psi_\lambda(S_0|Q)+\psi_\lambda(S_0).
\end{equation}
The positivity of this quantity is equivalent to the concavity of $\psi_\lambda(S)$ on the pre-ordered set of theories.\\
Assuming this property we obtain an analog of the Kullback-Leibler divergence.\\
If $\psi_\lambda(S)$ is strictly concave, that is the most convenient hypothesis, this function takes the value zero if and only if $S_0=S_1$.
Therefore it can be taken as a natural \emph{semantic distance}, depending on the data of $Q$, as candidate from a counter-example of $P$.\\
As in the case of $D_{KL}$ this function is not symmetric, then it could be more convenient to take the sum
\begin{equation}
\sigma_\lambda^{Q}(S_0;S_1)=\phi_\lambda^{Q}(S_0;S_1)+\phi_\lambda^{Q}(S_1;S_0)
\end{equation}
to have a good notion of distance between two theories.\\

\subsection*{Simplicial homogeneous space of histories of theories}

\noindent Another argument to justify the consideration of the homogeneity is the interest of taking a \emph{pushout} of the theories.\\

\noindent The sheaf of monoidal categories $\mathcal{D}_\lambda$ over $\mathcal{D}$ acts in two manners on the algebraic space of theories
$\Theta^{\bullet}_\lambda$:
\begin{equation}\label{tensorconditioning}
\gamma_Q.(S\otimes [\gamma_0;...;\gamma_n])=(S|Q)\otimes [\gamma_0;...;\gamma_n],
\end{equation}
\begin{equation}\label{tensorproduct}
\gamma_Q\wedge (S\otimes [\gamma_0;...;\gamma_n])=S\otimes [\gamma_Q\gamma_0;...;\gamma_Q\gamma_n].
\end{equation}
Then we can consider the colimit $\Theta^{\bullet}_\lambda/\mathcal{D}$ of these pairs of maps over all the arrows $\gamma_Q$, i.e.
over $\mathcal{D}_\lambda$: this colimit is the disjoint union of the
coequalizers for
each arrow. This is a quotient simplicial set. The homogeneous cochains are just the (measurable) functions on this simplicial set.\\

This can be realized directly as a pushout, or coequalizer, of a unique pair of maps, by taking the union $Z$
of the products $\Theta_\lambda^{\bullet}\times\mathcal{D}_\lambda$, and the two natural maps $\mu, \nu$ to $T=\Theta_\lambda^{\bullet}$ given by
multiplication and conditioning respectively.\\

Remark that the two operations in \eqref{tensorconditioning} and \eqref{tensorproduct} are adjoint of each other, then we can speak of
\emph{adjoint gluing}.\\

Also interesting is the \emph{homotopy quotient}, taking into account that, geometrically, $Z$ has a higher degree in propositions belonging to $\mathcal{D}_\lambda$,
due to the presence of $\gamma_Q$. This homotopy colimit is the simplicial set $\Sigma^{\bullet}$ obtained from the disjoint union $(Z\times [0,1])\sqcup (T\times\{0\})\sqcup (T\times\{1\})$
by taking the identification
of $(z,0)$ with $\mu(z)$ and of $(z,1)$ with $\nu(z)$. It can be named a \emph{homotopy gluing}, because the arrows are used geometrically as
continuous links between points in
$T\times\{0\}$ and $T\times\{1\}$. The simplicial set $\Sigma^{\bullet}$ is equipped with a natural projection onto the ordinary coequalizer
$\Theta_\lambda^{\bullet}/\mathcal{D}_\lambda$. See for instance Dugger \cite{Dugger2008APO} for a nice exposition of this notion, and its interest for
homotopical stability with respect to the ordinary colimit. Then we propose that a more convenient notion of homogeneous cochains could be the
functions on $\Sigma^{\bullet}$.\\

\indent Similarly, we have two natural actions of the category $\mathcal{D}$ of arrows leading to $\lambda$ and issued from $\lambda'$: the first one
being of the type
\begin{equation}
\Theta_{\lambda'}\otimes \mathcal{D}_\lambda^{\otimes (n+1)}\rightarrow \Theta^{n}_{\lambda};
\end{equation}
the second one of the type
\begin{equation}
\Theta_{\lambda'}\otimes \mathcal{D}_\lambda^{\otimes (n+1)}\rightarrow \Theta^{n}_{\lambda'}.
\end{equation}
They are respectively defined by the following formulas:
\begin{equation}\label{feedforwardtheories}
\gamma^{\star }(S'\otimes [\gamma_0;...;\gamma_n])=(\pi^{\star }_\gamma S')_\lambda \otimes [\gamma_0;...;\gamma_n]
\end{equation}
The second one is
\begin{equation}\label{feedbackpropositions}
\gamma_{\star }(S'\otimes [\gamma_0;...;\gamma_n])=S' \otimes [\pi^{\gamma}_\star \gamma_0;...;\pi^{\gamma}_\star \gamma_n]
\end{equation}

\noindent They are both compatibles with the quotient by the actions of the monoids, then they define
maps at the level of $\Sigma^{\bullet}$.\\

\noindent The \emph{natural} cochains are the functions that satisfy, for each $\gamma:\lambda\rightarrow \lambda'$, the equation
\begin{equation}
\phi_\lambda\circ \gamma^{\star }=\phi_{\lambda'}\circ \gamma_\star .
\end{equation}\\

\noindent Note that no one of the above equations, for homogeneity and naturality, necessitates numerical values, but the second necessitates values in a constant set or a
constant category, at least along the orbits of $\mathcal{D}$.\\
\indent And it is important for us that the cochains can take their values in a category $\mathcal{M}$ admitting limits, like ${\sf Set}$ or ${\sf Top}$,
non necessarily Abelian, because our aim is to obtain a theory of
\emph{information spaces} in the sense searched by Carnap and Bar-Hillel in $1952$ \cite{CBH52}.\\

Define a set $\Theta^{n}_{1}$ ({\em resp.} $\Theta^{n}_{0}$) by the coproduct, or disjoint union, over $\gamma:\lambda\rightarrow\lambda'$ ({\em resp.} $\lambda$)
of the sets $\Theta_{\lambda'}\otimes \mathcal{D}_\lambda^{\otimes (n+1)}$ ({\em resp.} $\Theta^{n}_{\lambda}$). When the integer $n$ varies,
we note the sum by $\Theta^{\bullet}_{1}$ ({\em resp.} $\Theta^{\bullet}_{0}$). They are canonically simplicial sets.\\
\indent The collections of maps $\gamma^{\star }$ and $\gamma_\star $ define two (simplicial) maps from $\Theta^{\bullet}_{1}$ to $\Theta^{\bullet}_{0}$,
that we will denote respectively $\varpi$ and $\vartheta$, for {\em past} and {\em future}. The colimit or \emph{coequalizer} of these two maps, is the
quotient $\mathrm{H}^{\bullet}_{0}$ of $\Theta^{\bullet}_{0}$ by the equivalence relation
\begin{equation}
(\pi^{\star }_\gamma S')_\lambda \otimes [\gamma_0;...;\gamma_n]_\lambda\quad \sim \quad S'_{\lambda'} \otimes [\pi^{\gamma}_\star \gamma_0;...;\pi^{\gamma}_\star \gamma_n]_{\lambda'}.
\end{equation}

\noindent Once iterated over the arrows, this relation represents the complete story of a theory, from the source of its formulation in the network to the final layer.\\
It is remarkably conform to the notion of cat's manifold, and compatible with the possible presence of inner sources in the network.\\

Remark that the two operations in \eqref{feedforwardtheories} and \eqref{feedbackpropositions} are also adjoint relative to each other, then again
the corresponding colimit can be named an adjoint gluing.\\
\begin{rmk*}
	\normalfont The above equivalence relation is more fine than the relation we would have found with the covariant functor, i.e.
	\begin{equation}
		(\pi_{\star }^{\gamma} S)_{\lambda'} \otimes [\pi^{\gamma}_{\star }\gamma_0;...;\pi^{\gamma}_{\star }\gamma_n]_{\lambda'}\quad
		\sim \quad S_{\lambda} \otimes [\gamma_0;...;\gamma_n]_{\lambda};
	\end{equation}
	because this relation is implied by the former, when we applied it to $S'=\pi_\star S$, in virtue of our hypothesis $\pi^{\star }\pi_\star ={\sf Id}$.\\
	\indent The two relations ar equivalent if and only if $\pi_\star \pi^{\star }={\sf Id}$, that is the case of isomorphic logics among the network.
\end{rmk*}
\noindent 
We define the \emph{natural cochains} as the (measurable) functions on $\mathrm{H}^{\bullet}_{0}$, and the \emph{natural homogeneous cochains} as the
functions on the quotient $\mathrm{H}^{\bullet}_{0}/\mathcal{D}$ by the identification of junction with conditioning.
And we are more interested in the homogeneous case.\\

However, in a non-Abelian context, the stability under homotopy will be an advantage, therefore we also consider the homotopy colimit
of the maps $\varpi$ and $\vartheta$, or homotopy gluing between  past and future, and propose that this colimit $\mathrm{I}^{\bullet}_{0}$
(or $ho\mathrm{I}$ if we reserve $\mathrm{I}$ for the usual
colimit) is a better notion of the histories of theories in the network.
It is also naturally a simplicial set. Then the \emph{natural homotopy homogeneous cochains} will be functions on the homotopy gluing
$ho\mathrm{I}$.\\

\noindent The homotopy type of the theories histories $\mathrm{I}^{\bullet}_{0}$
itself is an interesting candidate for representing the information, and information flow in the network.\\
\indent For instance, its connected components gives the correct notion of zero-cycles, and the functions on
them are zero-cocycles. The Abelian construction is sufficient to realize these cocycles.\\
\indent We will later consider functions from the space $\mathrm{I}^{\bullet}_{0}$ to a closed model category $\mathcal{M}$, their homotopy
type in the sense of Quillen can be seen as a non-Abelian set of cococycles.\\

What we just have made above for the cochains (homogeneous and/or natural) is a particular case of a \emph{homotopy limit}.\\

\noindent The notion of homotopy limit was introduced in Bousfield-Kan $1972$, \cite[chapter XI]{BK} where
it generalized the classical bar resolution in a non-linear context,
see MacLane's book "Homology" \cite{2012homology}. The authors attributed its origin to Milnor, in the article "On axiomatic homology theory" \cite{pjm/1103036730}.
For this notion and more recent developments (see \cite{Hirschhorn2003}, \cite{Hirschhorn},
or \cite{Dugger2008APO}).\\

In this spirit, we extend now the two maps $\varpi,\vartheta$ from $\Theta^{\bullet}_{1}$  to $\Theta^{\bullet}_{0}$, in higher degrees,
by using the nerve of the category $\mathcal{D}$.\\

\indent The nerve $\mathcal{N}=\mathcal{N}(\mathcal{D})$ of the category $\mathcal{D}$ is the simplicial set made by the sequences $A$
of successive arrows in $\mathcal{D}$. For $k\in \mathbb{N}$, $\mathcal{N}_k$ is the set of sequences of length $k$. A sequence is written $(\delta_1,...,\delta_k)$,
where $\delta_i;i=1,...,k$ goes from $\lambda_{i-1}$ to
$\lambda_{i}$ in $\mathcal{D}$. We use the symbols $\delta_i^{\star }$, or the letters $\gamma_i$ when there is no ambiguity,
for the arrow $\delta_i$ considered in the opposite category $\mathcal{D}^{op}=\widetilde{\mathcal{A}}'_{\rm strict}$;
this reverse the direction of the sequence, going now upstream. When necessary, we write $\delta_i(A), \lambda_{i-1}(A), ...$,
for the arrows and vertices of a chain $A$. \\
\indent For $k\in \mathbb{N}$, we define $\Theta^{n}_{k}$ as the disjoint union over $A=(\delta_1,...,\delta_k)$
of the sets $\Theta_{\lambda_0}\otimes \mathcal{D}_{\lambda_k}^{\otimes (n+1)}$. Thus the theory is attached to the beginning in the sense
of $\mathcal{D}$, and the involved propositions are at the end. The chain in $\mathcal{D}$ goes in the dynamical direction, downstream.
When the integers $n$ and $k$ vary,
we note $\Theta^{\bullet}_{\star }$ the sum (disjoint union). This is a bi-simplicial set.\\

We have $k+1$ canonical maps $\vartheta_i;i=1,...,k+1$ from $\Theta^{n}_{k+1}$ to $\Theta^{n}_{k}$. Each map
deletes a vertex, moreover at the extremities it also deletes the arrow, and inside the chain, it composes the arrows
at $i-1$ and $i$. In $\lambda_0$,
the map $\pi^{\star }_{\gamma_1}$ is applied to the theory, to be transmitted downstream, and in $\lambda_{k+1}$, the map
$\pi_{\star }^{\gamma_{k+1}}$ is applied to the $n+1$ elements $\gamma_{Q_j}$ in $\mathcal{D}_{\lambda_{k+1}}$,
to be transmitted upstream.\\

\noindent By analogy with the definition of the homotopy colimit of a diagram in a model category cf. references upcit, we take
for a more complete space of histories, the whole geometric realization of the simplicial functor $\Theta^{\bullet}_{\star }$, seen now as
a \emph{simplicial space} with the above skeleton in degree $k$, and the above gluing maps $\vartheta_i$.
The expression $g\mathrm{I}$ denotes this space, that we understand as the geometrical space of complete
histories of theories.\\
\indent The extension of information over the nerve incorporates the topology of the categories $\mathcal{C},\mathcal{F},\mathcal{D}$.
The degree $n$  was for the logic, the degree $k$ is for its transfer through the layers.\\
\indent $g\mathrm{I}$, or its homotopy type, represents for us the logical part of the available information;
it takes into account
\begin{enumerate}[label=\arabic*)]
	\item the architecture $\mathcal{C}$,
	\item the pre-semantic structure, through the fibration $\mathcal{F}$
	over $\mathcal{C}$, which constrains the possible weights, and also generates the logical transfers $\pi^{\star }$, $\pi_\star $,
	\item the terms of a language through $\widetilde{\mathcal{A}}$, and the propositional judgements through $\mathcal{D}$ and $\Theta$.
\end{enumerate}
The dynamic is given by the semantic functioning $S^{w}:X^{w}\rightarrow\Theta$, depending on the data and the learning. Its analysis needs
an intermediary, a notion of cocycles of information, that we describe now. \\

\noindent The information appears as a tensor $F^{\gamma_0,...,\gamma_n}_{\delta_1,...,\delta_k}(S)$.
\emph{A priori} its components take their values in the category $\mathcal{M}$, that  can be $\sf Set$ or $\sf Top$.\\

\noindent The points in $g\mathrm{I}$ are classes of elements
\begin{equation}
u=S\otimes[\gamma_0,...,\gamma_n]\otimes[\delta_1,...,\delta_k](t_0,...,t_n;s_1,...,s_k)
\end{equation}
where the $t_i;i=0,...,n$ and $s_j;j=1,...,k$ are respectively barycentric coordinates in $\Delta(n)$ and $\Delta(k-1)$.\\

\noindent It is tempting to interpret the coordinates $t_i$ as weights, or values, attributed to the propositions $Q_i$,
and the numbers $s_j$ as times, conduction times perhaps, along the chain of mappings.\\

\noindent Therefore we see the tensor $F$ as a local system $F_u;u\in g\mathrm{I}$ over $g\mathrm{I}$.

\subsection*{Simplicial dynamical space of a DNN, information content}

\noindent Considering a semantic functioning $S:X\rightarrow\Theta$, we can enrich it by the choice of propositions
in each layer $U$ and context $\xi_U$ (or better collections of elements of $\mathcal{D}_\lambda$), and consider sequences
over the networks, relating activities and enriched theories. Then, for each local activity, and each chain of arrows in
the network, equipped with propositions at one end (downstream), the function $F$ gives a space of information.\\

\noindent More precisely, we form the topological space of activities $g\mathbb{X}$, by taking the homotopy colimit
of the object $\mathbb{X}$, fibred over the object $\mathbb{W}$, in the classifying topos of $\mathcal{F}$, lifted to $\mathcal{D}$, and seen as
a diagram over $\mathcal{D}$. This space is defined in the same manner $gI_\star $ was defined from $\Theta_\star $ over $\mathcal{D}$;
it is the geometric realization of the simplicial set $g\mathrm{X}_\star $, whose $k$-skeleton is the sum of the pairs $(A_k,x_\lambda)$
where $A$ is an element of length $k$ in $\mathcal{N}(\mathcal{D})$ and $x_\lambda$ an element in $\mathbb{X}_\lambda$, at
the origin of $A$ in $\mathcal{D}$. The degeneracies $d_i;i=1,...,k+1$ from $\mathrm{X}_{k+1}$ to $\mathrm{X}_{k}$ are given
for $1< i< k+1$, by composition of the morphisms at $i$, by forgetting $\delta_{k+1}(A)$ for $i=k+1$, and by forgetting
$\delta_1$ and transporting $x_\lambda$ by $X_w^{\star }$ for $i=1$.\\

\noindent Then we can ask for an extension of the semantic functioning to a continuous or simplicial map
\begin{equation}
gS:g\mathrm{X}\rightarrow g\mathrm{I}.
\end{equation}

\noindent This implies a compatibility between dynamical functioning in $\mathbb{X}$ and logical functioning in $\Theta$. However,
this map factorizes by a quotient, that can be small, when the semantic functioning is poor. It is only for some regions
in the weight object $\mathbb{W}$, giving itself a geometrical space $g\mathbb{W}$, that the semantic functioning is interesting.\\

\noindent Given $F:g\mathrm{I}\rightarrow \mathcal{M}$, this gives a map $F\circ gS$ from $g\mathrm{X}$ to $\mathcal{M}$,
that can be seen as the information content of the network.\\
\indent To have a better analog on the Abelian quantities, we suppose that $\mathcal{M}$ is a closed model category, and we pass to the homotopy type
\begin{equation}
ho.F\circ gS:g\mathrm{X}\rightarrow ho\mathcal{M}.
\end{equation}

\noindent For real data inputs and spontaneous internal activities, this gives a homotopy type
for each image.\\

\noindent For instance, the degree one homogeneous cocycle $\phi_\lambda^{Q}(S)$ deduced from a precision function $\psi_\lambda(S)$
with real values, is replaced by a map to topological spaces, associated to some "propositional" paths between two points of $g\mathrm{I}$; a
degree two combinatorial cocycles, as the mutual information, is replaced by a varying space associated to a "propositional" triangle, up to homotopy.\\

\subsection*{Non-Abelian inhomogeneous fundamental cochains and cocycles. A tentative}\label{subsec:tentativeinfospaces}

\noindent Remember that the fundamental zero cochain $\psi^{Q_0}_\lambda$ with real coefficients, satisfied $\psi_\lambda^{Q}(S)=\psi_\lambda(S| Q)\geq \psi_\lambda(S)$.
Then, in the nonlinear framework, it is tempting to assume the existence in $\mathcal{M}$ of a class of morphisms replacing the inclusions of the sets,
namely cofibrations, and to generalize the increasing of the function $\psi_\lambda$ of $S$, by the existence of a cofibration,
$F(S)\rightarrowtail F(S|Q)$, or more generally a cofibration $F(S)\rightarrowtail F(S')$ each time $S\leq S'$.\\
\indent This is sufficient for defining an object of ambiguity, then an information object (non-homogeneous), by generalizing the relation between
precision and ambiguity of the Abelian case:
\begin{equation}
H^{Q}(S)=F(S|Q)\backslash F(S);
\end{equation}
where the subtraction is taken in a geometrical or homotopical sense.\\
All that supposes that $\mathcal{M}$ is a closed model category of Quillen.\\

\noindent This invites us to assume that $F$ is covariant under the action of the monoidal categories $\mathcal{D}_\lambda$,
i.e. for every arrow $\gamma_Q$ in $\mathcal{D}_\lambda$, and every theory $S$ in $\Theta_\lambda$, there exists a morphism $F(\gamma_Q;S):F(S)\rightarrow F(S|Q)$
in $\mathcal{M}$, and for two arrows $\gamma_Q$, $\gamma_{Q'}$,
\begin{equation}
F(\gamma_{Q'}\gamma_Q;S)=F(\gamma_{Q'};S|Q)\circ F(\gamma_Q;S)
\end{equation}
and we assume that every $F(\gamma_Q;S)$ is a cofibration.\\

\noindent In the same manner, the generalization of the concavity of the real function $\psi_\lambda^{Q}$ is the hypothesis
that, for two arrows $\gamma_Q$, $\gamma_{Q'}$, there exists a cofibration of the
quotient objects $H$:
\begin{equation}
H(Q,Q';S): H^{Q}(S|Q')\rightarrowtail H^{Q}(S).
\end{equation}
The same thing happening for $H^{Q'}(S|Q)\rightarrowtail H^{Q'}(S)$.\\
The difference space is the model category version of the \emph{mutual information} between $Q$ and $Q'$:\\
\noindent by definition
\begin{equation}
I_2(Q;Q')=H^{Q}\backslash[H^{Q\otimes Q'}\backslash H^{Q'}],
\end{equation}
or in other terms,
\begin{equation}
I_2(Q;Q')=(Q.F\backslash F)\backslash[(Q\otimes Q')F\backslash Q'.F],
\end{equation}

Reasoning on subsets of $H^{Q\otimes Q'}$, this gives the symmetric relation
\begin{equation}
I_2(Q;Q')\sim H^{Q}\cap H^{Q'}.
\end{equation}

\noindent The general concavity condition is the existence of a natural cofibration $H^{Q}(S')\rightarrowtail H^{Q}(S)$ as soon as
there is an inclusion $S\leq S'$.\\

\noindent This stronger property of concavity for the functor  $F$ implies in particular, for any pair of theories $S_0,S_1$,
the existence of a cofibration
\begin{equation}
J_{Q}(S_0;S_1):H^{Q}(S_0)\rightarrow H^{Q}(S_0\wedge S_1).
\end{equation}
This allows to define a homotopical notion of Kullback-Leibler
divergence space in $\mathcal{M}$, between two theories falsifying $P$, at a proposition $Q\geq P$:
\begin{equation}
D^{Q}(S_0;S_1)=H^{Q}(S_0\wedge S_1)\backslash F_\star H^{Q}(S_0).
\end{equation}

\subsection*{Comparison between homogeneous and inhomogeneous non-Abelian cochains and cocycles}

To be complete, we have to relate these maps $F,H,I,D, ...$ from theories and constellations of propositions
to $\mathcal{M}$ with the homogeneous
tensors $F^{\gamma_0,...,\gamma_n}_{\delta_1,...,\delta_k}(S)$. For that, the natural idea is to follow the path
we had described from the homogeneous Abelian bar-complex to the non-homogeneous one, at the beginning of this section.
This will give a homotopical/geometrical version of the MacLane comparison in homological algebra.\\

We consider the bi-simplicial set $\mathbf{I}^{\bullet}_\star $ as a simplicial set $\mathbf{I}_\star $
in the algebraic exponent $n$ for $\bullet$, then
it is a contravariant functor from the category $\Delta$ to the category of simplicial sets $\Delta_{Set}$. The morphisms
of $\Delta$ from $[m]$ to $[n]$ are the non-decreasing maps, their set is noted $\Delta(m,n)$.\\

Our hypothesis is that the above tensors form a \emph{cosimplicial local system} $\Phi$ with values
in the category $\mathcal{M}$ over the simplicial presheaf $\mathbf{I}_\star $,
in the sense of the preprint \emph{Extra-fine sheaves and interaction decompositions} \cite{bennequin2020extrafine}. In an equivalent manner, we consider the category $\mathcal{T}={\sf Set}(\mathbf{I}_\star )$ which
objects are the simplicial cells $u$ of $\mathbf{I}_\star $ and arrows from $v$ of dimension $n$ to $u$ of dimension $m$ are the non-decreasing
maps $\varphi\in \Delta(m,n)$ (morphisms in the category $\Delta$) such that $\varphi^{\star }(v)=u$. Here the map $\varphi^{\star }$
is simplicial in the index $k$ for $\star $, concerning the nerve complex of $\mathcal{D}$; then the cosimplicial local system is a
contravariant functor from $\mathcal{T}$ to $\mathcal{M}$.\\
All that is made to obtain a non-Abelian version of the propositional (semantic) bar-complex. Following a recent trend, we name \emph{spaces} the elements of $\mathcal{M}$.\\

\noindent We add that an inclusion of theories $S\leq S'$ gives a cofibration $\Phi(S')\rightarrowtail \Phi(S)$, in a functorial
manner over the poset of theories.\\

\noindent Let us repeat the arguments to go from homogeneous cochains or cocycles to non-homogeneous ones.\\

\noindent First, a zero-cochain is defined over the cells $S_\lambda\otimes [\gamma_0]$, where the arrow $\gamma_0$
abuts in a propositions $Q_0\geq P$. The associated non-homogeneous space $F(S)$ corresponds to $Q_0=\top$. The relation
between conditioning and multiplication gives the way to recover $\Phi^{Q_0}(S)$.\\

\noindent Second, we name degree one homogeneous cocycle a sheaf of spaces $\Phi^{[\gamma_0,\gamma_1]}(S)$, over the one skeleton
of $\varphi^{\star }$, which satisfies that for the triangle $[\gamma_0, \top,\gamma_1]$, the space $\Phi^{[\gamma_0,\gamma_1]}$
is homotopy equivalent to the \emph{difference} of the spaces $\Phi^{[\gamma_0,\top]}$ and $\Phi^{[\gamma_1,\top]}$.\\
Remark: more generally a degree one cocycle should satisfies this axiom for every zigzag $\gamma_0\leq \gamma_{\frac{1}{2}}\geq \gamma_1$.\\
This definition supposes that we have a notion of difference in $\mathcal{M}$, satisfying the same properties that the difference $A\backslash (A\cap B)$
satisfies in subsets of set. If all the theories considered contain a minimal one, then spaces are subspaces of a given space, and this hypothesis
has a meaning. However, this is the case in our situation, considering the sets $\Theta_P$, because we consider only propositions $Q,Q_0,Q_1,...$
that are implied by $P$.\\

\noindent To the degree one cocycle $\Phi^{[\gamma_0,\gamma_1]}(S)$ we associate the space $H^{\gamma_0}(S)=\Phi^{[\gamma_0,\top]}(S)$,
obtained by
replacing $\gamma_1$ by $\top$. The space $G^{\gamma_1}(S)$ is obtained by replacing $\gamma_0$ by $\top$ in $\Phi$.\\

\noindent Note the important point that $H$ and $G$ are in general non-homogeneous cocycles.\\

\noindent Applying the definition of $1$-cocycle to the triangle $[\gamma_0,\top,\gamma_1]$, we obtain that
\begin{equation}
\Phi^{[\gamma_0,\gamma_1]}(S)\sim  H^{\gamma_0}(S)\setminus H^{\gamma_1}(S).
\end{equation}
\begin{lem}
	The cocyclicity of $\Phi$ implies
	\begin{equation}
		Q.H^{Q}\sim H^{Q\otimes Q}\backslash H^{Q}.
	\end{equation}
\end{lem}
\begin{proof}
	\begin{equation}
		Q.H^{Q}=Q.\Phi^{Q|\top}=\Phi^{Q\otimes Q|Q\otimes\top}=\Phi^{Q\otimes Q|Q}=H^{Q\otimes Q}\backslash H^{Q}.
	\end{equation}
\end{proof}

\noindent From that we deduce,
\begin{prop}
	The homogeneity of $\Phi$ implies
	\begin{equation}\label{nonAbelianshannon}
		H^{Q\otimes Q'}\backslash H^{Q\otimes Q}\sim Q.H^{Q'}\backslash[H^{Q\otimes Q}\backslash H^{Q}].
	\end{equation}
\end{prop}
\begin{proof}
	\begin{equation}
		H^{Q\otimes Q'}\backslash H^{Q\otimes Q}=Q.H^{Q'}\backslash Q.H^{Q}\sim Q.H^{Q'}\backslash[H^{Q\otimes Q}\backslash H^{Q}].
	\end{equation}
\end{proof}

\noindent In the Abelian case of ordinary difference this is equivalent to
\begin{equation}
H^{Q\otimes Q'}\sim Q.H^{Q'}\cup H^{Q}.
\end{equation}

\noindent This is the usual Shannon equation; then \eqref{nonAbelianshannon} can be seen as a non-Abelian Shannon
equation. Taking homotopy in $Ho(\mathcal{M})$ probably gives a  more intrinsic meaning of semantic information.\\

\noindent It is natural to admit that, at the level of information spaces, $H^{Q\otimes Q}\sim H^{Q}$. Under this hypothesis,
we get the usual Shannon's formula under
\begin{equation}\label{naturalnonAbelianshannon1}
H^{Q\otimes Q'}\backslash H^{Q}\sim Q.H^{Q'}.
\end{equation}
That is, for every theory $S$ falsifying $P$:
\begin{equation}\label{naturalnonAbelianshannon2}
H^{Q\otimes Q'}(S)\backslash H^{Q}(S)\sim H^{Q'}(S|Q).
\end{equation}

\noindent Remind there is no reason \emph{a priori} that $H^{Q}\rightarrowtail H^{Q\otimes Q'}$. Then the above difference is
after intersection.\\

\noindent If $F$  is any non-homogeneous zero-cochain, we have a cofibration $F\rightarrowtail Q.F$, where $Q.F(S)=F(S|Q)$.
In this case we already defined a space $H^{Q}$ by
\begin{equation}
H^{Q}(S)=F(S|Q)\backslash F(S).
\end{equation}

\begin{prop}
	$H^{Q}$ automatically satisfies equation \eqref{nonAbelianshannon}.
\end{prop}
\begin{proof}
	we have $F\rightarrowtail (Q\otimes Q')F$ and $F\rightarrowtail (Q\otimes Q)F$, then
	\begin{align*}
		H^{Q\otimes Q'}\backslash H^{Q\otimes Q}&=((Q\otimes Q')F\backslash F)\backslash((Q\otimes Q)F\backslash F)\\
		&\sim(Q\otimes Q')F\backslash (Q\otimes Q)F.
	\end{align*}
	Using $F\rightarrowtail Q.F\rightarrowtail (Q\otimes Q)F$, and assuming $Q.F\rightarrowtail (Q\otimes Q')F$, we get
	\begin{align*}
		Q.H^{Q'} \backslash[H^{Q\otimes Q}\backslash H^{Q}]&=Q.(Q'F\backslash F)\backslash[((Q\otimes Q)F\backslash F)\backslash (Q.F\backslash F)]\\
		&=(Q\otimes Q')F\backslash Q.F)\backslash[(Q\otimes Q)F\backslash Q.F]\\
		&\sim (Q\otimes Q')F\backslash (Q\otimes Q)F.
	\end{align*}
	Therefore, as wanted,
	\begin{equation}
		H^{Q\otimes Q'}\backslash H^{Q\otimes Q}\sim Q.H^{Q'} \backslash[H^{Q\otimes Q}\backslash H^{Q}].
	\end{equation}
\end{proof}

\noindent We also had suggested above to define the mutual information $I_2(Q;Q')$ associated to a cocycle $H$
by the formula $I_2(Q:Q')=H^{Q}\backslash Q'.H^{Q}$.\\
The restricted concavity condition on $H$ is the existence of a natural cofibration $Q'.H^{Q}\rightarrowtail H^{Q}$.\\
\begin{rmk*}
	\normalfont This goes in the opposite direction to $F$: the more precise the theory $S$ is, the bigger $H^{Q}(S)$ is,
	i.e. $S\leq S'$ implies $H^{Q}(S')\rightarrowtail H^{Q}(S)$.
\end{rmk*}

\noindent We assume also that for all pair $Q,Q'$ we have $H^{Q\otimes Q'}\sim H^{Q'\otimes Q}$.\\
\begin{prop*}
	under the above hypothesis and the assumption that $H^{Q\otimes Q}\sim H^{Q}$ and $H^{Q'\otimes Q'}\sim H^{Q'}$, we can consider $H^{Q}$
	and $H^{Q'}$ as subsets of $H^{Q\otimes Q'}$, and we have
	\begin{equation}
		I_2(Q;Q')=I_2(Q';Q)=H^{Q}\cap H^{Q'}.
	\end{equation}
\end{prop*}
\begin{proof}
	The Shannon formula \eqref{naturalnonAbelianshannon1} tells that $Q.H^{Q'}$ is $H^{Q\otimes Q'}\backslash H^{Q}$
	and $Q'.H^{Q}$ is $H^{Q'\otimes Q}\backslash H^{Q'}$, then
	\begin{equation}
		I_2(Q;Q')=H^{Q}\backslash [H^{Q\otimes Q'}\backslash H^{Q'}]\sim H^{Q}\cap H^{Q'}.
	\end{equation}
\end{proof}
\begin{rmk*}
	\normalfont We cannot write the relation with the usual union, but, under the above hypotheses, there is a cofibration
	\begin{equation}
		j\vee j':H^{Q}\vee H^{Q'}\rightarrowtail H^{Q\otimes Q'},
	\end{equation}
	giving rise to a quotient
	\begin{equation}
		I_2(Q;Q')\cong  H^{Q}\times_{H^{Q\otimes Q'}}H^{Q'}.
	\end{equation}
\noindent Generalizing the suggestion of Carnap and Bar-Hillel, and a Shannon theorem in the case of probabilities, we propose, to tell
that $Q,Q'$ are \emph{independent} (with respect to $P$) at the theory $S$, when $H^{Q}\cap H^{Q'}$ is empty (initial element of $\mathcal{M}$).
\end{rmk*}

\noindent With $I_2$, we can continue and get a semantic version of the \emph{synergy} quantity of three variables:
\begin{equation}
I_3(Q_1;Q_2;Q_3)(S)=I_2(Q_1;Q_2)(S)\backslash I_2(Q_1;Q_2)(S|Q_3).
\end{equation}
However, there is no reason why it must be a true space, because in the Abelian case it can be a negative number; (see \cite{Baudot_2019}
for the relation with the Borromean rings).\\

\begin{rmk*}
	\normalfont This invites us to go to $Ho(\mathcal{M})$, where there exists a notion of relative objects:
	for a zigzag $A\twoheadleftarrow C\rightarrowtail B$, with a trivial fibration to the left, and a cofibration to the right, the deduced arrow $A\rightarrow B$
	in $Ho(\mathcal{M})$, can be considered as a kind of difference of spaces as in  Jardine, Cocyle categories \cite{Jardine2009}, and
	Zhen Lin Low, Cocycles in categories of fibrant objects \cite{low2015cocycles}. Before Quillen and Jardine this kind of homotopy construction was
	introduced by Gabriel and Zisman \cite{gabriel1967calculus}, as a calculus of fraction, in the framework of simplicial objects, their book being the
	first systematic exposition of the simplicial theory.
\end{rmk*}

With respect to the Shannon information, what is missing is an analog of the expectation of functions over the states
of the random variables. In some sense, this is replaced by the properties of growing and concavity of the function $\psi$, or spaces $F$
and $H$, which give a manner  to compare the theories. The true semantic information is not the value attributed to each individual theory,
it is the set of relations between these values, either numerical, either geometric, as expressed by functors over the simplicial
space $gI^{\bullet}_\star $, or better, more practical, over the part of ot that is accessible to a functioning network $g\mathbb{X}$. \\

\subsection*{The example of the theory $\mathcal{L}^{2}_3$ of Carnap and Bar-Hillel}

Let us try to describe the structure of Information, as we propose it, in the simple (static) example that was chosen for development by Carnap and Bar-Hillel in their report in 1952,
\cite{CBH52}.\\

The authors considered a language $\mathcal{L}^{\pi}_n$ with $n$ subjects $a,b,c,...$ and $\pi$ attributes of them $A,B,...$, taking some possible values,
respectively $\pi_A,\pi_B,...$. In their developed example $n=3$, $\pi=2$ and every $\pi_i$ equals $2$. The subjects are human persons, the two attributes are the gender $G$, male $M$
or female $F$, and the age $A$, old $O$ or young $Y$.\\
\indent The elementary, or ultimate, states, $e\in E$ of the associated Boolean algebra $\Omega=\Omega^{E}$ are given by choosing values of all the attributes for all
the subjects. For instance, in the language $\mathcal{L}^{2}_3$, we have $4^{3}=64$ elementary states.\\
\indent The propositions $P,Q,R,...$ are the subsets of $\Omega$, their number is $2^{64}$. The theories $S,T,...$, in this case, are also described by their
initial assertion, that is the truth of a given proposition, obtained by conjunction, and also named $S,T,...$.\\
\indent With our conventions, for conditioning and information spaces or quantities,
it appears practical to define the propositions by the disjunction of their elements $e_I=e_{i_1}\vee ...\vee e_{i_k}$ and the theories by the conjonction
of the complementary sets $\neg e_i=S_i$, that is $S_I=(\neg e_{i_1})\wedge ...\wedge (\neg e_{i_k})$. Experimentally \cite{logic-DNN} the theories exclude something,
like $P$, i.e. contain $\neg P$, then with $S_I$ we see that $P=e_I$ is excluded, as are all the $e_{i_j}$ for $1\leq j\leq k$. A proposition $Q$ which is implied by $P$,
corresponds to a subset which contains all the elementary propositions $e_{i_j}$ for $1\leq j\leq k$.\\

In what follows,  the models of "spaces of information" that are envisaged are mainly groupoids, or sets, or topological spaces.\\
\indent A zero cochain $F_P(S)$ gives a space for any theory excluding $P$, in a growing manner, in the sense that $S\leq S'$ (inclusion of sets) implies
$F(S)\leq F(S')$. The coboundary $\delta F=H$, gives a space $H^{Q}_P(S)$ for any proposition $Q$ such that $P\leq Q$, whose formula is
\begin{equation}
H^{Q}_P(S)=F_P(S\vee\neg Q)\backslash F_P(S).
\end{equation}
By concavity, this function (space) is assumed to be decreasing with $S$, i.e. if $S\leq S'$,
\begin{equation}
H^{Q}_P(S)\leftarrowtail H^{Q}_P(S').
\end{equation}
And by monotonicity of $F$, it is also decreasing in $Q$, i.e. if $Q\leq Q'$,
\begin{equation}
H^{Q}_P(S)\leftarrowtail H^{Q'}_P(S').
\end{equation}

\noindent In particular,  we can consider the smaller $F_P(S)$ that is $F_P(\bot)$, as it is contained in all the spaces $F_P(S)$,
we choose to take it as the empty space (or initial object in $\mathcal{M}$), then
\begin{equation}
H^{Q}_P(\bot)=F_P(\neg Q).
\end{equation}

\noindent As we saw in general for every one-cocycle, not necessarily a coboundary, we have for any pair $Q,Q'$ larger than $P$,
\begin{equation}
H_P^{Q\wedge Q'}(S)\backslash H_P^{Q'}(S)\approx H_P^{Q}(S|Q')=H_P^{Q}(S\vee\neg Q').
\end{equation}
Therefore, in the boolean case, every value of $H$ can be deduced from its value on the empty theory:
\begin{equation}
H_P^{Q}(\neg Q')\approx H_P^{Q\wedge Q'}(\bot)\backslash H_P^{Q'}(\bot).
\end{equation}

\noindent We note simply $H_P^{Q}(\bot)=H_P^{Q}=F_P(\neg Q)$.\\
And they are the spaces to determine.\\

\noindent The localization at $P$ (i.e. the fact to exclude $P$) consists in discarding the elements $e_i$ belonging
to $P$ from the analysis. Therefore we begin by considering the complete situation, which corresponds to $P=\bot$.\\
\indent In this case we note simply $H^{Q}=H_\bot^{Q}=F(\neg Q)$.\\

The concavity of $F$ is expressed by the existence of embeddings (or more generally cofibrations) associated to each
set of propositions $R_0, R_1, R_2, R_3$ such that $R_0\leq R_1\leq R_3$ and $R_0\leq R_2\leq R_3$:
\begin{equation}
F(R_3)\setminus F(R_1)\rightarrowtail F(R_2)\setminus F(R_0).
\end{equation}
In particular, for any pair of proposition $Q,Q'$, we have $\bot\leq \neg Q \leq \neg (Q\wedge Q')$ and $\bot\leq \neg Q' \leq \neg (Q\wedge Q')$,
and $F(\bot)=H^{\top}=\emptyset$, then
\begin{equation}\label{mutualinforight}
j:H^{Q\wedge Q'}\setminus H^{Q'}\rightarrowtail H^{Q},
\end{equation}
and
\begin{equation}\label{mutualinfoleft}
j':H^{Q\wedge Q'}\setminus H^{Q}\rightarrowtail H^{Q'}.
\end{equation}
Then we introduced the hypothesis that the  subtracted spaces of both situations give equivalent
results, and defined the mutual information $I_2(Q;Q')$:
\begin{equation}\label{mutualinfo}
H^{Q}\setminus(j(H^{Q\wedge Q'}\setminus H^{Q'}))\approx I_2(Q;Q')\approx H^{Q'}\setminus(j'(H^{Q\wedge Q'}\setminus H^{Q})).
\end{equation}

\noindent Importantly, to get a cofibration, the subtraction cannot be replaced by a collapse with marked point,
but it can in general be a collapse without marked point.\\

Consequently, the main axioms for the brut semantic spaces $H^{Q}$ are: $(i)$ the existence of natural embeddings (or cofibrations) when $Q\leq Q'$:
\begin{equation}\label{embedding}
H^{Q'}\rightarrowtail H^{Q},
\end{equation}
and
$(ii)$ the above formulas \eqref{mutualinforight} and \eqref{mutualinfoleft} defining the same space $I_2(Q;Q')$, as in \eqref{mutualinfo},
which can perhaps all be interpreted after intersection.\\

We left open the relation between $I_2(Q;Q')$ and $H^{Q\vee Q'}$, however the axioms $(ii)$ imply that there exist natural embeddings
\begin{equation}\label{unioninequality}
H^{Q\vee Q'}\rightarrowtail I_2(Q;Q').
\end{equation}
\\
The idea, to obtain a coherent set of non-trivial information spaces, is to exploit the symmetries of the language, or other elements
of structure, which give
an action of a category on the language, and generate constraints of naturalness for the spaces.\\

There exists a Galois group $G$ of the language, generated by the permutation of the $n$ subjects, the permutations
of the values of each attribute and the permutations of the attributes that have the same number of possible values.\\
\indent To be more precise, we order and label the subjects, the attribute and the values, with triples $xY_i$.
In our example, $x=a,b,c$, $Y=A,G$, $i=1,2$,
the group of subjects permutation is $\mathfrak{S}_3$, the transposition of values are $\sigma_A=(A_1 A_2)$ and $\sigma_G=(G_1 G_2)$, and
the four exchanges of attributes are $\sigma=(A_1 G_1)(A_2 G_2)$,  $\kappa=(A_1 G_1 A_2 G_2)$,
$\kappa^{3}=\kappa^{-1}=(A_1 G_2 A_2 G_1)$, and  $\tau=(A_1G_2)(A_2G_1)$.\\
We have
\begin{equation}
\sigma_A\circ\sigma_G=\sigma_G\circ\sigma_A=(A_1 A_2) (G_1 G_2)=\kappa^{2};
\end{equation}
\begin{equation}
\sigma\circ\sigma_A=\sigma_G\circ\sigma=\kappa;\quad \sigma_A\circ\sigma=\sigma\circ\sigma_G=\kappa^{-1};
\end{equation}
\begin{equation}
\sigma_A\circ \sigma\circ \sigma_G=\tau;\quad \sigma_A\circ \tau\circ \sigma_G=\sigma
\end{equation}
The group generated by $\sigma,\sigma_A,\sigma_G$ is of order $8$; it is the dihedral group $D_4$ of
all the isometries of the square with vertices $A_1G_1, A_1 G_2, A_2 G_2, A_2 G_1$.
The stabilizer of a vertex is a cyclic group $C_2$, of type $\sigma$ or $\tau$, the stabilizer of an edge is of type
$\sigma_A$ or $\sigma_G$, noted $C_2^{A}$ or $C_2^{A}$.\\

Therefore, in the example $\mathcal{L}^{2}_3$, the group $G$ is the product of $\mathfrak{S}_3$ with a dihedral group
$D_4$.\\

\indent In the presentation given by the present article, the language $\mathcal{L}$ is a sheaf over the category $G$,
which plays the role of the fiber $\mathcal{F}$.
We have only one layer $U_0$, but the duality of propositions and theories corresponds to the duality between questions and answers (i.e. theories) respectively.\\

\noindent The action of $G$ on the set $\Omega$ is deduced from its action on the set $E$, which can be described as follows:
\begin{enumerate}[label=\arabic*)]
	\item One orbit of four elements, where $a,b,c$ have the same gender and age. The stabilizer of each element is
	$\mathfrak{S}_3\times C_2$, or order $12$.
	\item One orbit of $24$ elements made by a pair of equal subjects and one that differs from them by one attribute only.
	The stabilizer being the $\mathfrak{S}_2$ of the pair of subjects.
	\item One orbit of $12$ elements made by a pair of equal subjects and one that differs from them by the two attributes.
	The stabilizer being the product $\mathfrak{S}_2\times C_2$, where $C_2$ stabilizes the characteristic of the pair,
	which is the same as stabilizing the character of the exotic subject.
	\item One last orbit of $24$ elements, where the three subjects are different, then two of them differ by one attribute
	and differ from the last one by the two attributes. The stabilizer is the stabilizer $C'_2$ of the missing pair of values
	of the attributes.
\end{enumerate}

\noindent The action of $G$ on the set $E$ corresponds to the conjugation of the inertia subgroups.\\
\begin{rmk*}
	\normalfont All that looks like a Galois theory; however there exist subgroups of $G$, even normal subgroups, that cannot happen as stabilizers
	in the language, without adding terms or concepts. For instance, the cyclic group $\mathfrak{A}_3\subset \mathfrak{S}_3$; if it stabilizes a proposition $P$, this means
	that the subjects appear in complete orbits of $\mathfrak{A}_3$, but these orbits are orbits of $\mathfrak{S}_3$ as well, then
	the stabilizer contains $\mathfrak{S}_3$. The notion of cyclic ordering is missing.
\end{rmk*}

The collection of all the ultimate states of a given type defines a proposition, noted $T$, describing $I,II,III,IV$. This proposition has for stabilizer
the group $G$ itself. Its space of information must have a form attached to $G$, but it also must take into account the structure of its elements.\\

\begin{ans}
The information space of type $T$ corresponds to the natural groupoid of type $T$
\end{ans}

\noindent Remark that each type corresponds to a well formed sentence in natural languages: type $I$ is translated by
"all the subjects have the same attributes"; type $II$ by "all the subjects have the same attributes except one which differs
by only one aspect"; type $III$ "one subject is opposite to all the others"; type $IV$ "all the subjects are distinguished by
at least one attribute".\\
The union of the types $II$ and $III$ is described by the sentence "all the subjects have the same attributes except one".\\
The information space of $(II)\vee (III)$ is (naturally) a groupoid with $12$ objects and fundamental group $\mathfrak{S}_2$.
A good exercise is to determine the information spaces of all the unions of the four orbits. It should convince the reader that something interesting
happens here, even if the whole tentative here evidently needs to be better formalized.\\

\noindent Remark that other propositions have non-trivial inertia, and evidently support interesting semantic information. The most
important for describing the system are the \emph{numerical statements}, for instance "there exist two female subjects in the population".
Its inertia is $\mathfrak{S}_3\times C_2^{A}$.\\

\indent By definition, a \emph{simple} proposition is given by the form $aA$, telling that one given subject has one given value for one given attribute.
There exist twelve such propositions, they are permuted by the group $G$. The simple propositions form an orbit of the group $G$, of type $III$ above.\\

\noindent Amazingly, the set of the twelve simples is selfdual under the negation:
\begin{equation}
\neg (aA)=a\overline{A},
\end{equation}
where $\overline{A}$ denotes the opposite value.\\

\begin{ans}
	Each simple corresponds to a groupoid with one object, and four arrows, that form a Klein sub-group of $G$
which fixes the subject $a$ and fixes the attribute $A$ corresponding to $C_2$, generated by the transposition $\sigma_A$, also preserving $\overline{A}$.\\
\end{ans}

\noindent Another ingredient, introduced by Carnap and Bar-Hillel, is the \emph{mutual independency} of the
$12$ \emph{simple} propositions.\\

\noindent According to the definition of the spaces $I_2(Q,Q')$, this implies:
\begin{ans}
	The spaces of the simples are disjoint; the maximal information spaces, associated to full populations $e$, are unions of them, after some gluing.
\end{ans}

It is natural to expect that for each individual population $e\in X$, the information space
$H^{e}$ is a kind of marked groupoid $H^{T}_e$, that is a groupoid with
a singularized object. A good manner to mark the point $e$ in $H^{e}$ is to glue to the space $H^{T}$ of its type
a space $H^{P}$, where $P$ is the proposition which characterizes $e$ among the elements of the orbit $T$.
The groupoid of this space $H^{P}$  can contains several objects. \\

\noindent All kinds of gluing that we had to consider are realized by identifying two spaces $H_1, H_2$ with marked points along
a subspace $K$ (representing a mutual information or the space of the "or"), as asked by the axiom $(ii)$ above.\\
Therefore in general, the subspace has strictly less marked points  than any of the spaces that are glued. \\
When we mention \emph{cylinders} in this context, this means that one of the spaces, say $H_1$ is a cylinder
with basis $K$, and we say that $H_1$ is grafted on the other space $H_2$.\\

\begin{ans}\label{ans:infospaces}
		The information space of the ultimate element $e$ is obtained by gluing a cylinder to the space of its type, 
		based on a subspace associated to it, and containing as many objects as we need simple pieces
\end{ans}

\noindent For type $I$, one object is added; for type $II$ and $III$, two objects are added and for type $IV$, three objects. \\

\begin{ill*}
	\normalfont Associate to each $e$ a trefoil knot, presented as a braid with three colored strands, corresponding to its three simple
	constituents.\\
	Each subject corresponds to a strand, each pair of values $A,G$ of the attributes to a color, red, blue, green and black for the vertices ofthe
	square, red and green and blue and black being in diagonal.
\end{ill*}

\noindent Any proposition is a union of elementary ones, then to go farther, we have to delete pieces of the maximal spaces $H^{e}$,
for obtaining its information spaces.\\

The existence of a full coherent set of spaces is non-trivial and is described in detail in the forthcoming preprint, {\em A search of semantic spaces} \cite{sem-spaces}.\\

Then to describe the information of the more general propositions, we have to combine the forms given by the groups and groupoids, as for $H^{T}$ and $H^{e}$,
with a combinatorial counting of information, deduced from the content, as in Carnap and Bar-Hillel.\\
\indent A suggestion is to represent the combinatorial aspect by a dimension: all propositions are ranged by their numerical content,
for instance $e$ has $c(e)=63$, $\neg e$ has $c=1$, and $aA$ has $c=58$. We represent the groups and groupoids by $CW-$ complexes
of dimension $2$ or $\infty$, associated to a presentation by generators and relations of their fundamental group, possibly marked by several
base points. The spaces of information $H^{Q}$ are obtained by thickening
the complexes, by taking the product with a simplex or a ball of the dimension corresponding to $Q$. However, note that any manner
to code this dimension by a number, for instance, connected components, would work as well.\\

\noindent For some propositions, we cannot expect a form of information in addition of the dimension. This concerns propositions that are complex and not used in natural languages; example: "in this population,
there is two old mans, or there is a young woman, or there exist a woman that has the same age of a man". This is pure logical
calculus, not really semantic.\\

The general construction shows that the number of non-trivial semantic spaces is far 
from $2^{64}$, it is of the order of $64^{\alpha}$,
with $\alpha$ between $3$ or $4$.\\

Then, on this simple example we see that "spaces" of semantic information are more interesting and justified than numerical
estimations, but also that this concerns only few propositions, the ones which seem too have more sense. Then the structure
of spaces has to be completed by calculus
and combinatorics for most of the $2^{64}$ sentences. This touches the sensitive departure point from the \emph{admissible} sentences,
more relevant to Shannon theory, and the \emph{significant} sentences, more relevant for a future semantic theory, that we hope
to find in the above direction of homotopy invariants of spaces of theories and questions.

\chapter{Unfoldings and memories, LSTMs and GRUs}\label{chap:unfolding}

This chapter presents evidences that some architectures of $DNNs$, which are known to be efficient in syntactic and semantic
tasks, rely on internal invariance supported by some groupoids of braids, which also appear in enunciative linguistic, in relation
with cognition and representation of notions in natural languages.\\

\section{RNN lattices, LSTM cells}

Artificial networks for analyzing or translating successions of words, or any timely ordered set of data, have a structure in lattice,
which generalizes the chain: the input layers are arranged in a corner: horizontally $x_{1,0}$, $x_{2,0}$, $...$, named data, vertically $h_{0,1}$,
$h_{0,2}$, $...$, named hidden memories.\\
\indent Generically, there is a layer $x_{i,j}$ for each $i=1,2,...,N$, $j=0,1,2,...,M$, and a layer $h_{i,j}$ for each
$i=1,2,...,N$, $j=0,1,2,...,M$. The information of $x_{i,j-1}$ and $h_{i-1,j}$  are joined in a layer $A_{i,j}$, which sends information
to $x_{i,j}$ and $h_{i,j}$.\\
Then in our representation, the category $\mathcal{C}_\mathbf{X}$ has one arrow from $x_{i,j}$ to $A_{i,j}$, from $h_{i,j}$ to $A_{i,j}$,
from $x_{i,j-1}$ to $A_{i,j}$ and from $h_{i-1,j}$ to $A_{i,j}$, and it is all (see figure \ref{fig:RNN-poset}). If we want, we could add the layers $A^{\star }_{i,j}$, but there
is no necessity.\\
\indent The output is  generally a up-right corner horizontally $y_1=x_{1,M}$, $y_2=x_{2,M}$, $...$, named the result (a classification
or a translation), and vertically $h_{N,1}$, $h_{N,2}$, $...$, (which could be named future memories).\\

\begin{figure}[ht]
	\begin{center}
		\includegraphics[width=13cm]{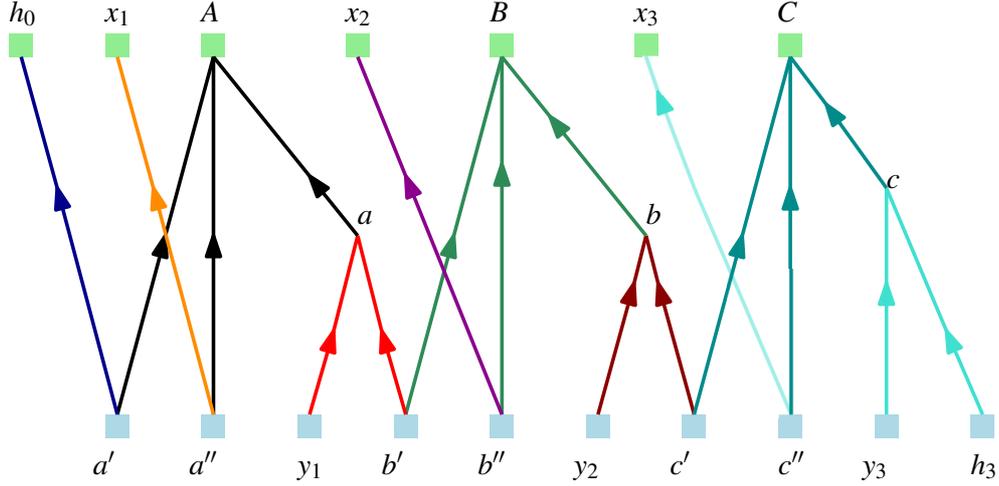}
		\caption{\label{fig:RNN-poset}Categorical representation of a RNN}
	\end{center}
\end{figure}

\noindent However, the inputs and outputs can have the shape of a more complex curves, transverse
to vertical and horizontal propagation. Things are organized as in a two dimensional Lorentz space, where a space coordinate is $x_{i,j-1}-h_{i-1,j}$ and a
time coordinate $x_{i,j-1}+h_{i-1,j}$. Input and output correspond to spatial sections, related by causal propagation.
\begin{rmk*}
	\normalfont In many applications, several lattices are used together, for instance a sentence or a book can be read backward after translation,
	giving reverse propagation, without trouble. We will discuss these aspects with the modularity.
\end{rmk*}

Most $RNNs$ have a dynamic of the type a non-linearity applied to a linear summation:\\
we denote the vectorial states of the layers by greek letters $\xi$ for layers $x$ and $\eta$ for layers $h$, like $\xi^{a}_{i,j}$ and $\eta^{b}_{k,l}$;
the lower indices denote the
coordinates of the layer and the upper indices denote the neuron, that is the real value of the state. In most applications, the basis of neurons plays an important role.\\
In the layer $A_{i,j}$ the vector of state is made by the pairs $(\xi^{a}_{i,j-1},\eta^{b}_{i-1,j});a\in x_{i,j-1}, b\in h_{i-1,j}$.\\
The dynamic $X^{w}$ has the following form:
\begin{equation}\label{fsigma}
\xi_{i,j}^{a}=f_x^{a}\left(\sum_{a'}w_{a';x,i,j}^{a}\xi_{i,j-1}^{a'}+\sum_{b'}u_{b';x,i,j}^{a}\eta_{i-1,j}^{b'}+\beta^{a}_{x,i,j}\right);
\end{equation}
\begin{equation}\label{fsigmah}
\eta_{i,j}^{b}=f_h^{b}\left(\sum_{a'}w_{a';h,i,j}^{b}\xi_{i,j-1}^{a'}+\sum_{b'}u_{b';h,i,j}^{b}\eta_{i-1,j}^{b'}+\beta^{b}_{x,i,j}\right).
\end{equation}
The functions $f$ are sigmoids or of the type $\tanh(Cx)$, the real numbers $\beta$ are named \emph{bias}, and the numbers $w$ and $u$ are the
weights.\\
In practice, everything here is important, the system being very sensitive, however theoretically, only the overall form
matters, thus for instance we can incorporate the bias in the weights, just by adding a formal neuron in $x$ or $h$, with fixed value $1$. The weights are
summarized by the matrices $W_{x,i,j}$, $U_{x,i,j}$, $W_{h,i,j}$, $ U_{h,i,j}$.\\
All these weights are supposed to be learned by backpropagation, or analog more general reinforcement.\\

Experiments during the eighties and nineties showed the strongness of the $RNN$s but also some weaknesses, in particular for learning or memorizing long
sequences. Then Hochreiter and Schmidhuber, in a remarkable paper in Neural Computation \cite{HochreiterSchmidhuber1997}, introduced a modification of the simple $RNN$, named the Long Short Term Memory,
or $LSTM$, which overcame all the difficulties so efficiently that more than thirty years after it continues to be the standard.\\
\indent The idea is to duplicate the layers $h$ by introducing parallel layers $c$, playing the role of longer time memory states, and just called
cell states, by opposition to hidden states for $h$.\\

\noindent In what follows we present the cell which replaces $A_{i,j}$ without insisting on the lattice aspect, which is unchanged for
many applications.\\

The sub-network which replaces the simple crux $A=A_{i,j}$ is composed of five tanks $A, F, I, H', V$, plus the inputs $C_{t-1},H_{t-1},X_{t-1}$,
and has nine tips $c'_{t-1},h'_{t-1},x'_{t},f,i,o,\widetilde{h},v_i,v_f$ plus the three outputs $c_t,h_t,y_t$. However, $y_t$ being a function
of $h_t$ only, it is forgotten in the analysis below.\\
In $A$, the two layers $h'$ and $x'$ (where we forget the indices $t-1$ and $t$ respectively)
join to give by formulas like \eqref{fsigmah} the four states of $i,f,o,\widetilde{h}$ respectively called input gate, forget gate, output gate,
combine gate, the first three are sigmoidal, the fourth one is of type $\tanh$, indicating a function of states separations. The weights
in these operations are the only parameters to adapt, they form matrices $W_i,U_i$, $W_f,U_f$, $W_o,U_o$ and $W_h,U_h$; which makes four times more than
for a $RNN$ (because the output $\xi_{i,j}$ is not taken in account).\\
Then the states in $v_f$ and $v_i$ are respectively given by combining $c'$ with $f$ and  $\widetilde{h}$ with $i$, in the simplest bilinear way:
\begin{equation}\label{hadamard}
\xi_v^{a}=\gamma^{a}\varphi^{a};a\in v;
\end{equation}
where $\gamma$ denotes the states of $c'$ or $\widetilde{h}$, and $\varphi$ the states of $f$ or $i$ respectively.\\
Note  that the above formulae have a sense if and only of the dimensions of $c$ and $f$ and $v_f$ are equal and
the dimension of $\widetilde{h}$ and $i$ and $v_i$ are equal. This is an important restriction.\\
At the level of vectors this diagonal product is name the
\emph{Hadamard product} and is written
\begin{equation}
\xi_v=\gamma\odot\varphi.
\end{equation}
It is free of parameters. Only the dimension is free for a choice.\\
Then, $v_i$ and $v_f$ are joined by a Hadamard sum, adding term by term, to give the new cell state
\begin{equation}
\xi_c=\xi_{v_f}\oplus\xi_{v_i};
\end{equation}
which implies that $v_i$ and $v_f$ have the same dimension.\\
And finally, a new Hadamard product gives the new hidden state:
\begin{equation}
\eta_h=\xi_o\odot\tanh \xi_c.
\end{equation}
We get an additional degree of freedom with the normalization factor $C$ in $\tanh Cx$ but this is all. However this implies that $c$
and $o$ and $h$ have the same dimension.\\
\indent Therefore the $LSTM$ has a discrete invariant, which is the dimension of the layers and is named its \emph{multiplicity} $m$.\\
Only the layers $x$ can have other dimensions; in what follows, we denote $n$ this dimension (see figure \ref{fig:lstm}).\\
\vspace{5mm}
\begin{figure}[ht]
	\centering
	\includegraphics[width=1\textwidth]{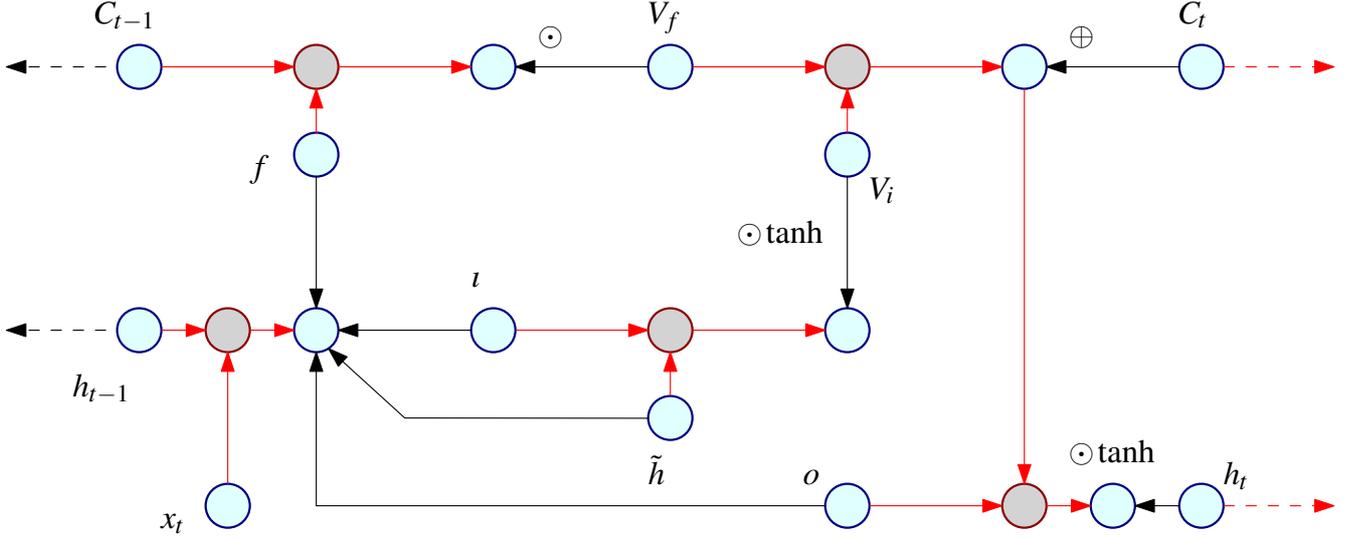}
	\caption{Grothendieck site representing a LSTM cell}
	\label{fig:lstm}
\end{figure}

\noindent Symbolically, the dynamics can be summarized by the two formulas:
\begin{equation}
c_t=c_{t-1}\odot \sigma_f(x_t,h_{t-1})\oplus \sigma_i(x_t,h_{t-1})\odot \tau_h(x_t,h_{t-1})
\end{equation}
\begin{equation}
h_t=\sigma_o(x_t,h_{t-1})\odot \tanh c_t,
\end{equation}
where $\sigma_k$ ({\em resp.} $\tau_k$) denotes the application of $\sigma$ ({\em resp.} $\tanh$) to a linear or affine form.\\
In what follows, $x_t$ is replaced by $x'$ and $h_{t-1}$, $c_{t-1}$ by $h'$, $c'$, like their tips.\\

Due to the non-linearities $\sigma$ and $\tanh$, there are several regimes of functioning, according to the fact that some of the
variables give or not a saturation; this can generate almost linear transformations or the opposite, a discrete-valued transformation. For instance,
$\pm 1$ when $\tanh$ is applied, or $\in \left\{0,1\right\}$ if $\sigma$ is applied. Here appears the fundamental aspect of discretization in the functioning of $DNNs$.\\

\noindent In the linear regime, the new state $c$ appears as a polynomial of degree $2$ in the vectors $x,h'$ and degree $1$
in $c'$, and $h$ appears as a polynomial of degree $3$ in $x',h'$.\\
Introducing the linear (or affine with bias) forms $\alpha_f,\alpha_i,\alpha_o,\alpha_h$, before application of $\sigma$ or $\tanh$, we have
\begin{equation}
h_t=\alpha_o\odot (c'\odot \alpha_f\oplus \alpha_i\odot \alpha_h).
\end{equation}
The dominant term in $x',h'$ is decomposable: $\alpha_o\odot \alpha_i\odot \alpha_h$; the term of degree $2$ in $x',h'$
is $\alpha_o\odot c'\odot\alpha_f$, and there is no linear term, because we forgot the bias. When separating $x'$ from
$h'$, we obtain all possible degrees $\leq 3$.\\
\noindent However, experiments with alternative memory cells, named $GRU$ and their simplifications, have shown that the degree
in $x'$ is apparently less important then the degree in $h'$. All trials with degree $<3$ in $h'$ gave a dramatic
loss of performance, but this was not the case for $x'$, where degree $1$ appears to be sufficient.\\

The number of parameters to tune is $4m^{2}+4mn$ or $4m^{2}+dmn$, with $1\leq d\leq 4$ is for the dependencies in $x$ in the four operations $\alpha_f,\alpha_i,\alpha_o,\alpha_h$. At least $d=1$ for $\alpha_h$ or for $\alpha_f$ seems to be necessary from the study of $MGU$.\\

\section{GRU, MGU}

Several attempts were made for diminishing the quantity of parameters to adapt in $LSTM$ without diminishing the performance.
The most popular solution is known as Gated Recurrent Unit, or $GRU$ (see  \cite{DBLP:journals/corr/ChoMBB14} and \cite{DBLP:journals/corr/ChungGCB14}
from Bengio's group). Then this cell has been simplified into several kinds of Minimal Gated Units, $MGU$ (\cite{DBLP:journals/corr/ZhouWZZ16} or \cite{DBLP:journals/corr/HeckS17}).\\
\indent The idea is to replace several gated layers by one, at the cost of a more complex architecture's topology.\\

\indent In the standard  $GRU$, the pair $h_t,c_t$  is replaced by $h_t$ alone, as in the original $RNN$; there exists two input
layers $X_t,H_{t-1}$, the number of
joins, our tanks, is six: $R,F,I,V,W,H'$, the number of tips is six, $z,r,v_{1-z},v_r,v_x,v_h$ and one output
$h_t$.\\
The dynamic begins with two non-linear linear transform, of type $\sigma\sum$, like \eqref{fsigmah} in $R$, giving $z$ and $r$ from $x'$ and $h'$; then in
$I$, there is a Hadamard product $v_z=h'\odot (1-z)$, where $1-z$ designates the Hadamard difference between the saturation and the values of the states of $z$.
Moreover, in $F$, there is another Hadamard product $v_r=h'\odot r$.
A $\tanh \sum$, like \eqref{fsigmah} with $f=\tanh$, joins $x'$ with $v_r$ in $W$ to give $v_x$, which joins $z$ in $H'$ to give
$v_h$ by a third Hadamard product. Finally, $v_h$ and $v_{1-z}$ are joined together by a Hadamard sum in $V$, giving $h=v_z\oplus v_h$.\\

\noindent Symbolically, with the same conventions used for $LSTM$, the dynamic can be summarized by the following formula
\begin{equation}
h_t=(1-\sigma_z(x_t,h_{t-1}))\odot h_{t-1}\oplus\sigma_z(x_t,h_{t-1})\odot \tanh(W_x(x_t)+U_x(\sigma_r(x_t,h_{t-1})\odot h_{t-1})).
\end{equation}

In a $GRU$ as in a $LSTM$ we have three Hadamard products and one Hadamard sum, plus three non-linear-linear transforms $NLL$ (one with $\tanh$); $LSTM$ had \
four $NLL$ transforms (two with $\tanh$), but the complexity of $GRU$ stays in the succession of two $NLL$ with adaptable parameters.\\
Remark that $LSTM$ also contains a succession of non-linearities, $\tanh$ being applied to $c_t$, which is a sum of product on non-linear terms
of type $\sigma$ or $\tanh$.\\

\noindent In the linear (or affine) regime, the $GRU$ gives
\begin{equation}
h_t=[(1-\alpha_z)\odot h_{t-1}]\oplus [\alpha_z\odot [Wx_t +U(\alpha_r\odot h_{t-1})]].
\end{equation}

For the same reason than $LSTM$ a $GRU$ has a multiplicity $m$, and a dimension $n$ of data input. The parameters to be adapted are the matrices $W_z,U_z$, $W_r,U_r$
and $W_x, U_x$ in $W$. This gives $3m^{2}+3mn$ real numbers to adapt, in place of $4m^{2}+4mn$ for a complete $LSTM$.\\

The simplification which was proposed by Zhou et al. in \cite{DBLP:journals/corr/ZhouWZZ16} for $MGU$ consists in taking $\sigma_z=\sigma_r$, thus reducing the parameters
to $2m^{2}+2mn$. This unique vector is denoted $\sigma_f$, assimilated to the forget gate $f$ of $LSTM$.\\
It seems that the performance of $MGU$ was as good as the ones of $GRU$, which are almost as good as $LSTM$ for many tasks.\\

Heck and Salem \cite{DBLP:journals/corr/HeckS17} suggested further radical simplifications, some of them being as good as $MGU$. $MGU1$ consists in suppressing the dependency of the unique $\sigma_f$ in $x'$, and $MGU2$ in suppressing also the bias
$\beta_f$. An $MGU3$ removed $x'$ and $h'$, just keeping a bias, but it showed poor learning and accuracy in the tests.\\

The experimental results proved that $MGU2$ is excellent in all tests, even better than $GRU$.\\
Note that both $MGU2$ and $MGU1$ continue to be of degree $3$ in $h'$. This reinforces the impression that this degree is an
important invariant of the memory cells. But these results indicate that the degree in $x'$ is not so important.\\

\noindent Consequently we may assume
\begin{equation}
	h_t=(1-\sigma_z(h_{t-1}))\odot h_{t-1}
	\oplus\sigma_z(h_{t-1})\odot \tanh(W_x(x_t)+U_x(\sigma_z(h_{t-1})\odot h_{t-1}))).
\end{equation}
And in the linear regime
\begin{equation}
h_t=[(1-\alpha_z)\odot h']\oplus [\alpha_z\odot [Wx_t +U(\alpha_z\odot h')]].
\end{equation}
Only two vectors of linear (or affine) forms intervene, $\alpha_z^{a}(h');a=1,...,m$ and $h'$ itself, i.e. $\eta^{a}(h');a=1,...,m$.\\
The parameters to adapt are $U_z$, giving $\alpha_z$, and $U_x=U$, $W_x=W$, giving the polynomial of degree two in parenthesis, i.e. the state
of the layer called $v_h$.\\

\noindent The number of free parameters in $MGU2$ is $2m^{2}+mn$, twice less than the most economical $LSTM$.\\

The graph $\Gamma$ of a $GRU$ or a $MGU$ has five independent loops, a fundamental group free of rank five; it is non-planar.
The categorical representation of
a $LSTM$ has only three independent loops, and is planar (see figure \ref{fig:lstm}).

\section{Universal structure hypothesis}

A possible form of dynamic covering the above examples is a vector of dimension $m$ of non-linear functions of several vectors
$\sigma_{\alpha^{a}}$, $\sigma_{\beta^{b}}$, ..., that are $\sigma$ of $th$ functions of linear (or perhaps affine) forms of the variables
$\xi^{a},\eta^{b}$, for $a,b,c$ varying from $1$ to $m$. More precisely
\begin{equation}
\eta_t^{a}=\sum_{b,c,d}t_{b}^{a}\sigma_{\alpha^{b}}\tanh\left[\sum_{c,d}u_{c,d}^{a}\sigma_{\beta^{c}}\sigma_{\gamma^{d}}+\sum_{c}v_{c}^{a}\sigma_{\beta^{c}}
+\sum_{d}w^{a}_{d}\sigma_{\gamma^{d}}+ \sigma_{\delta^{a}}\right].
\end{equation}

\noindent Remark:  we have written $\sigma_\alpha,\sigma_\beta,...$ for the application to a linear form of a sigmoid or a $\tanh$ indifferently;
but for a more precise discussion of the examples, we must distinguish and write $\tau_\alpha,\tau_\beta,...$ when $\tanh$ is applied.
However, sometimes in the following lines, we will use $\tau$ when we are sure that a $\tanh$ is preferable to a $\sigma$.\\

The tensor $u_{c,d}^{a}$ would introduce $m^{3}$ parameters, leading to great computational difficulties.
A natural manner to limit the degrees of freedom at $Km^{2}$, inspired by $LSTM$ and $GRU$, is to use the Hadamard product, for instance $\sigma_{\beta^{a}}\sigma_{\gamma^{a}}$.\\
A second simplification, justified by the success of $MGU$ consists to impose $\alpha^{a}=\gamma^{a}$.\\
A third one, justified by the success of $MGU2$ is to limit the degree in $x'$ to $1$. This can be done by reserving the dependency
on $x'$ to the forms $\beta$ and $\delta$.\\
All that gives
\begin{equation}\label{mgu}
\eta_t^{a}=\sigma_{\alpha^{a}}(\eta)\tanh\left[\sigma_{\alpha^{a}}(\eta)\sigma_
{\beta^{a}}(\eta,\xi)+\sigma_{\beta^{a}}(\eta,\xi)+ \sigma_{\delta^{a}}(\xi)\right].
\end{equation}
This contains $2m^{2}+2mn$ free parameters to be adapted.\\

\begin{rmk*}
	\normalfont Here we have neglected the addition of the alternative term in the dynamic which is $(1-\sigma_{\alpha^{a})}\eta^{a}$ in $GRU$ and
	$MGU$, but this term is probably very important, therefore, we must keep in mind that it can be added in the applications. At the end it will
	reappear in the formulas we suggest below.
\end{rmk*}

\noindent For $MGU1,2$, the term of higher degree has no dependency in $x'$, then we can simplify further in
\begin{equation}\label{mgus}
\eta_t^{a}=\sigma_{\alpha^{a}}(\eta)\tanh\left[\sigma_{\alpha^{a}}(\eta)\sigma_{\beta^{a}}(\eta)+\sigma_{y^{a}}(\xi)\sigma_{\beta^{a}}(\eta)+ \tau_{\delta^{a}}(\xi)\right].
\end{equation}
\noindent Moreover, as $MGU2$ is apparently better than $MGU1$ in the tested applications, the forms $\alpha^{a}$ can be taken linear, not affine.\\

It looks like a simplified $LSTM$, if we define for the state of $c_t$ the following vector:
\begin{equation}
\gamma_t^{a}=\sigma_{\alpha^{a}}(\eta)\sigma_{\beta^{a}}(\eta)+\sigma_{y^{a}}(\xi)\sigma_{\beta^{a}}(\eta)
+ \tau_{\delta^{a}}(\xi),
\end{equation}
and impose the recurrence $y^{a}(\xi)=\gamma_{t-1}^{a}$.\\
This gives a kind of minimal $LSTM$, so-called $MLSTM$,
\begin{equation}
\gamma_t^{a}=\sigma_{\alpha^{a}}(\eta)\sigma_{\beta^{a}}(\eta)+\gamma_{t-1}^{a}\sigma_{\beta^{a}}(\eta)
+ \tau_{\delta^{a}}(\xi),
\end{equation}
\begin{equation}
\eta_t^{a}=\sigma_{\alpha^{a}}(\eta)\tanh[\gamma_t^{a}].
\end{equation}
Or with the forgotten alternative term,
\begin{equation}\label{mlstm}
\eta_t^{a}=\sigma_{\alpha^{a}}(\eta)\tanh[\gamma_t^{a}]+(1-\sigma_{\alpha^{a}}(\eta))\eta^{a}.
\end{equation}

Now we suggest to look at these formulas from the point of view of the deformation of singularities having polynomial
universal  models,
and trying to keep the
main properties of the above dynamics:
\begin{enumerate}[label=\arabic*)]
	\item on a generic straight line in the input space $h'$, and in any direction of the output space $h$,
	we have \emph{every possible shape of a 1D polynomial function of degree $3$}, when modulating by the functions of $x'$;
	\item the presence of non-linearity $\sigma$ applied to forms in $h'$ and $th$ applied to forms in $x'$
	allow discretized regimes for the full application, but also a regime where the dynamic is close to a simple
	polynomial model.
\end{enumerate}

In the above formulas the last application of $th$ renders possible the degeneration to degree $1$
in $h'$ and $x'$, we suggest to forbid that, and to focus on the coefficients of the polynomial. In fact the truncation
of the linear forms by $\sigma$ or $th$ is sufficient to warranty the saturation of the polynomial map. \\
From this point of view the terms of degree $2$ are in general not essential, being absorbed by a Viete
transformation. Also the term of degree zero, does not change the shape, only the values; but this can be non-negligible.\\
In the simplest form this gives
\begin{equation}\label{simplepolynomialform}
\eta_t^{a}=\sigma_{\alpha^{a}}(\eta)^{3}+u^{a}(\xi)\sigma_{\alpha^{a}}(\eta) +v^{a}(\xi);
\end{equation}
where $u$ and $v$ are $th$ applied to a linear form of $\xi$, and $\sigma_{\alpha}$ is a $\sigma$
applied to a linear form in $\eta$. This gives only $m^{2}+2mn$ free parameters, thus one order
less than $MGU2$ in $m$.\\

However, we cannot neglect the forgotten alternative $(1-z)h'$ of $GRU$, or more generally the possible function in the transfer
of a term of degree two, even if structurally, from the point of view of the deformation of shapes, it seems not necessary, thus the following form could be
preferable:
\begin{equation}
\eta_t^{a}=\sigma_{\alpha^{a}}(\eta)^{3}+(1-\sigma_{\alpha^{a}}(\eta))\eta^{a}+u^{a}(\xi)\sigma_{\alpha^{a}}(\eta)+v^{a};
\end{equation}
or more generally, with $2m^{2}+2mn$ free parameters:
\begin{equation}\label{polynomform}
\eta_t^{a}=\sigma_{\alpha^{a}}(\eta)^{3}+\sigma_{\alpha^{a}}(\eta)[\sigma_{\beta^{a}}(\eta)+u^{a}(\xi)]+
v^{a}(\xi);
\end{equation}
where $\beta$ is a second linear map in $\eta$.\\

\noindent {\bf Description of an architecture for this dynamic :} it has two input layers $H_{t-1}, X_t$, three sources or tanks $A$, $B$,
$C$, and seven internal layers that give six tips, $\alpha$,$\beta$, $v_\beta$, $u$, $v$, $v_{\alpha\beta}$, $v_{\alpha\alpha\alpha}$,
and one output layer $h_t$. First $h_{t-1}$ gives $\sigma_\alpha$ and $\sigma_\beta$, and $x_t$ gives $u$ and $v$; then $\sigma_\beta$ joins $u$ in $A$
to give $v_\beta=\sigma_\beta \oplus u$, then $\sigma_\alpha$ joins $v_\beta$ in $B$ to give $v_{\alpha\beta}=\sigma_\alpha\odot v_\beta$. In parallel,
$\sigma_\alpha$ is transformed along an ordinary arrow in $v_{\alpha\alpha\alpha}=\sigma_\alpha^{\odot 3}$. And finally, in $C$, the sum
of $v$, $v_{\alpha\alpha\alpha}$ and $v_\beta$ produces the only output $h_t$.\\

\noindent The simplified network is for $\beta=0$. It has also three tanks, $A$, $B$ and $C$, but only five tips, $\alpha$, $u$, $v$, $v_\alpha$,
$v_{\alpha\alpha\alpha}$. The schema is the same, without the creation of $\beta$, and $v_\beta$ ({\em resp.} $v_{\alpha\beta}$) replaced by $v_\alpha$
({\em resp.} $v_{\alpha\alpha}$). \\
\begin{rmk*}
	\normalfont In the models with $\tanh$ like \eqref{mlstm} the sign of the terms of effective degree three
	can be minus or plus; in the model \eqref{polynomform} it is always plus, however this can be compensated by the change
	of sign of the efferent weights in the next transformation.
\end{rmk*}

Equation \eqref{mgu} could induce the belief that $0$ goes to $0$, but in general this is not the case, because
the function $\sigma$ contrarily to $\tanh$ has only strictly positive values. For instance the standard $\sigma(z)=1/1+\exp (-z)$ gives $\sigma(0)=1/2$.\\
\indent However, the point $0$ plays apparently an important role, even if it is not preserved: 1) in $MGU2$ the absence of bias in $\alpha^{a}$ confirms this point;
2) the functions $\sigma$ and $th$ are almost linear in the vicinity of $0$ and only here. Therefore, let us define the space $H$ of the activities of the memory
vectors $h_{t-1}$ and $h_t$, of real dimension $m$; it is pointed by $0$, and the neighborhood of this point is a region of special interest.\\
\indent We also introduce the line $U$ of coordinate $u$ and the plane $\Lambda=U\times\mathbb{R}$ of coordinates $u,v$, where $0$ and its neighborhood is also crucial.
The input from new data $x_t$ is sent to $\Lambda$, by the two maps $u(\xi)$ and $v(\xi)$. By definition this constitutes an \emph{unfolding}
of the degree three map in $\sigma_\alpha(\eta)$.\\

\noindent A more complex model of the same spirit is
\begin{equation}\label{umbilicform}
\eta_t^{a}=\sigma_{\alpha^{a}}(\eta)^{3}\pm \sigma_{\alpha^{a}}(\eta)[\sigma_
{\beta^{a}}(\eta)^{2}+u^{a}(\xi)]+
v^{a}(\xi)\sigma_{\beta^{a}}(\eta)+w^{a}(\xi)[\sigma_
{\alpha^{a}}(\eta)^{2}+\sigma_{\beta^{a}}(\eta)^{2}]+z^{a}(\xi);
\end{equation}
it has $2m^{2}+4mn$ free parameters. The expression of $x_t$ is much richer and we will see below that it shares many
good properties with the model \eqref{simplepolynomialform}, in particular stability and universality. The corresponding space $U$
has dimension $3$ and the corresponding space $\Lambda$ has dimension $4$.\\

\section{Memories and braids}\label{sec:memories}

In every $DNN$, the dynamic from one or several layers
to a deeper one must have a sort of stability, to be
independent of most of the details in the inputs, but it must also be plastic, and sensitive to the important details
in the data, then not too stable, able
to shift from a state to another one, for constructing a kind of discrete signification. These two aspects are complementary.
They were extensively discussed a long time before the apparition of $DNN$s in the theory of dynamical
systems. The framework was different because most concepts in this theory were asymptotic, pertinent when the time
tends to infinity, and here in deep learning, to  the contrary,  most concepts are transient: one shot transformations for feed forward,
and gradient descent or open exploration for learning; however, with respect to the shape of individual transformation,
or with respect to the parameters of deformation, the two domains encounter similar problems, and probably answer in similar
manners.\\
\indent Structural stability is the property to preserve the shape after small variation of the parameters.
In the case of individual map between layers, this means that little change in the input has little effect on the output.
In the case of a family of maps, taking in account a large set of different inputs, this means that varying a little the weights,
we get little change
in the global functioning and the discrimination between data. The second level is deeper, because it allows to understand
what are the regions of the manifolds of input data, where the individual dynamics are stable in the first sense,
and what happens when individual dynamics changes abruptly, how are made the transitions and what are the properties of the
inputs at the boarders. A third level of structural stability concerns the weights, selected by learning: in the space of weights
it appears regions where the global functioning in the sense of family is stable, and regions of transitions where the global
functioning changes; this happens when the tasks of the network change, for instance detect a cat versus a dog. This last notion
of stability depends on the architecture and on the forms of dynamical maps that are imposed.\\

With $LSTM$, $GRU$ and their simplified versions like $MGU$, $MGU2$, we have concrete examples of these notions of
structural stability.\\
\indent The transformation is $X^{w}$ from $(h_{t-1},x_t)$ to $h_t$. The weights $w$ are made
by the coefficients of the linear forms, $\alpha^{a}(\eta),\beta^{a}(\eta),u^{a}(\xi)$, $v^{a}(\xi)$, but the structure
depends on the fixed architecture and the non-linearities, of two types, the tensor products and sums,
and the applied sigmoids and $tanh$.\\
\indent For simplicity we assume a response of the cell of the form \eqref{simplepolynomialform}, but the discussion is not very different with the other cell families
\eqref{polynomform}, \eqref{mgus} or \eqref{mlstm}.\\
We have a linear  endomorphism $\alpha$ of coordinates $\alpha^{a};a\in h$ of $\mathbb{R}^{m}=H$; when we apply to it the sigmoid function coordinate by coordinate,
we obtain a map $\phi$ from $H$ to a compact domain in $H$.
The invariance of the multiplicity $m$ of the memory cell suggests the hypothesis (to be verified experimentally) that  $\phi$ is
a diffeomorphism from $H$ to its image. However, as we will see just below, other reasons
like redundancy suggests the opposite, therefore we left open this hypothesis, with a preference for  diffeomorphism, for mathematical
or structural reasons. Probably, depending on the application, there exists a range of dimensions $m$ which performs the task,
such that $\phi$ is invertible.\\
We also have the two mappings $u^{a}(\xi);a\in h$ and $v^{a}(\xi);a\in h$ from the space $X=\mathbb{R}^{n}$ of states $x_t$, to $\mathbb{R}^{m}$.\\
This gives a complete description of the set of weights $W_{h;h',x'}$.\\
The formula \eqref{simplepolynomialform} defines the map $X^{w}$ from $H\times X$ to $H$.\\
We also consider  the restriction $X^{w}_\xi$ at a fixed state $\xi$ of $x_t$.\\

\begin{thm}\label{thm:activities}
	The map $X^{w}$ is not structurally stable on $H$ or $H\times X$, but each coordinate $\eta_t^{a}$, seen as function
	on a generic line of the input $h_{t-1}$  and a generic line of the input $x_t$, or as a function on $H$ or $H\times X$, is stable (at least in the bounded
	regions where the discretization does not apply).
\end{thm}

\noindent These coordinates represent the activities of individual neurons,
then we get structural stability at the level of the neurons and not at the level of the layers.\\

\noindent As we justify in the following lines, this theorem follows from the results of the universal unfolding theory of smooth mappings,
developed by Whitney, Thom, Malgrange and Mather (see \cite{gibson1976} and  \cite{zbMATH03826815}).\\
The main point here (our hypothesis) is the observation that, for each neuron in the $h_t$ layer, the cubic degeneracy
$z^{3}$ can appear, together with its deformation by the function $u$. \\
For the deformation of singularities of functions, and their unfolding, see  \cite{zbMATH03438104} and
\cite{arnold2012singularities}.\\

\noindent The universal unfolding of the singularity $z^{3}$ is given by a polynomial
\begin{equation}
P_{u}(z)=z^{3}+uz,
\end{equation}
This means that for every smooth real function $F$, from a neighbor of a point $0$ in $\mathbb{R}^{1+M}$, such that
\begin{equation}
F(z,0,...,0)=z^{3},
\end{equation}
there exist a smooth map $u(Y)$ and a smooth family of maps $\zeta(z,Y)$ such that
\begin{equation}
F(z,Y)=\zeta(z,Y)^{3}+u(Y)\zeta(z,Y)
\end{equation}

\noindent Equivalently, the smooth map
\begin{equation}
(z,u)\mapsto (P_u(z),u),
\end{equation}
in the neighbor of $(0,0)$ is stable: every map sufficiently near to it can be transformed to it by a pair of diffeomorphisms of the
source and the goal. This result on maps from the plane to the plane, is the starting point of the whole theory, found by Whitney: the stability of the gathered
surface over the plane $v,u$.\\
The stability is not true for the product
\begin{equation}
(z,u,w,v)\mapsto (P_u(z),u,P_v(w),v)
\end{equation}
The infinitesimal criterion of Mather is not satisfied (see \cite{gibson1976}, \cite{zbMATH03826815}).\\

There also exists a notion of universal unfolding for maps from a domain of $\mathbb{R}^{n}$ to
$\mathbb{R}^{p}$ in the neighborhood of a point $0$, however in most cases, there exists no universal unfolding,
at the opposite of the case of functions, when $p=1$.\\
Here $n=p=m$, the transformation from $h_{t-1}$ to $h_t$ is an unfolding, dependent of $\xi\in x_t$,
but it does not admit a universal deformation. It has an infinite codimension in the space of germs of maps.\\
Also for mappings, universality of and unfolding and its stability as a map are equivalent (another theorem from Mather).\\

Our non-linear model from equation \eqref{simplepolynomialform} with $u$ free being equivalent to the polynomial model by diffeomorphism, we
can apply to it the above results. This establishes theorem \ref{thm:activities}.\\
\begin{cor*}
	Each individual cell plays a role.
\end{cor*}

This does not contradict the fact that frequently several cells send
similar message, i.e. there exists a redundancy, which is opposite to the stability or genericity of the whole layer.
However, as said before, in some regime and/or for $m$ sufficiently small, the redundancy is not a simple repetition,
it is more like a creation of characteristic properties.\\

Let us look at a neuron $a\in h_t$, and consider the model \eqref{simplepolynomialform}.
If $u=u^{a}(\xi)$ does not change of sign, the dynamic of the neuron $a$ is stable under small
perturbations. For $u> 0$, it looks like a linear function, it is monotonic. For $u< 0$ there exist a unique stable minimum and a unique
saddle point which limits its basin of attraction. But for $u=0$ the critical points collide, the individual map is unstable.
This is named the \emph{catastrophe} point. For the whole theory, see \cite{Thom1972}, \cite{arnold2012singularities}.\\

If we are interested in the value of $\eta^{a}_t$, as this is the case in the analysis of the cat's manifolds seen before,
for understanding the information flow layer by layer, we must also consider the levels of the function, involving $v^{a}$ then $\Lambda$.
This asks to follow a sort of inversion of the flow, going to the past,
by finding the roots $z$ of the equations
\begin{equation}
P^{a}(z)=c.
\end{equation}
Depending on $u$ and $v$, there exist one root or three roots. For instance, for $c=0$, the second case happens if an only if
the numbers $u^{a}(\xi),v^{a}(\xi)$ satisfy the inequality $4u^{3}+27v^{2}<0$.  When the point $(u^{a}(\xi),v^{a})$ in the plane $\Lambda$
belongs to the \emph{discriminant} curve $\Delta$
of equation $4u^{3}+27\eta^{2}=0$, things become ambiguous, two roots collide and disappear together for $4u^{3}+27v^{2}>0$.\\
\noindent These accidents create ramifications in the cat's manifolds.\\

\noindent This analysis must be applied independently to all the neurons $a=1,...,m$ in $h$, that is to all the axis in $H$. If $\alpha$ is an
invertible endomorphism, the set of inversions has a finite number of solutions, less than $3^{m}$.\\

Remind that the region around $0$ in the space $H$ is especially important, because it is only here that the polynomial model
applies numerically, $\sigma$ and $\tanh$ being almost linear around $0$. Therefore the set of data $\eta_{t-1}$ and $\xi_t$ which gives
some point $\eta_t$ in this region have a special meaning: they represent ambiguities in the past for $\eta_{t-1}$ and critical parameters
for $\xi_t$. Thus the discriminant $\Delta$ of equation $4u^{3}+27v^{2}=0$ in $\Lambda$ plays an important role in the global dynamic.\\

The inversion of $X^{w}_\xi: H\rightarrow H$ is impossible continuously along a curve in $\xi$ whose $u^{a},v^{a}$ meet $\Delta$ for some component $a$.
It becomes possible if we pass to complex numbers, and lift the curve in $\Lambda$ to the universal covering $\Lambda_\star ^{\sim}(\mathbb{C})$ of the complement
$\Lambda_\mathbb{C}^{\star }$ of $\Delta_\mathbb{C}$ in $\Lambda_\mathbb{C}$ \cite{arnold2012singularities2}.\\
The complex numbers have the advantage that every degree $k$ polynomials has $k$ roots, when counted with multiplicities. The ambiguity in distinguishing
individual roots along a path is contained in the Poincaré fundamental group $\pi_1(\Lambda_\mathbb{C}^{\star })$. However the precise definition of this group
requires the choice of a base point in $\Lambda_\mathbb{C}^{\star }$, then it is more convenient to consider the fundamental groupoid $\Pi(\Lambda_\mathbb{C}^{\star })=\mathcal{B}_3$,
which is a category, having for points the elements of $\Lambda_\mathbb{C}^{\star }$ and arrows the homotopy classes of paths between two points. The choice
of an object $\lambda_0$ determine $\pi_1(\Lambda_\mathbb{C}^{\star };\lambda_0)$, which is the group of homotopy classes of loops from $\lambda_0$
to itself, i.e. the isomorphisms of $\lambda_0$ in $\mathcal{B}_3$. This group is isomorphic to the Artin braid group $B_3$ of braids with three strands \cite{arnold2012singularities2}.\\

\begin{figure}[ht]
	\centering
	\includegraphics[width=0.9\textwidth]{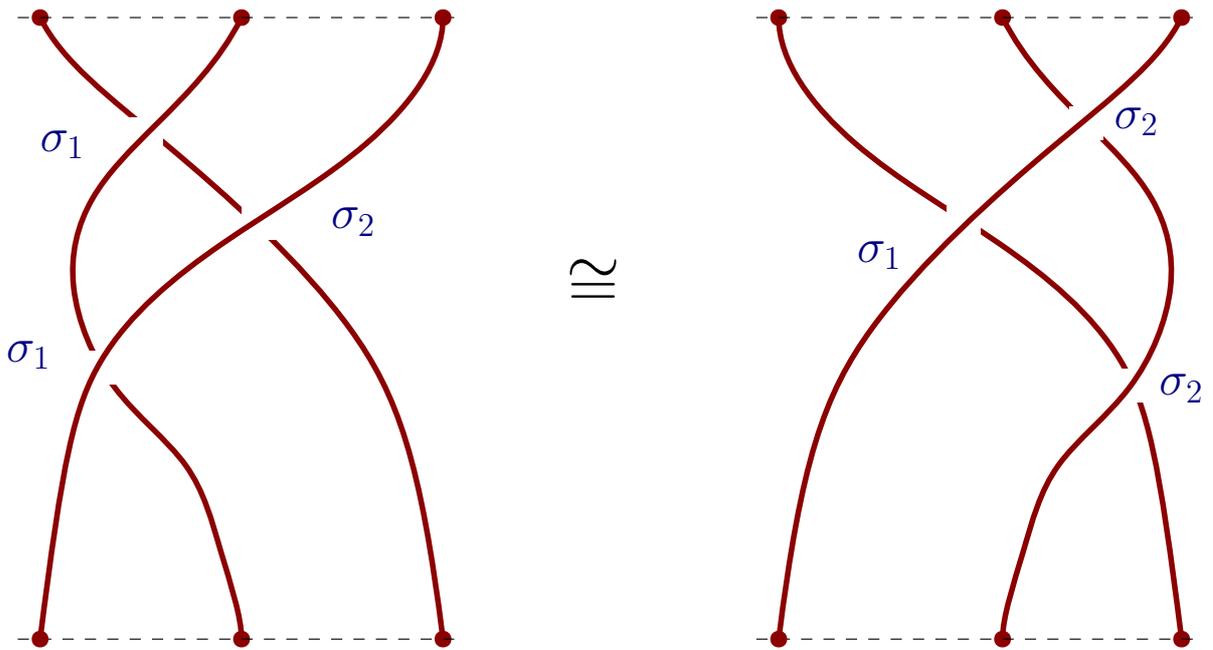}
	\caption{Two homotopic braids}
	\label{fig:braids}
\end{figure}

\noindent This group $B_3$ is generated by two loops $\sigma_1, \sigma_2$ that could be defined as follows: take a line $u=u_0\in \mathbb{R}_{-}\subset \mathbb{C}$,,
with complex coordinate $v$, and let $v^{+}_0, v_0^{-}$ be the positive and negative square roots of $-\frac{4}{27}u_0^{3}$; the loop $\sigma_1=\sigma^{+}$ ({\em resp.}
$\sigma_2=\sigma^{-}$) is based in $0$, contained in the line $u=u_0$ and makes one turn in the trigonometric sense around $v^{+}_0$  ({\em resp.} $v^{-}_0)$.
The relations between $\sigma_1$ and $\sigma_2$ are generated by $\sigma_1\sigma_2\sigma_1=\sigma_2\sigma_1\sigma_2$.\\
The center $C$ of $B_3$ is generated by $c=\left(\sigma_1\sigma_2\right)^{3}$. The quotient by this center is isomorphic to the group $B_3/C$ generated by $a=\sigma_1\sigma_2\sigma_1$
and $b=\sigma_1\sigma_2$ satisfying $a^{2}=b^{3}$; the quotient of $B_3/C$ by $a^{2}$ is the Möbius group $PSL_2\left(\mathbb{Z}\right)$ of integral homographies,
and the quotient of $B_3/C$ by $a^{4}$ is the
modular group $SL_2(\mathbb{Z})$ of integral matrices of determinant one, then a two fold covering of $PSL_2(\mathbb{Z})$. The quotient
$\mathfrak{S}_3$ of $B_3$ is defined by the relations $\sigma_1^{2}=\sigma_2^{2}=1$,
and by the relation which defines $B_3$, i.e. $\sigma_1\sigma_2\sigma_1=\sigma_2\sigma_1\sigma_2$ (see figure \ref{fig:braids}).\\

\indent Of course the disadvantage of the complex numbers is the difficulty to compute with them in $DNN$s, for instance $\sigma$ and $\tanh$ extended
to $\mathbb{C}$ have poles. Moreover all the dynamical regions are confounded in $\Lambda_\mathbb{C}^{\star }$; in some sense the room is too wide. Therefore,
we will limit ourselves to the sub-category $\Pi_\mathbb{R}=\mathcal{B}_3(\mathbb{R})$, made by the real points of
$\Lambda^{\star }$, but retaining all the morphisms between them, that is a full sub-category of $\mathcal{B}_3$. This means that only the paths are imaginary
in $\mathcal{B}_3(\mathbb{R})$.\\

\begin{figure}[ht]
	\centering
	\includegraphics[width=0.7\textwidth]{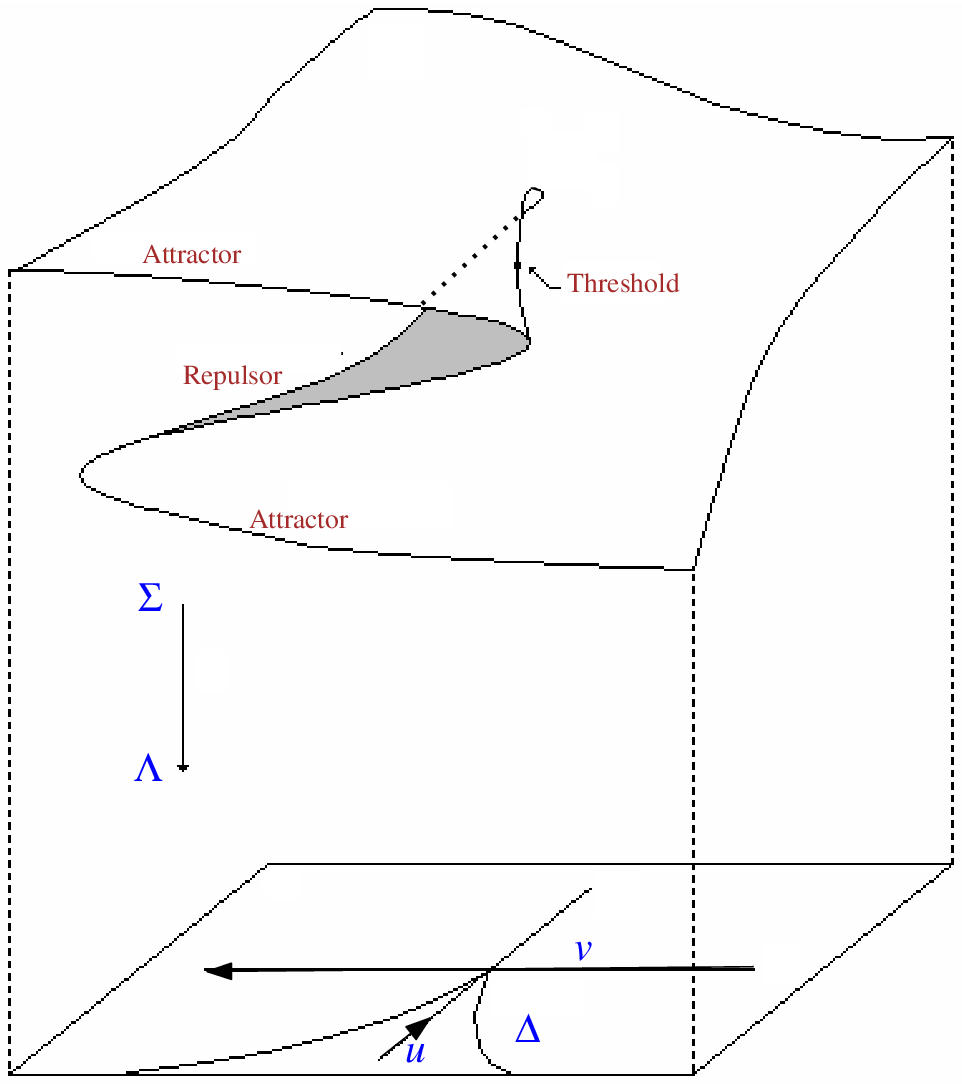}
	\caption{Cusp}
	\label{fig:cusp}
\end{figure}

Another sub-groupoid could be also useful (see figure \ref{fig:cusp}): consider the gathered surface $\Sigma$ in $\Lambda\times\mathbb{R}$ of equation
$z^{3}+uz+v=0$; let $\Delta_3$  be the natural lifting of $\Delta$ along the folding lines of $\Sigma$ over $\Lambda$,
the complement $\Sigma^{\star }$ of $\Delta_3$ in $\Sigma$ can be canonically embedded in the complex universal covering $\Lambda_\star^{\sim}$,
based in the real contractile region $\Lambda_0$ inside the real cusp, by taking, for each $(u,v)=\lambda$
in $\Lambda_0$ the points $\lambda_+$ and $\lambda_-$ respectively given by the paths $\sigma^{+}=\sigma_1$ and $\sigma^{-}=\sigma_2$, which
make simple turn over the branches of the cusp. When $\lambda$ approaches one of these branches, the corresponding point collides with it on $\Delta_3$, but the
other point continues to be isolated then the construction gives an embedding of $\Sigma^{\star }$. Therefore we can define the full sub-groupoid of $\mathcal{B}_3$ which has
as objects the points of $\Sigma^{\star }$, and name it $\mathcal{B}_3^{r}$ or $\Pi_r$.\\
\begin{rmk*}
	\normalfont The groupoid $\Pi_r$ can be further simplified, by taking one point in each region of interest: one point outside the preimage of the cusp $\Delta$,
	and three points in each region over the interior of the cusp.
\end{rmk*}
\begin{rmk*}
	\normalfont These four points correspond to the four real structures of Looijenga in
	the complex kaleidoscope \cite{Looijenga1978}.
\end{rmk*}

The groupoid $\mathcal{B}_3^{r}$ is naturally equipped with a covering (surjective) functor $\pi$ to the groupoid  $\mathcal{B}_3(\mathbb{R})$ of real points.\\
The interest of  $\mathcal{B}_3^{r}$ with respect to $\mathcal{B}_3(\mathbb{R})$ is that it distinguishes between the stable minimum and
the unstable one in the regime $u<0$. But the interest of $\mathcal{B}_3(\mathbb{R})$ with respect to $\mathcal{B}_3^{r}$ is that it speaks only of
computable quantities $u,v$ without ambiguity, putting all the ambiguities in the group $B_3$.\\

\noindent All these groupoids are connected, the two first ones, $\mathcal{B}_3(\mathbb{R})$ and $\mathcal{B}_3^{r}$ because they are full
subcategories of the connected groupoid $\mathcal{B}_3$, the other ones in virtue of the definition of a quotient (to the right) of a groupoid by a normal sub-group $H$
of its fundamental group $G$: it  has the
same objects, and two arrows $f,g$ from $a$ to $b$ are equivalent if they
differ by an element of $H$. This is meaningful because in $Aut_a$ ({\em resp.} $Aut_b$) the sub-group $H_a$ ({\em resp.} $H_b$) is well defined, being normal,
and moreover $f^{-1}g\in H_a$ is equivalent to $gf^{-1}\in H_b$.\\

Cardan formulas expresses the roots by using square roots and cubic roots. They give explicit formulas for the differences
of roots $z_2-z_1,z_3-z_1$. They can be seen directly in the surface $\Sigma$.\\
\begin{rmks*}
	\normalfont These formulas correspond to the simplest non trivial case of a map of period:
	\begin{enumerate}[label=(\roman*)]
		\item integral classes
		of the homology $H_0(P^{-1}_{u,v}(0)$ are transported along paths;
		\item the holomorphic form $dz$ is integrated
		on the integral classes.
	\end{enumerate}
	This gives a linear representation of $B_3$, which
	factorizes through $\mathfrak{S}_3$.\\
	
	\noindent Augment the variable $z$ by a variable $y$, the roots can be completed by the levels $Z_{u,v}$ over $(u,v)\in \Lambda$,
	which are the elliptic curves
	\begin{equation}
		P_{u,v}(z,y)=z^{3}+y^{2}+uz+v=0,
	\end{equation}
	the $2$-form $\omega=dz\wedge dy$ can be factorized as follows
	\begin{equation}
		\omega=-\frac{1}{2}dP\wedge \frac{dz}{y};
	\end{equation}
	the integral of $dx/y$ over the curve $Z_{u,v}$ is an elliptic integral, its periods over
	integral cycles, gives a linear representation of $B_3$ which factorizes through $SL_2(\mathbb{Z})$.\\
	Every stabilization of $z^{3}$ by a quadratic form gives rise to the representation of the first case in odd dimension and
	of the second case in even dimension.
\end{rmks*}

Natural groupoids smaller than $\mathcal{B}_3$ are given by quotienting the morphisms, replacing $B_3$ by $\mathfrak{S}_3$ or $SL_2(\mathbb{Z})$
or its projective version $PSL_2(\mathbb{Z})$ made by homographies.\\

\section{Pre-semantics}\label{sec:presemantics}

The natural languages have many functions, from everyday life to poetry and science, or politics and law, however all of them rely on
cognitive operations about meanings and shapes, as they appear in the many language-games of Wittgenstein or the
action/perception dimensions of Austin. Cf. \cite{Wittgenstein1953}, \cite{Austin1961-AUSPP}.\\
\indent The linguist Antoine Culioli, having studied in depth a great variety of natural languages, tried to characterize
some of these operations in meta-linguistic, for instance the generic structure and dynamics of a \emph{notional domain}. The notion here
can be "dog" or "cat" or "good" or "absent" or anything which has a meaning for most peoples, or specialists in some field.
To have a meaning must involve in general several occurrences and disappearances of the notion, a knowledge of its possible
properties and individuations, in a language
and in the world (data for instance, relations between them and classifications). \\
A good reference is the book \emph{Cognition and Representation in Linguistic Theory}, A. Culioli, Benjamins, \cite{Culioli}.\\
\indent The notional domain has an interior
$I$ where the properties of the notion are sure, an exterior $E$ where the properties are false, and a boundary $B$, where
things are more uncertain. A path through the boundary goes from "truly P" to "truly not P", through an uncertain region where "non-really P,
non really not P" can be said. In the center of $I$ are one or several prototypes of the notion. A kind of gradient vector leads
the mind to these archetypes, that Culioli named attracting centers, or attractors; however he wrote in 1989 (upcit.) the following
important precision: "Now the term attractor cannot be interpreted as an attainable last point (...) but as the representation
of the imaginary absolute value of the property (the predicate) which organizes an aggregate of occurrences into a structured
notional domain." Culioli also used the term of organizing center, but as we shall see this would conflict with another use.\\
\indent The division $I,B,E$ takes all its sense when interrogative mode is involved, or negation and double negation, or
intero-negative mode. In negation you go out of the interior, in interro-negation you come back inside from $E$.
"Is your brother really here" (it means that "I do not expect that your brother is here".)
"Now that, that is not a dog!" (you place yourself in front of proposition P, or inside the notion $I$, you know what is a dog,
then goes to $E$); "Shall I still call that a dog?" "I do not refuse to help"; here come back in $I$ of "help" after a turn in
its exterior $E$.
All these circumstances involve an imaginary place $IE$, where the regions are not separated, this is like the cuspidal point
before the separation of the branches $I$ and $E$ of the cusp.\\
\indent Mathematically this corresponds precisely to the creation of the external ({\em resp.}
internal) critical point of $z^{3}+uz+v$, on the curve $\Delta$. Example: "he could not have left the window open", the meaning
mobilizes the place $IE$ of indetermination, the maximum of ambiguity, where the two actions, "left" and "not to left" are possible,
then one of them is forbidden, and "not having left" is retained by the negation. In the terminology of Thom, the place $IE$
is the \emph{organizing center}, the function $z^{3}$ itself, the most degenerate one in the stable family, giving birth to the unfolding.\\
\indent To describe the mechanisms beyond these paths, Culioli used the model of the \emph{cam}: "the movement travels from one
place to another, only to return to the initial plane". Example: start from $IE$, then make a half-turn around $I$ which passes
by $E$ then come to $I$ by another half-turn. "This book is only slightly interesting." Here the meaning only appears if you imagine the
place where interesting and not interesting are not yet separated, then go to not interesting and finally temperate the judgment
by going to the boundary, near $I$; the compete turn leads you in another place, over the same point, thus the meaning is greatly
in the path, as an enclosed area. "This book is not uninteresting" means that it is more than interesting.
The paths here are well represented on the
gathered real surface $\Sigma$, of equation
\begin{equation}
z^{3}+uz+v=0,
\end{equation}
but they can also be made in the complement of $\Delta$ in $\Lambda$ in a complexified domain. It seems that
only the homotopy class is important, not the metric, however we cannot neglect a weakly quantitative aspect, on the way
of discretization in the nuances of the language. Consequently, the convenient representation of the moves of Culioli
is in the groupoid $\mathcal{B}_3^{r}$, that we propose to name the {\em Culioli groupoid}.\\

Remind that $LSTM$ and the other memory cells are mostly used in chains, to 
translate texts.\\
It is natural to make a rapprochement between their structural and dynamical properties and the meta-linguistic
description of Culioli. In many aspects René Thom was closed to Culioli in his own 
approach  of semantics,
see his book \emph{Mathematical Models of Morphogenesis} \cite{thom1983mathematical}, which is a translation of a French book published by Bourgois in 1980. The original theory was exposed in \cite{Thom1972}. In this approach, all the elementary catastrophes
having a universal unfolding of dimension less than 4
are used, through their sections and projections, for understanding in particular the valencies of the verbs, from
the semantic point of view,
according to Peirce, Tesnière, Allerton: impersonal, "it rains", intransitive "she sleeps", transitive "he kicks the ball",
triadic "she gives him a ball", quadratic "she ties the goat to a tree with a rope".\\
The list of organizing centers
is as follows:
\begin{multline}
y=x^{2},\quad y=x^{3},\quad y=x^{4},\quad y=x^{5},\quad y=x^{6},\\
y=x_1^{3}-x_2^{2}x_1,\quad y=x_1^{3}+x_2^{3}\quad (or \quad y=x_1^{3}+x_2^{2}x_1),\quad y=x_1^{4}+x_2^{2}x_1;
\end{multline}
respectively named: \emph{well}, \emph{fold}, \emph{cusp}, \emph{swallowtail}, \emph{butterfly},
\emph{elliptic umbilic}, \emph{hyperbolic umbilic} and \emph{ parabolic umbilic}, or with respect to the group which generalizes the Galois group $\mathfrak{S}_3$
for the fold, respectively: $A_1$, $A_2$, $A_3$, $A_4$, $A_5$, $D_4^{+}=D_4^{-}=D_4$ and $D_5$. The $A_n$ are the symmetric
groups $\mathfrak{S}_{n+1}$ and the $D_n$ index two subgroups of the symmetry groups of the hypercubes $I^{n}$ \cite{SB_1984-1985__27__19_0}.\\

\noindent It is not difficult to construct networks, on the model of $MLSTM$, such that the dynamics of neurons obey to the unfolding of
these singular functions. The various actors of a verb in a sentence could be separated input data, for different coordinates
on the unfolding parameters. The efficiency of these cells should be tested in translation.\\

Coming back to the memory cell \eqref{simplepolynomialform}, the critical parameters $x_t$ over $\Delta$
can be interpreted as boarders between regions of notional domains.\\
\indent The precise learned $2mn$ weights $w_x$ for the coefficients $u^{a}$ and $v^{a}$, for $a=1,...,m$, together with
the weights in the forms $\alpha^{a}$ for $h_{t-1}$ gives vectors (or more accurately matrices), which are like readers of the words $x$ in entry,
taking in account the contexts from the other words through $h$. Remember Frege: a word has a meaning only in the context od a sentence.
This is a citation of Wittgenstein, after he said that "Naming is not yet a move in a language-game" \cite[p. 49]{Wittgenstein1953}.\\
\indent To get "meanings", the names, necessarily
embedded in sentences, must resonate with other contexts and experiences, and must be situated with respect to the discriminant, along
a path, thus we suggest that the vector spaces of "readers" $W$, and the vector spaces of states $X$ are local systems $A$ over
a fibered category $\mathcal{F}$ in groupoids $\mathcal{B}_3^{r}$ over the network's category $\mathcal{C}$.\\
In some circumstances, the groupoid $\mathcal{B}_3^{r}$ can be replaced by the quotient over objects $\mathcal{B}_3(\mathbb{R})$,
or a quotient over morphisms giving $SL_2$ or $\mathfrak{S}_3$.\\

The case of $z^{3}$ corresponds to $A_2$. It is tempting to consider the case of $D_4$, i.e. the elliptic and hyperbolic umbilics, because
their formulas are very closed to $MGU2$ as mentioned at the end of the preceding section.\\
\indent This would allow the direct coding and translation of sentences by using three actant.\\
\begin{equation}
\eta=z^{3}\mp zw^{2}+uz+vw+x(z^{2}+w^{2})+y.
\end{equation}

\chapter{A natural $3$-category of deep networks}\label{chap:3-category}

In this chapter, we introduce a natural $3$-category for representing the morphisms, deformations and surgeries
of semantic functioning of $DNNs$ based on various sites and various stacks, which have connected models in their
fibers.\\
\indent Grothendieck's derivators will appear at two successive levels:
\begin{enumerate}
	\item formalizing internal aspects of this $3$-category;
	\item defining potential invariants of information over the objects of this $3$-category. Therefore we can expect that the interesting relations (for the theory and for its applications) appear at the level of a kind of "composition of derivators", and are analog to the spectral sequences of \cite{Grothendieck_1957}.
\end{enumerate}

\section{Attention moduli and relation moduli}

In addition to the chains of $LSTM$, another network's component is now recognized as essential for most of the tasks
in linguistic: to translate, to complete a sentence, to determine a context and to take into account a context for finding
the meaning of a word or sentence. This modulus has its origin in the \emph{attention operator}, introduced by Bahdanau et al. \cite{Bahdanauetal2014},
for machine translation of texts. The extended form that is the most used today was defined in the same context by Vaswani et al. 2017 \cite{DBLP:journals/corr/VaswaniSPUJGKP17},
under the common name of \emph{transformer} or simply \emph{decoder}. \\

\noindent Let us describe the steps of the algorithm: the input contains vectors $Y$ representing memories or hidden variables like contexts,
and external input data $X$ also in vectorial form.
\begin{enumerate}[label=\arabic*)]
	\item Three sets of linear operators are applied:
	\begin{align*}
		Q&=W^{Q}[Y],\\
		K&=W^{K}[Y,X],\\
		V&=W^{V}[Y];
	\end{align*}
	where the $W$'s are matrices of weights, to be learned.
	The vectors $Q,K,V$ are respectively called \emph{queries}, \emph{keys} and \emph{values}, from names used in Computer Science;
	they are supposed to be indexed by "heads" $i\in I$, representing individuals in the input,
	and by other indices $a\in A$, representing for instance different instant times, or aspects, to be integrated together.
	Then we have vectors $Q^{a}_i,K^{a}_i,V^{a}_i$.
	\item The inner products $E^{a}_i=k(Q^{a}_i|K^{a}_i)$  are computed (implying that $Q$ and $K$ have the same dimension),
	and the soft-max function is applied to them, giving a
	probability law, from the Boltzmann weights of energy $E^{a}_i$
	\begin{equation}
		p^{a}_i=\frac{1}{Z^{a}_i}e^{E^{a}_i},
	\end{equation}
\item a sum of product is computed
\begin{equation}
	V'_i=\sum_ap^{a}_iV^{a}_i.
\end{equation}
\item A new matrix is applied in order to mix the heads
\begin{equation}
	A_j=\sum_iw_j^{i}V'_i.
\end{equation}
\end{enumerate}

All that is summarized in the formula:
\begin{equation}
A_j(Y,X)=\sum_i\sum_a w_{j}^{i} {\rm softmax}\left[k(W^{Q}(Y)^{a}_i|W^{K}(Y,X)^{a}_i)\right]W^{V}(Y)^{a}_i.
\end{equation}
\indent A remarkable point is that, as it is the case for $MGU2$ or $LSTM$ and $GRU$ cells, the transformer corresponds to a mapping of degree $3$,
made by multiplying a linear form of $Y$ with non-linear function of a bilinear form of $Y$. Strictly speaking the degree $3$ is
only valid in a region of the parameters. In other regions, some saturation decreases the degree.\\

Chains of $LSTM$ were first used for language translations, and were later on used for image description helped by sentences predictions, as in
\cite{DBLP:journals/corr/KarpathyF14} or \cite{MaoYuille2015}, where they proved to outperform
other methods for detection of objects and their relations.\\
\indent In the same manner, the concatenation of attention cells has been proven to be very  beneficial in this context \cite{DBLP:journals/corr/abs-1806-01830}, then
it was extended to develop reasoning about the
the relations between objects in images and videos \cite{DBLP:journals/corr/RaposoSBPLB17},
\cite{pmlr-v80-barrett18a}, \cite{pmlr-v80-barrett18a}, \cite{San-Rap},
or \cite{DBLP:journals/corr/abs-2012-08508}.\\

In the $MHDPA$ (multi-head dot product attention) algorithm \cite{Santoro2018RelationalRN}, the inputs $X$ are either words, questions and features of objects
and their relations, coded into vectors, the inputs $Y$ combine hidden and external memories, the outputs $A$ are new memories, new relations and new questions.
\begin{rmk*}
	\normalfont Interestingly, the method combines fully supervised learning with unsupervised learning (or adaptation) by maximization of a learned functional
	of the above variables.
\end{rmk*}

\indent In particular, the memories or hidden variables issued from the transformer were re-introduced in the $LSTM$ chain; giving the
following symbolic formulas:
\begin{equation}
c_t=c_{t-1}\odot \sigma_f(x_t,h_{t-1})\oplus \sigma_i(x_t,m_{t})\odot \tau_h(x_t,h_{t-1});
\end{equation}
where $m_t$ results of transformer applied to the antecedent sequence of $h_s$, $c_s$ and $x_s$ ; and
\begin{equation}
h_t=\sigma_o(x_t,h_{t-1})\odot \tanh c_t.
\end{equation}

\noindent Geometrically, this can be seen as a \emph{concatenation of folds}, as proposed by Thom \emph{Esquisse d'une Sémiophysique} \cite{Thom1988}, to explain many
kinds of organized systems in biology and cognition. From this point of view, the concatenation of folds, giving the possibility of coincidence of cofolds \cite{argemi}, is a necessary condition for representing the emergence of a meaningful structure and oriented dynamic in a living system.\\
Note that, in the unsaturated regimes, $h_t$ has a degree $5$ in $h_{t-1}$, then its natural groupoid can be embedded in a braids groupoid of type $\mathcal{B}_5$. This augmentation, from the fold to the so called \emph{swallowtail}, could explain the greatest syntactic power of the $MHDPA$ with respect to
$LSTM$. However the concrete use of more memories in times $s$ before $t$ makes the cells much more complex than a simple mapping from
$t-1$ to $t$.\\

The above algorithm can be composed with other cells for detecting relations. For instance, Raposo et al. \cite{DBLP:journals/corr/RaposoSBPLB17}
have defined a \emph{relation operator}:
having produced contexts $H$ or questions $Q$ concerning two objects $o_i, o_j$ by a chain of $LSTM$ (that can be helped by external
memories and attention cells) the answer is taken from a formula:
\begin{equation}
A=f\left(\sum_{i,j}g\left(o_i,o_l;Q,H\right)\right),
\end{equation}
where $f$ and $g$ are parameterized functions, and $o_i:i\in I$ are vectors representing objects with their characteristics.\\
The authors insisted on the important invariance of this operator by the permutation group $\mathfrak{S}_n$ of the objects.\\

\noindent More generally, composed networks were introduced in 2016 by Andreas et al. \cite{DBLP:journals/corr/AndreasRDK16} for question answering about images.
The reasoning architecture $MAC$, defined by Hudson and Manning, \cite{HM18},
is composed of three attention operators named \emph{control}, \emph{write} and \emph{read}, in a $DNN$, inspired from the architecture of computers.\\

\noindent This leads us to consider the evolution of architectures and internal fibers of stacks and languages, in relation to the problems to be
solved in semantic analysis.\\

\section{The $2$-category of a network}

For representing languages in DNNs, we have associated to a small category $\mathcal{C}$ the class
$\mathcal{A}_\mathcal{C}={\sf Grpd}_\mathcal{C}^{\wedge}$ of presheaves over the category of fibrations in groupoids over $\mathcal{C}$.
The objects of $\mathcal{A}_\mathcal{C}$ were described in terms of presheaves $A_U$ on the fibers $\mathcal{F}_U$ for $U\in \mathcal{C}$
satisfying gluing conditions, cf. sections $3$ and $4$.
\begin{rmk*}
	\normalfont Other categories than groupoids, for instance posets or fibrations in groupoids over posets, can replace the
	groupoids in this section, and are useful in the applications, as we mentioned before, and as we will show in the forthcoming article on semantic communication.
\end{rmk*}

\indent Natural morphisms between objects $(\mathcal{F},A)$ and $(\mathcal{F}',A')$ of $\mathcal{A}_\mathcal{C}$ are defined
by a family of functors $F_U:\mathcal{F}_U\rightarrow\mathcal{F}'_U$, such that for any morphism $\alpha:U\rightarrow U'$ in $\mathcal{C}$,
\begin{equation}
F'_\alpha\circ F_{U'}=F_U\circ F_\alpha;
\end{equation}
and by a family of natural transformations $\varphi_U:A_U\rightarrow F_U^{\star}(A'_U)=A'_U\circ F_U$, such that for any morphism $\alpha:U\rightarrow U'$ in $\mathcal{C}$,
\begin{equation}
F^{\star}_{U'}(A'_\alpha)\circ\varphi_{U'}=F_\alpha^{\star}(\varphi_U)\circ A_\alpha,
\end{equation}
from $A_{U'}$ to $F_\alpha^{\star}(F_U^{\star}A'_U)=F_{U'}^{\star}\left((F'_\alpha)^{\star}A'_U\right)$.\\
Note that the family $\left\{F_U;U\in \mathcal{C}\right\}$ is equivalent to a $\mathcal{C}$-functor $F:\mathcal{F}\rightarrow \mathcal{F}'$ of fibered categories in groupoids, and the family $\varphi_U$ is equivalent to a morphism $\varphi$ in the topos $\mathcal{E}_\mathcal{F}$ from the object $A$ to the object $F^{\star}(A')$.\\
\begin{rmk*}
	\normalfont These morphisms include the morphisms already defined for the individual classifying topos $\mathcal{E}_\mathcal{F}$. But, even for one fibration $\mathcal{F}$
	and its topos $\mathcal{E}$, we can consider non-identity end-functor from $\mathcal{F}$ to itself, which give new morphisms in $\mathcal{A}_\mathcal{C}$.
\end{rmk*}

\indent The composition of $(F_U,\varphi_U);U\in \mathcal{C}$ with $(G_U,\psi_U)$ from $(\mathcal{G},B)$ to $(\mathcal{F},\mathcal{A})$ is defined
by the ordinary composition of functors $F_U\circ G_U$, and the twisted composition of natural transformation
\begin{equation}
(\varphi\circ\psi)_U=G_U^{\star }(\varphi_U)\circ\psi_U: B_U\rightarrow (F_U\circ G_U)^{\star}A'_U.
\end{equation}
This rule gives a structure of category to $\mathcal{A}_\mathcal{C}$.\\

\indent In addition, the natural transformations between functors give the vertical arrows in ${\sf Hom}_\mathcal{A} (\mathcal{F},A: \mathcal{F}',A')$, that form
categories:\\
a morphism from $(F,\varphi)$ to $(G,\psi)$ is a natural transformations $\lambda: F\rightarrow G$, which in this case with groupoids, is an homotopy
in the nerve, plus a morphism $a:A\rightarrow A$, such that
\begin{equation}
A'(\lambda)\circ\varphi=\psi\circ a:A\rightarrow G^{\star }A'.
\end{equation}
For a better understanding of this relation, we can introduce the points $(U,\xi)$ in $\mathcal{F}$ over $\mathcal{C}$, and read
\begin{equation}
A'_U(\lambda_U(\xi))\circ\varphi_U(\xi)=\psi_U(\xi)\circ a_U(\xi):A_U(\xi)\rightarrow A'_U(G_U(\xi)).
\end{equation}
This can be understood geometrically, as a lifting of the deformation $\lambda$ to a deformation of the presheaves.\\
\indent Vertical composition is defined by usual composition for the deformations $\lambda$ and ordinary composition in ${\sf End}(A)$
for $a$. Horizontal compositions are for $\mathcal{F}\rightarrow \mathcal{F}'\rightarrow \mathcal{F}"$.\\

\noindent Horizontal arrows and vertical arrows satisfy the axioms of a $2$-category  \cite{giraud-cohomologie}, \cite{maclane:71}.\\
This structure encodes the relations between several semantics over the same network.\\

The relations between several networks, for instance moduli inside a network, or networks that are augmented  by external links,
belong to a $3$-category, whose objects are the above semantic triples, and the $1$-morphism are lifting of functors between sites $u:  \mathcal{C}\rightarrow\mathcal{C}'$.\\

\cite[Theorem 2.3.2]{giraud-cohomologie} tells us that, as for ordinary presheaves, there exist natural right and left adjoints $u_\star $ and $u_!$ respectively of the
pullback $u^{\star }$ from the $2$-category ${\sf Cat}_{\mathcal{C}'}$ of fibrations over $\mathcal{C}'$ to the $2$-category ${\sf Cat}_{\mathcal{C}}$ of fibrations over $\mathcal{C}$.
They are natural $2$-functors, adjoint in the extended sense. These $2$-functors define adjoint $2$-functors between the above $2$-categories of classifying toposes
$\mathcal{A}_{\mathcal{C}}$ and $\mathcal{A}_{\mathcal{C}'}$, by using the natural constructions of $SGA 4$ for the categories of presheaves. They can be seen
as substitutions of stacks and languages induced by functors $u$.\\

\noindent The construction of $\mathcal{A}_\mathcal{C}$ from $\mathcal{C}$ is a particular case of Grothendieck's derivators \cite{AMBP_2003__10_2_195_0}.\\

\section{Grothendieck derivators and semantic information}

For $\mathcal{M}$ a closed model category, the map  $\mathcal{C}\mapsto \mathcal{M}_\mathcal{C}$, or $\mathcal{M}^{\wedge}_\mathcal{C}$ (see section \ref{modelML}), is an example of \emph{derivator} in the sense
of Grothendieck. References are \cite{grothendieck1983pursuing}, \cite{grothendieck1990derivateurs}, the three articles of
Cisinski \cite{AMBP_2003__10_2_195_0}, and the book of Maltsiniotis on the homotopy theory of Grothendieck \cite{AST_2005__301__R1_0}.\\
\indent A derivator generalizes the passage from a category to its topos of presheaves, in order to develop homotopy theory, as topos were made to develop
cohomology theory. It is a $2$-functor $\mathbb{D}$ from
the category ${\sf Cat}$ (or a special sub-category of diagrams, for instance ${\sf Poset}$) to the $2$-category ${\sf CAT}$, satisfying four axioms.
\begin{enumerate}[label=\alph*)]
	\item The first one tells us that $\mathbb{D}$ transforms sums of categories into products,
	\item The second one that isomorphisms of images can be tested on objects,
	\item the third one
	that there exists, for any functor $u:\mathcal{C}\rightarrow \mathcal{C}'$, a right adjoint $u_\star $ (defining homotopy limit)
	and a left adjoint $u_!$ (defining homotopy colimit) of the functor
	$u^{\star }=\mathbb{D}(u)$;
	\item the fourth axiom requires that these adjoints are defined locally; for instance, if $X'\in \mathcal{C}'$, and $F\in \mathbb{D}(C)$,
	therefore $u_\star F\in \mathbb{D}(C)'$, the fourth axiom tells us that
	\begin{equation}
		(u_\star F)_{X'}\cong p_\star j^{\star }F;
	\end{equation}
	where $j$ is the canonical map from $\mathcal{C}|X'$ to $\mathcal{C}$, and $p$ the unique morphism from $\mathcal{C}|X'$ to $\star$.
\end{enumerate}

\indent Another formula that expresses the same thing is
\begin{equation}
(u_\star F)_{X'}\cong H^{\star }\left(\mathcal{C}|X';F|_{\mathcal{C}|X'}\right),
\end{equation}
abstract version of a Kan extension formula.\\
In general, the cohomology is defined by
\begin{equation}
H^{\star }(\mathcal{C};F)=(p_{\mathcal{C}})_\star F\in \mathbb{D}(\star ).
\end{equation}

A first example of derivator is given by an Abelian category ${\sf Ab}$, like commutative groups or real vector spaces, and it is defined by the
derived category of differential complexes, where quasi-isomorphisms (isomorphisms in cohomology) are formally inverted,
\begin{equation}
\mathbb{D}(I)={\sf Der}({\sf Hom}(I^{\rm op}, {\sf Ab})).
\end{equation}

Another kind of example is a \emph{representable  derivator}
\begin{equation}
\mathbb{D}_\mathcal{M}(I)={\sf Funct}(I^{\rm op}, \mathcal{M}),
\end{equation}
where $\mathcal{M}$
is a closed model category. This can be seen as a non-Abelian generalization of the above first example.\\

A third kind of examples is given by the topos of sheaves over a representable derivator $\mathcal{M}_{\mathcal{C}}^{\wedge}$.\\

Then representable derivators allow to compare the elements of semantic functioning between several networks, for instance
a network with a sub-network of
this network, playing the role of a module in computation.\\

Consider the sub-categories $\Theta_P$, over the languages $\mathcal{A}_\lambda,\lambda\in \mathcal{F}_U$, made by the theories that exclude a
rigid proposition $P=!\Gamma$, in the sense they contain $P\Rightarrow\Delta$, for a given chosen $\Delta$, (see appendix \ref{app:nat-lang}). The right slice category $P|\mathcal{A_\lambda}$
acts on  $\Theta_P$. The information spaces $F$ define an object of $\mathcal{M}_{\Theta_P}$, its cohomology allow us to generalize the cat's manifolds, that
we defined below with the connected components of the category $\mathcal{D}$, in the following way: the dynamical object $\mathbb{X}$ is assumed to be defined over the stack
$\mathcal{F}$, then the dynamical space $g\mathbb{X}$ is defined over the nerve of $\mathcal{F}$, and the semantic functioning gives a simplicial map $gS:g\mathbb{X}\rightarrow gI^{\bullet}$
from $g\mathbb{X}$ space to the equipped theories, then we can consider the inverse image of $\Theta_P$ over the functioning network.
Composing with $F$ we obtain a parameterized object $M_P$ in $\mathcal{M}$, defining a local system over the category associated to $g\mathbb{X}$,
which depends on $\Gamma,\Delta$.
This represents the semantic information in $\mathbb{X}$ about the problem of (rigidly) excluding $P$ when considering that $\Delta$ is
(thought to be) false. Seen as an element of $\mathbb{D}(g\mathbb{X})$, its cohomology is an homotopical invariant of the information.\\

In this text, we have defined information quantities, or information spaces, by applying cohomology or homotopy limits, over the category $\mathcal{D}$ which expresses
a triple $\mathcal{C},\mathcal{F},\mathcal{A}$, made by a language over a pre-semantic over a site. The Abelian situation was studied through the bar-complex
of cochains of the module
of functions $\Phi$ on the fibration $\mathcal{T}$ of theories $\Theta$ over the category  $\mathcal{D}$. A non-Abelian tentative, for defining spaces of information,
was also proposed  at this level, using (in the non-homogeneous form) the functors $F$ from $\Theta_{\rm loc}$ to a model category $\mathcal{M}$ (see section \ref{sec:homotopy}).
Therefore information spaces were defined at the level of $\mathcal{M}_\mathcal{T}$, not at a level $\mathcal{M}_\mathcal{C}$.\\

Information spaces belong to $\mathbb{D}_\mathcal{M}(\mathcal{T})$. To compare spaces of information flows in two theoretical semantic networks, we have
at disposition the adjoint functors $\varphi_\star , \varphi_!$ of the functors $\varphi^{\star }=\mathbb{D}(\varphi)$ associated to $\varphi: \mathcal{T}\rightarrow\mathcal{T}'$,
between categories of theories. Those functors $\varphi$ can be associated to changes of languages $\mathcal{A}$, changes of stacks $\mathcal{F}$
and/or changes of basic architecture $\mathcal{C}$.\\

An important problem to address, for constructing networks and applying deep learning efficiently to them, is the realization of information relations
or correspondences,
by relations or correspondences between the underlying invariance structures. For instance, to realize a family of homotopy equivalences ({\em resp.} fibration, {\em resp.} cofibration)
in $\mathcal{M}$, by transformations of languages, stacks or sites having some properties, like enlargement of internal symmetries.\\

The analog problem for presheaves (set valued) is to realize a correspondence (or relation) between the topos $\mathcal{I}^{\wedge}$ and $(\mathcal{I}')^{\wedge}$
from a correspondence between convenient sites for them.\\

\noindent For toposes morphisms this is a classical result (see \cite[4.9.4]{SGA4} or the Stacks project \cite[7.16n 2.29]{stack-project}) that any geometric morphism $f_\star :{\sf Sh}(I)\rightarrow {\sf Sh}(J)$ comes
from a morphism of sites up to topos equivalence between $I$ and $I'$. More precisely, there exists a site $I'$ and a cocontinuous and continuous
functor $v:I\rightarrow I'$ giving an equivalence $v_!:{\sf Sh}(I)\rightarrow {\sf Sh}(I')$ extending $v$, and a site morphism $J\rightarrow I'$, given by a continuous
functor $u:I'\rightarrow J$ such that $f_\star=u_\star\circ v_!$.\\

\noindent From \cite{Shulman_2012}, a geometric morphism between ${\sf Sh}(I)$ and ${\sf Sh}(J)$ comes from a morphism of site
if and only if it is compatible with the Yoneda embeddings.\\

\section{Stacks homotopy of DNNs}

The characterization of fibrant and cofibrant objects in $\mathcal{M}_\mathcal{C}$ was the main result of chapter \ref{chap:stacks}.
All objects of $\mathcal{M}_\mathcal{C}$ are cofibrant and
the fibrant objects are described by theorem \ref{thm:fibrations-dnn}; we saw that they correspond to ideal semantic flows, where the condition $\pi^{\star }\pi_\star ={\sf Id}$
holds. They also correspond to the contexts and the types of a natural $M-L$ theory. The objects of $Ho(\mathcal{M}_{\mathcal{C}})$, \cite{quillen1967homotopical},
are  these fibrant and cofibrant objects of $\mathcal{M}_{\mathcal{C}}$, the $Ho$ morphisms being the homotopy classes of morphisms in $\mathcal{M}_{\mathcal{C}}$, generated
by inverting formally zigzags similar to the above ones. Thus we get a direct access to the homotopy category $Ho\mathcal{M}_{\mathcal{C}}$.
The $Ho$ morphisms are the homotopy equivalences classes of the substitutions of variables in the $M-L$ theory.\\

\indent From the point of view of semantic information, we just saw that homotopy is pertinent at the next level: looking first at languages over
the stacks, then at some functors from the posets of theories to a test model category $\mathcal{M}'$, then going to $Ho(\mathcal{M}')$.
However, the fact that we restrict to theories over fibrant objects and fibrations between them, implies that the homotopy of
semantic information only depends on the images of these theories over the category $Ho(\mathcal{M}_\mathcal{C})$.
How to use this fact for functioning networks?\\

\appendix
\noappendicestocpagenum
\addappheadtotoc
\renewcommand{\thesection}{\Alph{section}}
\chapter*{Appendices}
\section{Localic topos and Fuzzy identities}\label{app:localic}
\begin{defns*}
	let $\Omega$ be a complete Heyting algebra; a {\rm set over} $\Omega$, $(X,\delta)$, also named an {\rm $\Omega$-set}, is a set $X$ equipped with a map $\delta:X\times X\rightarrow \Omega$, which is symmetric
	and transitive, in the sense that for any triple $x,y,z$, we have $\delta(x,y)=\delta(y,x)$ and
	\begin{equation}
		\delta(x,y)\wedge \delta(y,z)\leq \delta(x,z).
	\end{equation}
	Note that $\delta(x,x)$ can be different from $\top$.\\
	But we always have $\delta(x,y)=\delta(x,y)\cap \delta(y,x)\leq \delta(x,x)$, and $\delta(x,y)\leq \delta(y,y)$.\\
	As $\Omega$ is made for fixing a notion of relative values of truth, $\delta$ is interpreted as {\rm fuzzy equality} in $X$; it generalizes the characteristic function
	of the diagonal when $\Omega$ is boolean. In our context of DNN, it can be understood as the
	progressive decision about the outputs on the trees of layers rooted in a given layer.\\
	A {\rm morphism} from $(X,\delta)$ to $(X',\delta')$ is an application $f:X\times X'\rightarrow \Omega$, such that, for every, $x,x',y,y'$
	\begin{align}\label{omegamap}
		\delta(x,y)\wedge f(x,x')&\leq f(y,x'),\\
		f(x,x')\wedge \delta'(x',y')&\leq f(x,y');\\
		f(x,x')\wedge f(x,y')&\leq \delta'(x',y').
	\end{align}
	Moreover
	\begin{equation}\label{domaine}
		\delta(x,x)=\bigvee_{x'\in X'}f(x,x').
	\end{equation}
	Which generalizes the usual properties of the characteristic function of the graph of a function in the boolean case.\\
	The {\rm composition} of a map $f:X\times X'\rightarrow \Omega$ with a map $f':X'\times X"\rightarrow\Omega$ is given by
	\begin{equation}
		(f'\circ f)(x,x")=\bigvee_{x'\in X'}f(x,x')\wedge f(x',x").
	\end{equation}
	And the {\rm identity morphism} is defined by
	\begin{equation}
		{\sf Id}_{X,\delta}=\delta.
	\end{equation}
	This gives the category ${\sf Set}_\Omega$ of sets over $\Omega$, also named {\rm $\Omega$-sets}.
\end{defns*}

\indent The Heyting algebra $\Omega$ of a topos $\mathcal{E}$ is made by the subobjects of the final object $\mathbf{1}$; the elements of $\Omega$ are named the
open sets of $\mathcal{E}$. In fact, there exists an object $\boldsymbol{\Omega}$ in $\mathbf{E}$, the Lawvere object, such that for every object $X\in \mathcal{E}$, the
set of subobjects of $X$ is naturally identified with the set of morphisms $\boldsymbol{\Omega}^{X}$. When $\mathcal{E}={\sf Sh}(\mathbf{X})$ is a Grothendieck topos, $\boldsymbol{\Omega}$ is the sheaf over $X$, which
is defined by $\boldsymbol{\Omega}(x)=\Omega(\mathcal{E}|x)$, the subobjects of $\mathbf{1}|x$. In the Alexandrov case, $\boldsymbol{\Omega}(x)$ is the set of open sets for the Alexandrov topology contained in $\Lambda_x$.\\

\noindent According to Bell, \cite{Bell}, a localic topos, as the one of a DNN, is naturally equivalent to the category ${\sf Set}_\Omega$ of $\Omega$-sets, i.e. sets equipped with fuzzy identities
with values in $\Omega$. We now give a direct explicit construction of this equivalence, because it offers a view of the relation between the network layers directly connected to the intuitionist logic of the topos.\\
\noindent Let us mention the PhD thesis of Johan Lindberg \cite[part III]{Lindberg-phd}, developing this point of view, and studying in details the naturalness of the geometric morphism of topos induced by a morphism of locale.
\\

\begin{defn}
	On the poset $\left(\Omega,\leq\right)$, the canonical Grothendieck topology $K$ is defined by the coverings by open subsets of the open sets.\\
	In the localic case, where we are, the topos is isomorphic to the Grothendieck topos $\mathcal{E}=Sh(\Omega,K)$.\\
	We assume that this is the case in the following exposition.
\end{defn}
\noindent In the particular case $\mathcal{E}=\mathbf{X}^{\wedge}$, where $\mathbf{X}$ is a poset, $\Omega$ is the poset of lower Alexandrov open sets and
the isomorphism with $Sh(\Omega,K)$ is given explicitly by proposition \ref{prop:alexandrov}.\\

Let $X$ be an object of $\mathcal{E}$; we associate to it the set $X^{\boldsymbol{\Omega}}$ of natural
transformation from $\boldsymbol{\Omega}$ to $X$. For two elements $x,y$ of $X^{\boldsymbol{\Omega}}$, we define $\delta_X(x,y)\in \Omega$ as the
largest open set over which $x$ and $y$ coincide.\\
An element $u$ of $X^{\boldsymbol{\Omega}}$ is nothing else than a sub-singleton in $X$, its domain $\omega_u$ is $\delta_X(u,u)$.
In other terms, in the localic case, $u$ is a section of the presheaf $X$ over an open subset $\omega_u$ in $\Omega$.\\
Then, if $u$, $v$ and $w$ are three elements of $X^{\boldsymbol{\Omega}}$, the maximal open set where $u=w$ contains the intersection
of the open sets where $u=v$ and $v=w$. Thus $X^{\boldsymbol{\Omega}}$ is a set over $\Omega$.\\
\indent In the same manner, suppose we have a morphism $f:X\rightarrow Y$ in $\mathcal{E}$, if we take $x\in  X^{\boldsymbol{\Omega}}$ and $y\in  Y^{\boldsymbol{\Omega}}$
we define $f(x,y)\in \Omega$ as the largest open set of $\boldsymbol{X}$ where $y$ coincides with $f_\star x$. This gives a morphism of $\Omega$-sets.\\
\noindent All that defines a functor from $\mathcal{E}$ to $Set_\Omega$.\\

\indent A canonical functor from ${\sf Set}_\Omega$ to $\mathcal{E}$ is given by a similar construction:\\
for $U\in \Omega$, $\Omega_U=\boldsymbol{\Omega}(U)$ is an $\Omega$-set, with the fuzzy equality defined by the internal equality
\begin{equation}
\delta_U(\alpha,\alpha')=(\alpha\asymp \alpha'),
\end{equation}
that is the restriction of the characteristic map of the diagonal subset:
$\Delta:\boldsymbol{\Omega}\hookrightarrow \boldsymbol{\Omega}\times\boldsymbol{\Omega}$. The set $\Omega_U$ can be identified with
the $\Omega$-set $U^{\boldsymbol{\Omega}}$ associated to the Yoneda presheaf defined by $U$. More concretely, an element $\omega$ of $\Omega_U$
is an open subset of $U$, and its domain $\delta(\omega,\omega)$ is $\omega$ itself.\\
Now, for any $\Omega$-set $(X,\delta)$, and for any element $U\in \Omega$, we define the set (see \eqref{omegamap},\eqref{domaine}), 
\begin{equation}
X_\Omega(U)={\sf Hom}_{{\sf Set}_\Omega}(\Omega_U,X)=\{f:\Omega_U\times X\rightarrow \Omega\}.
\end{equation}
In what follows, we sometimes write $X_\Omega=X$, when the notation does not introduce too much ambiguity.\\
If $V\leq W$, the formula $f(\omega_V,\omega_W)=\omega_V\cap \omega_W$ defines a $\Omega$-morphism from $\Omega_V$ to $\Omega_W$, which gives
a map from $X(W)$ to $X(V)$. Then $X_\Omega$ is a presheaf over $\Omega$.\\
\begin{prop}
	A morphism of $\Omega$-set $f:X\times Y\rightarrow \Omega$ gives by composition a natural transformation $f_\Omega:X_\Omega\rightarrow Y_\Omega$
	of presheaves over $\Omega$.
\end{prop}

\begin{proof}
	Consider $f_U\in X(U)$; the axiom \eqref{domaine} tells that for every open set $V\subset U$, the family of open sets $f_U(V,u);u\in X$ is an open covering $f_U^{V}$
	of $V$.\\
	The first axiom of \eqref{omegamap}, which represents the substitution of the first variable, tells that on $V\cap W$ the two coverings $f_U^{V}$
	and $f_U^{W}$ coincide. Therefore, for every $u\in X$, the value $f_U(u)=f(U,u)$ of $f_U$ on the maximal element $U$ determines by intersection
	all the values $f_U(V,u)$ for $V\subset U$.\\
	\indent For $f_U\in X(U)$ and $V\leq U$, the functorial image $f_V$ of $f_U$ in $X(V)$ is the trace on $V$:
	\begin{equation}\label{heredity}
		\forall u\in X,\quad f_V(u)=\rho_{VU}f_U(u)=f_U(u)\cap V.
	\end{equation}
	
	\noindent This implies that $X_{\Omega}$ is a sheaf: consider a covering $\mathcal{U}$ of $U$, $(1)$  for two elements $f_U,g_U$ of $X(U)$,
	if the families of restrictions $f_U \cap V; V\in \mathcal{U}$, $g_U\cap V;V\in \mathcal{U}$, then $f_U=g_U$; $(2)$ if a family of coverings $f_V; V\in \mathcal{U}$
	is given, such that for any intersection $W=V\cap V'$, the restriction $f_V|W$ and $f_{V'}|W$ coincide, as open coverings, we can define an element $f_U$ of $X(U)$
	by taking for each $u\in X$ the open set $f_U(u)$ which is the reunion of all the $f_V(u)$ for $V\in \mathcal{U}$. The union of the sets $f_V(u)$ over $u\in X$
	is $V$, and the union of the sets $V$ is $U$, then the union of the $f_U(u)$ when $u$ describes $X$ is $U$.
\end{proof}

\noindent The second axiom of substitution tells that for any $u,v\in X$, $\delta(u,v)\cap f(u)=\delta(u,v)\cap f(v)$. The third axiom of \eqref{omegamap}, which expresses the functional
character of $f$, tells that for any $u,v\in X$, $\delta(u,v)\supseteq f(u)\cap f(v)$.\\
\indent Consequently, the elements of $X(\alpha)$ can be identified with the open coverings
$f_U(u);u\in X$ of the open set $U$, such that, in $\Omega$, we have
\begin{equation}\label{coveringlogic}
\forall u,v \in X, \quad f_U(u)\cap f_U(v)\subseteq \delta(u,v)\subseteq (f_U(u)\Leftrightarrow f_U(v));
\end{equation}
where $\Leftrightarrow$ denotes the internal equivalence $\Leftarrow\wedge\Rightarrow$ in $\Omega$.\\
Remind that $\alpha\Rightarrow \beta$ is the largest element $\gamma \in \Omega$ such that $\gamma\wedge\alpha\leq\beta$, and in our topological setting $\Omega=\mathcal{U}(\mathbf{X})$ it is the union of the open sets $V$ such that $V\cap \alpha\subseteq\beta$, therefore $f(u)\Leftrightarrow f(v)$
is the union of the elements $V$ of $\Omega$ such that $V\cap f(u)=V\cap f(v)$.\\

\begin{prop}\label{prop:etale}
	Let $\Omega$ be any complete Heyting algebra (i.e. a locale); the two functors $F:(X,\delta)\mapsto (U\mapsto X(U)={\sf Hom}_\Omega(\Omega_U,X)$
	and $G:X\mapsto (X^{\boldsymbol{\Omega}},\delta_X)={\sf Hom}_\mathcal{E}(\boldsymbol{\Omega},\mathbf{X})$
	define an equivalence of category between $\emph{Set}_\Omega$ and $\mathcal{E}=Sh(\Omega, K)$.
\end{prop}
\begin{proof}
	The composition $F\circ G$ sends a sheaf $X(U);U\in \Omega$ to the sheaf $X^{\boldsymbol{\Omega}}(U);U\in \Omega$
	made by the open coverings of $U$ by sets indexed by the sub-singletons $u$ of $X$ satisfying the two inclusions \eqref{coveringlogic}.\\
	\indent Consider an element $s_U\in X(U)$, identified with a section of $X$ over $U$. For each sub-singleton $v\in X^{\boldsymbol{\Omega}}$,
	we define the open set $f(v)=f^{s}_U(v)$ by the largest open set in $U$ where $v=s_{U}$. As the sub-singletons generate $X$, this forms an open covering of $U$.
	It satisfies \eqref{coveringlogic} for any pair $(u,v)$: $\delta(u,v)$ is the largest open set where $u$ coincides with $v$, then the first inclusion is
	evident, for the second one, consider the intersection $\delta(u,v)\cap f(u)$, on it we have $u=v$ and $u=s$, then it is included in $\delta(u,v)\cap f(v)$. 
	If $V\subset U$ and $s_V=s_{U}|V$, the open covering of $V$ defined by $s_V$ is the trace of the open covering defined by $s_U$.\\
	Moreover, a morphism $\phi:X\rightarrow Y$ in $\mathcal{E}$ sends sub-singletons to sub-singletons and induces injections of the maximal domain of extension;
	therefore the above construction defines a natural transformation $\eta_{\mathcal{E}}$
	from $Id_{\mathcal{E}}$ to $F\circ G$.\\
	This transformation is invertible: take an element $f$ of $X^{\boldsymbol{\Omega}}(U)$, and for every $U\in \Omega$, consider
	the set $S(f,U)$ of sub-singletons $u$ of $X$ such that $f_U(u)\neq\emptyset$. If $u$ and $v$ belong to this set, the first inequality of \eqref{coveringlogic}
	implies that $u=v$ on the intersection $f_U(u)\cap f_U(v)$, then, by the sheaf property $3$, $S(f,U)$ defines a unique element $u_U\in X(U)$.\\
	\indent In the other direction, the composition $G\circ F$ associates to a $\Omega$-set $(X,\delta)$ the $\Omega$-set $(X^{\boldsymbol{\Omega}},\delta_{X,\boldsymbol{\Omega}})$ made by the sub-singletons of the presheaf $X_\Omega$, i.e. the families $(f,U)$ of compatible coverings $f_V(v),v\in X$ of $V; V\subset U$. We have $\delta((f,U),(f,U))=U$; therefore, for simplifying the notations, we denote the singleton by $f$, and $U$ is $\delta(f,f)$.\\
	We saw that, for two elements $f$, $(f'$, the open set $\delta(f,f')$ is the maximal open subset of
	$U\cap U'$ where the coverings $f_V(u)$ and $f'_V(u)$ coincide for every $u\in X$ and $V\subset U$.\\
	For a pair $(u,f)$, of $u\in X$ and $(f\in X^{\boldsymbol{\Omega}}$, we define $H(u,f)\in \Omega$ as the unions of the open sets $f_V(u)$, over $V\subset \delta(f,f)\cap \delta(u,u)$.\\
	The formula \eqref{heredity} implies that $H(u,f)$ is also the union of open sets $\alpha$ such that $\alpha\subset f_\alpha(u)$, i.e. $f_\alpha(u)=\alpha$.\\
	We verify that $H$ is a morphism of $\Omega$-sets: the first axiom
	\begin{equation}
		\delta(u,v)\wedge H(u,f)\leq H(v,f)
	\end{equation}
	results from
	\begin{equation}
		\delta(u,v)\wedge f_\alpha(u)\leq f_\alpha(v)
	\end{equation}
	for every $\alpha\in \Omega$.\\
	The second axiom
	\begin{equation}
		H(u,f)\wedge \delta(f,f')\leq H(u,f')
	\end{equation}
	comes from the definition of $\delta(f,f')$ as an open set where the induced coverings
	coincide.\\
	For the third axiom,
	\begin{equation}
		H(u,f)\wedge H(u,f')\leq \delta (f,f');
	\end{equation}
	if $\alpha$ is included in the intersection we have $f_\alpha(u)=\alpha=f'_\alpha(u)$, then $\alpha\leq\delta (f,f')$.\\
	From \eqref{coveringlogic}, we have $f_\alpha(u)\subset\delta(u,u)$, then
	\begin{equation}
		H(u,f)\subset \delta(u,u)
	\end{equation}
	And for every $\alpha\leq \delta(u,u)$, we can define a special covering $f^{u}_\alpha$ by
	\begin{equation}
		f^{u}_\alpha(u)=\alpha,\quad f^{u}_\alpha(v)=\alpha\wedge \delta(u,v);
	\end{equation}
	it satisfies \eqref{coveringlogic}. Then
	\begin{equation}
		\delta(u,u)=\bigvee_{f\in X(U)}H(u,f)
	\end{equation}
	The $\Omega$-map $H$ is natural in $X\in {\sf Set}_\Omega$. To terminate the proof of proposition \ref{prop:etale}, we have to show
	that $H$ is invertible, that is to find a $\Omega$-map $H':X_\Omega^{\boldsymbol{\Omega}}\times X\rightarrow\Omega$, such that
	$H'\circ H=\delta_X$ and $H\circ H'=\delta_{X_\Omega, \boldsymbol{\Omega}}$. We note the first fuzzy identity by $\delta$
	and the second one by $\delta'$.\\
	\indent In fact $H'(f,u)=H(u,f)$ works; in other terms $H$ is an involution of $\Omega$-sets. let us verify this fact:\\
	by definition of the composition
	\begin{equation}
		H'\circ H(u,v)=\bigvee_fH(u,f)\wedge H'(f,v)
	\end{equation}
	is the reunion of the $\alpha\in \Omega$ such that there exists $f$ with $\alpha=f_\alpha(u)=f_\alpha(v)$, then by the first inequality in \eqref{coveringlogic}
	it is included in $\delta(u,v)$. Now consider $\alpha\leq \delta(u,v)\subseteq \delta(u,u)$, and define a covering of $\alpha$ by $f_\alpha^{u}(w)=\alpha\cap \delta(u,w)$
	for any $w\in X$, this gives $\alpha\leq f_\alpha^{u}(v)$ then $\alpha\subseteq H(v,f^{u})$, then $\alpha\subset H(u,f^{u})\wedge H'(f^{u},v)$.\\
	On the other side,
	\begin{equation}
		H\circ H'(g,f)=\bigvee_u H(g,u)\wedge H(u,f),
	\end{equation}
	is the reunion of the $\alpha\in \Omega$ such that there exists $u$ with $\alpha=f_\alpha(u)=g_\alpha(u)$. In this case, we consider
	the set $S(f,\alpha)$ of elements $v\in X$ such that $f_\alpha(v)\neq\emptyset$. If $v$ and $w$ belong to this set, the first inequality of \eqref{coveringlogic}
	implies that $v=w$ on the intersection $f_\alpha(v)\cap f_z(w)$, then, by the sheaf property, $S(f,\alpha)$ defines a unique element $u_\alpha\in X$. This element
	must be equal to $u$. The same thing being true for $g$, this implies that $f_\alpha(v)=g_\alpha(v)$ for all the elements $v$ of $X$, some of them giving $\alpha$
	the other giving the empty set. Consequently, $H\circ H'(g,f)\subseteq \delta'(f,g)$.\\
	The other inclusion $\delta'(f,g)\subseteq H\circ H'(g,f)$ being obvious, this terminates the proof of the proposition.
\end{proof}

This proposition generalizes to the localic Grothendieck topos the construction of the sheaf space
(espace étalé in French) associated to a usual topological sheaf. However the accent in $\Omega$-sets is put more on the gluing of sections
than on a well defined set of germs of sections, as in the sheaf space. In some sense, the more general $\Omega$-sets give
also a more global approach, as in the original case of Riemann surfaces. Replacing a dynamics for instance by its solutions, pairs of domains and functions on them, with the relation of prolongation over sub-domains. This seems to be well adapted to the understanding of a DNN, on sub-trees of its architectural graph $\boldsymbol{\Gamma}$.\\

The localic Grothendieck topos $\mathcal{E}_\Omega$ are the "elementary topos" which are sub-extensional (generated by sub-singletons) and defined over ${\sf Set}$ \cite[p. 207]{Bell}. \\
Particular cases are characterized by special properties of the lattice structure of the locale $\Omega$ \cite[pp. 208-210]{Bell}:
\begin{itemize}
	\item we say that two elements $U,V$ in $\Omega$ are separated by another element $\alpha\in \Omega$ when one of them is smaller than $\alpha$ but not the other one.
\end{itemize}

\noindent $\mathcal{E}_\Omega$ is the topos of sheaves over a topological space $\mathbf{X}$ if and only if $\Omega$ is \emph{spatial}, which means by definition, that any pair of elements of $\Omega$
is separated by a \emph{large} element, i.e. an element $\alpha$ such that $\beta\wedge\gamma\leq\alpha$ implies $\beta\leq\alpha$ or $\gamma\leq\alpha$.\\
\noindent Moreover, in this case, $\Omega$ is the poset of open sets of $\mathbf{X}$, and the large elements are the complement of the closures of points of $\mathbf{X}$.\\
The topological space is not unique, only the \emph{sober} quotient is unique. A topological space is sober when every irreducible closed set is the closure of one and only one point.\\

$\mathcal{E}_\Omega$ is the topos of presheaves over a poset $\mathcal{C}_\mathbf{X}$ if and only if $\Omega$ is an Alexandrov lattice, i.e. any pair of elements of $\Omega$
is separated by a \emph{huge} (very large) element, i.e. an element $\alpha$ such that $\bigwedge_{i\in I}\beta_i\leq\alpha$ implies that $\exists i\in I, \beta_i\leq\alpha$.\\
In this case $\Omega$ is the set of lower open sets for the Alexandrov topology on the poset.\\
If $\Omega$ is finite, large and huge coincide, then spatial is the same as Alexandrov.\\

\section{Topos of DNNs and spectra of commutative rings}\label{app:top-of-dnn}

A finite poset with the Alexandrov topology is sober. This is a particular case of 
Scott's topology. Then it is also a particular case of spectral
spaces \cite{10.2307/1995344}, \cite{priestley94}, that are (prime) spectra of a commutative ring with the Zariski topology.\\

\noindent From the point of view of spectrum, a tree in the direction described in theorem \ref{thm:dnn}, corresponds to a ring with a unique maximal ideal, i.e., by definition a local ring.\\
The minimal points correspond to minimal primes. The gluing of two posets along an ending vertex corresponds to the fiber product of the two rings over the
simple ring with only one prime ideal \cite{Tedd-thesis}. A ring with a unique prime ideal is a field, in this case the maximal ideal is $\{0\}$.
This gives the following result:
\begin{prop}
	The canonical (i.e. sober) topological space of a $DNN$ is the Zariski spectrum of a commutative ring which is the
	fiber product of a finite set of local rings over a product of fields.
\end{prop}

\noindent The construction of a local rings for a given finite poset can be made by recurrence over the number of primes, by successive application of two
operations: gluing a poset along an open subset of another poset,  and joining several maximal points; this method is due to Lewis 1973 \cite{Tedd-thesis}.\\
\begin{exs}\normalfont 
	\begin{enumerate}[label=\Roman* $-$]
		\item The topos of Shadoks \cite{proute-shadok} corresponds to the poset $\beta< \alpha$ with two points; this is the spectrum of any \emph{discrete valuation ring} only containing the ideal $\{0\}$
		and a non-zero maximal ideal. Such a ring is the subset of a commutative field $\mathbb{K}$ with a valuation $v$ valued in $\mathbb{Z}$, defined by $\{a\in \mathbb{K}| v(a)\geq0\}$.
		An example is $\mathbb{K}((x))$ the field of fractions of the formal series $\mathbb{K}[[x]]$, with the valuation given by the smallest power of $x$ (and $\infty$) for $a=0$.
		The valuation ring is $\mathbb{K}[[x]]$, also noted $K\{x\}$, its maximal ideal is $\mathfrak{m}_x=x\mathbb{K}[[x]]$.
		\item Consider the poset of length three: $\gamma<\beta<\alpha$. Apply the gluing construction to the ring $A=\mathbb{K}\{x\}$ embedded in $\mathbb{K}((x))$ and the
		ring $B=\mathbb{K}((x))\{y\}$ projecting to $\mathbb{K}((x))$; this gives the following local ring:
		\begin{equation}
			D=\mathbb{K}\{x\}\times_{\mathbb{K}((x))}\mathbb{K}((x))\{y\}\cong \{d=a+yb|a\in A, b\in B\}\subset B.
		\end{equation}
		The sequence of prime ideals is
		\begin{equation}
			\{0\}\subset yB\subset \mathfrak{m}_x+yB.
		\end{equation}
	\item Continuing this process, we get a natural local ring which spectral space 
	is the chain of length $n+1$, $\alpha_n< ...< \alpha_0$ or simplest $DNN$s.
	There is one such ring for any commutative field $\mathbb{K}$:
	\begin{equation}
		\begin{split}
			D_n=\{d=a_n+x_{n-1}b_{n-1}+...+x_1b_1\in \mathbb{K}((x_1,x_2,...,x_n))|\\a_n\in \mathbb{K}\{x_n\}, b_{n-1}\in \mathbb{K}((x_{n}))\{x_{n-1}\},...,b_1\in \mathbb{K}((x_2,...,x_{n}))\{x_1\}.
		\end{split}
	\end{equation}
	The sequence of prime ideals is
	\begin{equation}
		\begin{split}
			\{0\}\subset x_1\mathbb{K}((x_2,...,x_{n}))\{x_1\}\subset \\x_1\mathbb{K}((x_2,...,x_{n}))\{x_1\}+x_{2}\mathbb{K}((x_3,...,x_{n}))\{x_{2}\}\subset\\...
			\subset x_1\mathbb{K}((x_2,...,x_{n}))\{x_1\}+...+ x_n\mathbb{K}\{x_n\}.
		\end{split}
	\end{equation}
	\end{enumerate}
\end{exs}

\section{Classifying objects of groupoids}\label{app:classifying}
\begin{prop}
	There exists an equivalence of category between any connected groupoid $\mathcal{G}$ and its fundamental group $G$.
\end{prop}
\begin{proof}
	let us choose an object $O$ in $\mathcal{G}$, the group $G$ is represented by the group $G_O$ of automorphisms of $O$.
	The inclusion gives a natural functor $J:G\rightarrow \mathcal{G}$ which is full and faithful. In the other direction, we choose for any
	object $x$ of $\mathcal{G}$, a morphism (path) $\gamma_x$ from $x$ to $O$, we choose $\gamma_O=id_O$, and we define a functor $R$ from $\mathcal{G}$ to $G$ by sending
	any object to $O$ and any arrow $\gamma:x \rightarrow y$ to the endomorphism $\gamma_y\circ\gamma\circ \gamma_x^{-1}$ of $O$. The rule of
	composition follows by cancellation. A natural isomorphism between $R\circ J$ and $Id_G$ is the identity. A natural transformation $T$ from $J\circ R$
	to $Id_\mathcal{G}$ is given by sending $x\in \mathcal{G}$ to $\gamma_x$, which is invertible for each $x$. The fact that it is natural results
	from the definition of $R$: for every morphism $\gamma:x \rightarrow y$, we have
	\begin{equation}
		T(y)\circ Id(\gamma)=\gamma_y\circ\gamma=(\gamma_y\circ\gamma)\circ\gamma_x^{-1}\circ\gamma_x=JR(\gamma)\circ T(x).
	\end{equation}
	
	\noindent What is not natural in general (except if $\mathcal{G}=G=\{1\}$) is the choice of $R$. This makes groupoids strictly richer than groups, but not from the
	point of view of homotopy equivalence. Every functor between two groupoids that induces an isomorphism of $\pi_0$, the set of connected
	components, and of $\pi_1$, the fundamental group, is an equivalence of category.
\end{proof}

One manner to present the topos $\mathcal{E}=\mathcal{E}_\mathcal{G}$ of presheaves over a small groupoid $\mathcal{G}$ (up to category equivalence)
is to decompose $\mathcal{G}$ in connected components $\mathcal{G}_a;a\in A$, then $\mathcal{E}$ will be product of the topos $\mathcal{E}_a;a\in A$
of presheaves over each component. For each $a\in A$, the topos $\mathcal{E}_a$ is the  category of $G_a$-sets,
where $G_a$ denotes the group of auto-morphisms of any object in $\mathcal{G}_a$.\\
The classifying object $\Omega=\Omega_\mathcal{G}$ is the boolean algebra of the subsets of $A$.\\
\indent In the applications, we are frequently interested by the subobjects of a fixed object $X=\{X_a;a\in A\}$.
The algebra of subobjects $\Omega^{X}$, has for elements all the subsets that are preserved by $G_a$ for each component $a\in A$ independently.\\
Thus we can consider what happens for a given $a$. Every element $Y_a\in \Omega^{X_a}$  has a complement $Y_a^{c}=\neg Y_a$, which is
also invariant by $G_a$, and we have $\neg\neg=Id$. Here the relation of negation $\leq$ is the set-theoretic one. It is also true for the operations $\wedge$
(intersection of sets), $\vee$ (union of sets), and the internal implication $p\Rightarrow q$, which is defined in this case by $(p\wedge q)\vee\neg p$.\\
\indent All the elements $Y_a$ of $\Omega^{X_a}$ are reunions of orbits $Z_i;i \in K(X_a)$ of the group $G_a$ in the $G_a$-set $X_a$.
On each orbit, $G_a$ acts transitively.\\
Each subobject of $X$ is a product of subobjects of the $X_a$ for $a\in A$. The product over $a$ of the $K(X_a)$ is a set $K=K(X)$. \\
\indent The algebra $\Omega^{X}$ is the Boolean algebra of the subsets of the set of elements  $\{Z_i;i\in K\}$, that we can note simply $\Omega_K$.\\
\indent The arrows in this category, $p\rightarrow q$, correspond to the pre-order $\leq$, or equivalently to the inclusion of sets,
and can be understood as implication of propositions. This is the implication in the external sense, if $p$ is true then $q$ is true,
not in the internal sense $q^{p}$, also denoted $p\Rightarrow q$, that is also the maximal element $x$ such that $x\wedge p\leq q$).\\
On this category, there exists a natural Grothendieck topology, named the canonical topology, which is the largest (or the finest) Grothendieck topology
such that, for any $p\in \Omega$, the presheaf $x\mapsto {\sf Hom}(x,p)$ is a sheaf. For any $p\in \Omega$, the set of coverings $J_K(p)$ is the set of collections
of subsets $q$ of
$p$ whose reunion is $p$. In particular $J_K(\emptyset)$ contains the empty family; this is a singleton.\\
\begin{prop}
	The topos $\mathcal{E}$ is isomorphic to the topos ${\sf Sh}(\Omega;K)$ of sheaves for this topology
	$J_K$ (see for instance {\rm Bell, Toposes and local set theories \cite{Bell}}).
\end{prop}
\begin{proof}
	For all $p$, any covering of $p$ has for refinement the covering made by the disjoint singletons $Z_i$ that belong to $p$,
	seen as a set; then, for every sheaf $F$ over $\Omega$,
	the restriction maps give a canonical isomorphism from $F(p)$ with the product of the sets $F(Z_i)$ over $p$ itself.\\
	In particular, any sheaf has for value in $\bot=\emptyset$ a singleton.
\end{proof}

\section{Non-Boolean information functions}\label{app:non-boolean}

\noindent This is the case of chains and injective presheaves on them.\\

\noindent The site $S_n$ is the poset $0\rightarrow 1\rightarrow ...\rightarrow n$. A finite  object $E$ is chosen
in the topos of presheaves $S_n^{\wedge}$, such  that each map $E_i\rightarrow E_{i-1}$ is an injection,
and we consider the Heyting algebra $\Omega^{E}$, that is made
by the subobjects of $E$. The inclusion, the intersection and the union of subobjects are evident.
The
only non-trivial internal operations are the exponential, or internal implication $Q\Rightarrow T$, and the negation $\neg Q$, that is a particular case $Q\Rightarrow\emptyset$.\\
\begin{lem}\label{lem:implication}
	Let $T_n\subset T_{n-1}\subset ...\subset T_0$ and $Q_n\subset Q_{n-1} \subset... \subset Q_0$
	be two elements of $\Omega^{E}$, then the implication $U=(Q\Rightarrow T)$ is inductively defined by the following formulas:
	\begin{align*}
		U_0&=T_0\vee (E_0\backslash Q_0),\\
		U_1&=U_0\wedge (T_1\vee (E_1\backslash Q_1),\\
		...&\\
		U_k&=U_{k-1}\wedge (T_k\vee (E_k\backslash Q_k),\\
		...&
	\end{align*}
\end{lem}
\begin{proof}
	By recurrence. For $n=0$ this is the well known boolean formula. Let us assume
	the result for $n=N-1$, and prove it for $n=N$. The set $U_N$ must belong to $U_{N-1}$ and must be the union
	of all the sets $V\subset E_N\cap U_{N-1}$ such that $V\wedge Q_N\subset T_N$, then it is the union of
	$T_N\cap U_{N-1}$ and $(E_N\backslash Q_N)\cap U_{N-1}$.
	
	\noindent In particular the complement $\neg Q$ is made by the sequence
\begin{equation}
\bigcap_{k=0}^{n} (E_k\backslash Q_k)\subset \bigcap_{k=0}^{n-1} (E_k\backslash Q_k)\subset ...\subset E_0\backslash Q_0.
\end{equation}
\end{proof}
\begin{defn}
	We choose freely a strictly positive function $\mu$ on $E_0$; for any subset $F$ of $E_0$,
	we note $\mu(F)$ the sum of the numbers $\mu(x)$ for $x\in F$.\\
	In practice $\mu$ is the constant function equal to $1$, or to $|F|^{-1}$.
\end{defn}
\begin{defn}
	Consider a strictly decreasing sequence $[\delta]$ of strictly positive real numbers $\delta_0> \delta_1> ...> \delta_n$;
	the function $\psi_{\delta}:\Omega^{E}\rightarrow\mathbb{R}$ is defined by the formula
	\begin{equation}
		\psi_\delta(T_n\subset T_{n-1}\subset ...\subset T_0)=\Sigma_{k=0}^{n}\delta_k\mu(T_k).
	\end{equation}
\end{defn}
\begin{lem}\label{lem:increase}
	The function $\psi_\delta$ is strictly increasing.
\end{lem}

\noindent This is because index by index, $T'_k$ contains $T_k$.\\
\begin{defn}
	A function $\varphi:\Omega^{E}\rightarrow\mathbb{R}$ is \emph{concave} ({\em resp.}
	strictly concave), if for any pair of
	subsets $T\leq T'$ and any proposition $Q$, the following expression is positive ({\em resp.} strictly positive),
	\begin{equation}
		\Delta\varphi(Q;T,T')=\varphi(Q\Rightarrow T)-\varphi(T)-\varphi(Q\Rightarrow T')+\varphi(T').
	\end{equation}
\end{defn}

\noindent \textbf{Hypothesis on $\delta$}: for each $k$, $n\geq k\geq 0$, we assume that $\delta_k> \delta_{k+1}+...+\delta_n$.\\
This hypothesis is satisfied for instance for $\delta_0=1, \delta_1=1/2, ..., \delta_k=1/2^{k},...$.\\
\begin{prop}\label{prop:concavity}
	Under this hypothesis, the function $\psi_\delta$ is concave.
\end{prop}
\begin{proof}
	Let $T\leq T'$ in $\Omega^{E}$. We define inductively an increasing sequence $T^{(k)}$ of $S_n$-sets by taking
	$T^{(0)}=T$ and, for $k> 0$, $T_j^{(k)}$ equal to $T_j^{(k-1)}$ for $j< k$ or $j> k$, but equal to $T'_j$ for $j=k$. In other terms,
	the sequence is formed by enlarging $T_k$ to $T'_k$, index after index. Let us prove
	that $\Delta\psi_\delta(Q;T^{(k-1)},T^{(k)})$ is positive, and strictly positive when at the index $k$, $T_k$ is strictly included in $T'_k$.
	The theorem follows by telescopic cancellations.\\
	The only difference between $T^{(k-1)}$ and $T^{(k)}$ is the enlargement of $T_k$ to $T'_k$, and this generates a difference
	between $T_j^{(k-1)}|Q$ and $T_j^{(k)}|Q$ only for the indices $j> k$. This allows us to simplify the notations by assuming $k=0$.\\
	The contribution of the index $0$ to the double difference $\Delta\psi_\delta$ is the difference between the sum of $\delta_0\mu$ over the points
	in $E_0\backslash Q_0$ that do not belong to $T_0$ and the sum of $\delta_0\mu$ over the points
	in $E_0\backslash Q_0$ that do not belong to $T'_0$, then it is the sum of $\delta0\mu$ over the points
	in $E_0\backslash Q_0$ that belong to $T'_0\backslash T_0$.\\
	As in lemma \ref{lem:implication}, let us write $U_0=T_0\vee(E_0\backslash Q_0)$ and $U'_0=T'_0\vee(E_0\backslash Q_0)$. And for $k\geq 1$,
	let us write $V_k=T_k\vee(E_k\backslash Q_k)$, and $W_k=V_1\cap ...\cap V_k$.\\
	From the lemma 1, the contribution of the index $1$ to the double difference $\Delta\psi_\delta$, is the simple difference between
	the sum of $\delta_k\mu$ over the points in $U_0\cap W_k$ and its sum over the points in $U'_0\cap W_k$, then it is equal to the
	opposite of the sum
	of $\delta_k\mu$ over the points in $(T'_0\backslash T_0)\cap (E_0\backslash Q_0)\cap W_k$. The hypothesis on the sequence $\delta$
	implies that the sum over $k$ of these sums is smaller than the difference given by the index $0$.
\end{proof}

\begin{rmk*}
	\normalfont In general the function $\psi_\delta$, whatever being the sequence $\delta$, is not strictly concave,
	because it can happen that $T'_0$ is strictly larger than $T_0$, and the intersection of $T'_0\backslash T_0$ with
	$E_0\backslash Q_0$ is empty. Therefore, to get a strictly concave function, we take the logarithm, or another function from
	$\mathbb{R}_+^{\star }$ to $\mathbb{R}$ that transforms strictly positive strictly increasing concave functions to strictly increasing strictly concave functions.\\
	This property for the logarithm comes from the formulas
	\begin{equation}
		(\ln \varphi)"=[\frac{\varphi'}{\varphi}]'=\frac{\varphi\varphi"-(\varphi')^{2}}{\varphi^{2}}< 0.
	\end{equation}
\end{rmk*}

\noindent In what follows we take $\psi=\ln \psi_\delta$ as the fundamental function of precision.\\
By normalizing $\mu$ and taking $\delta_0=1$, we get $0<\psi_\delta\leq 1$, $-\infty< \psi\leq 0$.\\
\begin{rmk*}
	\normalfont Lemmas \ref{lem:implication}, \ref{lem:increase} and proposition \ref{prop:concavity} can easily be extended to the case where the basic site $\mathcal{S}$ is
	a rooted (inverse) tree, i.e. the poset that comes from an oriented graph with several initial vertices and a unique
	terminal vertex. The computation with intersections works in the same manner. The hypothesis on $\delta$ concerns
	only the descending branches to the terminal vertex.\\
	\indent Now, remember that the poset of a $DNN$ is obtained by gluing such trees on some of their initial vertices,
	interpreted as tips (of forks) or output layers. The maximal points correspond to tanks (of forks) of input layers.
	Therefore it is natural to expect that the existence of $\psi$ holds true for any site of a $DNN$.
\end{rmk*}

\section{Closer to natural languages: linear semantic information}\label{app:nat-lang}

Several attempts were made by logicians and computer scientists, since Frege and Russel, Tarski and Carnap, to approach the properties of
human natural languages by formal languages and processes. In particular, a computational grammar was proposed by Lambek \cite{Lambek58}: a syntactic category
is defined with sentences as objects and applications of grammatical rules as arrows, a second category is defined, that contains products and
exponentials, for instance a topos, and semantic is seen as some functor from the first category to the second one.
This is the first place where semantic is defined as interpretations of types and propositions in a topos. Precursors of the kind of grammar considered by Lambek
were Adjukiewicz in 1935 \cite{Adjukiewicz} and Bar-Hillel in 1953 \cite{Bar-Hillel-53}.\\
\indent Then a decisive contribution was made by Montague in 1970, \cite{Montague-grammar}, who developed in particular a formal treatment of pieces of English \cite{10.2307/4177871}. Also in this approach, semantics appears as a transformation from a syntactic algebraic structure, having lexis and multiple operations,
to a coarser structure. In the nineties
mathematicians and linguists observed that the categorical point of view, as in Lambek, gives a good framework for developing further Montague's theory \cite{10.2307/30226424}.\\
\indent The next step used intensional type theories, like Martin-Löf's theory \cite{ML}, named modern TT by Luo \cite{Luo-MTT}, or rich TT
by Cooper et al. \cite{Cooper2015ProbabilisticTT}. New types were introduced, corresponding to the many structural notions of linguistic,
e.g. noun, verb, adjective, and so on.
Also modalities like interrogative, performative, can be introduced (see Brunot \cite{brunot1936} for the
complexity of the enterprise in French). Recent experiment with
programming languages have shown that many properties of languages can be captured by extending TT. For instance, in Martin-Löf TT
it is possible to construct ZFT theories but also alternative Non-well-founded set theories,
like in \cite{aczel88}, taking into account paradoxical vicious circles as natural languages do \cite{10.2307/2275015}.
Even more powerful is the
homotopical type theory (HoTT) of Voevodski, Awodey, Kapulkin, Lumsdaine, Shulman, ..., \cite{KLV}. Also see Gylterud and Bonnevier \cite{gylterud2020nonwellfounded} 
for the inclusion of non-well-founded sets theories.\\
\indent These formal theories do not give a true definition of what is \emph{meaning}, (see the fundamental objections of Austin \cite{Austin1961-AUSPP}),
but they give an insight of the various ways the meanings can
be combined and how they are related to grammar, compatible with the intuition we have of human interpretations. We do not suggest that the
categorial presentation defines the natural languages, but here also we think that its capture something of toys languages, an some
languages games that can help the understanding of semantic functioning in networks, including properties of natural semantics of
human peoples.\\

\indent In what follows, we consider that a given category $\mathcal{A}$ represents the semantic for a given language, or some language
game \cite{Wittgenstein1953}, and reflects properties of a language, not the abstract rules, as in the algebra $\Omega^{\mathbb{L}}$ before. The objects
of $\mathcal{A}$ represent interpretations of sentences, or images, corresponding to the "types as propositions" (Curry-Howard) in a given grammar,
and its arrows represent the evocations, significations, or deductions, corresponding to
proofs or application of rules in grammar. Oriented cycles are \emph{a priori} admitted.\\
\indent We simply assume that $\mathcal{A}$ is a \emph{closed monoidal category} \cite{Eilenberg-Kelly}
that connects with linear logic and linear type theory as in
Mellies, "Categorical Semantics of Linear Logic"  \cite{Mellis2009CATEGORICALSO}.\\
In such a category, a bifunctor $(X,Y)\mapsto X\otimes Y$ is given, that is associative up to natural transformation, with a neutral element $\star$ also up to linear  transformation,
satisfying conditions of coherence. This product representing aggregation of sentences. Moreover there exists classifiers objects of morphisms, i.e. objects $A^{Y}$ defined
for any pair of objects $A,Y$, such that for any $X$, there exist natural isomorphisms
\begin{equation}
{\sf Hom}(X\otimes Y,A)\simeq {\sf Hom}(X,A^{Y}).
\end{equation}
The functor $X\mapsto X\otimes Y$ has for right-adjoint the functor $A\mapsto A^{Y}$.\\
\indent For us, this defines the semantic
conditioning, the effect on the interpretation $A$ that $Y$ is taken into account, when $A$ is evoked by a composition with $Y$.
Thus we also denote $A^{Y}$ by $Y\Rightarrow A$ or $A|Y$.\\
\indent When $A$ is given, and if $Y'\rightarrow Y$ we get $A|Y\rightarrow A|Y'$.\\
From $X\otimes \star\cong X$, it follows that canonically $A^{\star}\cong A$. We make the supplementary hypothesis that $\star$ is a final object,
then we get a canonical arrow $A\rightarrow A|Y$, for any object $Y$. This represents the internal constants.\\
\begin{rmk*}
	\normalfont In the product $X\otimes Y$, the ordering plays a role, and in linguistic, in the spirit of Montague, two functors can appear, the one we just said $Y\mapsto X\otimes Y$ and the other one $X\mapsto X\otimes Y$.
	If both have a left adjoint, we get two exponentials: $A^{Y}=A|Y$ and $^{X}A=X\ A$; the natural axiomatic becomes the \emph{bi-closed category} of Eilenberg and Kelly \cite{Eilenberg-Kelly}.
	Dougherty \cite{10.1305/ndjfl/1093634562} gave a clear exposition of part of the Lambek calculus
	in the Montague grammar in terms of this structure (same in \cite{Lambek1988grammar}). A theory of semantic information should benefit of this possibility, where composition depends on the
	ordering, but in what follows, to begin, we assume that $\mathcal{A}$ is \emph{symmetric}: there exist natural isomorphisms exchanging the two factors
	of the product.
\end{rmk*}

\indent All that can be localized in a context $\Gamma \in \mathcal{A}$ by considering the category $\Gamma\backslash\mathcal{A}$ of morphisms
$\Gamma\rightarrow A$, where $A$ describes $\mathcal{A}$, with morphisms given by the commutative triangles. For $\Gamma\rightarrow A$, and
$Y\in \mathcal{A}$, we get a morphism $\Gamma\rightarrow A|Y$ by composition with the canonical morphism $A\rightarrow A|Y$.
This extends the conditioning. We will discuss the existence of a restricted tensor product later on; it asks restrictions on $\Gamma$.\\

The analog of a theory, that we will also name theory here, is a collection $S$ of propositions $A$, that is stable by morphisms to the right, i.e. $A\in S$
and $A\rightarrow B$ implies $B\in S$. This can be seen as the consequences of a discourse. A theory $S'$ is said
weaker than a theory $S$ if it is contained in it, noted $S\leq S'$.
Then the analog of the conditioning of $S$ by $Y$
is the collection of the objects $A^{Y}$ for $A$ in $S$. The collection of theories is partially ordered.\\
We have $S|Y'\leq S|Y$ when there exists $Y'\rightarrow Y$.
In particular $S|Y\leq S$, as it was the case in simple type theory.\\
\indent When a context is given, it defines restricted theories, because it introduces a constraint of commutativity for $A\rightarrow B$,
to define a morphism from $\Gamma\rightarrow A$ to $\Gamma\rightarrow B$.\\

\noindent The monoidal category $\mathcal{A}$ acts on the set of functions from the theories to a fixed commutative group, for instance the real numbers.\\
We will later discuss how the context $\Gamma$ can be included in a category generalizing the category $\mathcal{D}$ of sections \ref{sec:sem-info} and \ref{sec:homotopy}, to obtain the analog of the classical ordinary logical case with the
propositions $P$ excluded. This needs a notion of negation, which, we will see, are many.\\
\begin{rmk*}
	\normalfont The model should be more complete if we introduce a syntactic type theory, as in Montague 1970, such that $\mathcal{A}$ is an
	interpretation of part of the types, compatible with products and exponentials. Then some of the arrows can interpret transformation
	rules in the grammar. The introduction of syntaxes will be necessary for communication between networks.
\end{rmk*}

Let us use the notations of chapter \ref{chap:stacks}. Between two layers $\alpha: U\rightarrow U'$ lifted by $h$ to $\mathcal{F}$, we assume the existence of a
functor $\pi_\star{\alpha, h}$ from $\mathcal{A}_{U,\xi}$ to $\mathcal{A}_{U',\xi'}$, with a  left adjoint $\pi^{\star}_{\alpha, h}$, such that
$\pi^{\star}\pi_\star={\sf Id}$, in such a manner that $\mathcal{A}$ becomes a pre-cosheaf over $\mathcal{F}$ for $\pi_\star$ and the sets of theories
$\Theta$ form a presheaf for $\pi^{\star}$.\\
\indent The information quantities are defined as before, by the natural bar-complex associated to the action of $\mathcal{A}$
on the pre-cosheaf $\Phi'$ of functions on the functor $\Theta$.\\
\indent The passage to a network gives a dynamic to the semantic, and the consideration of weights gives a model of learning semantic.
Even if they are caricature of the natural ones, we hope this will help to  capture some interesting aspects of them.\\

\noindent A big difference with the ordinary logical case, is the absence of "false", then in general, the absence of the negation operation. This can make the
cohomology of information non-trivial.\\

\noindent Another big difference is that the category $\mathcal{A}$ is not supposed to be
a poset, the sets ${\sf Hom}$ can be more complex than $\emptyset$ and $\star$, and they can contain isomorphisms. In particular loops can be present.\\

\noindent Consider for instance any function $\psi$ on the collection of theories; and suppose that there exist arrows from $A$ to $B$ and from $B$ to $A$;
then the function $\psi$ must take the same value on the theories generated by $A$ and $B$. This tells in particular that they contain the same information.\\

\noindent The homotopy construction of a bi-simplicial set $g\Theta$ can be made as before, representing the propagation feed-forward of theories and propagation
backward of the propositions, and the information can be defined by a natural increasing and concave map $F$ with values in a closed model category $\mathcal{M}$
of Quillen (see chapter \ref{chap:stacks}).\\
\indent The semantic functioning becomes a simplicial map $gS:g\mathbb{X}\rightarrow g\Theta$, and the semantic spaces are given by the composition $F\circ gS$.\\

Here is another interest of this generalization: we can assume that a measure of complexity $K$ is attributed to the objects, seen as expressions in a language,
and that this complexity is additive in the product, i.e. $K(X\otimes Y)=K(X)+K(Y)$, and related to the combinatorics of the syntax, and the complexity
of the lexicon, and the grammatical rules of formation. In this framework, we could compare the values of $K$ in the category, and define the \emph{compression}
as the ratio $F/K$ of information by complexity.\\
\begin{rmk*}
	\normalfont It is amazing and happy that the bar-complex for the information cocycles and the homotopy limit, can also be defined
	for the bi-closed generalization. The two exponentials $^{X}A$ and $A^{Y}$ an action of the monoid $\mathcal{A}$ to the right
	and to the left that commute on the functions of theories, and on the bi-simplicial set $g\Theta$. Then we can apply the work of MacLane, Beck on bi-modules
	and the work of Schulman on enriched categories.\\
	Taking into account the network, we get a tri-simplicial set $\Theta_\star^{\bullet\bullet}$ of information elements,
	or tensors, giving rise to a bi-simplicial space of histories of theories, with multiple left and right conditioning, $gI^{\bullet\bullet}$, that is the
	geometrical analog of the bar-complex of semantic information.
\end{rmk*}

\subsection*{Links with Linear Logic (intuitionist) and negations.}

The generalized framework corresponds to a fragment of an intuitionist Linear Logic (see Bierman and de Paiva \cite{10.2307/20016199}, Mellies \cite{Mellis2009CATEGORICALSO}). The arrows $A\rightarrow B$ in the category
are the expression of the assertions of consequence $A\vdash B$, and the product expresses the joint of the elements of the left members of consequences, in the sense that
a deduction $A_1,...,A_n\vdash B$
corresponds to an arrow $A_1\otimes ...\otimes A_n\rightarrow B$. There
is no necessarily a "or" for the right side, but there is an internal implication $A\multimap B$ which satisfies all the axioms of the above implication $A\Rightarrow B$,
right adjoint of the tensor product.
The existence of the final element corresponds to the existence of (multiplicative) 
truth ${\bf 1}=\star$. To be more complete, we should suppose that all the finite products exist
in the category $\mathcal{A}$. Then the (categorial) product of two corresponds to an additive disjunction $\oplus$, then a "or", that can generate the right side of
sequents $B_1,...,B_m$ in $A_1,...,A_n/B_1,...,B_m$; however, a neutral element for $\oplus$ could be absent,
even if it is always present in the full theory of Girard \cite{GIRARD19871}. No right adjoint is required for $\oplus$. And in what follows we do not assume the data $\oplus$.\\
\indent One of the main ideas of \cite{GIRARD19871} was to incorporate the fact that in real life the proposition $A$ that is used in a consequence $A\multimap B$ does not remain unchanged
after the event, however it is important to give a special status for propositions that continue to hold after the event. For that purpose
Girard introduced  an operator on the formulas, named a \emph{linear exponential}, and written $!$. It is named "of course" and has the meaning of a reaffirmation,
something stable. The functor $!$ is required to be naturally equivalent to $!!$, then a projector in the sense
of categories, such that, in a natural manner, the
objects $!A$ and the morphisms $!f$ between them satisfy the Gentzen rules of weakening and contraction, respectively $(\Gamma\vdash\Delta)/(\Gamma,!A\vdash \Delta)$
and $(\Gamma,A,A\vdash \Delta)/(\Gamma,A\vdash \Delta)$. (This corresponds to the traditional assertions $A\wedge B\leq  A$ and $A\leq A\wedge A$.)
Further axioms state, when translated in categorical terms, that $!$ is a monoidal functor equipped with two natural transformations
$\epsilon_A:!A\rightarrow A$ and $\delta_A:!A\rightarrow !!A$, that are monoidal transformations, satisfying the coherence rules of a comonad, and with natural
transformations $e_A:!A\rightarrow 1$ (useful when $1$ is not assumed final) and $d_A:!A\rightarrow !A\otimes !A$, that is a diagonal operator, also satisfying
coherence axioms telling that each $!A$ is a commutative comonoid, and each $!f$ a morphism of commutative comonoid. From all these axioms, it is proved that under $!$
the monoidal product becomes a usual categorial product in the category $!\mathcal{A}:=\mathcal{A}^{!}$,
\begin{equation}
!(A\otimes B)\cong !A\otimes !B \cong !(A \times B);
\end{equation}
and the category $\mathcal{A}^{!}$, named the  Kleisli category of $\left(\mathcal{A},!\right)$, is cartesian closed. More precisely, under $!$
the multiplicative exponential becomes the usual exponential:
\begin{equation}
!(A\multimap B)\cong !B^{!A}.
\end{equation}

\noindent Remind that a \emph{comonad} in a category is a functor $T$ of this category to itself, equipped with two natural transformations $T\rightarrow T\circ T$
and $\varepsilon:T\rightarrow {\sf Id}$, satisfying coassociativity and counity axioms. This the dual of a \emph{monad}, $T\circ T\rightarrow T$ and ${\sf Id}\rightarrow T$,
that is the generalization of monoids to categories.
The functor $!$ is an example of comonad \cite{maclane:71}.\\

The axioms of a closed symmetric monoidal category, plus the existence of finite products, plus the functor $!$, give the largest part of the Gentzen rules,
as they were generalized by Jean-Yves Girard in 1987 \cite{GIRARD19871}.\\
\begin{prop}\label{prop:lin-exp}
	The linear exponential $!$ allows to localize the product at a given proposition, in the sense that the slice category to the right $\Gamma |\mathcal{A}$
	is closed by products of linear exponential objects as soon as $\Gamma$ belongs to $\mathcal{A}^{!}$.
\end{prop}

\begin{proof}
	If we restrict us to the arrows $!\Gamma\rightarrow Q$, then the product $!\Gamma\rightarrow Q\otimes Q'$
	is obtained by composing the diagonal $d_{!\Gamma}:!\Gamma\rightarrow !\Gamma\otimes !\Gamma$ with the tensor product $!\Gamma\otimes !\Gamma\rightarrow Q\otimes Q'$.\\
	Its right adjoint is given by $!\Gamma\rightarrow (Q\multimap R)$, obtained by composing $!\Gamma\rightarrow Q$ with the natural map $Q\rightarrow Q|R$.
\end{proof}

To localize the theories themselves at $P$, for instance at a $!\Gamma$, we used, in the Heyting case, a notion of negation. To exclude a given proposition
was the only coherent choice from the point of view of information, and this was also in accord with the experiments of spontaneous logics
in small networks \cite{logic-DNN}.\\

In the initial work of Girard, negation was a fundamental operator, verifying the hypothesis of involution $\neg\neg={\sf Id}$, thus giving  a duality.
That explains that the initial theory is considered as a classical Linear Logic; it generalizes the usual Boolean logic in another direction than
intuitionism. In a linear intuitionist theory, the negation is not necessary, but it is also not forbidden, and axioms were discussed in the nineties.\\
\indent We follow here the exposition of Paul-André Melliès in \cite{Mellis2009CATEGORICALSO} and of his article with Nicolas Tabareau \cite{mellies:hal-00339154}.
The authors work directly in a monoidal category $\mathcal{A}$, without assuming that it is
closed, and define negation as a functor $\neg:\mathcal{A}\rightarrow \mathcal{A}^{\rm op}$, such that the opposite functor $\neg^{\rm op}$ from $\mathcal{A}^{\rm op}$
to $\mathcal{A}$, also denoted by $\neg$, is the left-adjoint of $\neg$, giving a unit $\eta:{\sf Id}\rightarrow \neg \neg$ and a counit $\epsilon:\neg\neg\rightarrow {\sf Id}$,
that are not equivalence in general. Then there exist for any objects $A,B$ a canonical bijection bijection between ${\sf Hom}_{\mathcal{A}}(\neg A,B)$ and ${\sf Hom}_{\mathcal{A}}(\neg B,A)$.
Note that in this case $\varepsilon$ and $\eta$ coincide, because the morphisms in $\mathcal{A}^{\rm op}$ are the morphisms in $\mathcal{A}$ written in the reverse order.\\
\noindent The double negation $T=\neg^{\rm op}\neg$ forms a monad whose $\eta$ is the unit; the multiplication $\mu:\neg\neg\neg\neg\rightarrow\neg\neg$ is obtained
by composing ${\sf Id}_\neg$ with $\neg(\eta)$, to the left or to the right, that is $\mu_A=\neg(\eta_A)\circ {\sf Id}_{\neg A}={\sf Id}_{\neg \neg\neg A}\circ \neg(\eta_A)$.\\
In theoretical computer science, $T$ is called the \emph{continuation monad}, and plays an important role in computation and games logics as in the works of Kock, Moggi, Mellies,
Tabareau.\\

\noindent In the case of the Heyting algebra of a topos (elementary), this continuation defines a topology, named after Lawvere and Tierney, which defines the unique
subtopos that is Boolean and dense (i.e. contains the initial object $\emptyset$ \cite{caramello_2012}).\\

\indent The second important axiom tells how the (multiplicative) product $\otimes$ is transformed : it is required that for any objects $B,C$ the object $\neg (B\otimes C)$ represents
the functor $A\mapsto {\sf Hom}(A\otimes B,\neg C)\cong {\sf Hom}(C,\neg(A\otimes B)$; that is
\begin{equation}
{\sf Hom}(A\otimes B,\neg C)\cong {\sf Hom}(A,\neg (B\otimes C)).
\end{equation}
This bijection being natural in the three argument and coherent with the associativity and unit for the
product $\otimes$.\\
For instance all the sets ${\sf Hom}(ABC,\neg D)$, ${\sf Hom}(AB,\neg(CD)$, ${\sf Hom}(A, \neg(BCD))$, are identified with ${\sf Hom}(ABCD,\neg {\bf 1})$.\\
Mellies and Tabareau \cite{mellies:hal-00339154} called such a structure a \emph{tensorial negation}, and named the monoidal category $\mathcal{A}$, equipped with $\neg$, a \emph{dialogue category}.\\

\noindent The special object $\neg {\boldsymbol 1}$ is canonically associated to the chosen negation; it is named the \emph{pole} and frequently denoted by $\bot$.
It has no reason in general to be an initial object of $\mathcal{A}$.\\

\noindent A monoidal structure of (multiplicative) disjunction is deduced from the tensor product by duality:
\begin{equation}
A\wp B=\neg(\neg A\otimes \neg B).
\end{equation}
Its neutral element is the pole of $\neg$.\\
This implies that  the notion of "or" is parameterized by the variety of negations, that we will see equivalent to $\mathcal{A}$ itself.\\

\noindent In the same manner an additive conjonction is defined by
\begin{equation}
A \& B=\neg(\neg A\oplus \neg B).
\end{equation}
Its neutral element is $\top=\neg \emptyset$, when an initial element $\emptyset$ exists, that is the additive "false".\\

\noindent An operator $?$ was introduced by Girard in classical linear logic, that satisfies
\begin{equation}
?\neg A=\neg ! A,\quad \neg ? A=!\neg A
\end{equation}
For us, just these relations are not sufficient to define it, because $\neg$ is not a bijection.\\

The Girard operator $?$ means "why not?", as the operator $!$ means "of course"; they are examples of modalities,
and correspond to the modalities more frequently denoted $\Box$ and $\diamond$ in modal logics.\\

However, Hasegawa \cite{Hasegawa2003CoherenceOT}, Moggi \cite{MOGGI199155}, Mellies and Tabareau \cite{mellies:hal-00339154} have remarked that more convenient tensorial negations must satisfy a further axiom. Note that
this story started with Kock \cite{Kock} inspired by Eilenberg and Kelly \cite{Eilenberg-Kelly}.\\
\begin{lem}
	From the second axiom of a tensorial negation it results two natural transformations
	\begin{align}
		\neg\neg A\otimes B &\rightarrow \neg\neg(A\otimes B);\\
		A\otimes \neg\neg B &\rightarrow \neg\neg (A\otimes B).
	\end{align}
	
	\noindent A monad where such maps exist in a monoidal category, is named a \emph{strong monad} \cite{Kock} and
	\cite{MOGGI199155}.\\
	The first transformation is named the \emph{strength} of the monad $T=\neg\neg$, the second one its \emph{costrength}.
\end{lem}
\begin{proof}
	Let us start with the Identity morphism of $\neg (A\otimes B)$; by the axiom, it can be interpreted as a morphism
	$B\otimes \neg (A\otimes B)\rightarrow \neg A$, then applying the functor $\neg$, we get a morphism
	\begin{equation}
		\neg\neg A\rightarrow \neg [B\otimes \neg (A\otimes B)];
	\end{equation}
	then, applying the axiom again, we obtain a natural transformation
	\begin{equation}
		\neg\neg A\otimes B\rightarrow \neg\neg (A\otimes B).
	\end{equation}
	Exchanging the roles of $A$ and $B$ gives the other transformation. \\
	Said in other terms, we have natural bijections given by the tensorial axiom, applied two times,
	\begin{multline}
		{\sf Hom}(\neg (A\otimes B),\neg (A\otimes B))\cong {\sf Hom}(\neg (A\otimes B)\otimes B,\neg A)\\
		\cong {\sf Hom}(A,\neg [B\otimes \neg (A\otimes B)]\cong {\sf Hom}(A\otimes B,\neg\neg (A\otimes B));
	\end{multline}
	and also natural bijections, obtained in the same manner,
	\begin{multline}
		{\sf Hom}(\neg (A\otimes B),\neg (A\otimes B))\cong {\sf Hom}(\neg (A\otimes B)\otimes A,\neg B)\\
		\cong {\sf Hom}(B,\neg [A\otimes \neg (A\otimes B)]\cong {\sf Hom}(A\otimes B,\neg\neg (A\otimes B));
	\end{multline}
	The identity of $\neg (A\otimes B)$ in the first term gives a natural marked point, that is also identifiable
	with $\eta_{A\otimes B}$ in the last term.\\
	On the set ${\sf Hom}((\neg (A\otimes B)\otimes B,\neg A)$ ({\em resp.} ${\sf Hom}(A\otimes \neg (A\otimes B),\neg B)$)
	we can apply the functor $\neg$; this gives a map to ${\sf Hom}(\neg\neg A,\neg [B\otimes \neg (A\otimes B)])$ ({\em resp.} ${\sf Hom}(\neg\neg B, \neg [A\otimes \neg (A\otimes B)])$),
	then the strength ({\em resp.} the costrength) after applying the second axiom.
\end{proof}

\indent The strength and costrength taken together give two \emph{a priori} different transformations $TA\otimes TB\rightarrow T(A\otimes B)$ (see \emph{ n lab cafe},
Kock, Moggi, Hazegawa).\\
The first one is the composition starting with the costrength of $TA$ followed by the strength of $B$, then ending with the product:
\begin{equation}
TA\otimes TB\rightarrow T(TA\otimes B)\rightarrow TT(A\otimes B)\rightarrow T(A\otimes B);
\end{equation}
the other one starts with the strength, then uses the costrength, and ends with the product
\begin{equation}
TA\otimes TB\rightarrow T(A\otimes TB)\rightarrow TT(A\otimes B)\rightarrow T(A\otimes B).
\end{equation}

\indent Then a third axiom was suggested by Kock in general for strong monads, and reconsidered by Hazegawa, Moggi, Mellies and Tabareau, it
consist to require that these two morphisms coincide. This is named, since Kock, a {\em commutative monad}, or a {\em monoidal monad}. We will say that
the negation itself is monoidal.\\

\noindent According to Mellies and Tabareau, Hasegawa observed that $T=\neg\neg$ is commutative, if and only if $\eta$ gives an isomorphism $\neg\cong\neg\neg$
on the objects of $\neg\mathcal{A}$, if and only if $\mu$ gives an isomorphism on the objects of $\mathcal{A}$.\\
\begin{prop}
	A necessary and sufficient condition for having $\neg$ monoidal is that for each object $A$, the
	transformation $\eta_{\neg A}$ is an equivalence from $\neg A$ and $\neg\neg\neg A$ in the category $\mathcal{A}$.
\end{prop}
\begin{cor*}
	Define $\mathcal{A}^{\eta}$ as the collection of objects $A'$ of $\mathcal{A}$, such that $\eta_{A'}$ is
	an isomporphism; in the commutative case, $\neg\mathcal{A}$ is a sub-category $\neg$ induces an equivalence of the full subcategory $\mathcal{A}^{\eta}$ of $\mathcal{A}$ with its opposite \cite[Proposition 1.31]{Bell}.
\end{cor*}

Thus we recover most of the usual properties of negation, without having a notion of false.\\

\noindent Now assume that $\mathcal{A}$ is symmetric monoidal and closed; we get natural isomorphisms
\begin{equation}
\neg (A\otimes B)\approx A\Rightarrow \neg B\approx B\Rightarrow \neg A.
\end{equation}
And using the neutral element ${\bf 1}=\star $ for $C$, and denoting $\neg {\bf 1}$ by $P$, we obtain that $\neg B=B\multimap P$.\\
\begin{prop}
	For any object $P\in \mathcal{A}$, the functor $A\mapsto (A\multimap P)=P|A$ is a tensor negation
	whose pole is $P$.
\end{prop}

\begin{proof}
	First, this is a contravariant functor in $A$.\\
	Secondly, for any pair $A,B$ in $\mathcal{A}$, using the symmetry hypothesis, we get natural bijections
	\begin{equation}
		{\sf Hom}(B,A\multimap P)\cong {\sf Hom}(B\otimes A,P)\cong {\sf Hom}(A,B\multimap P).
	\end{equation}
	This gives the basic adjunction.\\
	Third, for any triple $A,B,C$ in $\mathcal{A}$, the associativity gives
	\begin{equation}
		{\sf Hom}(A\otimes B,C\multimap P)\cong {\sf Hom}(A\otimes B\otimes C, P)\cong {\sf Hom}(A,(B\otimes C)\multimap P).
	\end{equation}
	This gives the tensorial condition.
\end{proof}

\noindent The transformation $\eta$ is given by the Yoneda lemma, from the following
natural map
\begin{equation}
{\sf Hom}(X,A)\rightarrow {\sf Hom}^{\rm op}(\neg X,\neg A)\cong {\sf Hom}(X, \neg\neg A).
\end{equation}

\noindent There is no reason for asserting that this negation is commutative.\\
From proposition \ref{prop:lin-exp}, the necessary and sufficient condition is that, for any object $A$, the following map is an
isomorphism
\begin{equation}
\eta_{A\Rightarrow P}: (A\Rightarrow P)\rightarrow (((A\Rightarrow P)\Rightarrow P)\Rightarrow P).
\end{equation}

\noindent Even  for $A=1$ this is a non-trivial condition: $P\approx ((P\Rightarrow P)\Rightarrow P)$.\\
\noindent The fact that $1\Rightarrow P\equiv P$ being obvious.\\

\noindent Choose an arbitrary object $\Delta$ and define $\neg Q$ as $Q\multimap \Delta$.
This $\Delta$ will play the role of "false".\\
We say that a theory $\mathbb{T}$ excludes $P$ if it contains $P\multimap \Delta$. This is equivalent to say that there exists $R$
in $\mathbb{T}$ such that $R\rightarrow (P\multimap \Delta)$, i.e. $R\otimes P \rightarrow \Delta$,
that is by symmetry: there exists $P\rightarrow (R\multimap \Delta)$. In particular, if $P\rightarrow R$, we obtain such a map by
composition with $R\rightarrow (R\multimap \Delta)$.\\

\noindent To localize the action of the proposition at $P$, we have to prove the following lemma:\\
\begin{lem}
	Conditioning by $Q$ such that $P\rightarrow P\otimes Q$ is non-empty, sends a theory $\mathbb{T}$ that excludes
	$P$ into a theory $\mathbb{T}$
	that also excludes $P$.
\end{lem}
\begin{proof}
	From the hypothesis we have a morphism $\neg (P\times Q)\rightarrow \neg P$, but $\neg (P\times Q)$
	is isomorphic to $Q\Rightarrow (P\Rightarrow\Delta)=(\neg P)|Q$.
\end{proof}

\noindent This is analog to the statement of Proposition \ref{prop:conditioning} in section \ref{sec:fibrations-and-cofibrations}, because in this case $P\leq Q$
is equivalent to $P=P\wedge Q$ and to $P\leq P\wedge Q$. The proof does not use that $P$ is a linear exponential object.\\

Now assume that $P$ belongs to the category $\mathcal{A}^{!}$, i.e. $P=!\Gamma$ for a given object $\Gamma\in \mathcal{A}$; we saw that the
set $\mathcal{A}_P$ of $Q$ such that $P\rightarrow Q$ forms a closed monoidal category, and by the above lemma, it acts on the set of theories
excluding $P$. That is because $P\rightarrow Q$ implies $P\rightarrow P\otimes P\rightarrow P\otimes Q$\\
\indent Therefore, all the ingredients of the information topology of chapter \ref{chap:stacks} are present in this situation.\\

\bibliographystyle{alpha}

\bibliography{./ToposStacksDNNs_arxiv-Report_V4}

\end{document}